\documentclass[11pt,a4paper]{article}

\usepackage{microtype}
\usepackage{graphicx}
\usepackage{subfigure}
\usepackage{booktabs} 
\usepackage{multirow}
\usepackage{makecell}

\usepackage{hyperref}
\usepackage{undertilde}

\usepackage{algorithm}
\usepackage{amsthm}
\usepackage{amsmath}
\usepackage{amssymb}
\usepackage{graphicx}
\usepackage{color}
\usepackage{ifpdf}
\usepackage{url}
\usepackage{algorithm}
\usepackage[usenames,dvipsnames]{xcolor}
\usepackage{paralist}

\usepackage{algorithm}
\usepackage{algorithmic}
\usepackage[T1]{fontenc} 
\usepackage[letterpaper, margin=1.2in]{geometry}
\usepackage{multicol} 
\usepackage[hang, small,labelfont=bf,up,textfont=it,up]{caption} 
\usepackage{booktabs} 
\usepackage{float} 
\theoremstyle{plain}
\newtheorem{theorem}{Theorem}[section]
\newtheorem{corollary}{Corollary}[section]
\newtheorem{lemma}{Lemma}[section]

\theoremstyle{definition}
\newtheorem{definition}{Definition}[section]
\newtheorem{assumption}{Assumption}[section]
\newtheorem{example}{Example}[section]
\newtheorem{remark}{Remark}[section]
\usepackage{todonotes}

\newcommand{\Id}{\mathbb{I}}

\newcommand{\R}{\mathbb{R}}

\newcommand{\Rext}{\R\cup\{+\infty\}}

\newcommand{\set}[1]{\left\{#1\right\}}
\newcommand{\sets}[1]{\{#1\}}
\newcommand{\norm}[1]{\left\Vert#1\right\Vert}
\newcommand{\norms}[1]{\Vert#1\Vert}

\newcommand{\Eproof}{\hfill $\square$}

\newcommand{\proj}{\mathrm{proj}}

\newcommand{\argmin}{\mathrm{arg}\!\displaystyle\min}

\newcommand{\dom}[1]{\mathrm{dom}(#1)}

\newcommand{\zero}[1]{{\boldsymbol{0}}}

\newcommand{\zer}[1]{\mathrm{zer}(#1)}

\newcommand{\gra}[1]{\mathrm{gra}(#1)}
\newcommand{\range}[1]{\mathrm{range}(#1)}
\newcommand{\mcal}[1]{\mathcal{#1}}

\newcommand{\xopt}{x^{\star}}

\newcommand{\Hc}{\mathcal{H}}

\newcommand{\Bc}{\mathcal{B}}

\newcommand{\Xc}{\mathcal{X}}

\newcommand{\Dc}{\mathcal{D}}
\newcommand{\Lc}{\mathcal{L}}
\newcommand{\Qc}{\mathcal{Q}}
\newcommand{\Gc}{\mathcal{G}}

\newcommand{\Tc}{\mathcal{T}}

\newcommand{\Vc}{\mathcal{V}}
\newcommand{\Cc}{\mathcal{C}}
\newcommand{\Ec}{\mathcal{E}}
\newcommand{\Nc}{\mathcal{N}}
\newcommand{\Pc}{\mathcal{P}}
\newcommand{\Rc}{\mathcal{R}}

\newcommand{\iprod}[1]{\left\langle #1\right\rangle}
\newcommand{\iprods}[1]{\langle #1\rangle}

\newcommand{\ri}[1]{\mathrm{ri}\left(#1\right)}

\newcommand{\kron}{\otimes}
\newcommand{\dist}[1]{\mathrm{dist}\left(#1\right)}

\newcommand{\diag}[1]{\mathrm{diag}\left(#1\right)}

\newcommand{\BigOs}[1]{\mathcal{O}\big(#1\big)}
\newcommand{\SmallOs}[1]{o\big(#1\big)}

\newcommand{\mbf}[1]{\mathbf{#1}}
\newcommand{\mbb}[1]{\mathbb{#1}}

\newcommand{\beforesubsec}{\vspace{-1.5ex}}
\newcommand{\aftersubsec}{\vspace{-1ex}}
\newcommand{\beforesec}{\vspace{-1.25ex}}
\newcommand{\aftersec}{\vspace{-1ex}}
\newcommand{\beforesubsubsec}{\vspace{-1.5ex}}
\newcommand{\aftersubsubsec}{\vspace{-1.5ex}}

\usepackage{bbm}

\title{Distributed Fast Fixed-Point Algorithms for Composite Monotone Inclusions over Networks}

\author{Nghia Nguyen-Trung$^{\star}$, Ion Necoara$^{\dagger}$, and Quoc Tran-Dinh$^{\star\ddagger}$ \vspace{0.25ex}\\
\newline {$^{\star}$Department of Statistics and Operations Research}\\
\newline The University of North Carolina at Chapel Hill\\
318 Hanes Hall, UNC-Chapel Hill, NC 27599-3260. ($^{\ddagger}$\textit{Corresponding author}).\\
\newline {$^{\dagger}$Automatic Control and Systems Engineering Department},\\ 
\newline {National University of Science and Technology Politehnica Bucharest},\\ 
\newline {060042, Bucharest, Romania.}\\
\newline \textit{Email:} \url{nghiant@unc.edu}, \url{ion.necoara@upb.ro}, \url{quoctd@email.unc.edu}.}

\date{}

\begin{document}
\maketitle

\begin{abstract}
\normalfont
This paper aims to develop new and efficient distributed algorithms for solving a class of monotone inclusions, $0 \in \sum_{i=1}^n (G_ix + T_ix)$, over a connected network of $n$ agents, where the single-valued operator $G_i$ and the possibly multivalued operator $T_i$ remain private to agent $i$. 
Existing distributed algorithms for this problem class are primarily non-accelerated, and their exact convergence rates in the original primal space are largely unexplored. 
To bridge this gap, we propose two Decentralized Fast Fixed-Point-based algorithms, \texttt{ND-DFFP} and \texttt{NI-DFFP}, which integrate Nesterov-type acceleration with primal-dual techniques under two prominent settings: (i) \textit{Lipschitz continuity of $G_i$ and maximal monotonicity of $G_i+T_i$}; and (ii) \textit{co-coercivity of $G_i$ and maximal monotonicity of $T_i$}. 
While \texttt{ND-DFFP} utilizes a homogeneous network-dependent stepsize, \texttt{NI-DFFP} reformulates the problem into a three-operator inclusion to decouple the network topology, enabling heterogeneous network-independent stepsizes. 
Under appropriate assumptions, we establish an $\mathcal{O}(1/k)$ convergence rate for the consensus error and an $\mathcal{O}(1/k)$ rate for both the restricted gap function and the squared forward-backward splitting residual, with the latter two metrics evaluated at the network-average iterate or its projection onto the effective domain.
Finally, numerical experiments on distributed bilinear matrix games and a virtual power plant problem demonstrate the competitive performance and computational efficiency of our methods over recent decentralized baselines in the literature.

\vspace{0.75ex}
\end{abstract}

\beforesec
\section{Introduction}\label{sec:intro}
\aftersec

The \textit{composite monotone inclusion} provides a powerful and unifying mathematical framework for modeling a vast range of complex problems in computational mathematics and related disciplines.
By formulating such a problem as a root-finding problem of the sum of multivalued operators, one can capture challenging structural properties such as nonsmoothness, constraints, and asymmetry commonly appearing in optimization, minimax optimization, variational inequalities, complementarity problems, game theory, and fixed-point problems within a single theoretical paradigm, see, e.g., \cite{Bauschke2011,reginaset2008,Facchinei2003,phelps2009convex,Rockafellar2004,Rockafellar1976b,ryu2016primer}.
This remarkable versatility has led to the widespread adoption of this template across diverse foundational and applied fields such as operations research, economics, uncertainty quantification, and transportation, see, e.g., \cite{Ben-Tal2009,Facchinei2003,giannessi1995variational,harker1990finite,Konnov2001}.
For instance, in modern machine learning, the composite inclusion template naturally captures robust training formulations, adversarial training, regularized empirical risk minimization, and the training dynamics of generative adversarial networks (GANs), see, e.g., \cite{arjovsky2017wasserstein,Ben-Tal2009,goodfellow2014generative,levy2020large,madry2018towards,rahimian2019distributionally}.
In game theory and economics, it serves as the mathematical backbone for identifying Nash equilibria in multi-agent environments, see, e.g., \cite{Bauschke2011,briceno2013monotone,Facchinei2003,scutari2014real}.
Furthermore, in physical sciences and engineering, composite inclusions are often utilized in signal and image recovery, optimal control, and wireless communication, see, e.g., \cite{condat2022distributed,Luo2006Analysis,Odonoghue2012a,pesquet2021learning,Yu2002Distributed}.

Despite the theoretical elegance of these formulations, traditional centralized approaches for solving composite inclusions are increasingly inadequate for modern applications.
The rapid growth in dataset sizes often renders it computationally intractable or physically impossible to store all data on a single centralized server or device.
Furthermore, in domains such as healthcare, finance, or telecommunications, strict privacy regulations and communication bottlenecks severely restrict the sharing of raw, local data to a central coordinator or to other entities in the network.
To address these critical challenges, there has been a growing shift toward decentralized architectures.
In a decentralized network, a collection of agents must efficiently coordinate and collaboratively solve a global mathematical problem by relying exclusively on local computations and peer-to-peer communication with their immediate neighbors, entirely eliminating the need for a central coordinator.

\vspace{1ex}
\noindent$\textrm{(a)}$~\textbf{Problem statement.}
In this paper, we study a \textbf{\textit{distributed composite inclusion}} over a networked system consisting of $n$ agents that communicate and cooperate to solve a common problem.
The communication network is modelled as an undirected connected graph $\Gc = (\Vc, \Ec)$, where $\Vc = [n] := \sets{1, \cdots, n}$ represents the set of computing agents and $\Ec$ represents the available communication links.
The global objective, which is more general than optimization, of this multi-agent system is to collaboratively find a consensus root of the aggregated sum of network-wide operators in the following form:
\begin{equation}\label{eq:DGE}
	\textrm{Find $x^{\star} \in \R^p$ such that:}~ 0 \in Gx^{\star} + Tx^{\star} \equiv \sum_{i=1}^n\big(G_ix^{\star} + T_ix^{\star}\big).
	\tag{DCI}
\end{equation}
Here, we denote by $\Phi := G + T$, and for each agent $i \in \Vc$, $G_i : \R^p  \to \R^p$ is a single-valued operator and $T_i : \R^p\rightrightarrows\R^p$ is a possibly multivalued operator.
Because the operators $G_i$ and $T_i$ represent private local data or structural constraints that must be kept private to agent $i$, the agents must reach a consensus solution without sharing their underlying mathematical operators.
Throughout this paper, we assume that the solution set $\zer{\Phi} := \sets{x^{\star} \in \R^p : 0 \in \Phi{x}^{\star}}$ of \eqref{eq:DGE} is nonempty.

The distributed composite inclusion formulation is sufficiently general and encapsulates several fundamental problem classes as special cases.
Here, we mention some concrete examples.

\vspace{1ex}
\noindent$\mathrm{(i)}$~\textbf{Distributed nonlinear equations.}
When $T_i = 0$, \eqref{eq:DGE} reduces to the following \textit{distributed [non]linear equation}:
\begin{equation}\label{eq:DNE}
	\textrm{Find $x^{\star} \in \R^p$ such that:}~ \sum_{i=1}^n G_ix^{\star} = 0.
	\tag{DNE}
\end{equation}
In particular, if $F_i := \Id - \lambda G_i$ for some $\lambda > 0$ and $i \in [n]$, then \eqref{eq:DNE} is equivalent to the following \textit{distributed fixed-point problem}:
\begin{equation}\label{eq:DFP}
	\textrm{Find $x^{\star} \in \R^p$ such that:}~ x^{\star} =  \frac{1}{n}\sum_{i=1}^n F_ix^{\star}.
	\tag{DFP}
\end{equation}
These formulations naturally arise in networked multi-agent systems and federated machine learning, see, e.g., \cite{chraibi2019distributed, kovalev2020optimal,malinovsky2020local,nguyen2025class}.

\vspace{1ex}
\noindent$\mathrm{(ii)}$~\textbf{Distributed convex optimization problems.}
This mathematical formulation is perhaps the most attractive problem class in the multi-agent community, which considers a network of agents cooperatively minimizing a common additive convex objective function:
\begin{equation}\label{eq:DO}\tag{DOP}
	\text{Find } x^{\star} \in \argmin_{x \in \R^p} \Big\{ f(x) := \sum_{i=1}^n \big( f_i(x) + g_i(x) \big) \Big\}.
\end{equation}
Here, $f_i : \R^p \to \R \cup \sets{+\infty}$ is a local convex objective function with $L$-Lipschitz gradient of agent $i$ and $g_i : \R^p \to \R \cup \sets{+\infty}$ is a nonsmooth component (e.g., the regularizer or constraints) of agent $i$, which must be kept private from all other agents (see, e.g., \cite{chen2012fast,Jakovetic2011,Nedic2009,qu2018harnessing,shi2015extra}).
If $f_i$ is differentiable and $g_i$ is a proper, closed, and convex function, $i \in [n]$, then under appropriate qualification conditions, the common goal is to approximate a solution of \eqref{eq:DO} satisfying its optimality condition: $0 \in \sum_{i=1}^n (G_i x^{\star} + T_i x^{\star})$, where $G_i := \nabla f_i$, the gradient of $f_i$, and $T_i := \partial g_i$, the subdifferential of $g_i$, respectively.
This is a special case of \eqref{eq:DGE}.

\vspace{1ex}
\noindent$\mathrm{(iii)}$~\textbf{Distributed variational inequalities.}
If $T_i = \partial g_i$, the subdifferential of a proper, closed, and convex function $g_i$, then \eqref{eq:DGE} reduces to a \textit{distributed mixed variational inequality problem} (MVIP):
\begin{equation}\label{eq:DMVIP}
	\textrm{Find $x^{\star} \in \R^p$ such that:}~ \Big\langle{\sum_{i=1}^n G_ix^{\star}, x - x^{\star}}\Big\rangle + \sum_{i=1}^n \left(g_i(x) - g_i(x^{\star})\right) \geq 0, \ \forall x \in \R^p.
	\tag{DMVIP}
\end{equation}
In particular, if $T_i = \Nc_{\Xc_i}$, the normal cone of the convex set $\Xc_i$ (i.e., $g_i := \delta_{\Xc_i}$, the indicator function of $\Xc_i$), then, under appropriate constraint qualification conditions, \eqref{eq:DMVIP} recovers the following \textit{distributed variational inequality problem} (see, e.g., \cite{beznosikov2022decentralized,beznosikov2023similarity}):
\begin{equation}\label{eq:DVIP}
	\textrm{Find $x^{\star} \in \Xc$ such that:}~ \Big\langle{\sum_{i=1}^n G_ix^{\star}, x - x^{\star}}\Big\rangle \geq 0, \quad\textrm{for all} \ x \in \Xc := \bigcap_{i=1}^n \Xc_i.
	\tag{DVIP}
\end{equation}
These distributed variational inequality templates cover a broad spectrum of fundamental problems, from constrained convex minimization and saddle-point problems to complex multi-agent network games, including the computation of generalized Nash equilibria across distributed systems.

\vspace{1ex}
\noindent\textrm{(iv)}~\textbf{Distributed convex-concave minimax problem.}
This problem considers a network of $n$ agents working cooperatively to solve a convex-concave minimax problem of the form:
\begin{equation}\label{eq:Dminmax}
	\min_{u \in \R^{p_1}} \max_{v \in \R^{p_2}} \Big\{\Phi(u, v) :=  \sum_{i=1}^n \set{f_i(u) + \Lc_i(u,v) - g_i(v)} \Big\},
\end{equation}
where $f_i : \R^{p_1} \to \R \cup \sets{+\infty}$ and $g_i : \R^{p_2} \to \R \cup \sets{+\infty}$ are proper, closed, and convex functions, and $\Lc_i : \R^{p_1} \times \R^{p_2} \to \R$ is a convex-concave saddle function, often assumed differentiable.
All three functions are known only to node $i$ (see, e.g., \cite{malitsky2026firstorder,sharma2022federated,tam2025decentralised,xian2021faster,Yang2019cooperative}).
If we define $x := [u,v] \in \R^p$ with $p := p_1 + p_2$, $G_ix := [\nabla_u \Lc_i(u,v), -\nabla_v \Lc_i(u,v)]$, and $T_ix := [\partial_u f_i(u), \partial_v g_i(v)]$, then the optimality condition of \eqref{eq:Dminmax} can be written exactly as a special case of \eqref{eq:DGE}.

\vspace{1ex}
\noindent$\textrm{(b)}$~\textbf{Related work.}
Most existing results in the literature are devoted to distributed [un]constrained optimization problems.
These methods typically allow each agent $i$ to maintain a local copy of the common decision variable, with the goal of driving all local copies to an optimal consensus state.
A fundamental approach is the distributed [sub]gradient descent (DGD) scheme \cite{chen2012fast,Jakovetic2011,Nedic2009}, which intertwines local gradient steps with a consensus mixing step over the network.
However, it is well-known that when a constant stepsize is utilized, the iterates generated by DGD only converge to a neighborhood of the exact solution, achieving a sublinear $\BigOs{1/k}$ convergence rate on the optimality gap, see, e.g., \cite{shi2015extra}.
While exact convergence can be recovered by employing diminishing stepsizes, this approach inevitably degrades the practical convergence speed.

To achieve exact convergence with fixed stepsizes, a major development is EXTRA and its variants \cite{shi2015extra,shi2015proximal}.
These methods modify the standard DGD scheme by utilizing two different weight matrices and performing updates based on the difference between two consecutive historical gradients, which effectively cancels out the steady-state error.
Another dominant direction for achieving exact convergence is the gradient tracking (GT) method \cite{dilorenzo2015distributed,dilorenzo2016next,nedic2017achieving,qu2018harnessing}.
Motivated by the discrete-time dynamic average consensus problem \cite{zhu2010discrete}, GT algorithms augment the local state with a tracking variable  designed to dynamically estimate the global average gradient.
By carefully injecting the difference of consecutive local gradients as a correction term, GT ensures that the sum of the tracking estimators perfectly matches the sum of the true gradients at every iteration.
Because the tracking direction eventually converges to the true average gradient, these methods successfully achieve exact linear convergence rates for strongly convex objective functions.

While decentralized optimization is extensively studied, extending these techniques to distributed minimax optimization and composite inclusions introduces significant structural challenges.
Unlike standard optimization, monotone operators, particularly those arising in minimax games and saddle-point problems, often exhibit asymmetric Jacobians, making direct adaptation of optimization methods inapplicable.
Early works in this domain primarily focused on distributed projection and subgradient methods for solving saddle-point problems over networks \cite{mateos2015distributed,mateos2017distributed}.
More recently, driven by the massive demand for several modern machine learning applications such as adversarial training, GANs, or robust multi-agent reinforcement learning, decentralized minimax optimization has garnered significant attention.
For instance, \cite{malitsky2026firstorder} has successfully adapted advanced operator splitting techniques, combining an EXTRA-type consensus scheme with the forward-reflected-backward splitting method to obtain the following algorithm for monotone inclusions:
\begin{equation}\label{eq:malitsky2026_alg}\tag{ND-DFRB}
\arraycolsep=0.2em
\left\{
\begin{array}{lcl}
	\mbf{v}^k &=& 2\mbf{G}\mbf{x}^k - \mbf{G}\mbf{x}^{k-1}, \vspace{1ex}\\
	
	\mbf{u}^{k+1} &=& \mbf{u}^k + W\mbf{x}^{k} - \frac{1}{2}(\Id + W)\mbf{x}^{k-1} - \tau (\mbf{v}^k - \mbf{v}^{k-1}), \vspace{1ex}\\
	
	\mbf{x}^{k+1} &=& J_{\tau \mbf{T}}(\mbf{u}^{k+1}),
\end{array}
\right.
\end{equation}
where $W$ is the mixing matrix (see Definition~\ref{def:mixing_matrix}), $\tau \in \left(0, \frac{1 + \lambda_{\min}(W)}{4\max_i L_i}\right)$ in which $\lambda_{\min}(W)$ denotes the minimum eigenvalue of $W$, and $\mbf{G}$ and $\mbf{T}$ are defined as in \eqref{eq:GE_notation}.

This method relaxes the co-coercivity requirement on the $G_i$ operators to monotonicity and $L_i$-Lipschitz continuity.
However, a limitation of this method is its reliance on a fixed, homogeneous stepsize across all agents, which is tightly coupled with the network topology through $\lambda_{\min}(W)$.
This coupling not only forces all agents to adopt a uniform, conservative stepsize restricted by the worst local operator, but it also requires global knowledge of the graph's spectral properties prior to execution.

To overcome these severe coordination bottlenecks, a notable theoretical breakthrough was recently achieved by developing a decentralized forward-backward-type algorithm that permits heterogeneous local stepsizes \cite{tam2025decentralised}.
In this framework, agents can independently select their fixed stepsizes $\alpha_i < \frac{1}{8L_i}$ with upper bounds that have absolutely no dependence on the global communication graph.
The update scheme of this method can be written as
\begin{equation}\label{eq:tam2026_alg}\tag{NI-DFB}
\arraycolsep=0.2em
\left\{
\begin{array}{lcl}
	\mbf{v}^k &=& 2\mbf{G}\mbf{y}^k - \mbf{G}\mbf{y}^{k-1}, \vspace{1ex}\\
	\mbf{z}^{k+1} &=& \mbf{z}^k - \mbf{x}^k + \widetilde{\mbf{W}}(2\mbf{x}^k - \mbf{x}^{k-1} - \boldsymbol{\Lambda}(\mbf{v}^k - \mbf{v}^{k-1})), \vspace{1ex}\\
	\mbf{x}^{k+1} &=& J_{\boldsymbol{\Lambda} \mbf{T}}(\mbf{z}^{k+1}), \vspace{1ex}\\
	\mbf{y}^{k+1} &=& \mbf{x}^k + \mbf{z}^{k+1} - \mbf{z}^k,
\end{array}
\right.
\end{equation}
where $\Lambda := \diag{\alpha_1, \cdots, \alpha_n}$ , $\boldsymbol{\Lambda} := \Lambda \kron \Id$, $\widetilde{\mbf{W}} := (\Id - \frac{\beta}{2}\Lambda(\Id - W)) \kron \Id$ for some parameter $\beta$ such that $0 < \beta < \norms{\Lambda^{1/2}\left(\frac{\Id - W}{2}\right)\Lambda^{1/2}}^{-1}$.

However, despite the significant progress in eliminating network dependencies, the literature on \textbf{accelerated} exact methods for distributed composite inclusions remains surprisingly limited. 
Most existing exact decentralized algorithms for monotone inclusions, including \eqref{eq:malitsky2026_alg} and \eqref{eq:tam2026_alg}, are fundamentally built upon \textbf{non-accelerated} operator-splitting techniques. 
As a result, their best-known worst-case convergence rate for the squared residual norm is typically no better than $\BigOs{1/k}$, where $k$ denotes the number of iterations. 
By contrast, accelerated methods have the potential to improve this rate to $\BigOs{1/k^2}$, thereby substantially reducing the number of communication rounds required to achieve a desired accuracy.

Beyond computational efficiency, the current theoretical analysis remains incomplete. 
Existing convergence results are largely restricted to asymptotic guarantees or convergence rates established in reformulated primal-dual spaces. 
Explicit non-asymptotic convergence guarantees for the original primal iterates are still largely unavailable. These observations naturally motivate the following two research questions:
\begin{center}
\textit{\textbf{
\begin{enumerate}
\item How can we design distributed accelerated algorithms for efficiently solving the distributed monotone composite inclusion problem \eqref{eq:DGE}?
\item Can we establish explicit non-asymptotic convergence rates in the original primal space, rather than only in an equivalent primal-dual formulation?
\end{enumerate}
}}
\end{center}
This paper addresses both questions by developing two new decentralized accelerated fixed-point methods based on the primal-dual frameworks of \cite{malitsky2026firstorder} and \cite{tam2025decentralised}. 
We establish the convergence of the generated iterate sequences and derive explicit non-asymptotic convergence rates directly in the original primal space.

\vspace{1ex}
\noindent$\mathrm{(c)}$~\textbf{Our contributions.}
Our contributions in this paper can be summarized as follows:
\begin{compactitem}
	\item[$\mathrm{(i)}$] 
	First, we develop two centralized Fast Fixed-Point (FFP) algorithms for solving \eqref{eq:DGE}, without yet exploiting its distributed structure.
	These algorithms are new and serve as the base methods for deriving the distributed algorithms presented later.
	The first algorithm addresses \eqref{eq:DGE} under the Lipschitz continuity of $G$ and the maximal monotonicity of $G+T$.
	The second algorithm considers a different setting in which $G$ is co-coercive and $T$ is co-hypomonotone.
	These two settings overlap, but neither is a special case of the other.
	For both algorithms, we rigorously establish an $\BigOs{1/k^2}$ convergence rate for the squared norm of the operator residual.
	Furthermore, using a primal-dual approach, we extend this framework to the three-operator inclusion \eqref{eq:3op_inclusion}, which provides the theoretical foundation for developing our decentralized algorithms in the sequel.
	
	\item[$\mathrm{(ii)}$]
	Second, we reformulate \eqref{eq:DGE} as a standard composite inclusion \eqref{eq:GE} in a primal-dual space. 
	By applying the proposed FFP framework to this reformulation, we derive a novel decentralized algorithm, denoted as \texttt{ND-DFFP}, to solve \eqref{eq:DGE}.
	We then establish convergence of the original primal iterates and derive explicit non-asymptotic rates for consensus error and aggregated operator residuals, including $\SmallOs{1/k}$ rates in norm and $\SmallOs{1/k^2}$ rates for their squared counterparts.
	We can also prove an $\SmallOs{1/k}$ rate for the restricted dual gap function and for the squared forward-backward splitting residual.
	
	\item[$\mathrm{(iii)}$]
	Third, we employ an alternative reformulation that transforms \eqref{eq:DGE} into a three-operator inclusion \eqref{eq:3op_inclusion}.
	This critical transformation successfully decouples the local problem operators from the network communication information.
	By applying the three-operator variant of FFP to this decoupled structure, we obtain \texttt{NI-DFFP}, a fundamentally new decentralized algorithm.
	In this scheme, each agent can independently utilize a heterogeneous stepsize that is completely free from any dependence on the network topology.
	Finally, we demonstrate that this network-independent algorithm simultaneously achieves an exact $\BigOs{1/k^2}$ convergence rate in the original primal space. 
\end{compactitem}
Let us further elaborate on our contributions and highlight the key aspects that distinguish our work from the existing literature.

First, the centralized algorithms developed in $\mathrm{(i)}$ differ from existing methods, including those in our recent works \cite{TranDinh2025a,tran2025accelerated}.
For instance, the second and last steps of \eqref{eq:FKM_scheme2} differ from the method in \cite[Section 5.1]{tran2025accelerated}, particularly through the stepsize $\eta\gamma_k-\beta_k$ and the use of $\hat{w}^{k+1}$ instead of $\hat{w}^k$ in the last step.
These modifications fundamentally change the underlying convergence analysis.
Similarly, the main step in the second line of \eqref{eq:FKM2_scheme} differs from the scheme in \cite{TranDinh2025a}, where $G$ is evaluated at the intermediate state $\hat{x}^k$ instead of the primal iterate $x^k$.
This seemingly small modification again requires a substantially different analysis.

Second, to the best of our knowledge, our distributed methods are the first accelerated algorithms for solving \eqref{eq:DGE}, achieving faster convergence rates than existing non-accelerated methods such as \cite{malitsky2026firstorder,tam2025decentralised}.
We emphasize that developing accelerated methods in this setting is nontrivial, since much of the existing analysis for accelerated methods in convex optimization does not directly extend to monotone inclusions due to the absence of objective functions.

Third, unlike \cite{malitsky2026firstorder,tam2025decentralised}, we establish not only convergence guarantees for the reformulated problem, but also convergence rates and complexity guarantees for the original problem \eqref{eq:DGE}.

Finally, since our methods address the general problem \eqref{eq:DGE}, they can be specialized to obtain suitable variants for monotone variational inequalities and convex-concave saddle-point problems.
In particular, some of these specializations can lead to new distributed primal-dual variants for solving convex-concave minimax problems as well as convex optimization.

\vspace{1ex}
\noindent$\mathrm{(d)}$ \textbf{Paper organization.}
The remainder of this paper is organized as follows.
Section~\ref{sec:background} presents some necessary mathematical background used in this paper.
Section~\ref{sec:FKM} develops two new centralized algorithms for \eqref{eq:DGE}, which serve as the base methods for developing distributed algorithms in the sequel. 
Section~\ref{sec:ND_DFFP_alg} develops our first class of distributed algorithms with network-dependent stepsizes.
Section~\ref{sec:NI-DFKM} derives a novel class of distributed FFP algorithms with network-independent stepsizes.
In Section~\ref{sec:numerical_experiments}, we implement our proposed algorithms to evaluate their efficacy and compare them with some state-of-the-art benchmarks.
All the technical proofs of the results in the main text are given in the appendix.

\beforesec
\section{Mathematical Background and Preliminaries}\label{sec:background}
\aftersec
In this section, we briefly recall necessary notions and concepts which will be used in the sequel.
Then, we recall the restricted gap function and its properties for our model \eqref{eq:DGE}.

\beforesubsec
\subsection{Notation and mathematical tools}\label{subsec:sec2_math_background}
\aftersubsec
\noindent$\mathrm{(a)}$~\textbf{Notations and basic concepts.}
Let $\R^p$ be a finite-dimensional Euclidean space equipped with an inner product $\iprods{\cdot, \cdot}$ and a norm $\norms{\cdot}$.
Let $u$ be a vector in $\R^p$, and $u_{[i]} \in \R^p$ be the local vector of agent $i \in [n] := \sets{1,\cdots, n}$ of a distributed system.
For $u_{[i]} \in \R^p$ (the local copy of $u$ at node $i$), and given the operators $G_i$ and $T_i$ in \eqref{eq:DGE}, $i = 1, \cdots, n$, let us define an element and operators in the product space $\R^{n \times p}$ as
\begin{equation}\label{eq:GE_notation}
\arraycolsep=0.2em
\begin{array}{c}
	\mbf{u} = (u_{[1]}, \cdots, u_{[n]})^\top = \begin{bmatrix} u_{[1]}^{\top} \\ \vdots \\ u_{[n]}^{\top}\end{bmatrix} \in \R^{n\times p}, \vspace{1ex}\\
	\mbf{G}\mbf{u} =\left( G_1u_{[1]}, \cdots, G_nu_{[n]}\right) \qquad\textrm{and} \qquad
	\mbf{T}\mbf{u} = T_1u_{[1]} \times \cdots\times T_nu_{[n]}.
\end{array}
\end{equation}
Here, $T_1u_{[1]}\times\cdots\times T_nu_{[n]}$ represents the Cartesian product mapping.
We define
\begin{equation}\label{eq:consensus_space}
	\Dc := \sets{\mbf{u} \in \R^{n\times p} : u_{[1]} = \cdots = u_{[n]}},
\end{equation}
the linear subspace in the product space $\R^{n\times p}$, often called the \textit{consensus space} as all rows of a matrix in this space are identical.

For a single-valued or multivalued mapping $T : \R^p \rightrightarrows \R^p$, we denote by $\dom{T} := \set{x \in \R^p : Tx \neq\emptyset}$ the domain of $T$ and by $\gra{T} := \sets{(x, u) \in \R^p\times \R^p : u \in Tx}$ the graph of $T$.
The inverse mapping $T^{-1}$ of $T$ is defined as $T^{-1}u := \sets{x \in \R^p : u \in Tx}$.
The resolvent of $T$ is defined by $J_Tu := \set{ x \in \R^p : u \in x + Tx} = (\Id + T)^{-1}u$.
We say that $T$ is \emph{closed} if $\gra{T}$ is closed.
For a nonempty, closed, and convex set $\Xc$ in $\R^p$, we denote by $\Nc_{\Xc}$ the normal cone of $\Xc$.
We use $\mathrm{ri}(\Xc)$ and $\mathrm{int}(\Xc)$ to denote the relative interior and the interior of $\Xc$, respectively.
For a proper, closed, and convex function $f : \R^p\to\Rext$, $\dom{f} := \sets{x \in \R^p : f(x) < +\infty}$ denotes the domain of $f$, $\partial{f}$ denotes the subdifferential of $f$, and $\nabla{f}$ stands for the [sub]gradient of $f$.
We also use $\R_{+}$ and $\R_{++}$ to denote the set of nonnegative and positive real numbers, respectively.
Given two functions $g, h : [\underline{t}, +\infty) \to \R$, we write $g(t) = \BigOs{h(t)}$ if there exist constants $M > 0$ and $t_0 \geq \underline{t} > 0$ such that $g(t) \leq Mh(t)$ for all $t \geq t_0$.
We write $g(t) = o(h(t))$ if $\lim_{t\to +\infty}g(t)/h(t) = 0$.

For a linear mapping $X : \R^p\to\R^p$, $\mathrm{Null}(X) := \sets{u \in \R^p : Xu = 0}$ denotes its null space and $\norms{X}$ denotes its spectral norm.
The notation $X \succeq 0$ (resp., $X \succ 0$) means $X$ is positive semidefinite (resp., positive definite).
For two linear mappings $X$ and $Y$, we say that $X \succeq Y$ (resp., $X \succ Y$) if $X - Y \succeq 0$ (resp., $X-Y\succ 0$).
For a symmetric matrix $X$ in $\R^{p\times p}$, $\lambda_{\max}(X)$ and $\lambda_{\min}(X)$ denote its largest and smallest eigenvalues, respectively.
For a symmetric positive definite matrix $P \in \R^{p\times p}$, we define the corresponding weighted inner product $\iprods{x, y}_P := \iprods{Px, y}$ and  weighted norm $\norms{x}_P := \iprods{Px, x}^{1/2}$ for any $x, y \in \R^p$. 

\vspace{0.75ex}
\noindent$\mathrm{(b)}$~\textbf{Monotone and Lipschitz continuous operators.}
A mapping $T: \R^p \rightrightarrows \R^p$ is called $\rho$-co-monotone if there exists $\rho \in \R$ such that $\iprods{u - v, x - y} \geq \rho \norms{u - v}^2$ for all $(x, u), (y, v) \in \gra{T}$.
When $T$ is single-valued, the last inequality becomes $\iprods{Tx - Ty, x - y} \geq \rho \norms{Tx - Ty}^2$ for all $x, y \in \dom{T}$.
If $\rho > 0$, then $T$ is called $\rho$-co-coercive.
If $\rho = 0$, then $T$ is said to be monotone.
If $\rho < 0$, then $T$ is called $\vert\rho\vert$-co-hypomonotone.
A mapping $T$ is called maximally $\rho$-[co-]monotone if its graph is not contained in the graph of any other $\rho$-[co-]monotone operator.
Clearly, $T$ is $\mu$-monotone iff $T^{-1}$ is $\mu$-co-monotone.

A single-valued operator $G: \R^p \to \R^p$ is $L$-Lipschitz continuous if $\norms{Gx - Gy} \leq L \norms{x - y}$ for all $x, y \in \R^p$.
By the triangle inequality, we can prove that if each $G_i: \R^p \to \R^p$ is $L_i$-Lipschitz continuous, then $\sum_{i=1}^{n} G_i$ is $L$-Lipschitz continuous with $L = \sum_{i=1}^n L_i$.
Similarly, if each $G_i: \R^p \to \R^p$ is $\frac{1}{L_i}$-co-coercive, then $\sum_{i=1}^{n} G_i$ is $\frac{1}{L}$-co-coercive with $L = \sum_{i=1}^n L_i$.

\vspace{0.75ex}
\noindent$\mathrm{(c)}$~\textbf{Mixing matrices.}
The agents in a given network communicate to exchange information, where the communication is characterized by a \textit{mixing matrix}, defined as follows.
\begin{definition}\label{def:mixing_matrix}
	Given an undirected connected graph $\Gc = (\Vc, \Ec)$, where $\Vc = \sets{1, 2, \cdots, n}$ is the set of vertices and $\Ec$ is the set of edges.
	A matrix $\mbf{W} = (W_{ij}) \in \R^{n \times n}$ is called a mixing matrix if it satisfies the following four conditions:
	\begin{compactenum}[(1)]
		\item $W_{ij} = 0$ for every $i \ne j$ with $(i, j) \notin \Ec$.
		\item $\mbf{W}$ is symmetric.
		\item $\operatorname{Null}\big((\Id_n -\mbf{W})\otimes \Id_p\big) = \mathcal D$, where $\Id_n$ is the $n \times n$ identity matrix.
		\item $\Id_n \succeq \mbf{W} \succ -\Id_n$.
	\end{compactenum}
\end{definition}
\begin{example}
	We recall here some standard constructions of mixing matrices satisfying Definition~\ref{def:mixing_matrix}.
	These matrices have been widely used in the literature, including \cite{chen2012fast,Jakovetic2011,mateos2015distributed,mateos2017distributed,Nedic2009,shi2015extra,shi2015proximal}.
	\begin{compactitem}
		\item[(i)] Symmetric doubly stochastic matrix: 
		Let $\mbf{W} = \mbf{W}^\top$ be a nonnegative doubly stochastic
		matrix compatible with $\Gc$, i.e.,
		$W_{ij}=0$ whenever $i\neq j$ and $(i,j) \notin \Ec$.
		If its support graph is connected and $W_{ii} > 0$ for all $i\in[n]$, then $\mbf{W}$ satisfies Definition~\ref{def:mixing_matrix}.
		
		\item[(ii)] Laplacian-based matrix: $\mbf{W} = \Id - \frac{\Lc}{\gamma}$, where $\Lc$ is the Laplacian matrix of the graph $\Gc$ and $\gamma > \frac{1}{2}\lambda_{\max}(\Lc)$ is a scaling parameter.
		Let $\Vc_i := \sets{j \in \Vc : (i, j) \in \Ec}$ be the neighbor set of node $i$ and $d_i = \vert \Vc_i \vert$ denote the degree of node $i$.
		When $\lambda_{\max}(\Lc)$ is unavailable, $\gamma = \max_{i \in \Vc}\sets{d_i} + \epsilon$ for some small $\epsilon > 0$ (e.g., $\epsilon = 1$) can be used.
		
		\item[(iii)] Metropolis constant edge weight matrix:
		\begin{equation*}
			\arraycolsep=0.2em
			W_{ij} = \left\{ 
			\begin{array}{ll}
				\frac{1}{\epsilon + \max\sets{d_i, d_j}} & \text{if } (i,j) \in \Ec \ \text{and} \ i \neq j \vspace{1ex}\\
				1 - \sum_{k \in \Nc_i} \frac{1}{\epsilon + \max\sets{d_i, d_k}} \quad & \text{if } i = j, \vspace{1ex}\\
				0 & \text{otherwise},
			\end{array}
			\right.
		\end{equation*}
		for some small $\epsilon > 0$ (e.g., $\epsilon = 1$), where $\Vc_i$ and $d_i$ are defined in (ii).
	\end{compactitem}
\end{example}

\beforesec
\subsection{Restricted gap function and its properties}\label{subsec:sec2_restricted_gap_func}
\aftersec
To characterize the properties of the restricted gap function defined below, we impose the following assumption on \eqref{eq:DGE}.

\begin{assumption}[\textbf{Restricted gap regularity}]\label{as:restricted_dual_gap}
	For \eqref{eq:DGE}, there exists an interior solution $u^{\dagger}$ of \eqref{eq:DGE}, i.e., $u^{\dagger} \in \zer{\Phi} \cap \bigcap_{i=1}^n \mathrm{int}(\dom{T_i})$.
\end{assumption}

\begin{assumption}[\textbf{Boundary regularity}]\label{as:boundary_regularity}
	For \eqref{eq:DGE}, the following conditions hold:
	\begin{compactitem}
	\item[$\mathrm{(i)}$] $\bigcap_{i=1}^n \dom{T_i}$ is bounded and $\bigcap_{i=1}^n \ri{\dom{T_i}} \ne \emptyset$.

	\item[$\mathrm{(ii)}$] There exist nonempty closed convex sets $\Cc_i \subseteq \R^p$ and a mapping $S_i: \Cc_i \rightrightarrows \R^p$ with $\dom{S_i} = \Cc_i$, $i = 1, \cdots, n$, such that $T_i = S_i + \Nc_{\Cc_i}$.
	Moreover, $S_i$ satisfies
	\begin{equation}\label{eq:S_bounded}
		M_S := \max_{i=1, \cdots, n} \sup\Big\{ \norms{s_i} : u \in \bigcap_{i=1}^n \Cc_i, \ s_i \in S_iu \Big\} < +\infty.
	\end{equation}
	\end{compactitem}
\end{assumption}
Assumption~\ref{as:restricted_dual_gap} is a standard interiority condition and is automatically satisfied whenever all $T_i$ have full domain as long as the solution set $\zer{\Phi}$ is nonempty.
The first condition in Assumption~\ref{as:boundary_regularity} accommodates constrained problems over compact convex sets (e.g., variational inequalities, matrix games, etc.), whose solutions may lie on the boundary of the common domain, including lower-dimensional compact feasible set such as simplices.
In this situation, Assumption~\ref{as:boundary_regularity}(ii) serves as a boundary regularity to handle the unboundedness of the set-valued operators $T_i$ on the boundary of its domain.
It is automatically satisfied for variational inequalities and matrix games with $S_i = 0$, and more generally, for constrained composite optimization problems with $T_i = \partial g_i + \Nc_{\Cc_i}$, whenever $\Cc_i \subseteq \dom{\partial g_i}$ and $\partial g_i$ is uniformly bounded on the common domain.
Note that we only require these assumptions when we characterize the convergence of our method through the restricted gap function.

Motivated by \cite{Nesterov2007a}, for a given nonempty subset $\Bc \subseteq \bigcap_{i=1}^n \dom{T_i}$, let us define the restricted dual (or Minty) gap function $\mathrm{Gap}_{\Bc}(\cdot)$ for \eqref{eq:DGE} as follows:
\begin{equation}\label{eq:res_gap_func}
	\mathrm{Gap}_{\Bc}(x) := \sup_{u \in \Bc}\Big\{  \sum_{i=1}^n \iprods{G_i u + \xi_i, x - u}, \quad\text{subject to} \quad \xi_i \in T_iu, \ i \in [n] \Big\}.
\end{equation}
The following lemma states the characterizations of $\mathrm{Gap}_{\Bc}$, whose proof is given in Appendix~\ref{apdx:le:gap_function}.

\begin{lemma}\label{le:gap_function}
	For \eqref{eq:DGE}, suppose that $\Phi_i := G_i + T_i$ is maximally monotone.
	Furthermore, suppose that the set $\Bc$ is chosen as follows:
	\begin{compactitem}
		\item Under Assumption~\ref{as:restricted_dual_gap}, let $\Bc$ be  a nonempty compact convex set satisfying $u^{\dagger} \in \mathrm{int}(\Bc)$ and $\Bc \subset \bigcap_{i=1}^n \mathrm{int}(\dom{T_i})$.
		
		\item Under Assumption~\ref{as:boundary_regularity}$\mathrm{(i)}$, let $\Bc := \bigcap_{i=1}^n \dom{T_i}$.
	\end{compactitem}
	Then, 	$\mathrm{Gap}_{\Bc}(\cdot)$ defined by \eqref{eq:res_gap_func} has the following properties:
	\begin{compactitem}
		\item[$\mathrm{(i)}$] $\mathrm{Gap}_{\Bc}$ is well-defined on $\dom{\Phi}$. 
		Moreover, $\mathrm{Gap}_{\Bc}(u) \geq 0$ for all $u \in \dom{\Phi}$;
		\item[$\mathrm{(ii)}$]
		If $u^{\star} \in \zer{\Phi}$, then $\mathrm{Gap}_{\Bc}(u^{\star}) = 0$.
		\item[$\mathrm{(iii)}$]
		If $u^{\star} \in \Bc$ and $\mathrm{Gap}_{\Bc}(u^{\star}) = 0$, then $u^{\star} \in \zer{\Phi}$.
	\end{compactitem}
\end{lemma}

For a given tolerance $\epsilon > 0$, from Lemma~\ref{le:gap_function}, since $\mathrm{Gap}_{\Bc}(x) \geq 0$ for all $x\in\dom{\Phi}$, we can  state that if $\tilde{x}^{*} \in \Bc$ and $\mathrm{Gap}_{\Bc}(\tilde{x}^{*}) \leq \epsilon$, then $\tilde{x}^{*}$ is an $\epsilon$-solution of \eqref{eq:DGE}.
Due to the construction of $\Bc$, it is clear that $\mathrm{Gap}_{\Bc}$ is a restricted merit function under Assumption~\ref{as:restricted_dual_gap} and a global merit function under Assumption~\ref{as:boundary_regularity}(i), thus it can be used to characterize convergence and convergence rates of algorithms.

\beforesec
\section{New Centralized Fast Fixed-Point-Based Methods}\label{sec:FKM}
\aftersec
To develop distributed algorithms for solving \eqref{eq:DGE}, we first derive two new centralized fast fixed-point methods to solve a composite monotone inclusion, when we do not take into account its distributed structure.
Our algorithm relies on Nesterov's acceleration techniques \cite{Nesterov2004}.

For this purpose, we first state the composite monotone inclusion as follows:
\begin{equation}\label{eq:GE}
	\textrm{Find $x^{\star} \in \R^p$ such that: }~ 0 \in \Phi x^{\star}, \quad \text{where} \quad \Phi x := Gx + Tx.
	\tag{CI}
\end{equation}
Here, $G: \R^p \to \R^p$ is single-valued and $T: \R^p \rightrightarrows \R^p$ is [possibly] multivalued.
Let $\zer{\Phi} := \sets{x^{\star} \in \R^p : 0 \in \Phi{x}^{\star}}$ be the solution set of \eqref{eq:GE}.
Throughout this section, we assume that $\zer{\Phi} \neq \emptyset$ (i.e., \eqref{eq:GE} has a solution).

\beforesubsec
\subsection{The derivation of the algorithm}\label{subsec:alg1_derivation}
\aftersubsec
\noindent\textbf{$\mathrm{(a)}$~Proposed algorithm.}
We propose the following algorithm to solve \eqref{eq:GE}.
Our algorithm mimics Nesterov's acceleration technique \cite{Nesterov2004} and Popov's past-extragradient method \cite{popov1980modification}.
It is presented as follows:
Starting from an initial point $x^0 \in \dom{T}$, we set $y^{-1} = z^0 = x^0$, and at each iteration $k \geq 0$, given $\xi^k \in Tx^k$, we update
\begin{equation}\label{eq:FKM_scheme}
\tag{FFP}
\arraycolsep=0.2em
\left\{\begin{array}{lcl}
	\hat{x}^k &=& \frac{t_k - r}{t_k} x^k + \frac{r}{t_k} z^k \vspace{1ex}\\
	y^k &=& \hat{x}^k - (\eta \gamma_k - \beta_k)(Gy^{k-1} + \xi^k) \vspace{1ex}\\
	x^{k+1} &=& \hat{x}^k - \eta (Gy^k + \xi^{k+1}) + \beta_k(Gy^{k-1} + \xi^k) \vspace{1ex}\\
	z^{k+1} &=& z^k - \frac{\nu_k}{r}(Gy^k + \xi^{k+1}),
\end{array}
\right.
\end{equation}
where $r > 0$, $\eta > 0$, $t_k > 0$, $\gamma_k > 0$, $\beta_k > 0$, and $\nu_k > 0$ are given parameters, determined later.
\begin{compactitem}
\item 
The first and last lines of \eqref{eq:FKM_scheme} implement a Nesterov-type acceleration mechanism as often seen in convex optimization, see, e.g., \cite{Nesterov2004}.
One major difference is that $\nu_k = \BigOs{1}$ in the last line instead of $\nu_k = \BigOs{t_k}$ as in convex optimization.

\item 
The second and third lines represent Popov's past-extragradient update \cite{popov1980modification}, a variant of the well-known extragradient method.
\end{compactitem}
For convenience of our analysis, given $\xi^k \in Tx^k$, we define the following quantities
\begin{equation}\label{eq:FKM_w_quantities}
\arraycolsep=0.2em
\begin{array}{lcl}
	w^k := Gx^k + \xi^k, \quad \hat{w}^k := Gy^{k-1} + \xi^k, \quad \text{and} \quad e^k := Gy^{k-1} - Gx^k = \hat{w}^k - w^k.
\end{array}
\end{equation}
Then,   \eqref{eq:FKM_scheme} can be rewritten equivalently as
\begin{equation}\label{eq:FKM_scheme2}
\arraycolsep=0.2em
\left\{\begin{array}{lcl}
	\hat{x}^k &=& \frac{t_k - r}{t_k} x^k + \frac{r}{t_k} z^k \vspace{1ex}\\
	y^k &=& \hat{x}^k - (\eta \gamma_k - \beta_k)\hat{w}^k \vspace{1ex}\\
	x^{k+1} &=& \hat{x}^k - \eta \hat{w}^{k+1} + \beta_k\hat{w}^k \vspace{1ex}\\
	z^{k+1} &=& z^k - \frac{\nu_k}{r}\hat{w}^{k+1}.
\end{array}
\right.
\end{equation}
\noindent\textbf{$\mathrm{(b)}$~Equivalent form.}
While \eqref{eq:FKM_scheme2} is convenient for our convergence analysis in the centralized setting, we will need to derive a simpler form that is equivalent to \eqref{eq:FKM_scheme2} to develop distributed algorithms.
To this end, first, it follows from the first line of \eqref{eq:FKM_scheme2} that $r(z^k - \hat{x}^k)  = (t_k - r)(\hat{x}^k - x^k)$.
Then, from the second and the third lines of \eqref{eq:FKM_scheme2}, we have $x^{k+1} = y^k - \eta \hat{w}^{k+1} + \eta \gamma_k \hat{w}^k$.
Using these expressions, we can derive from the second line of \eqref{eq:FKM_scheme2} that
\begin{equation*} 
	\arraycolsep=0.2em
	\begin{array}{lcl}
		y^{k+1} &=& y^k + \theta_k \left(y^k - y^{k-1} \right) - \hat{\eta}_k\hat{w}^{k+1} + \hat{\gamma}_k \hat{w}^k  - \hat{\lambda}_k \hat{w}^{k-1},
	\end{array}
\end{equation*}
where $\theta_k := \frac{t_k - r}{t_{k+1}}$, $\hat{\eta}_k := \frac{\eta (t_{k+1} - r)}{t_{k+1}} + \frac{\nu_k}{t_{k+1}} + \eta \gamma_{k+1} - \beta_{k+1}$, $\hat{\gamma}_k := \frac{\beta_k (t_{k+1} - r)}{t_{k+1}} + (\eta \gamma_k - \beta_k)\left(1 + \frac{t_k - r}{t_{k+1}}\right) + \frac{\eta(t_k - r)}{t_{k+1}}$, and $\hat{\lambda}_k := \frac{\eta\gamma_{k-1}(t_k - r)}{t_{k+1}}$.
Combining this relation, $x^{k+1} = y^k - \eta \hat{w}^{k+1} + \eta \gamma_k \hat{w}^k$ from the first and the second line of \eqref{eq:FKM_scheme2}, and $\hat{w}^{k+1} = Gy^k + \xi^{k+1}$, we obtain the following scheme:
\begin{equation}\label{eq:FKM_scheme3}
	\tag{FFP$_{+}$}
	\arraycolsep=0.2em
	\left\{
	\begin{array}{lcl}
		x^{k+1} &=& y^k - \eta \hat{w}^{k+1} + \eta \gamma_k \hat{w}^k, \quad \xi^{k+1} \in Tx^{k+1},  \vspace{1ex}\\
		\hat{w}^{k+1} &=& Gy^k + \xi^{k+1}, \vspace{1ex}\\
		y^{k+1} &=& y^k + \theta_k \left(y^k - y^{k-1} \right) - \hat{\eta}_k\hat{w}^{k+1} + \hat{\gamma}_k \hat{w}^k  - \hat{\lambda}_k \hat{w}^{k-1}.
	\end{array}
	\right.
\end{equation}
This is an equivalent form of \eqref{eq:FKM_scheme2} and will be used to develop distributed algorithms.
Here, we need to choose $y^{-1} = x^0$, $\hat{w}^{-1} = \hat{w}^0 = w^0 = Gx^0 + \xi^0$ for $\xi^0 \in Tx^0$, and $y^0 := x^0 - (\eta\gamma_0 - \beta_0)\hat{w}^0$.

\beforesubsec
\subsection{Maximally monotone and Lipschitz continuous setting}
\aftersubsec

In this subsection, we consider the composite inclusion \eqref{eq:GE} under the following assumption.
\begin{assumption}\label{as:FKM_assumption}
	The composite inclusion \eqref{eq:GE} satisfies the following conditions:
	\begin{compactitem}
		\item[(i)] (\textbf{\textit{Maximal monotonicity}}) $G + T$ is maximally monotone.
		\item[(ii)] (\textbf{\textit{Lipschitz continuity}}) $G$ is $L$-Lipschitz continuous.
	\end{compactitem}
\end{assumption}
In Assumption~\ref{as:FKM_assumption}(ii), the Lipschitz continuity of $G$ is standard in optimization and operator theory.
Regarding (i), since \eqref{eq:FKM_scheme2} is an abstract theoretical framework and does not involve the resolvent operator, its convergence can be established using the maximal monotonicity of the composite operator $\Phi := G + T$ rather than requiring a strict condition on the multivalued operator $T$.
However, practical implementations may necessitate computing its resolvent.
Under Assumption~\ref{as:FKM_assumption}, the resolvent of $T$ is still well-defined and has nice properties, although $T$ may not be necessarily maximally monotone but maximally hypomonotone.
\begin{lemma}\label{le:resolvent_properties}
	Let $\lambda \in (0, \frac{1}{L})$.
	Under Assumption~\ref{as:FKM_assumption}, the resolvent $J_{\lambda T} := (\Id + \lambda T)^{-1}$ is well-defined, single-valued, and $\frac{1}{1-\lambda L}$-Lipschitz continuous.
\end{lemma}
\begin{proof}
	Since $G$ is $L$-Lipschitz continuous, using Cauchy-Schwarz inequality, we can prove that
	\begin{equation*}
		\arraycolsep=0.2em
		\begin{array}{lcl}	
			\iprods{(L\Id - G)u - (L\Id - G)u', u - u'} &=& L\norms{u - u'}^2 - \iprods{Gu - Gu', u - u'} \vspace{1ex}\\
			&\geq& L\norms{u - u'}^2 - \norms{Gu - Gu'}\norms{u - u'} \vspace{1ex}\\
			&\geq& 0, \quad \forall u, u' \in \R^p.
		\end{array}
	\end{equation*}
	Hence, $L\Id - G$ is monotone and continuous, and has full domain.
	Moreover, since $\Phi = G + T$ is maximally monotone, it also follows that $T + L\Id = \Phi + (L\Id - G)$ is maximally monotone.
	Then, since $0 < \lambda < \frac{1}{L}$, it is easy to show that the operator $\Id + \lambda T = (1 - \lambda L)\Id + \lambda(T + L\Id)$ is maximally $(1-\lambda L)$-strongly monotone. Therefore, the resolvent $J_{\lambda T} = (\Id + \lambda T)^{-1}$ is single-valued on $\R^p$ and $\frac{1}{1 - \lambda L}$-Lipschitz continuous.
\end{proof}

\beforesubsubsec
\subsubsection{Convergence analysis}
\aftersubsubsec
\noindent\textbf{$\mathrm{(a)}$~Technical lemmas.}
First, we define the Lyapunov function $\Pc_k$ and derive a descent property of this function in the following lemma, whose proof is given in Appendix~\ref{apdx:le:FKM_key_est1}.

\begin{lemma}\label{le:FKM_key_est1}
	Suppose that Assumption~\ref{as:FKM_assumption} holds for \eqref{eq:GE} and $\zer{\Phi}\neq\emptyset$.
	For given $r > 0$, $\eta > 0$, and $\nu > 0$, let the parameters in \eqref{eq:FKM_scheme} be updated by
	\begin{equation}\label{eq:FKM_key_est1_params}
	\arraycolsep=0.2em
	\begin{array}{lcl}
		t_{k+1} = t_k + 1,  \quad \beta_k := \frac{c_1 \eta(t_k - r)}{t_k}, \quad \nu_k := \frac{\nu(t_k - 1)}{t_k}, \quad \text{and} \quad \gamma_k := 1,
	\end{array}
	\end{equation}
	for some $c_1 > 0$ and $t_0 > \max\sets{1, r, \frac{\nu}{\eta}}$.
	Let us define the following \textbf{Lyapunov function}:
	\begin{equation}\label{eq:FKM_Lyapunov_func}
		\arraycolsep=0.2em
		\begin{array}{lcl}
			\Pc_k &:=& \frac{a_k}{2}\norms{w^k}^2 + r(t_k - r) \iprods{w^k, x^k - z^k} + \frac{r^2(r-1)}{2\nu_k}\norms{z^k - \xopt}^2 +  \frac{d_k}{2} \norms{e^k}^2 ,
		\end{array}
	\end{equation}
	then we have
	\begin{equation}\label{eq:FKM_key_est1}
		\arraycolsep=0.2em
		\begin{array}{lcl}
			\Pc_k - \Pc_{k+1} &\geq& \frac{\hat{a}_{k+1} - a_{k+1}}{2} \norms{w^{k+1}}^2 + r(r-1)\iprods{w^{k+1}, x^{k+1} - \xopt} + \frac{b_k}{2} \norms{\hat{w}^{k+1}}^2  \vspace{1ex}\\ 
			&& + {~} \frac{\phi_k}{2} \norms{\hat{w}^{k+1} - w^k}^2 + \frac{\hat{d}_{k+1} - d_{k+1}}{2}\norms{e^{k+1}}^2,
		\end{array}
	\end{equation}
	where $x^{\star} \in \zer{\Phi}$,  $M^2 := 2(1+\omega)L^2$ for some $\omega \geq 0$, and 
	\begin{equation}\label{eq:FKM_key_est1_coeffs}
	\arraycolsep=0.2em
	\left\{
	\begin{array}{lcl}
		a_k &:=& \eta t_k (t_k - r) \big[ 1 - \frac{c_1 (t_k - r)}{t_k} \big], \vspace{1ex}\\
		\hat{a}_{k+1} &:=& t_k [(1-c_1)\eta t_k - \nu_k], \vspace{1ex}\\
		b_k &:=& t_k (r \eta - \nu_k) + \frac{(r-1)\nu (t_k-1)}{t_k}, \vspace{1ex}\\
		\phi_k &:=& t_k[\eta (t_k - r) - M^2 \eta^2 (\eta t_k - \nu_k)], \vspace{1ex}\\
		d_k &:=& t_k \big[ \frac{\beta_k^2 t_k}{c_1 \eta} + M^2 \eta^2 (\eta t_k - \nu_k) \big], \vspace{1ex}\\
		\hat{d}_{k+1} &:=& t_k \big[\omega (\eta t_k - \nu_k) - (r-1)\nu (t_k-1)\big].
	\end{array}
	\right.
	\end{equation}
\end{lemma}
Next, we can further simplify Lemma~\ref{le:FKM_key_est1} by choosing concrete values of $\omega$ and $c_1$ to achieve a simpler bound for $\Pc_k - \Pc_{k+1}$.
To this end, for given $r > 2$ and $\delta \in (0,r-2)$, let us define
\begin{equation*}
\arraycolsep=0.2em
\begin{array}{lcl}		
	c_1 := \frac{\delta}{2(r-1)} \quad \text{and} \quad \omega := (r-1)(r-2) + 1 = r^2 - 3r + 3, 
\end{array}
\end{equation*}
and choose $t_0$ such that
\begin{equation}\label{eq:FKM_key_est2_t0}
\arraycolsep=0.2em
\begin{array}{lcl}		
	t_0 &:=& \max\set{\frac{\nu}{\eta}, r\eta - 2\nu, 4r\omega}.
\end{array}
\end{equation}
Then, we have the following lemma, whose proof is given in Appendix~\ref{apdx:le:FKM_key_est2}.

\begin{lemma}\label{le:FKM_key_est2}
	Under the same settings as in Lemma~\ref{le:FKM_key_est1}, suppose that $t_0$ satisfies \eqref{eq:FKM_key_est2_t0} and other parameters are chosen such that
	\begin{equation}\label{eq:FKM_key_est2_params}
	\arraycolsep=0.2em
	\hspace{-1ex}
	\begin{array}{lcl}
		r > 2, \quad \delta \in (0, r - 2), \quad \omega = r^2 - 3r + 3, \quad 0 < \eta \leq \frac{1}{2\sqrt{1+\omega}L}, \ \text{and} \ 0 < \nu < (r - 2 - \delta)\eta.
	\end{array}
	\hspace{-1ex}
	\end{equation}
	Then, $\Pc_k$ defined in \eqref{eq:FKM_Lyapunov_func} satisfies
	\begin{equation}\label{eq:FKM_key_est2}
	\arraycolsep=0.2em
	\begin{array}{lcl}
		\Pc_k - \Pc_{k+1} &\geq& \frac{[(r-2-\delta)\eta - \nu]t_{k+1}}{2} \norms{w^{k+1}}^2 + \frac{\eta t_{k+1}}{2} \norms{\hat{w}^{k+1}}^2 \vspace{1ex}\\
		&&  + {~} \frac{\eta t_k^2}{8} \norms{\hat{w}^{k+1} - w^k}^2 + \frac{\eta t_{k+1}^2}{8}\norms{e^{k+1}}^2.
	\end{array}
	\end{equation}
\end{lemma}
We highlight that while the parameter choice in Lemma~\ref{le:FKM_key_est2} yields a clean lower bound for $\Pc_k - \Pc_{k+1}$, it is not guaranteed to be optimal.
In practice, exploring parameters slightly outside the conditions in \eqref{eq:FKM_key_est2_params} may be used if it improves empirical performance.

Finally, we derive a lower bound for the Lyapunov function $\Pc_k$ as in the following lemma, whose proof is given in Appendix~\ref{apdx:le:FKM_key_est3}.
\begin{lemma}\label{le:FKM_key_est3}
	Under the same settings as in Lemmas~\ref{le:FKM_key_est1} and \ref{le:FKM_key_est2}, $\Pc_k$ given in \eqref{eq:FKM_Lyapunov_func} satisfies
	\begin{equation}\label{eq:FKM_key_est3}
	\arraycolsep=0.2em
	\begin{array}{lcl}
		\Pc_k &\geq& \frac{\delta \eta}{2(r-1)}(t_k - r)^2 \norms{w^k}^2 + \frac{r^2}{2\nu}\norms{z^k - \xopt}^2 + \frac{d_k}{2}\norms{Gy^{k-1} - Gx^k}^2.
	\end{array}
	\end{equation}
\end{lemma}
\vspace{0.75ex}
\noindent\textbf{$\mathrm{(b)}$~Convergence guarantees.}
Now, we state the convergence of \eqref{eq:FKM_scheme2} in the following theorem.
\begin{theorem}\label{th:FKM1_convergence}
	Suppose that Assumption~\ref{as:FKM_assumption} holds for \eqref{eq:GE}, $\zer{\Phi} \neq\emptyset$, and $t_0$ is given by \eqref{eq:FKM_key_est2_t0}. 
	Let $\sets{(\hat{x}^k, x^k, y^k, z^k)}$ be generated by \eqref{eq:FKM_scheme} where the parameters are chosen as
	\begin{equation}\label{eq:FKM_th1_params}
		\arraycolsep=0.2em
		\begin{array}{c}
			r > 2, \quad \delta \in (0, r-2), \quad \omega := r^2 - 3r + 3, \quad 0 < \eta \leq \frac{1}{2\sqrt{1+\omega}L}, \quad 0 < \nu <  (r - 2 - \delta)\eta,
			\vspace{1ex}\\
			t_{k} := k + t_0, 
			\quad \beta_k := \frac{\delta \eta (t_k - r)}{2(r-1)t_k}, \quad \nu_k := \frac{\nu(t_k - 1)}{t_k}, \quad \text{and} \quad \gamma_k := 1.
		\end{array}
	\end{equation}
	Then, the following statements hold. 
	\begin{itemize}
	\itemsep=-0.3em
	\item[$\mathrm{(i)}$]~\textbf{$($The $\BigOs{1/k^2}$ convergence rates$)$} 
	There exists $\xi^k \in Tx^k$ such that
	\begin{equation}\label{eq:FKM_th1_BigO_rates}
	\arraycolsep=0.2em
	\begin{array}{lcllcl}		
		\norms{Gx^k + \xi^k}^2 &\leq& \dfrac{r-1}{\delta \eta(k + t_0 - r)^2}\Rc_0^2 \ \ \text{and} \ \
		\norms{Gy^{k-1} + \xi^k}^2 &\leq& \dfrac{4(r-1)}{\delta\eta(k + t_0 - r)^2}\Rc_0^2,
	\end{array}
	\end{equation}
	where $\Rc_0^2 := \frac{\eta (t_0 - r) [(2(r-1) - \delta)t_0 + r\delta]}{2(r-1)}\norms{w^0}^2 + \frac{r^2(r-1)t_0}{\nu (t_0 - 1)}\norms{x^0 - \xopt}^2$.
	
	\item[$\mathrm{(ii)}$]~\textbf{$($The $\SmallOs{1/k^2}$ convergence rates$)$} 
	There exist $\xi^k \in Tx^k$ for all $k \geq 0$ such that 
	\begin{equation}\label{eq:FKM_th1_SmallO_rates}
	\arraycolsep=0.2em
	\begin{array}{lcl}		
		\lim_{k \to \infty} k^2 \norms{Gx^k + \xi^k}^2 = 0 \qquad \text{and} \qquad 
		\lim_{k \to \infty} k^2 \norms{Gy^{k-1} + \xi^k}^2 = 0.
	\end{array}
	\end{equation}
	\item[$\mathrm{(iii)}$]~\textbf{$($The convergence of iterates$)$} 
	All $\sets{x^k}$, $\sets{\hat{x}^k}$, $\sets{y^k}$, and $\sets{z^k}$ converge to $\xopt \in \zer{\Phi}$.
	\end{itemize}
\end{theorem}

\begin{proof}
	For readability, we divide the proof of this theorem into the following steps.
	
	\vspace{1ex}
	\noindent\textbf{Step 1. (\textit{The summability bounds}).}
	Summing up \eqref{eq:FKM_key_est2} from $k := 0$ to $k := K$ and noting that $\Pc_k \geq 0$ for all $k \geq 0$ (see \eqref{eq:FKM_key_est3}), taking the limit of the result as $K \to \infty$, we can show that
	\begin{equation}\label{eq:FKM_th1_proof1}
	\arraycolsep=0.2em
	\begin{array}{lcl}		
		\sum_{k=1}^\infty  (k + t_0) \norms{w^k}^2 &\leq& \frac{2}{(r-2-\delta)\eta - \nu}\Pc_0, \vspace{1ex}\\
		\sum_{k=1}^\infty (k + t_0) \norms{\hat{w}^k}^2 &\leq& \frac{2}{\eta}\Pc_0, \vspace{1ex}\\
		\sum_{k=0}^\infty (k + t_0)^2 \norms{\hat{w}^{k+1} - w^k}^2 &\leq& \frac{8}{\eta}\Pc_0, \vspace{1ex}\\
		\sum_{k=1}^\infty (k + t_0)^2 \norms{Gy^{k-1} - Gx^k}^2 &\leq& \frac{8}{\eta} \Pc_0.
	\end{array}
	\end{equation}
	Since $y^{-1} = z^0 = x^0$, utilizing $\Rc_0$ from \eqref{eq:FKM_th1_BigO_rates}, we can prove that
	\begin{equation*}
	\arraycolsep=0.2em
	\begin{array}{lcl}		
		\Pc_0 &=& \frac{a_0}{2}\norms{w^0}^2 + \frac{r^2(r-1)}{2\nu_0}\norms{x^0 - \xopt}^2 \vspace{1ex}\\
		
		&=& \frac{\eta t_0 (t_0 - r)}{2} \left[ 1 - \frac{\delta (t_0 - r)}{2(r-1)t_0}\right]\norms{w^0}^2 + \frac{r^2(r-1)t_0}{2\nu (t_0 - 1)}\norms{x^0 - \xopt}^2 \vspace{1ex}\\
		
		& = & \frac{\eta (t_0 - r) [(2(r-1) - \delta)t_0 + r\delta]}{4(r-1)}\norms{w^0}^2 + \frac{r^2(r-1)t_0}{2\nu (t_0 - 1)}\norms{x^0 - \xopt}^2 = \frac{\Rc_0^2}{2}.
	\end{array}
	\end{equation*}
	Substituting this bound into \eqref{eq:FKM_th1_proof1}, we obtain 
	\begin{equation}\label{eq:FKM_th1_summability_bounds}
	\arraycolsep=0.2em
	\begin{array}{lclcl}		
		\sum_{k=1}^\infty  (k + t_0) \norms{Gx^k + \xi^k}^2 &\leq& \frac{1}{(r-2-\delta)\eta - \nu} \Rc_0^2 &<& +\infty, \vspace{1ex}\\
		\sum_{k=1}^\infty (k + t_0) \norms{Gy^{k-1} + \xi^k}^2 &\leq& \frac{1}{\eta}\Rc_0^2 &<& +\infty, \vspace{1ex}\\
		\sum_{k=1}^\infty (k + t_0)^2 \norms{Gy^{k-1} - Gx^k}^2 &\leq& \frac{4}{\eta} \Rc_0^2 &<& +\infty,
	\end{array}
	\end{equation}
	
	\vspace{1ex}
	\noindent\textbf{Step 2. (\textit{The $\BigOs{1/k^2}$ convergence rates in \eqref{eq:FKM_th1_BigO_rates}}).}
	Now, from \eqref{eq:FKM_key_est2} of Lemma~\ref{le:FKM_key_est2}, we can use induction to show that $\Pc_k \leq \Pc_0 \leq \frac{\Rc_0^2}{2}$ for all $k \geq 0$.
	Moreover, from \eqref{eq:FKM_key_est3}, we have
	\begin{equation*}
	\arraycolsep=0.2em
	\begin{array}{lcl}
		\Pc_k &\geq& \frac{\delta \eta}{2(r-1)}(t_k - r)^2 \norms{w^k}^2 + \frac{r^2}{2\nu}\norms{z^k - \xopt}^2 + \frac{d_k}{2}\norms{Gy^{k-1} - Gx^k}^2 \geq \frac{\delta \eta}{2(r-1)}(t_k - r)^2 \norms{w^k}^2.
	\end{array}
	\end{equation*}
	Combining these facts, we get
	\begin{equation*}
	\arraycolsep=0.2em
	\begin{array}{lcl}
		\norms{Gx^k + \xi^k}^2 = \norms{w^k}^2 &\leq& \frac{2(r-1)}{\delta \eta (k + t_0 - r)^2}\Pc_k \leq \frac{r-1}{\delta \eta(k + t_0 - r)^2}\Rc_0^2,
	\end{array}
	\end{equation*}
	which is the first bound of \eqref{eq:FKM_th1_BigO_rates}, where $\xi^k \in Tx^k$.
	
	Next, from \eqref{eq:FKM_key_est1_coeffs}, using $\beta_k$ from \eqref{eq:FKM_th1_params} and $\eta t_k - \nu_k \geq 0$, we achieve
	\begin{equation*}
	\arraycolsep=0.2em
	\begin{array}{lcl}
		d_k &=&  t_k \left[ \frac{\beta_k^2 t_k}{c_1\eta} + M^2 \eta^2 \left(\eta t_k - \nu_k\right) \right] 
		= \frac{\delta \eta}{2(r-1)}(t_k - r)^2 + M^2\eta^2 t_k (\eta t_k - \nu_k) 
		\geq \frac{\delta \eta}{2(r-1)}(t_k - r)^2.
	\end{array}
	\end{equation*}
	Then, utilizing \eqref{eq:FKM_key_est3} and Young's inequality as $ 2\norms{w^k}^2 + 2\norms{\hat{w}^k - w^k}^2 \geq \norms{\hat{w}^k}^2$, we can show that
	\begin{equation*}
	\arraycolsep=0.2em
	\begin{array}{lcl}
		\Pc_k &\geq& \frac{\delta \eta}{2(r-1)}(t_k - r)^2 \norms{w^k}^2 + \frac{d_k}{2}\norms{Gy^{k-1} - Gx^k}^2 \vspace{1ex}\\
		&\geq& \frac{\delta \eta}{4(r-1)}(t_k - r)^2 \norms{w^k}^2 + \frac{\delta \eta}{4(r-1)}(t_k - r)^2 \norms{\hat{w}^k - w^k}^2 \vspace{1ex}\\
		&\geq& \frac{\delta \eta}{8(r-1)}(t_k - r)^2\norms{\hat{w}^k}^2.
	\end{array}
	\end{equation*}
	Combining this inequality and $\Pc_k \leq \frac{\Rc_0^2}{2}$, we obtain	
	\begin{equation*}
	\arraycolsep=0.2em
	\begin{array}{lcl}
		\norms{Gy^{k-1} + \xi^k}^2 = \norms{\hat{w}^k}^2 &\leq& \frac{8(r-1)}{\delta \eta (k + t_0 - r)^2}\Pc_k = \frac{4(r-1)}{\delta\eta(k + t_0 - r)^2}\Rc_0^2,
	\end{array}
	\end{equation*}
	which is the second bound of \eqref{eq:FKM_th1_BigO_rates}, where $\xi^k \in Tx^k$.
	
	\vspace{1ex}
	\noindent\textbf{Step 3. (\textit{Some intermediate results}).}
	Rearranging \eqref{eq:FKM_key_est2}, noting that $e^k = \hat{w}^k - w^k$, and then applying Lemma~\ref{le:A1}, we can prove that $\lim_{k \to \infty} \Pc_k$ exists, and
	\begin{equation}\label{eq:FKM_th1_proof2}
	\arraycolsep=0.2em
	\begin{array}{lcl}
		\sum_{k=0}^\infty t_k \norms{w^k}^2 &<& +\infty, \vspace{1ex}\\
		\sum_{k=0}^\infty t_k \norms{\hat{w}^k}^2 &<& +\infty, \vspace{1ex}\\
		\sum_{k=0}^\infty t_k^2 \norms{\hat{w}^{k+1} - w^k}^2 &<& +\infty, \vspace{1ex}\\
		\sum_{k=0}^\infty t_k^2 \norms{\hat{w}^k - w^k}^2 &<& +\infty.
	\end{array}
	\end{equation}
	By Young's inequality, we have
	\begin{equation*}
	\arraycolsep=0.2em
	\begin{array}{lcl}
		\norms{w^{k+1} - w^k}^2 &\leq& 2\norms{w^{k+1} - \hat{w}^{k+1}}^2 + 2\norms{\hat{w}^{k+1} - w^k}^2, \vspace{1ex}\\
		\norms{\hat{w}^{k+1} - \hat{w}^k}^2 &\leq& 2\norms{\hat{w}^{k+1} - w^k}^2 + 2\norms{\hat{w}^k - w^k}^2.
	\end{array}
	\end{equation*}
	Utilizing these relations, we can easily derive from \eqref{eq:FKM_th1_proof2} that
	\begin{equation}\label{eq:FKM_th1_proof3}
		\arraycolsep=0.2em
		\begin{array}{lcllcl}
			\sum_{k=0}^\infty t_k^2 \norms{w^{k+1} - w^k}^2 &<& +\infty
			\qquad \text{and} \qquad 
			\sum_{k=0}^\infty t_k^2 \norms{\hat{w}^{k+1} - \hat{w}^k}^2 &<& +\infty.
		\end{array}
	\end{equation}
	
	\vspace{1ex}
	\noindent\textbf{Step 4. (\textit{The $\SmallOs{1/k^2}$ convergence rates}).}
	Let us denote $v^k := r(z^k - x^k)$.
	Then, from the first and the second lines of \eqref{eq:FKM_scheme2}, we can derive that
	\begin{equation}\label{eq:FKM_th1_proof4}
	\arraycolsep=0.2em
	\begin{array}{lcl}
		v^k = t_k (\hat{x}^k - x^k) = t_k(y^k - x^k) + t_k (\eta - \beta_k)\hat{w}^k.
	\end{array}
	\end{equation}
	Using this expression and $x^{k+1} - y^k = \eta (\hat{w}^k - \hat{w}^{k+1})$ from the second and the third lines of \eqref{eq:FKM_scheme2}, we can derive
	\begin{equation*}
	\arraycolsep=0.2em
	\begin{array}{lcl}
		v^{k+1} - \frac{t_k - r}{t_k} v^k &=& r(z^{k+1} - z^k) - r(x^{k+1} - y^k) + r(\eta - \beta_k)\hat{w}^k \vspace{1ex}\\
		&=& -\nu_k \hat{w}^{k+1} - r\eta(\hat{w}^k - \hat{w}^{k+1})  + r(\eta - \beta_k)\hat{w}^k \vspace{1ex}\\
		&=& (r\eta - \nu_k)\hat{w}^{k+1} - r\beta_k \hat{w}^k.
	\end{array}
	\end{equation*}
	This is equivalent to $v^{k+1} = \left(1 - \frac{r}{t_k}\right)v^k + \frac{r}{t_k} \cdot \frac{t_k}{r}\left[ (r\eta - \nu_k)\hat{w}^{k+1} - r\beta_k \hat{w}^k \right]$.
	Since $\frac{r}{t_k} \in \left(0, 1\right)$, from the convexity of $\norms{\cdot}^2$, Young's inequality, $\nu_k = \frac{\nu(t_k - 1)}{t_k} \leq \nu$, and $\beta_k = \frac{\delta \eta (t_k - r)}{2(r-1)t_k} \leq \frac{\eta}{2}$, we can show that
	\begin{equation*}
	\arraycolsep=0.2em
	\begin{array}{lcl}
		\norms{v^{k+1}}^2 &\leq&  \left(1 - \frac{r}{t_k}\right) \norms{v^k}^2 + \frac{t_k}{r} \norms{(r\eta - \nu_k)\hat{w}^{k+1} - r\beta_k \hat{w}^k}^2 \vspace{1ex} \\
		&\leq&  \norms{v^k}^2 - \frac{r}{t_k} \norms{v^k}^2 + \frac{2t_k(r\eta - \nu_k)^2}{r}\norms{\hat{w}^{k+1}}^2 + 2r\beta_k^2 t_k \norms{\hat{w}^k}^2 \vspace{1ex}\\
		&\leq& \norms{v^k}^2 - \frac{r}{t_k} \norms{v^k}^2 + \frac{4(r^2\eta^2 + \nu^2)t_k}{r}\norms{\hat{w}^{k+1}}^2 + \frac{r\eta^2}{4} t_k \norms{\hat{w}^k}^2.
	\end{array}
	\end{equation*}
	Applying Lemma~\ref{le:A1} to the last inequality and using the second bound in \eqref{eq:FKM_th1_summability_bounds}, we get
	\begin{equation}\label{eq:FKM_th1_proof5}
	\arraycolsep=0.2em
	\begin{array}{lcl}
		\lim\limits_{k \to \infty} \norms{v^k}^2 = r^2 \lim\limits_{k \to \infty} \norms{x^k - z^k}^2 \text{ exists \quad and \quad } \sum\limits_{k=0}^\infty \frac{r}{t_k}\norms{x^k - z^k}^2 < \infty.
	\end{array}
	\end{equation}
	Utilizing Lemma~\ref{le:A2}, these relations imply
	\begin{equation}\label{eq:FKM_th1_proof5.1}
	\arraycolsep=0.2em
	\begin{array}{lcl}
		\lim\limits_{k \to \infty} \norms{x^k - z^k}^2 = 0.
	\end{array}
	\end{equation}
	Next, using $x^{k+1} - y^k = -\eta (\hat{w}^{k+1} - \hat{w}^k)$ from \eqref{eq:FKM_scheme2} and the second bound in \eqref{eq:FKM_th1_proof3}, we have
	\begin{equation}\label{eq:FKM_th1_proof6}
	\arraycolsep=0.2em
	\begin{array}{lcl}
		\sum_{k=0}^\infty t_k^2\norms{x^{k+1} - y^k}^2 = \eta^2 \sum_{k=0}^{\infty} t_k^2 \norms{\hat{w}^{k+1} - \hat{w}^k}^2 < +\infty.
	\end{array}
	\end{equation}
	Using again \eqref{eq:FKM_th1_proof4}, by Young's inequality and $\beta_k \leq \frac{\eta}{2}$, we can show that
	\begin{equation*}
	\arraycolsep=0.2em
	\begin{array}{lcl}
		t_k \norms{y^k - x^k}^2 &\leq& \frac{2r^2}{t_k}\norms{z^k - x^k}^2 + 2t_k(\eta - \beta_k)^2\norms{\hat{w}^k}^2 \leq \frac{2r^2}{t_k}\norms{z^k - x^k}^2 + 2\eta^2 t_k \norms{\hat{w}^k}^2, \vspace{1ex}\\
		t_k \norms{x^{k+1} - x^k}^2 &\leq& 2t_k\norms{x^{k+1} - y^k}^2 + 2t_k \norms{y^k - x^k}^2.
	\end{array}
	\end{equation*}
	Using the second summability bound in \eqref{eq:FKM_th1_proof2}, \eqref{eq:FKM_th1_proof5}, and \eqref{eq:FKM_th1_proof6}, we can derive from the last inequalities that
	\begin{equation}\label{eq:FKM_th1_proof7}
	\arraycolsep=0.2em
	\begin{array}{lcl}
		\sum_{k=0}^\infty t_k \norms{y^k - x^k}^2 < +\infty
		\text{\quad and \quad}
		\sum_{k=0}^\infty t_k \norms{x^{k+1} - x^k}^2 < +\infty.
	\end{array}
	\end{equation}
	Now, since $\Pc_k \leq \frac{\Rc_0^2}{2} < +\infty$, from \eqref{eq:FKM_key_est3} of Lemma~\ref{le:FKM_key_est3}, we have
	\begin{equation*}
	\arraycolsep=0.2em
	\begin{array}{lcl}
		\frac{\delta \eta}{2(r-1)}(t_k-r)^2 \norms{w^k}^2 + \frac{r^2}{2\nu}\norms{z^k - \xopt}^2 + \frac{d_k}{2}\norms{Gy^{k-1} - Gx^k}^2 &\leq& \Pc_k \leq \frac{\Rc_0^2}{2} < +\infty,
	\end{array}
	\end{equation*}
	which implies that $\norms{z^k - \xopt}^2 \leq M := \frac{\nu \Rc_0^2}{r^2}$.
	Using this bound, Young's inequality, and $\nu_k \leq \nu$, we can show that
	\begin{equation*}
	\arraycolsep=0.1em
	\begin{array}{lcl}
		-\frac{2\nu_k}{r} \iprods{e^{k+1}, z^{k+1} - \xopt} &\leq& \frac{\nu_k t_k^2}{r}\norms{e^{k+1}}^2 + \frac{\nu_k}{rt_k^2}\norms{z^{k+1} - \xopt}^2 \leq \frac{\nu t_k^2}{r}\norms{e^{k+1}}^2 + \frac{M\nu}{rt_k^2}, \vspace{1ex}\\
		-\frac{2\nu_k}{r} \iprods{w^{k+1}, z^{k+1} - x^{k+1}} &\leq& \frac{\nu_k t_k}{r} \norms{w^{k+1}}^2 + \frac{\nu_k}{rt_k}\norms{z^{k+1} - x^{k+1}}^2 \leq \frac{\nu t_k}{r} \norms{w^{k+1}}^2 + \frac{\nu}{rt_k}\norms{z^{k+1} - x^{k+1}}^2.
	\end{array}
	\end{equation*}
	Using the last line of \eqref{eq:FKM_scheme2}, $\nu_k \leq \nu$, and the last two relations, we have
	\begin{equation*}
	\arraycolsep=0.2em
	\begin{array}{lcl}
		\norms{z^{k+1} - \xopt}^2 &=& \norms{z^k - \xopt}^2 + 2\iprods{z^{k+1} - z^k, z^{k+1} - \xopt} - \norms{z^{k+1} - z^k}^2 \vspace{1ex}\\
		&=& \norms{z^k - \xopt}^2 - \frac{2\nu_k}{r}\iprods{w^{k+1} + e^{k+1}, z^{k+1} - \xopt} - \frac{\nu_k^2}{r^2}\norms{w^{k+1} + e^{k+1}}^2 \vspace{1ex}\\
		&\leq& \norms{z^k - \xopt}^2 - \frac{2\nu_k}{r}\iprods{w^{k+1}, z^{k+1} - \xopt} - \frac{2\nu_k}{r}\iprods{e^{k+1}, z^{k+1} - \xopt} \vspace{1ex}\\
		&=& \norms{z^k - \xopt}^2 - \frac{2\nu_k}{r}\iprods{w^{k+1}, x^{k+1} - \xopt} - \frac{2\nu_k}{r}\iprods{w^{k+1}, z^{k+1} - x^{k+1}}\vspace{1ex}\\
		&& - {~} \frac{2\nu_k}{r}\iprods{e^{k+1}, z^{k+1} - \xopt}  \vspace{1ex}\\
		&\leq&  \norms{z^k - \xopt}^2 + \left(\frac{2\nu}{r} + \frac{\nu t_k}{r}\right) \norms{w^{k+1}}^2 + \frac{\nu}{rt_k}\norms{z^{k+1} - x^{k+1}}^2 + \frac{\nu t_k^2}{r}\norms{e^{k+1}}^2 + \frac{M\nu}{rt_k^2}.
	\end{array}
	\end{equation*}
	Since the last four terms on the right-hand side are summable due to \eqref{eq:FKM_th1_proof2} and \eqref{eq:FKM_th1_proof5}, applying Lemma~\ref{le:A1}, we conclude that $\lim_{k \to \infty} \norms{z^k - \xopt}^2$ exists.
	However, since $\lim_{k \to \infty} \norms{x^k - z^k}^2 = 0$ due to \eqref{eq:FKM_th1_proof5.1}, using the triangle inequality as $|\norms{x^k - \xopt} - \norms{z^k - \xopt}| \leq \norms{x^k - z^k}$, we conclude that $\lim_{k \to \infty} \norms{x^k - \xopt}^2$ also exists.
	
	From the first line of \eqref{eq:FKM_th1_BigO_rates}, we have $\sets{(t_k - r)^2 \norms{w^k}^2}$ is bounded, i.e., there exists $\bar{M} > 0$ such that $(t_k - r)^2 \norms{w^k}^2 \leq \bar{M}^2$ for all $k \geq 0$.
	Using this inequality and $\lim_{k \to \infty} \norms{x^k - z^k}^2 = 0$  from \eqref{eq:FKM_th1_proof5.1}, we can show that as $k \to \infty$,
	\begin{equation*}
	\arraycolsep=0.2em
	\begin{array}{lcl}
		(t_k - r)^2[\iprods{w^k, x^k - z^k}]^2 \leq (t_k - r)^2 \norms{w^k}^2 \norms{x^k - z^k}^2 \leq \bar{M}^2\norms{x^k - z^k}^2 \to 0.
	\end{array}
	\end{equation*}
	This inequality implies that $\lim_{k \to \infty} (t_k - r)\iprods{w^k, x^k - z^k} = 0$.
	Moreover, from the last line of \eqref{eq:FKM_th1_proof2}, we also have $\lim_{k \to \infty} t_k^2 \norms{e^k}^2 = 0$.
	
	Now, from \eqref{eq:FKM_Lyapunov_func}, we have
	\begin{equation*}
	\arraycolsep=0.2em
	\begin{array}{lcl}
		\Pc_k &:=& \frac{a_k}{2}\norms{w^k}^2 + r(t_k - r) \iprods{w^k, x^k - z^k} + \frac{r^2(r-1)}{2\nu_k}\norms{z^k - \xopt}^2 +  \frac{d_k}{2} \norms{e^k}^2 ,
	\end{array}
	\end{equation*}
	Since $\lim_{k \to \infty} \Pc_k$ and $\lim_{k \to \infty}\norms{z^k - \xopt}^2$ exist, and $\lim_{k \to \infty} (t_k - r)\iprods{w^k, x^k - z^k} = 0$ and $\lim_{k \to \infty} d_k \norms{e^k}^2 = 0$ due to $d_k = \BigOs{t_k^2}$, we conclude that $\lim_{k \to \infty} a_k \norms{w^k}^2$ exists.
	By the definition of $a_k$, one can easily show that $\frac{t_k^2}{a_k} \to \frac{1}{\eta(1-c_1)}$ as $k \to \infty$.
	Therefore, the existence of $\lim_{k \to \infty} a_k \norms{w^k}^2$ implies the existence of $\lim_{k \to \infty} t_k^2 \norms{w^k}^2$.
	However, since $\sum_{k=0}^\infty t_k \norms{w^k}^2 < +\infty$ by the first line of \eqref{eq:FKM_th1_proof2}, by Lemma~\ref{le:A2}, we can argue that $\lim_{k \to \infty} t_k^2 \norms{w^k}^2 = 0$, which proves the first limit of \eqref{eq:FKM_th1_SmallO_rates}.
	
	Using Young's inequality as $\norms{\hat{w}^k}^2 \leq 2\norms{w^k}^2 + 2\norms{\hat{w}^k - w^k}^2$, the first limit of \eqref{eq:FKM_th1_SmallO_rates}, and the last line of \eqref{eq:FKM_th1_proof2}, we obtain the second limit of \eqref{eq:FKM_th1_SmallO_rates}.
	
	\vspace{1ex}
	\noindent\textbf{Step 5. (\textit{The convergence of the iterates}).}
	We have proved in \textbf{Step 4} that $\lim_{k \to \infty} \norms{x^k - \xopt}$ exists for every $\xopt \in \zer{\Phi}$. 
	Moreover, from the first line of \eqref{eq:FKM_th1_proof1}, we also have $\lim_{k \to \infty} \norms{Gx^k + \xi^k} = 0$ for some $\xi^k \in Tx^k$, which means that $\lim_{k \to \infty} \norms{w^k} = 0$ for $(x^k, w^k) \in \gra{\Phi}$.
	Applying Opial's lemma, we conclude that $\sets{x^k}$ converges  to a solution $\xopt \in \zer{\Phi}$.
	Next, since $\lim_{k \to \infty} \norms{x^k - z^k} = 0$ and $\lim_{k \to \infty} \norms{x^k - y^k} = 0$ due to \eqref{eq:FKM_th1_proof5.1} and \eqref{eq:FKM_th1_proof7}, we conclude that $\sets{y^k}$ and $\sets{z^k}$ also converge  to a solution $\xopt \in \zer{\Phi}$.
	Finally, from the first line of \eqref{eq:FKM_scheme}, we can also show that $\sets{\hat{x}^k}$ converges to $x^{\star} \in \zer{\Phi}$.
\end{proof}

\beforesubsubsec
\subsubsection{Convergence of \eqref{eq:FKM_scheme3}}
\aftersubsubsec
Since \eqref{eq:FKM_scheme3} is mathematically equivalent to \eqref{eq:FKM_scheme2}, its convergence properties are immediately inherited from the latter, leading to the following corollary.

\begin{corollary}\label{co:FKM_scheme3_convergence}
	Suppose that Assumption~\ref{as:FKM_assumption} holds for \eqref{eq:GE}, $\zer{\Phi} \neq \emptyset$, and $t_0$ satisfies \eqref{eq:FKM_key_est2_t0}. 
	Let $\sets{(x^k, y^k)}$ be generated by \eqref{eq:FKM_scheme3} where the parameters are chosen as
	\begin{equation}\label{eq:FKM_scheme3_params}
		\arraycolsep=0.2em
		\hspace{-2ex}
		\begin{array}{c}
			r > 2, \quad \delta \in (0, r-2), \quad \omega := r^2 - 3r + 3,
				\quad
			0 < \eta \leq \frac{1}{2\sqrt{1+\omega}L}, 
			\vspace{1ex}\\
			0 < \nu < (r-2-\delta)\eta, \qquad t_{k} := k + t_0,  \qquad \theta_k := \frac{t_k - r}{t_{k+1}}, \qquad \hat{\lambda}_k := \eta\theta_k, \vspace{1ex}\\
			\hat{\eta}_k := \frac{[2(r-1)-\delta]\eta(t_{k+1} - r)}{2(r-1)t_{k+1}} + \frac{\nu(t_k - 1)}{t_k t_{k+1}} + \eta, \quad \text{and} \quad
			\hat{\gamma}_k := \frac{[4(r-1) - \delta]\eta (t_k - r)}{2(r-1)t_{k+1}} + \eta.
		\end{array}
		\hspace{-2ex}
	\end{equation}
	Then, all the $\BigOs{1/k^2}$ and $\SmallOs{1/k^2}$ convergence rates in \eqref{eq:FKM_th1_BigO_rates} and \eqref{eq:FKM_th1_SmallO_rates} still hold.
	Moreover, the sequences $\sets{x^k}$ and $\sets{y^k}$ converge to a solution $\xopt \in \zer{\Phi}$.
\end{corollary}

\begin{proof}
Utilizing the parameters in \eqref{eq:FKM_th1_params}, we can compute that $\hat{\eta}_k = \frac{[2(r-1) - \delta]\eta (t_{k+1} - r)}{2(r-1) t_{k+1}} + \frac{\nu(t_k - 1)}{t_k t_{k+1}} + \eta$, and $\hat{\gamma}_k = \frac{[4(r-1) - \delta]\eta (t_k - r)}{2(r-1)t_{k+1}} + \eta$,  which proves \eqref{eq:FKM_scheme3_params}.
Then, since \eqref{eq:FKM_scheme3} is equivalent to \eqref{eq:FKM_scheme}, all the results of Theorem~\ref{th:FKM1_convergence} still hold.
\end{proof}

\beforesubsec
\subsection{Co-coercive and maximally co-hypomonotone setting}\label{subsec:cocoercive_methods}
\aftersubsec
\noindent\textbf{$\mathrm{(a)}$~Co-coercivity and co-hypomonotonicity assumptions.}
Next, we consider a special subclass of \eqref{eq:GE} that satisfies the following assumptions.
\begin{assumption}\label{as:FKM2_cocoercive}
	For problem \eqref{eq:GE}, there exist $L > 0$ and $\rho \geq 0$ with $L\rho < 1$ such that:
	\begin{compactitem}
		\item[(i)] (\textbf{Co-coercivity}) $G$ is $\frac{1}{L}$-co-coercive.
		\item[(ii)] (\textbf{Maximal co-hypomonotonicity}) $T$ is maximally $\rho$-co-hypomonotone.
	\end{compactitem}
\end{assumption}
The co-coercivity of $G$ naturally arises in numerous practical applications, such as the minimization of smooth convex functions.
A function $f$ is convex and $L$-smooth iff $\nabla{f}$ is $\frac{1}{L}$-co-coercive \cite{Nesterov2004}.
For a general operator, the co-coercivity condition is more restrictive than the Lipschitz continuity required in Assumption~\ref{as:FKM_assumption}.
However, as a compensation, it facilitates the use of simpler algorithms with fewer parameters.
The maximal co-hypomonotonicity of $T$ in Assumption~\ref{as:FKM2_cocoercive}(ii) is milder than the maximal monotonicity commonly used in the literature and covers a class of nonmonotone operators.
A number of examples can be found in \cite{TranDinh2025a}.

\vspace{0.75ex}
\noindent\textbf{$\mathrm{(b)}$~The derivation of the algorithm.}
To handle the co-hypomonotonicity of $T$, we do not directly apply the accelerated method to \eqref{eq:GE}, but rather to its \textbf{forward-backward splitting (FBS)} reformulation.
More specifically, for a given $\eta > 0$, we define a FBS residual:
\begin{equation}\label{eq:FBS_mapping}
\Gc_{\eta}x := \tfrac{1}{\eta}\left(x - J_{\eta T}(x - \eta Gx) \right).
\end{equation}
Then, it is well-known that $x^{\star} \in \zer{G+T}$ iff $\Gc_{\eta}x^{\star} = 0$.
Moreover, the following lemma provides an important property of $\Gc_{\eta}$, whose proof can be found in \cite[Lemma 3]{TranDinh2025a}. 

\begin{lemma}\label{le:FB_cocoercive}
Suppose that Assumption~\ref{as:FKM2_cocoercive} holds for \eqref{eq:GE}.
Then, for any $\eta$ such that $\frac{2(1 - \sqrt{1 - L\rho})}{L} < \eta < \frac{2(1 + \sqrt{1 - L\rho})}{L}$, $\Gc_{\eta}$ defined by \eqref{eq:FBS_mapping} is $\bar{\beta}$-co-coercive with $\bar{\beta} := \frac{\eta (4 - L\eta) - 4\rho}{4(1 - L\rho)} > 0$.
\end{lemma}

We derive an algorithm to solve \eqref{eq:GE} under Assumption~\ref{as:FKM2_cocoercive}, which is in fact a special variant of \eqref{eq:FKM_scheme}.
By setting $\gamma_k = \beta_k = 0$, \eqref{eq:FKM_scheme} (or its equivalent form \eqref{eq:FKM_scheme2}) reduces to  
\begin{equation}\label{eq:FKM2_scheme} 
\arraycolsep=0.2em
\left\{\begin{array}{lcl}
	y^k &=& \frac{t_k - r}{t_k} x^k + \frac{r}{t_k} z^k \vspace{1ex}\\
	x^{k+1} &=& y^k - \eta \hat{w}^{k+1} \vspace{1ex}\\
	z^{k+1} &=& z^k - \frac{\nu_k}{r}\hat{w}^{k+1},
\end{array}
\right.
\end{equation}
where $\hat{w}^k := Gy^{k-1} + \xi^k$ for $\xi^k \in Tx^k$.

Since $x^{k+1}$ appears on both sides of the second line of \eqref{eq:FKM2_scheme} in $\hat{w}^{k+1}$, we can use the resolvent operator $J_{\eta T}$ of $T$ and get $x^{k+1} = J_{\eta T}(y^k - \eta Gy^k)$.
Then, $\hat{w}^{k+1}$ can be written as
\begin{equation*}
	\arraycolsep=0.2em
	\begin{array}{lcl}
		\hat{w}^{k+1} = \frac{1}{\eta}(y^k - x^{k+1}) = \frac{1}{\eta}(y^k - J_{\eta T}(y^k - \eta Gy^k)) = \Gc_{\eta}y^k,
	\end{array}
\end{equation*}
where $\Gc_{\eta}$ is the \textbf{forward-backward splitting (FBS)} residual defined by \eqref{eq:FBS_mapping}.
In this case, the scheme \eqref{eq:FKM2_scheme} can be rewritten equivalently to
\begin{equation}\label{eq:FKM2_scheme_reformulation}\tag{FFP2}
	\arraycolsep=0.2em
	\left\{\begin{array}{lcl}
		y^k &=& \frac{t_k - r}{t_k} x^k + \frac{r}{t_k} z^k \vspace{1ex}\\
		x^{k+1} &=& y^k - \eta \Gc_{\eta}y^k \vspace{1ex}\\
		z^{k+1} &=& z^k - \frac{\nu_k}{r} \Gc_{\eta}y^k.
	\end{array}
	\right.
\end{equation}
This scheme has a form similar to Nesterov's accelerated gradient method in convex optimization, see, e.g., \cite{Nesterov2004}. 
However, as discussed in \eqref{eq:FKM_scheme}, its parameter update for $z^k$ is different. 
Compared to \eqref{eq:FKM_scheme}, this scheme is much simpler and uses fewer parameters.

\beforesubsubsec
\subsubsection{Convergence analysis}
\aftersubsubsec
\noindent\textbf{$\mathrm{(a)}$~Technical lemmas.}
For given $r > 1$, $\eta > 0$, and $0 < \nu < (r-1)\eta$, let us choose a lower bound constant $t_0$ such that
\begin{equation}\label{eq:t0_cond2}
	\arraycolsep=0.2em
	\begin{array}{lcl}
		t_0 \geq \max\set{r, \frac{(r-1)(2r-1)\eta}{(r-1)\eta - \nu}, \frac{r-1}{r+1}}.
	\end{array}
\end{equation}
We also denote $a_k := t_k (t_k - r)$, and consider the following \textbf{Lyapunov function}:
\begin{equation}\label{eq:FKM2_Lyapunov_function}
	\arraycolsep=0.2em
	\begin{array}{lcl}
		\Lc_k &:=& \beta a_k \norms{\Gc_{\eta}y^k}^2 + rt_k \iprods{\Gc_{\eta}y^k, y^k - z^k} + \frac{r^2(r-1)}{2\nu_k} \norms{z^k - \xopt}^2,
	\end{array}
\end{equation}
where $\beta$, $r$, and $\nu_k$ are given, specified below.
Then, the following lemma states a  descent property of  $\Lc_k$, whose proof is deferred to Appendix~\ref{apdx:le:FKM2_key_est1}.
\begin{lemma}\label{le:FKM2_key_est1}
	Suppose that Assumption~\ref{as:FKM2_cocoercive} holds for \eqref{eq:GE}, $\zer{\Phi} \neq \emptyset$, and \eqref{eq:t0_cond2} holds.
	Let
	\begin{equation}\label{eq:FKM2_key_est1_params}
		\arraycolsep=0.2em
		\begin{array}{lcl}
			r > 1, \quad t_{k+1} = t_k + 1 , \quad 2\rho < \eta < 2\bar{\beta}, \quad 0 < \nu < (r-1)\eta, \quad \text{and} \quad \nu_k := \frac{\nu(t_k - 1)}{t_k}.
		\end{array}
	\end{equation}
	Then, $\Lc_k$ defined in \eqref{eq:FKM2_Lyapunov_function} with $\beta := \frac{\eta}{2}$ satisfies
	\begin{equation}\label{eq:FKM2_key_est1}
		\arraycolsep=0.2em
		\begin{array}{lcl}
			\Lc_k - \Lc_{k+1} &\geq& \left[ \frac{(r-1)\eta - \nu}{2}t_k + r(r-1)\bar{\beta} \right] \norms{\Gc_{\eta}y^k}^2 + \frac{(2\bar{\beta} - \eta) t_{k+1}(t_{k+1} - r)}{2} \norms{\Gc_{\eta}y^{k+1} - \Gc_{\eta}y^k}^2\vspace{1ex}\\
			&& + {~} \frac{r^2(r-1)}{2\nu t_k (t_k - 1)} \norms{z^{k+1} - \xopt}^2.
		\end{array}
	\end{equation}
\end{lemma}

The next lemma  provides a lower bound of $\Lc_k$, whose proof  can be found in Appendix~\ref{apdx:le:FKM2_key_est2}.
\begin{lemma}\label{le:FKM2_key_est2}
	Under the same settings as in Lemma~\ref{le:FKM2_key_est1},  $\Lc_k$ in \eqref{eq:FKM2_Lyapunov_function} satisfies
	\begin{equation}\label{eq:FKM2_key_est2}
		\arraycolsep=0.2em
		\begin{array}{lcl}
			\Lc_k &\geq& \frac{A_k}{4} \norms{\Gc_{\eta}y^k}^2 + \frac{r^2[((r-1)\eta  - 2\nu) t_k + 2\nu]}{2\nu\eta(t_k - 1)} \norms{z^k - \xopt}^2,
		\end{array}
	\end{equation}
	where $A_k := \eta(2a_k - t_k^2) + 4r\bar{\beta}t_k \geq \eta t_k^2$.
\end{lemma}

\noindent\textbf{$\mathrm{(b)}$~Convergence guarantees.}
Now, we state the convergence of \eqref{eq:FKM2_scheme_reformulation} in Theorem~\ref{th:FKM2_convergence}.

\begin{theorem}\label{th:FKM2_convergence}
	For \eqref{eq:GE}, suppose that Assumption~\ref{as:FKM2_cocoercive} holds, $\bar{\beta}$ is given in Lemma~\ref{le:FB_cocoercive}, and $t_0$ satisfies \eqref{eq:t0_cond2}. 
	Let $\sets{ (x^k, y^k, z^k)}$ be generated by \eqref{eq:FKM2_scheme_reformulation} using the following parameters
	\begin{equation}\label{eq:FKM_th2_params}
		\arraycolsep=0.2em
		\begin{array}{lcl}
			r > 1, \quad t_{k} := k + t_0 , \quad 2\rho < \eta < 2\bar{\beta}, \quad 0 < \nu < (r-1)\eta, \quad \text{and} \quad \nu_k := \frac{\nu(t_k - 1)}{t_k}.
		\end{array}
	\end{equation}
	Then, the following statements hold.
	\begin{itemize}
	\itemsep=-0.3em
	\item[$\mathrm{(i)}$]~\textbf{\textit{$($The $\BigOs{1/k^2}$ rates$)$}} 
	There exist $\xi^k \in Tx^k$ for all $k \geq 0$ such that:
	\begin{equation}\label{eq:FKM2_BigO_convergence}
		\arraycolsep=0.2em
		\begin{array}{lcl}
			\norms{Gy^{k-1} + \xi^k}^2 \leq \dfrac{2\mcal{R}_0^2}{\eta (k + t_0 - 1)^2} \quad \text{and} \quad \norms{Gx^k + \xi^k}^2 \leq \dfrac{2(1+L\eta)^2 \Rc_0^2}{\eta (k + t_0 - 1)^2},
		\end{array}
	\end{equation}
	where $\mcal{R}_0^2 :=  \eta t_0 (t_0 - r) \norms{\Gc_{\eta}y^0}^2 + \frac{r^2(r-1)t_0}{\nu (t_0 - 1)}\norms{x^0 - \xopt}^2$.
	
	\item[$\mathrm{(ii)}$]~\textbf{\textit{$($The $\SmallOs{1/k^2}$ rates$)$}}
	There exist $\xi^k \in Tx^k$ for all $k \geq 0$ such that:
	\begin{equation}\label{eq:FKM2_SmallO_convergence} 
	\arraycolsep=0.2em
	\begin{array}{lcl}
		\lim_{k \to \infty} k^2 \norms{Gy^{k-1} + \xi^k}^2 = 0 \qquad \text{and} \qquad
		\lim_{k \to \infty} k^2 \norms{Gx^k + \xi^k}^2 = 0.
	\end{array}
	\end{equation}
	
	\item[$\mathrm{(iii)}$]~\textbf{$($The convergence of iterates$)$} 
	All $\sets{x^k}$, $\sets{y^k}$, and $\sets{z^k}$ converge to $\xopt \in \zer{\Phi}$.
	\end{itemize}
\end{theorem}

\begin{proof}
	\noindent$\mathrm{(i)}$~\textbf{\textit{$($The $\BigOs{1/k^2}$ convergence rates$)$.}}
	First, telescoping \eqref{eq:FKM2_key_est1} from $k:=0$ to $k := K$, then taking the limit of the resulting inequality as $K \to \infty$ and noticing that $\Lc_k \geq 0$, we get
	\begin{equation}\label{eq:FKM2_convergence1_proof1}
		\arraycolsep=0.2em
		\begin{array}{lcl}
			\Lc_k &\leq& \Lc_0, \vspace{1ex}\\
			\sum_{k=0}^\infty \left[ \frac{(r-1)\eta - \nu}{2}t_k + r(r-1)\bar{\beta} \right] \norms{\Gc_{\eta}y^k}^2 &\leq& \Lc_0, \vspace{1ex}\\
			\sum_{k=0}^\infty \frac{(2\bar{\beta} - \eta) t_{k+1}(t_{k+1} - r)}{2} \norms{\Gc_{\eta}y^{k+1} - \Gc_{\eta}y^k}^2 &\leq& \Lc_0.
		\end{array}
	\end{equation}
	Since $x^0 = z^0 = y^0 = y^{-1}$, $\nu_0 = \frac{\nu(t_0 - 1)}{t_0}$, and $\Rc_0$ given in \eqref{eq:FKM2_BigO_convergence}, we have
	\begin{equation*}
		\arraycolsep=0.2em
		\begin{array}{lcl}
			\Lc_0 = \frac{\eta t_0 (t_0 - r)}{2} \norms{\Gc_{\eta}y^0}^2 + \frac{r^2(r-1)t_0}{2\nu (t_0 - 1)}\norms{x^0 - \xopt}^2 = \frac{\mcal{R}_0^2}{2}.
		\end{array}
	\end{equation*}
	Substituting this expression into the last two lines of \eqref{eq:FKM2_convergence1_proof1}, we obtain
	\begin{equation}\label{eq:FKM2_summability_bounds}
	\arraycolsep=0.2em
	\begin{array}{lcl}
		\sum_{k=1}^\infty (k + t_0 - 1) \norms{Gy^{k-1} + \xi^k}^2 &\leq& \frac{\mcal{R}_0^2}{(r-1)\eta - \nu}, \vspace{1ex}\\
		\sum_{k=1}^\infty (k + t_0) (k + t_0 - r)\norms{\Gc_{\eta}y^k - \Gc_{\eta}y^{k-1}}^2 &\leq& \frac{\mcal{R}_0^2}{2\bar{\beta} - \eta}.
	\end{array}
	\end{equation}
	Now, since $\hat{w}^{k+1} = Gy^k + \xi^{k+1}$, we have $\xi^{k+1} = \hat{w}^{k+1} - Gy^k = \Gc_{\eta}y^k - Gy^k \in Tx^{k+1}$.
	Moreover, since $G$ is $\frac{1}{L}$-co-coercive, it is also $L$-Lipschitz continuous.
	Using the triangle inequality, the $L$-Lipschitz continuity of $G$, and the second line of \eqref{eq:FKM2_scheme_reformulation}, we can show that
	\begin{equation}\label{eq:FKM2_convergence1_proof2}
		\arraycolsep=0.2em
		\begin{array}{lcl}
			\norms{Gx^{k+1} + \xi^{k+1}} &=& \norms{\Gc_{\eta}y^k + (Gx^{k+1} - Gy^k)} \vspace{1ex}\\
			&\leq& \norms{\Gc_{\eta}y^k} + \norms{Gx^{k+1} - Gy^k} \vspace{1ex}\\
			&\leq& \norms{\Gc_{\eta}y^k} + L\norms{x^{k+1} - y^k} \vspace{1ex}\\
			&=& (1 + L\eta) \norms{\Gc_{\eta}y^k} = (1 + L\eta) \norms{Gy^{k-1} + \xi^k}.
		\end{array}
	\end{equation}
	Putting this relation into the first line of \eqref{eq:FKM2_summability_bounds}, we get  
	\begin{equation*} 
	\arraycolsep=0.2em
	\begin{array}{lcl}
		\sum_{k=1}^\infty (k+t_0) \norms{Gx^k + \xi^k}^2 &\leq& \frac{(1+L\eta)^2\mcal{R}_0^2}{(r-1)\eta - \nu},
	\end{array}
	\end{equation*}
	Moreover, from \eqref{eq:FKM2_key_est2}, we also have
	\begin{equation*}
		\arraycolsep=0.2em
		\hspace{-1ex}\begin{array}{lcl}
			\Lc_k \geq \frac{A_k}{4} \norms{\Gc_{\eta}y^k}^2 \geq \frac{\eta t_k^2}{4} \norms{\Gc_{\eta}y^k}^2 = \frac{\eta (k + t_0)^2}{4}\norms{\Gc_{\eta}y^k}^2. 
		\end{array}\hspace{-1ex}
	\end{equation*}
	Combining the last two expressions and the first line of \eqref{eq:FKM2_convergence1_proof1}, we obtain
	\begin{equation*}
		\arraycolsep=0.2em
		\hspace{-1ex}\begin{array}{lcl}
			\eta (k + t_0)^2\norms{\Gc_{\eta}y^k}^2 \leq 4 \Lc_k \leq 4 \Lc_0 = 2\mcal{R}_0^2,
		\end{array}\hspace{-1ex}
	\end{equation*}
	which proves the first bound of \eqref{eq:FKM2_BigO_convergence}.
	Utilizing this bound and \eqref{eq:FKM2_convergence1_proof2}, we obtain the second bound of  \eqref{eq:FKM2_BigO_convergence}.
	
	\vspace{1ex}
	\noindent$\mathrm{(ii)}$ and $\mathrm{(iii)}$~\textbf{\textit{$($The $\SmallOs{1/k^2}$ rates and convergence of iterates$)$.}}
	The proof of the claims in (ii) and (iii) is very similar to those in Theorem~\ref{th:FKM1_convergence}.
	Thus, we omit them here.
\end{proof}

\beforesubsubsec
\subsubsection{Equivalent form of \eqref{eq:FKM_scheme3}}
\aftersubsubsec
Since $\beta_k = \gamma_k = 0$, the scheme \eqref{eq:FKM_scheme3} reduces to
\begin{equation}\label{eq:FKM_scheme2+}\tag{FFP2$_{+}$}
\arraycolsep=0.2em
\left\{
\begin{array}{lcl}
	x^{k+1} &=& y^k - \eta \hat{w}^{k+1}, \quad \xi^{k+1} \in Tx^{k+1},  \vspace{1ex}\\
	\hat{w}^{k+1} &=& Gy^k + \xi^{k+1}, \vspace{1ex}\\
	y^{k+1} &=& y^k + \theta_k \left(y^k - y^{k-1} \right) - \hat{\eta}_k\hat{w}^{k+1} + \hat{\gamma}_k \hat{w}^k,
\end{array}
\right.
\end{equation}
where $\theta_k = \frac{t_k - r}{t_{k+1}}$, $\hat{\eta}_k = \frac{\eta (t_{k+1} - r)}{t_{k+1}} + \frac{\nu (t_k - 1)}{t_k t_{k+1}}$, and $\hat{\gamma}_k = \eta \theta_k$.
The scheme is initialized at $y^{-1} = y^0 := x^0$ and $\hat{w}^0 := 0$.
The convergence of this scheme can be stated in the following corollary.

\begin{corollary}\label{co:FKM_scheme2+_convergence}
	Suppose that Assumption~\ref{as:FKM2_cocoercive} holds for \eqref{eq:GE}, $\zer{\Phi} \neq\emptyset$, and $t_0$ satisfies \eqref{eq:t0_cond2}. 
	Let $\sets{(x^k, y^k)}$ be generated by \eqref{eq:FKM_scheme2+} where the parameters are chosen as
	\begin{equation*} 
		\arraycolsep=0.2em
		\begin{array}{c}
			r > 1, \qquad t_k = k + t_0, \qquad 2\rho < \eta < \frac{2}{L}, \qquad 0 < \nu < (r-1)\eta, 
			\vspace{1ex}\\
			\theta_k := \frac{t_k - r}{t_{k+1}}, \qquad  \hat{\eta}_k = \frac{\eta (t_{k+1} - r)}{t_{k+1}} + \frac{\nu (t_k - 1)}{t_k t_{k+1}}, \qquad \text{and} \qquad \hat{\gamma}_k = \eta \theta_k.
		\end{array}
	\end{equation*}
	Then, all the $\BigOs{1/k^2}$ and $\SmallOs{1/k^2}$ convergence rates in \eqref{eq:FKM2_BigO_convergence} and \eqref{eq:FKM2_SmallO_convergence} still hold.
	Moreover, $\sets{x^k}$ and $\sets{y^k}$ converge to $\xopt \in \zer{\Phi}$.
\end{corollary}

\begin{proof}
	Since \eqref{eq:FKM_scheme2+} is equivalent to \eqref{eq:FKM2_scheme}, this corollary follows directly from Theorem~\ref{th:FKM2_convergence}.
\end{proof}

\beforesubsec
\subsection{Fast fixed-point-based methods for three-operator inclusions}\label{subsec:3o_methods}
\aftersubsec
The next step of our development is to apply our methods \eqref{eq:FKM_scheme3} and \eqref{eq:FKM_scheme2+} developed above to the following three-operator inclusion:
\begin{equation}\label{eq:3op_inclusion}\tag{CI3}
	\arraycolsep=0.2em
	\begin{array}{lcl}
		\text{Find $u^{\star} \in \R^p$ such that:} \quad 0 \in Au^{\star} + Bu^{\star} + Cu^{\star},
	\end{array}
\end{equation}
where $A, C: \R^p \rightrightarrows \R^p$ are two possibly set-valued mappings and $B: \R^p \to \R^p$ is a single-valued mapping.
The key step here is to reformulate \eqref{eq:3op_inclusion} into \eqref{eq:GE} by expanding the space.

\beforesubsubsec
\subsubsection{Composite inclusion reformulation and its preconditioned form}
\aftersubsubsec
To reformulate \eqref{eq:3op_inclusion} into \eqref{eq:GE}, we adopt a primal-dual approach as in \cite{malitsky2026firstorder,tam2025decentralised}.
Let $v^{\star} \in Cu^{\star}$ be an auxiliary variable.
Then, $0 \in C^{-1}v^{\star} - u^{\star}$.
Consequently, \eqref{eq:3op_inclusion} is equivalent to 
\begin{equation}\label{eq:3op_inclusion2}
	\arraycolsep=0.2em
	\left\{\begin{array}{lcl}
		0 \in Au^{\star} + Bu^{\star} + v^{\star}, \vspace{1ex}\\
		0 \in C^{-1}v^{\star} - u^{\star}.
	\end{array}\right.
\end{equation}
Let $x := [u, v] \in \R^{2p}$, and  $G: \R^{2p} \to \R^{2p}$ and $T: \R^{2p} \rightrightarrows \R^{2p}$ respectively be defined by
\begin{equation}\label{eq:3op_reform2}
	\arraycolsep=0.2em
	\begin{array}{lcl}
		Gx := \begin{bmatrix}
			B & 0 \\ 0 & 0
		\end{bmatrix} \begin{bmatrix}
			u \\ v
		\end{bmatrix} = \begin{bmatrix}
			Bu \\ 0
		\end{bmatrix} \quad \text{and} \quad
		Tx := \begin{bmatrix}
			A & 0 \\ 0 & C^{-1}
		\end{bmatrix}
		\begin{bmatrix}
			u \\ v
		\end{bmatrix}
		+ \begin{bmatrix}
			0 & \Id \\ -\Id & 0
		\end{bmatrix}
		\begin{bmatrix}
			u \\ v
		\end{bmatrix} 
		= \begin{bmatrix}
			Au + v \\ C^{-1}v - u
		\end{bmatrix}.
	\end{array}
\end{equation}
Then, \eqref{eq:3op_inclusion2} becomes a composite inclusion $0 \in Gx^{\star} + Tx^{\star}$ as in \eqref{eq:GE}.

Next, for $\tau > 0$ and  $\sigma > 0$, we consider the following symmetric matrix in $\R^{2p\times 2p}$: 
\begin{equation}\label{eq:3op_P} 
P = \begin{bmatrix}
	\frac{1}{\tau} \Id & -\Id \\ -\Id & \frac{1}{\sigma} \Id
\end{bmatrix}.
\end{equation}
It is straightforward to verify that  if $\tau\sigma < 1$, then $P$ is positive definite.
Then, \eqref{eq:3op_inclusion} is equivalent to the following preconditioned composite inclusion:
\begin{equation}\label{eq:3op_inclusion3}
	\arraycolsep=0.2em
	\begin{array}{lcl}
	\textrm{Find $x^{\star} \in \R^{2p}$ such that:}~ 0 \in P^{-1}Gx^{\star} + P^{-1}Tx^{\star}.
	\end{array}
\end{equation}
Our goal is to apply our FFP algorithmic framework developed in Subsections~\ref{subsec:alg1_derivation} and \ref{subsec:cocoercive_methods}  to \eqref{eq:3op_inclusion3}.
The first step is to verify that this transformation satisfies the underlying assumptions required by these algorithms.
To this end, we establish some structural properties of the operators $P^{-1}G$ and $P^{-1}T$ in the following lemma, whose proof is similar to the proof of Lemmas~\ref{le:DGE_assumption} and \ref{le:DGE_cocoercive} below with $\mbf{K} = \Id$, and thus we omit.

\begin{lemma}\label{le:3op_inclusion_assumptions}
	For \eqref{eq:3op_inclusion}, suppose that $A: \R^p \rightrightarrows \R^p$ and $C: \R^p \rightrightarrows \R^p$ are maximally monotone, and $B : \R^p \to \R^p$ is single-valued.
    Suppose further that $\tau > 0$ and $\sigma > 0$ in matrix $P$ from \eqref{eq:3op_P} satisfy $\tau \sigma < 1$.
	Then, for $T$ and $G$ defined by \eqref{eq:3op_reform2}, the following statements hold.
	\begin{itemize}
	\itemsep=-0.3em
		\item[$\mathrm{(i)}$] 
		The mapping $P^{-1}T: \R^{2p} \rightrightarrows \R^{2p}$ is maximally monotone w.r.t. $\iprods{\cdot, \cdot}_P$.
		
		\item[$\mathrm{(ii)}$] 
		If $B$ is monotone and $L_B$-Lipschitz continuous, then $P^{-1}G$ is monotone w.r.t. $\iprods{\cdot, \cdot}_P$ and $L_G$-Lipschitz continuous w.r.t. $\norms{\cdot}_P$, where $L_G := \frac{\tau L_B}{1 - \tau \sigma}$.
		
		\item[$\mathrm{(iii)}$] 
		If $B$ is $\frac{1}{L_B}$-co-coercive, then $P^{-1}G$ is $\frac{1}{L_G}$-co-coercive w.r.t. $\norms{\cdot}_P$, where $L_G := \frac{\tau L_B}{1 - \tau \sigma}$.		
	\end{itemize}
\end{lemma}

Now, to characterize a solution of \eqref{eq:3op_inclusion}, we define the following FBS residual:
\begin{equation}\label{eq:FBS_residual3}
	\arraycolsep=0.2em
	\begin{array}{lcl}
		\Gc_{\sigma^{-1}}u := \sigma \left(u - J_{\sigma^{-1}C}(u - \sigma^{-1} (\zeta + Bu))\right), \quad \text{for a given }\zeta \in Au.
	\end{array}
\end{equation}
Then, by the definition of the resolvent $J_{\sigma^{-1}C}$, we can easily show that $\Gc_{\sigma^{-1}}u^{\star} = 0$ if and only if $0 \in \zeta^{\star} + Bu^{\star} + Cu^{\star}$, where $\zeta^{\star} \in Au^{\star}$.
This means that $u^{\star}$ is a solution of \eqref{eq:3op_inclusion}.

\beforesubsubsec
\subsubsection{Derivation of the algorithm}
\aftersubsubsec
Now, we apply our method \eqref{eq:FKM_scheme} with  $\eta = 1$ to derive new algorithms for solving \eqref{eq:3op_inclusion3}.
We first denote $\hat{x}^k := [ \hat{u}^k, \hat{v}^k]$ and  $z^k :=  [z_u^k, z_v^k]$.
Then, we perform the following steps.
\begin{compactitem}
	\item \textbf{\textit{Step 1:}} 
	Update
	\begin{equation*}
		\arraycolsep=0.2em
		\begin{array}{lcl}
			\hat{u}^k := \frac{t_k - r}{t_k} u^k + \frac{r}{t_k} z_u^k 
			\qquad \text{and} \qquad 
			\hat{v}^k := \frac{t_k - r}{t_k} v^k + \frac{r}{t_k} z_v^k.
		\end{array}
	\end{equation*}
	\item \textbf{\textit{Step 2:}} For $\hat{w}^k = P^{-1}Gy^{k-1} + \xi^k$ with $\xi^k \in P^{-1}Tx^k$, we define $\tilde{w}^k = P\hat{w}^k = Gy^{k-1} + P\xi^k$.
	Then, using $P^{-1} = \frac{1}{1 - \tau \sigma} \begin{bmatrix}
		\tau \Id & \tau \sigma \Id \\ 
		\tau \sigma \Id & \sigma \Id
	\end{bmatrix}$, we have $\hat{w}^k = P^{-1}\tilde{w}^k = \frac{1}{1 - \tau\sigma} \begin{bmatrix}
		\tau \tilde{w}_u^k + \tau \sigma \tilde{w}_v^k \\
		\tau \sigma \tilde{w}_u^k + \sigma \tilde{w}_v^k
	\end{bmatrix}$.
	Then, we update
	\begin{equation*}
		\arraycolsep=0.2em
		\begin{array}{lcl}
			y_u^k = \hat{u}^k - \frac{\gamma_k - \beta_k}{1 - \tau\sigma} (\tau \tilde{w}_u^k + \tau \sigma \tilde{w}_v^k). 
		\end{array}
	\end{equation*}
	\item \textbf{\textit{Step 3:}} From the third line of \eqref{eq:FKM_scheme}, we have
	$x^{k+1} = \hat{x}^k - (P^{-1}Gy^k + \xi^{k+1}) + \beta_k \hat{w}^k$.
	Using the definitions of $G$, $T$, and $P$, we can derive from the last relation that
	\begin{equation*}
		\begin{cases}
			u^{k+1} = J_{\tau A} (\hat{u}^k - \tau \hat{v}^k - \tau By_u^k + \tau \beta_k \tilde{w}_u^k) \vspace{1ex}\\
			v^{k+1} = J_{\sigma C^{-1}} (\hat{v}^k + \sigma (2u^{k+1} - \hat{u}^k) + \sigma \beta_k \tilde{w}_v^k)
		\end{cases}
	\end{equation*}
	Since $C$ is maximally monotone, we have the inverse-resolvent identity $J_{\sigma C^{-1}}(x) = x - \sigma J_{\sigma^{-1} C} \left(\frac{x}{\sigma}\right)$.
	Thus,  the update of $v^{k+1}$ can be expressed as
	\begin{equation*}
		\arraycolsep=0.2em
		\begin{array}{lcl}
			\begin{cases}
				\tilde{v}^{k+1} = \hat{v}^k + \sigma (2u^{k+1} - \hat{u}^k) + \sigma \beta_k \tilde{w}_v^k \vspace{1ex}\\
				v^{k+1} = \tilde{v}^{k+1} - \sigma J_{\sigma^{-1} C} (\frac{1}{\sigma} \tilde{v}^{k+1}).
			\end{cases}
		\end{array}
	\end{equation*}
	
	\item \textbf{\textit{Step 4:}} 
	From the third line of \eqref{eq:FKM_scheme2}, we have $\hat{w}^{k+1} = \hat{x}^k - x^{k+1} + \beta_k \hat{w}^k$.
	Multiplying both sides of this equation by $P$ and using $\tilde{w}^k = P\hat{w}^k$, we get $\tilde{w}^{k+1} = P(\hat{x}^k - x^{k+1}) + \beta_k \tilde{w}^k$, which leads to the following updates
	\begin{equation*}
		\arraycolsep=0.2em
		\left\{\begin{array}{lcl}
				\tilde{w}_u^{k+1} & = &  \frac{1}{\tau} (\hat{u}^k - u^{k+1}) - (\hat{v}^k - v^{k+1}) + \beta_k \tilde{w}_u^k \vspace{1ex}\\ 
				\tilde{w}_v^{k+1}  & = & - (\hat{u}^k - u^{k+1}) + \frac{1}{\sigma} (\hat{v}^k - v^{k+1}) + \beta_k \tilde{w}_v^k.
		\end{array}\right.
	\end{equation*}
	\item \textbf{\textit{Step 5:}} 
	Using this $\tilde{w}^{k+1}$, we have $\hat{w}^{k+1} = P^{-1}\tilde{w}^{k+1} = \frac{1}{1 - \tau\sigma} \begin{bmatrix}
		\tau \tilde{w}_u^{k+1} + \tau \sigma \tilde{w}_v^{k+1} \\
		\tau \sigma \tilde{w}_u^{k+1} + \sigma \tilde{w}_v^{k+1}
	\end{bmatrix}$.
	Then, from the fourth line of \eqref{eq:FKM_scheme2}, we update
	\begin{equation*}
		\arraycolsep=0.2em
		\begin{array}{lcl}
			z_u^{k+1} = z_u^k - \frac{\nu_k}{r} \hat{w}_u^{k+1} 
			\qquad \text{and} \qquad
			z_v^{k+1} = z_v^k - \frac{\nu_k}{r} \hat{w}_v^{k+1}.
		\end{array}
	\end{equation*}
\end{compactitem}
\textbf{Variant of \eqref{eq:FKM_scheme}.}
Putting the above steps together, we obtain the following method for solving  the three-operator inclusion \eqref{eq:3op_inclusion}:
\begin{equation}\label{eq:FKM_3OP_scheme}\tag{FFP3}
	\arraycolsep=0.2em
	\left\{\begin{array}{lcl}
		\hat{u}^k  & := & \frac{t_k - r}{t_k} u^k + \frac{r}{t_k} z_u^k, \vspace{1ex}\\
		\hat{v}^k & := & \frac{t_k - r}{t_k} v^k + \frac{r}{t_k} z_v^k, \vspace{1ex}\\
		y_u^k & := & \hat{u}^k - \frac{\gamma_k - \beta_k}{1 - \tau\sigma} (\tau \tilde{w}_u^k + \tau \sigma \tilde{w}_v^k), \vspace{1ex}\\
		u^{k+1} & := & J_{\tau A} (\hat{u}^k - \tau \hat{v}^k - \tau By_u^k + \tau \beta_k \tilde{w}_u^k), \vspace{1ex}\\
		\tilde{v}^{k+1} & :=  & \hat{v}^k + \sigma (2u^{k+1} - \hat{u}^k) + \sigma \beta_k \tilde{w}_v^k, \vspace{1ex}\\
		v^{k+1} & := & \tilde{v}^{k+1} - \sigma J_{\sigma^{-1} C} (\frac{1}{\sigma} \tilde{v}^{k+1}), \vspace{1ex}\\
		\tilde{w}_u^{k+1} & := & \frac{1}{\tau} (\hat{u}^k - u^{k+1}) - (\hat{v}^k - v^{k+1}) + \beta_k \tilde{w}_u^k, \vspace{1ex}\\ 
		\tilde{w}_v^{k+1} & := & - (\hat{u}^k - u^{k+1}) + \frac{1}{\sigma} (\hat{v}^k - v^{k+1}) + \beta_k \tilde{w}_v^k, \vspace{1ex}\\ 
		z_u^{k+1} & := & z_u^k - \frac{\nu_k}{r (1 - \tau \sigma)} (\tau \tilde{w}_u^{k+1} + \tau \sigma \tilde{w}_v^{k+1}), \vspace{1ex}\\
		z_v^{k+1} & := & z_v^k - \frac{\nu_k}{r (1 - \tau \sigma)} (\tau \sigma \tilde{w}_u^{k+1} + \sigma \tilde{w}_v^{k+1}).
	\end{array}\right.
\end{equation}
\noindent\textbf{Variant of \eqref{eq:FKM2_scheme_reformulation}.}
In particular, if $\gamma_k = \beta_k = 0$, then our scheme \eqref{eq:FKM_3OP_scheme} reduces to
\begin{equation}\label{eq:FKM_3OP_scheme2}\tag{FFP3$_+$}
	\arraycolsep=0.2em
	\left\{\begin{array}{lcl}
		\hat{u}^k & := & \frac{t_k - r}{t_k} u^k + \frac{r}{t_k} z_u^k, \vspace{1ex}\\
		\hat{v}^k & := & \frac{t_k - r}{t_k} v^k + \frac{r}{t_k} z_v^k, \vspace{1ex}\\
		u^{k+1} & := & J_{\tau A} (\hat{u}^k - \tau \hat{v}^k - \tau B\hat{u}^k), \vspace{1ex}\\
		\tilde{v}^{k+1} & := & \hat{v}^k + \sigma (2u^{k+1} - \hat{u}^k), \vspace{1ex}\\
		v^{k+1} & := & \tilde{v}^{k+1} - \sigma J_{\sigma^{-1} C} (\frac{1}{\sigma} \tilde{v}^{k+1}), \vspace{1ex}\\
		\tilde{w}_u^{k+1} & := & \frac{1}{\tau} (\hat{u}^k - u^{k+1}) - (\hat{v}^k - v^{k+1}), \vspace{1ex}\\ 
		\tilde{w}_v^{k+1} & := & - (\hat{u}^k - u^{k+1}) + \frac{1}{\sigma} (\hat{v}^k - v^{k+1}), \vspace{1ex}\\ 
		z_u^{k+1} & := & z_u^k - \frac{\nu_k}{r (1 - \tau \sigma)} (\tau \tilde{w}_u^{k+1} + \tau \sigma \tilde{w}_v^{k+1}), \vspace{1ex}\\
		z_v^{k+1} & := & z_v^k - \frac{\nu_k}{r (1 - \tau \sigma)} (\tau \sigma \tilde{w}_u^{k+1} + \sigma \tilde{w}_v^{k+1}).
	\end{array}\right.
\end{equation}
Clearly, both schemes require one evaluation of $B$, one resolvent $J_{\tau A}$ and one resolvent $J_{\sigma^{-1}C}$.
Therefore, their per-iteration complexity is essentially the same as that of the non-accelerated operator splitting methods for solving \eqref{eq:3op_inclusion}.

\beforesubsubsec
\subsubsection{Convergence analysis}
\aftersubsubsec
By viewing \eqref{eq:FKM_3OP_scheme} as a direct application of our master template \eqref{eq:FKM_scheme} to the preconditioned problem \eqref{eq:3op_inclusion3}, we can establish its convergence by invoking Theorem~\ref{th:FKM1_convergence} and Theorem~\ref{th:FKM2_convergence}.

\begin{theorem}[\textbf{Monotone and Lipschitz continuous setting}]\label{th:FKM_3OP_convergence}
	For Inclusion \eqref{eq:3op_inclusion}, suppose that $A$ and $C$ are maximally monotone and $B$ is monotone and $L_B$-Lipschitz continuous.
	Suppose further that $\zer{A+B+C} \ne \emptyset$ and $t_0$ is given by \eqref{eq:FKM_key_est2_t0}.
	Let $\sets{(u^k, \hat{u}^k, y_u^k, z_u^k)}$ and $\sets{(v^k, \hat{v}^k, z_v^k)}$ be generated by \eqref{eq:FKM_3OP_scheme} where the parameters are chosen as
	\begin{equation}\label{eq:FKM_3OP_th1_params}
		\arraycolsep=0.2em
		\begin{array}{c}
			r > 2, \quad \delta \in (0,r-2), \quad \omega := r^2 - 3r + 3, \quad 0 < \nu < r - 2 - \delta, \vspace{1ex}\\
			2\tau L_B\sqrt{1+\omega} + \tau \sigma \leq 1, \quad t_{k} := k + t_0, \quad \beta_k := \frac{\delta (t_k - r)}{2(r-1)t_k}, \quad \nu_k := \frac{\nu(t_k - 1)}{t_k}, \quad \gamma_k := 1.
		\end{array}
	\end{equation}
	Then, the following statements hold.
	\begin{itemize}
	\itemsep=-0.3em 
	\item[$\mathrm{(i)}$] $\mathrm{(\text{\textbf{The $\BigOs{1/k^2}$ and $\SmallOs{1/k^2}$ rates}})}$ 
	For $\Gc_{\sigma^{-1}}$ defined by \eqref{eq:FBS_residual3},  $\norms{\Gc_{\sigma^{-1}}u^k}^2$ converges to zero at the rates of $\BigOs{1/k^2}$ and $\SmallOs{1/k^2}$.
	
	\item[$\mathrm{(ii)}$] $\mathrm{(\text{\textbf{The convergence of the iterates}})}$ 
	The iterate sequences  $\sets{u^k}$, $\sets{\hat{u}^k}$, $\sets{y_u^k}$, and $\sets{z_u^k}$ all converge  to $u^{\star}$, and the iterate sequences $\sets{v^k}$, $\sets{\hat{v}^k}$, and $\sets{z_v^k}$ all converge to $v^{\star}$, where $u^{\star} \in \zer{A+B+C}$.
	Consequently, all these sequences are bounded.
	\end{itemize}
\end{theorem}

\begin{proof}
	First, since $2\tau L_B\sqrt{1+\omega} + \tau \sigma \leq 1$, we can show that $1 \leq \frac{1 - \tau \sigma}{2\sqrt{1+\omega}\tau L_B} = \frac{1}{2\sqrt{1+\omega}L_G}$.
	Thus, all the conditions in \eqref{eq:FKM_3OP_th1_params} are satisfied for $\eta = 1$ by \eqref{eq:FKM_th1_params}.
	
	Now, for $x^k = [u^k,  v^k]$, by definition, we have $Gx^k = \begin{bmatrix}
		Bu^k \\ 0
	\end{bmatrix}$ and $Tx^k = \begin{bmatrix}
		Au^k + v^k \\ C^{-1}v^k - u^k
	\end{bmatrix}$.
	Take $\zeta^k \in Au^k$ and $\vartheta^k \in C^{-1}v^k$, then we have $\xi^k := \begin{bmatrix}
		\zeta^k + v^k \\ \vartheta^k - u^k
	\end{bmatrix} \in Tx^k$, and thus $Gx^k + \xi^k = \begin{bmatrix}
		\zeta^k + Bu^k + v^k \\ \vartheta^k - u^k
	\end{bmatrix}$, leading to
	\begin{equation}\label{eq:FKM_3OP_convergence_proof1}
		\norms{Gx^k + \xi^k}^2 = \norms{\zeta^k + Bu^k + v^k}^2 + \norms{\vartheta^k - u^k}^2.
	\end{equation}
	Since \eqref{eq:FKM_3OP_scheme} is equivalent to \eqref{eq:FKM_scheme} (or \eqref{eq:FKM_scheme2}), applying \eqref{eq:FKM_th1_BigO_rates} and \eqref{eq:FKM_th1_SmallO_rates}, we have
	\begin{equation*}
		\arraycolsep=0.2em
		\begin{array}{ccc}
			\norms{P^{-1}Gx^k + P^{-1}\xi^k}_P^2 = \BigOs{1/k^2} \quad \text{and} \quad \norms{P^{-1}Gx^k + P^{-1}\xi^k}_P^2 = \SmallOs{1/k^2}.
		\end{array}
	\end{equation*}
	Applying Lemma~\ref{le:norm_equivalence}, we can derive that
	\begin{equation*}
		\arraycolsep=0.2em
		\begin{array}{ccc}
			\norms{Gx^k + \xi^k}^2 = \BigOs{1/k^2} \quad \text{and} \quad \norms{Gx^k + \xi^k}^2 = \SmallOs{1/k^2}.
		\end{array}
	\end{equation*}
	Combining these results and \eqref{eq:FKM_3OP_convergence_proof1}, we obtain	
	\begin{equation}\label{eq:FKM_3OP_convergence_proof2}
		\arraycolsep=0.2em
		\begin{array}{ccc}
			\norms{\zeta^k + Bu^k + v^k}^2 = \BigOs{1/k^2} \quad &\text{and}& \quad \norms{\zeta^k + Bu^k + v^k}^2 = \SmallOs{1/k^2}, \vspace{1ex}\\
			\norms{\vartheta^k - u^k}^2 = \BigOs{1/k^2} \quad &\text{and}& \quad \norms{\vartheta^k - u^k}^2 = \SmallOs{1/k^2}.
		\end{array}
	\end{equation}
	Now, since $\vartheta^k \in C^{-1}v^k$, we have $v^k \in C\vartheta^k$, or equivalently, $\vartheta^k = J_{\sigma^{-1}C}(\vartheta^k + \sigma^{-1}v^k)$.
	Let us recall the FBS residual from \eqref{eq:FBS_residual3}:
	\begin{equation*}
		\arraycolsep=0.2em
		\begin{array}{lcl}
			\Gc_{\sigma^{-1}}u^k &:=& \sigma \left(u^k - J_{\sigma^{-1}C}(u^k - \sigma^{-1} (\zeta^k + Bu^k))\right) \vspace{1ex}\\
			&=& \sigma (u^k - \vartheta^k) + \sigma\left(\vartheta^k - J_{\sigma^{-1}C}(u^k - \sigma^{-1} (\zeta^k + Bu^k))\right).
		\end{array}
	\end{equation*}
	By the nonexpansiveness of $J_{\sigma^{-1}C}$ and the triangle inequality, we can easily show that
	\begin{equation*}
		\arraycolsep=0.2em
		\begin{array}{lcl}
			\norms{\vartheta^k - J_{\sigma^{-1}C}(u^k - \sigma^{-1} (\zeta^k + Bu^k))} &=& \norms{J_{\sigma^{-1}C}(\vartheta^k + \sigma^{-1}v^k) -  J_{\sigma^{-1}C}(u^k - \sigma^{-1} (\zeta^k + Bu^k))} \vspace{1ex}\\
			&\leq& \norms{\vartheta^k + \sigma^{-1}v^k - u^k + \sigma^{-1} (\zeta^k + Bu^k)} \vspace{1ex}\\
			&\leq& \norms{u^k - \vartheta^k} + \sigma^{-1}\norms{\zeta^k + Bu^k + v^k}.
		\end{array}
	\end{equation*}
	Using this bound and triangle inequality, we can derive that 
	\begin{equation*}
		\arraycolsep=0.2em
		\begin{array}{lcl}
			\norms{\Gc_{\sigma^{-1}}u^k} &\leq& 2\sigma \norms{u^k - \vartheta^k} + \norms{\zeta^k + Bu^k + v^k}.
		\end{array}
	\end{equation*}
	Combining this relation and \eqref{eq:FKM_3OP_convergence_proof2}, we get
	\begin{equation*} 
		\arraycolsep=0.2em
		\begin{array}{ccc}
			\norms{\Gc_{\sigma^{-1}}u^k}^2 = \BigOs{1/k^2} \quad &\text{and}& \quad \norms{\Gc_{\sigma^{-1}}u^k}^2 = \SmallOs{1/k^2}.
		\end{array}
	\end{equation*}
	Finally, let $x^{\star} = [u^{\star}, v^{\star}]$ be a solution of \eqref{eq:3op_inclusion3}.
	By the definitions of $G$ and $T$, we have
	\begin{equation}\label{eq:FKM_3OP_convergence_proof4}
		\arraycolsep=0.2em
		\begin{array}{ccc}
			0 \in Au^{\star} + Bu^{\star} + v^{\star} \quad &\text{and}& \quad 0 \in C^{-1}v^{\star} - u^{\star}.
		\end{array}
	\end{equation}
	From the second inclusion of \eqref{eq:FKM_3OP_convergence_proof4}, we get $v^{\star} \in Cu^{\star}$.
	Substituting this relation into the first inclusion of \eqref{eq:FKM_3OP_convergence_proof4} yields $0 \in Au^{\star} + Bu^{\star} + Cu^{\star}$, which implies that $u^{\star}$ is a solution to \eqref{eq:3op_inclusion}.
	Thus, the convergence of the iterates generated by \eqref{eq:FKM_3OP_scheme} follows directly from part (iii) of Theorem~\ref{th:FKM1_convergence}.
\end{proof}

\begin{theorem}[\textbf{Co-coercive setting}]\label{th:FKM_3OP_convergence2}
	For Inclusion \eqref{eq:3op_inclusion}, suppose that $A$ and $C$ are maximally monotone  and $B$ is $\frac{1}{L_B}$-co-coercive.
	Suppose further that $\zer{A+B+C} \ne \emptyset$ and $t_0$ is given by \eqref{eq:t0_cond2}.
	Let $\sets{(u^k, \hat{u}^k, z_u^k)}$ and $\sets{(v^k, \hat{v}^k, z_v^k)}$ be generated by \eqref{eq:FKM_3OP_scheme2} where the parameters are chosen as
	\begin{equation}\label{eq:FKM_3OP_th2_params}
		\arraycolsep=0.2em
		\begin{array}{c}
			r > 1, \quad 0 < \nu < r - 1, \quad \frac{1}{2}\tau L_B + \tau \sigma < 1, \quad t_{k} := k + t_0, \quad \text{and} \quad \nu_k := \frac{\nu(t_k - 1)}{t_k}.
		\end{array}
	\end{equation}
	Then, the following statements hold. 
	\begin{itemize}
	\itemsep=-0.3em
	\item[$\mathrm{(i)}$] $\mathrm{(\text{\textbf{The $\BigOs{1/k^2}$ and $\SmallOs{1/k^2}$ rates}})}$ 
	For $\Gc_{\sigma^{-1}}$ defined by \eqref{eq:FBS_residual3}, $\norms{\Gc_{\sigma^{-1}}u^k}^2$ converges to zero at the rates of $\BigOs{1/k^2}$ and $\SmallOs{1/k^2}$.
	
	\item[$\mathrm{(ii)}$] $\mathrm{(\text{\textbf{The convergence of the iterates}})}$ 
		The iterate sequences $\sets{u^k}$, $\sets{\hat{u}^k}$, and $\sets{z_u^k}$ all converge  to $u^{\star}$ and the iterate sequences $\sets{v^k}$, $\sets{\hat{v}^k}$, and $\sets{z_v^k}$ all converge to $v^{\star}$, where $u^{\star} \in \zer{A+B+C}$.
		Consequently, all these sequences are bounded.
	\end{itemize}
\end{theorem}

\begin{proof}
	Since $\frac{1}{2}\tau L_B + \tau \sigma < 1$, we have $1 < \frac{2(1 - \tau\sigma)}{\tau L_B} = \frac{2}{L_G}$.
	Thus, all the conditions in \eqref{eq:FKM_3OP_th2_params} are satisfied for $\eta = 1$ by \eqref{eq:FKM_th2_params}.
	The remainder of this proof is very similar to that of Theorem~\ref{th:FKM_3OP_convergence}.
	Thus, we omit it here.
\end{proof}

\beforesec
\section{Distributed FFP Algorithms with ND Stepsizes}\label{sec:ND_DFFP_alg}
\aftersec
The goal of this section is to customize our methods \eqref{eq:FKM_scheme3} and \eqref{eq:FKM_scheme2+} to develop distributed FFP algorithms with network-dependent (ND) stepsize for solving \eqref{eq:DGE}.
However, acceleration is obtained not by directly applying FFP algorithms to \eqref{eq:DGE}, but by first constructing an appropriate primal-dual/three-operator representation in which the network dependence can be separated from the local operator information.

\beforesubsec
\subsection{The reformulation of \eqref{eq:DGE}}\label{subsec:DN_reformulation}
\aftersubsec
Our first step is to reformulate \eqref{eq:DGE} into \eqref{eq:GE}.
First, we recast \eqref{eq:DGE} into the product space as
\begin{equation}\label{eq:GE2}
	0 \in \mbf{G}\mbf{u}^{\star} + \mbf{T}\mbf{u}^{\star} + \mbf{K}^{\top}\Nc_{\sets{\mbf{0}}}(\mbf{K}\mbf{u}^{\star}),
\end{equation}
where $\mbf{K} \in \R^{n\times n}$ such that $\mathrm{Null}(\mbf{K}) = \mathrm{span}\set{\mbf{1}}$.

Next, let $\mbf{v}^{\star} \in \Nc_{\sets{0}}(\mbf{K}\mbf{u}^{\star})$.
Then, we have $\mbf{K}\mbf{u}^{\star} = 0$.
Thus, \eqref{eq:GE2} is equivalent to 
\begin{equation}\label{eq:opt_cond1}
	\arraycolsep=0.2em
	\left\{\begin{array}{lcl}
		0 & \in & \mbf{G}\mbf{u}^{\star} + \mbf{T}\mbf{u}^{\star} + \mbf{K}^{\top}\mbf{v}^{\star}, \vspace{1ex}\\
		0 & = & \mbf{K}\mbf{u}^{\star}.
	\end{array}\right.
\end{equation}
Now, for some $\tau > 0$ and $\sigma > 0$, we define 
\begin{equation}\label{eq:DFKM_components}
\arraycolsep=0.2em
\begin{array}{lcl}
	\mbf{x} := \begin{bmatrix} \mbf{u} \\ \mbf{v} \end{bmatrix}, \quad 
	\mbb{G}\mbf{x} := \begin{bmatrix} \mbf{G}\mbf{u}\\ 0 \end{bmatrix}, \quad 
	\mbb{T}\mbf{x} := \begin{bmatrix} \mbf{T}\mbf{u} + \mbf{K}^{\top}\mbf{v}  \\ -\mbf{K}\mbf{u} \end{bmatrix}, \quad \text{and} \quad
	\mbb{P} := \begin{bmatrix} \frac{1}{\tau}\Id & \mbf{K}^\top \\ \mbf{K} & \frac{1}{\sigma}\Id \end{bmatrix}.
\end{array}
\end{equation}
Then, the system  \eqref{eq:opt_cond1} can be rewritten in the form of \eqref{eq:3op_inclusion3} (a special case of \eqref{eq:GE}), i.e.:
\begin{equation}\label{eq:precond_GE}
	0 \in \mbb{P}^{-1} \mbb{G}\mbf{x}^{\star} + \mbb{P}^{-1} \mbb{T}\mbf{x}^{\star}.
\end{equation}
Finally, to measure the consensus violation of \eqref{eq:DGE}, we define the operator $\Pi := \Id - \frac{\boldsymbol{1}\boldsymbol{1}^\top}{n}$ measuring the consensus violation, i.e.:
\begin{equation*}
	\Pi {\mbf{u}} = {\mbf{u}} - \boldsymbol{1}\bar{u} = \begin{bmatrix}
		{u}_{[1]}^\top - \bar{u}^\top \\
		\vdots \\
		{u}_{[n]}^\top - \bar{u}^\top 
	\end{bmatrix}, \quad \text{where} \quad \bar{u} = \frac{1}{n}\sum_{i=1}^n {{u}}_{[i]}.
\end{equation*}
We can see that if $\Pi {\mbf{u}}^{\star} = 0$, then ${u}^{\star}_{[1]} = \cdots = {u}_{[n]}^{\star} = \bar{u}^{\star}$, i.e., the consensus $\bar{u}^{\star}$ is achieved.

The following lemma from \cite{malitsky2026firstorder} links consensus to a solution of \eqref{eq:GE2}.

\begin{lemma}[\cite{malitsky2026firstorder}, Proposition 7]\label{le:consensus}
$\mbf{u}^{\star} = (u_{[1]}^{\star}, \cdots, u_{[n]}^{\star})^\top$ is a solution of \eqref{eq:GE2} if and only if $u_{[1]}^{\star} = \cdots = u_{[n]}^{\star}$ and $u^{\star} := u_{[1]}^{\star}$ is a solution of \eqref{eq:DGE}.
\end{lemma}

\beforesubsec
\subsection{Derivation of the algorithm}\label{subsec:DN_alg_derivation}
\aftersubsec
Our derivation consists of the following steps.
\begin{compactitem}
\item\textbf{Step 1:} Applying directly \eqref{eq:FKM_scheme3} with  stepsize $\eta = 1$ to \eqref{eq:precond_GE}, we obtain
\begin{equation}\label{eq:DFKM_scheme0}
\arraycolsep=0.2em
\left\{
\begin{array}{lcl}
	\mbf{x}^{k+1} &=& \mbf{y}^k - \hat{\mbf{w}}^{k+1} + \gamma_k \hat{\mbf{w}}^k, \quad \boldsymbol{\xi}^{k+1} \in \mbb{P}^{-1}\mbb{T}\mbf{x}^{k+1}, \vspace{1ex}\\
	\hat{\mbf{w}}^{k+1} &=& \mbb{P}^{-1} \mbb{G}\mbf{y}^k + \boldsymbol{\xi}^{k+1}, \vspace{1ex}\\
	\mbf{y}^{k+1} &=& \mbf{y}^k + \theta_k \left(\mbf{y}^k - \mbf{y}^{k-1} \right) - \hat{\eta}_k\hat{\mbf{w}}^{k+1} + \hat{\gamma}_k \hat{\mbf{w}}^k  - \hat{\lambda}_k \hat{\mbf{w}}^{k-1}.
\end{array}
\right.
\end{equation}
\item\textbf{Step 2:}
Let $\mbf{x}^k :=  [\mbf{u}^k,  \mbf{v}^k]$ and  $\mbf{y}^k := [\hat{\mbf{u}}^k, \hat{\mbf{v}}^k]$. 
Since $\boldsymbol{\xi}^{k+1} \in \mbb{P}^{-1}\mbb{T}\mbf{x}^{k+1}$, we have $\mbb{P}\boldsymbol{\xi}^{k+1} \in \mbb{T}\mbf{x}^{k+1}$. 
From the second line of \eqref{eq:DFKM_scheme0}, we have $\mbb{P}\hat{\mbf{w}}^{k+1} = \mbb{G}\mbf{y}^k + \mbb{P}\boldsymbol{\xi}^{k+1}$.
From the first line of \eqref{eq:DFKM_scheme0}, we can derive that $\mbb{P} \hat{\mbf{w}}^{k+1} = \mbb{P}\mbf{y}^k - \mbb{P}\mbf{x}^{k+1} + \gamma_k \mbb{P}\hat{\mbf{w}}^k$.
Combining these two relations and $\mbb{P}\boldsymbol{\xi}^{k+1} \in \mbb{T}\mbf{x}^{k+1}$, we obtain
\begin{equation*} 
\arraycolsep=0.2em
\begin{array}{lcl}
	\mbb{P}\mbf{y}^k - \mbb{G}\mbf{y}^k + \gamma_k \mbb{P}\hat{\mbf{w}}^k &\in& \mbb{P}\mbf{x}^{k+1} + \mbb{T}\mbf{x}^{k+1}.
\end{array}
\end{equation*}
Using the definition of $\mbb{G}$, $\mbb{T}$, and $\mbb{P}$, we can derive from the last inclusion that
\begin{equation}\label{eq:DFKM_derivation1}
	\arraycolsep=0.2em
	\left\{
	\begin{array}{lcl}
		\frac{1}{\tau} \hat{\mbf{u}}^k + \mbf{K}^\top \hat{\mbf{v}}^k - \mbf{G}\hat{\mbf{u}}^k + \gamma_k (\frac{1}{\tau}\hat{\mbf{w}}_u^k + \mbf{K}^\top \hat{\mbf{w}}_v^k) &\in& \frac{1}{\tau}\mbf{u}^{k+1} + \mbf{T}\mbf{u}^{k+1} + 2 \mbf{K}^\top \mbf{v}^{k+1}, \vspace{1ex}\\
		\mbf{K} \hat{\mbf{u}}^k + \frac{1}{\sigma} \hat{\mbf{v}}^k + \gamma_k (\mbf{K} \hat{\mbf{w}}_u^k + \frac{1}{\sigma}\hat{\mbf{w}}_v^k) &=& \frac{1}{\sigma}\mbf{v}^{k+1}.
	\end{array}
	\right.
\end{equation}
\item\textbf{Step 3:}
Denote $\mbf{s}^k := \hat{\mbf{u}}^k + \gamma_k \hat{\mbf{w}}_u^k$.
Then, we can derive from the second line of \eqref{eq:DFKM_derivation1} that 
\begin{equation*} 
	\arraycolsep=0.2em
	\begin{array}{lcl}	
		\mbf{v}^{k+1} &=& \hat{\mbf{v}}^k + \gamma_k \hat{\mbf{w}}_v^k + \sigma \mbf{K}\mbf{s}^k.
	\end{array}
\end{equation*}
Substituting this into the first line of \eqref{eq:DFKM_derivation1}, then rearranging the result, we obtain
\begin{equation*} 
	\arraycolsep=0.2em
	\begin{array}{lcl}	
		\hat{\mbf{u}}^k - \tau \mbf{K}^\top \hat{\mbf{v}}^k - \tau \mbf{G}\hat{\mbf{u}}^k + \gamma_k (\hat{\mbf{w}}_u^k - \tau \mbf{K}^\top \hat{\mbf{w}}_v^k) -  2 \tau \sigma \mbf{K}^\top \mbf{K}\mbf{s}^k&\in& \mbf{u}^{k+1} + \tau \mbf{T}\mbf{u}^{k+1},
	\end{array}
\end{equation*}
which is equivalent to 
\begin{equation*} 
	\arraycolsep=0.2em
	\begin{array}{lcl}	
		\mbf{u}^{k+1} &=& J_{\tau \mbf{T}} \left(\hat{\mbf{u}}^k - \tau \mbf{K}^\top \hat{\mbf{v}}^k - \tau \mbf{G}\hat{\mbf{u}}^k + \gamma_k (\hat{\mbf{w}}_u^k - \tau \mbf{K}^\top \hat{\mbf{w}}_v^k) -  2\tau \sigma \mbf{K}^\top \mbf{K}\mbf{s}^k\right), \vspace{1ex}\\
		&=& J_{\tau \mbf{T}} \left(\mbf{s}^k - \tau \mbf{K}^\top \hat{\mbf{v}}^k - \tau \mbf{G}\hat{\mbf{u}}^k - \tau \gamma_k  \mbf{K}^\top \hat{\mbf{w}}_v^k -  2 \tau \sigma \mbf{K}^\top \mbf{K}\mbf{s}^k\right).
	\end{array}
\end{equation*}
\item\textbf{Step 4:}
Putting the relations of $\mbf{s}^k$, $\mbf{v}^{k+1}$, $\mbf{u}^{k+1}$, and \eqref{eq:DFKM_scheme0} together, we obtain
\begin{equation}\label{eq:DFKM_derivation2}
	\arraycolsep=0.2em
	\left\{
	\begin{array}{lcl}	
		\mbf{s}^k &:=& \hat{\mbf{u}}^k  + \gamma_k \hat{\mbf{w}}_u^k, \vspace{1ex}\\
		\mbf{v}^{k+1} &:=& \hat{\mbf{v}}^k + \gamma_k \hat{\mbf{w}}_v^k + \sigma \mbf{K}\mbf{s}^k, \vspace{1ex}\\
		\mbf{u}^{k+1} &:=& J_{\tau \mbf{T}} \left(\mbf{s}^k - \tau \mbf{K}^\top \hat{\mbf{v}}^k - \tau \mbf{G}\hat{\mbf{u}}^k - \tau \gamma_k  \mbf{K}^\top \hat{\mbf{w}}_v^k -  2\tau \sigma \mbf{K}^\top \mbf{K}\mbf{s}^k\right), \vspace{1ex}\\
		\hat{\mbf{w}}_u^{k+1} &:=& \hat{\mbf{u}}^k - \mbf{u}^{k+1}  + \gamma_k \hat{\mbf{w}}_u^k, \vspace{1ex}\\
		\hat{\mbf{w}}_v^{k+1} &:=& \hat{\mbf{v}}^k - \mbf{v}^{k+1}  + \gamma_k \hat{\mbf{w}}_v^k, \vspace{1ex}\\
		\hat{\mbf{u}}^{k+1} &:=& \hat{\mbf{u}}^k + \theta_k \left(\hat{\mbf{u}}^k - \hat{\mbf{u}}^{k-1} \right) - \hat{\eta}_k\hat{\mbf{w}}_u^{k+1} + \hat{\gamma}_k \hat{\mbf{w}}_u^k  - \hat{\lambda}_k \hat{\mbf{w}}_u^{k-1}, \vspace{1ex}\\
		\hat{\mbf{v}}^{k+1} &:=& \hat{\mbf{v}}^k + \theta_k \left(\hat{\mbf{v}}^k - \hat{\mbf{v}}^{k-1} \right) - \hat{\eta}_k\hat{\mbf{w}}_v^{k+1} + \hat{\gamma}_k \hat{\mbf{w}}_v^k  - \hat{\lambda}_k \hat{\mbf{w}}_v^{k-1}.
	\end{array}
	\right.
\end{equation}
Substituting $\mbf{v}^{k+1} = \hat{\mbf{v}}^k + \gamma_k \hat{\mbf{w}}_v^k + \sigma \mbf{K}\mbf{s}^k$ into the fifth line of \eqref{eq:DFKM_derivation2}, we get
\begin{equation*} 
	\arraycolsep=0.2em
	\begin{array}{lcl}	
		\hat{\mbf{w}}_v^{k+1} &=& -\sigma \mbf{K}\mbf{s}^k.
	\end{array}
\end{equation*}
Putting this expression into the third line of \eqref{eq:DFKM_derivation2}, we obtain
\begin{equation*} 
	\arraycolsep=0.2em
	\begin{array}{lcl}	
		\mbf{u}^{k+1} &:=& J_{\tau \mbf{T}} \left(\mbf{s}^k - \tau \mbf{K}^\top \hat{\mbf{v}}^k - \tau \mbf{G}\hat{\mbf{u}}^k + \tau \sigma \gamma_k  \mbf{K}^\top \mbf{K}\mbf{s}^{k-1} -  2\tau \sigma \mbf{K}^\top \mbf{K}\mbf{s}^k\right), \vspace{1ex}\\
	\end{array}
\end{equation*}
Using the transformation $\tilde{\mbf{v}}^k := \mbf{K}^\top \hat{\mbf{v}}^k$ and eliminating $\mbf{v}^k$ and $\hat{\mbf{w}}_v^k$ from \eqref{eq:DFKM_derivation2}, we have
\begin{equation*} 
	\arraycolsep=0.2em
	\left\{
	\begin{array}{lcl}	
		\mbf{s}^k &:=& \hat{\mbf{u}}^k  + \gamma_k \hat{\mbf{w}}_u^k, \vspace{1ex}\\
		\mbf{u}^{k+1} &:=& J_{\tau \mbf{T}} \left(\mbf{s}^k - \tau \tilde{\mbf{v}}^k - \tau \mbf{G}\hat{\mbf{u}}^k + \tau \sigma \gamma_k  \mbf{K}^\top \mbf{K}\mbf{s}^{k-1} -  2 \tau \sigma \mbf{K}^\top \mbf{K}\mbf{s}^k\right), \vspace{1ex}\\
		\hat{\mbf{w}}_u^{k+1} &:=& \hat{\mbf{u}}^k - \mbf{u}^{k+1} + \gamma_k \hat{\mbf{w}}_u^k, \vspace{1ex}\\
		\hat{\mbf{u}}^{k+1} &:=& \hat{\mbf{u}}^k + \theta_k \left(\hat{\mbf{u}}^k - \hat{\mbf{u}}^{k-1} \right) - \hat{\eta}_k\hat{\mbf{w}}_u^{k+1} + \hat{\gamma}_k \hat{\mbf{w}}_u^k  - \hat{\lambda}_k \hat{\mbf{w}}_u^{k-1}, \vspace{1ex}\\
		\tilde{\mbf{v}}^{k+1} &:=& \tilde{\mbf{v}}^k + \theta_k \left(\tilde{\mbf{v}}^k - \tilde{\mbf{v}}^{k-1} \right) + \sigma \mbf{K}^\top \mbf{K} (\hat{\eta}_k\mbf{s}^k - \hat{\gamma}_k \mbf{s}^{k-1}  + \hat{\lambda}_k \mbf{s}^{k-2}).
	\end{array}
	\right.
\end{equation*}
\item\textbf{Step 5:}
Observe that $\mbf{K}^\top \mbf{K}$ is only multiplied by $\mbf{s}$.
Therefore, if we define $\mbf{d}^k := \mbf{K}^\top \mbf{K} \mbf{s}^k$, then the last scheme becomes
\begin{equation*} 
	\arraycolsep=0.2em
	\left\{
	\begin{array}{lcl}	
		\mbf{s}^k &:=& \hat{\mbf{u}}^k  + \gamma_k \hat{\mbf{w}}_u^k, \vspace{1ex}\\
		\mbf{d}^k &:=& \mbf{K}^\top \mbf{K} \mbf{s}^k, \vspace{1ex}\\
		\mbf{u}^{k+1} &:=& J_{\tau \mbf{T}} \left(\mbf{s}^k - \tau \tilde{\mbf{v}}^k - \tau \mbf{G}\hat{\mbf{u}}^k + \tau \sigma \gamma_k  \mbf{d}^{k-1} -  2\tau \sigma \mbf{d}^k\right), \vspace{1ex}\\
		\hat{\mbf{w}}_u^{k+1} &:=& \hat{\mbf{u}}^k - \mbf{u}^{k+1}  + \gamma_k \hat{\mbf{w}}_u^k, \vspace{1ex}\\
		\hat{\mbf{u}}^{k+1} &:=& \hat{\mbf{u}}^k + \theta_k \left(\hat{\mbf{u}}^k - \hat{\mbf{u}}^{k-1} \right) - \hat{\eta}_k\hat{\mbf{w}}_u^{k+1} + \hat{\gamma}_k \hat{\mbf{w}}_u^k  - \hat{\lambda}_k \hat{\mbf{w}}_u^{k-1}, \vspace{1ex}\\
		\tilde{\mbf{v}}^{k+1} &:=& \tilde{\mbf{v}}^k + \theta_k \left(\tilde{\mbf{v}}^k - \tilde{\mbf{v}}^{k-1} \right) + \sigma (\hat{\eta}_k\mbf{d}^k - \hat{\gamma}_k \mbf{d}^{k-1}  + \hat{\lambda}_k \mbf{d}^{k-2}).
	\end{array}
	\right.
\end{equation*}
\item\textbf{Step 6:}
We choose the matrix $\mbf{K}$ such that $\mbf{K}^\top \mbf{K} = \frac{\Id - W}{2}$.
In this case, the condition $\sigma\tau\norms{\mbf{K}}^2 < 1$ is equivalent to $\sigma\tau\norms{\Id - W} < 2$.

\item\textbf{Initialization:}
First, starting from $\mbf{u}^0 \in \dom{\mbf{T}}$, take $\boldsymbol{\xi}^0 \in \mbf{T}\mbf{u}^0$ arbitrarily, and set $\mbf{w}_u^0 := \mbf{G}\mbf{u}^0 + \boldsymbol{\xi}^0$.
We need to guarantee that $\mbb{P}\hat{\mbf{w}}^0 \in \mbb{G}\mbf{x}^0 + \mbb{T}\mbf{x}^0$, where $\hat{\mbf{w}}^0 = [\hat{\mbf{w}}_u^0; \hat{\mbf{w}}_v^0]$ and $\mbf{x}^0 = [\mbf{u}^0, \mbf{v}^0]$. 
If we choose $\hat{\mbf{w}}_u^0 := \tau \mbf{w}_u^0$, then this condition implies $\hat{\mbf{w}}_v^0 = \mbf{v}^0$ and $\mbf{v}^0 = -\sigma\mbf{K}(\mbf{u}^0 + \tau \mbf{w}_u^0)$.

Second, since $\mbf{y}^{-1} = \mbf{x}^0$, we have $\hat{\mbf{u}}^{-1} = \mbf{u}^0$ and $\hat{\mbf{v}}^{-1} = \mbf{v}^0$.
Next, since $\mbf{y}^0 = \mbf{x}^0 - (\gamma_0 -  \beta_0)\hat{\mbf{w}}^0$, we have $\hat{\mbf{u}}^0 = \mbf{u}^0 - (\gamma_0-\beta_0)\tau\mbf{w}_u^0$ and $\hat{\mbf{v}}^0 = \mbf{v}^0 - (\gamma_0-\beta_0)\mbf{v}^0 = (1 + \beta_0 - \gamma_0) \mbf{v}^0$.

Third, let us choose $\mbf{d}^{-1} := \frac{1}{2}(\Id - \mbf{W})(\mbf{u}^0 + \tau \mbf{w}_u^0)$.
Then, since $\mbf{K}^\top \mbf{K} = \frac{\Id - \mbf{W}}{2}$, we have $\tilde{\mbf{v}}^{-1} = \mbf{K}^\top \hat{\mbf{v}}^{-1} = \mbf{K}^\top \mbf{v}^0 = -\sigma\mbf{K}^\top \mbf{K}(\mbf{u}^0 + \tau \mbf{w}_u^0) = -\sigma\mbf{d}^{-1}$.
Hence, $\tilde{\mbf{v}}^0 = \mbf{K}^\top \hat{\mbf{v}}^0  = (1 + \beta_0 - \gamma_0) \mbf{K}^\top \mbf{v}^0 = -\sigma(1 + \beta_0 - \gamma_0)\mbf{d}^{-1}$.

Finally, we set $\hat{\mbf{w}}^{-1} = \hat{\mbf{w}}^0$, leading to $\hat{\mbf{w}}_u^{-1} = \hat{\mbf{w}}_u^0 = \tau\mbf{w}_u^0$ and $\hat{\mbf{w}}_v^{-1} = \hat{\mbf{w}}_v^0 = \mbf{v}^0$, and thus $\mbf{d}^{-2} = \mbf{d}^{-1}$.
\end{compactitem}
Summarizing the above derivations, we can write the entire scheme into Algorithm~\ref{alg:ND_DFFP}.

\begin{algorithm}[hpt!]\caption{(Network-Dependent Fast Fixed-Point-Based Algorithm for solving \eqref{eq:DGE})}\label{alg:ND_DFFP}
	\normalsize
	\begin{algorithmic}[1]
		\STATE\label{step:A1_i0}{\bfseries Initialization:} Choose $\mbf{u}^0 \in \dom{\mbf{T}}$, $\boldsymbol{\xi}^0 \in \mbf{T}\mbf{u}^0$ arbitrarily, and a mixing matrix $\mbf{W}$.
		\STATE\hspace{0ex}Choose appropriate stepsizes $\tau > 0$ and $\sigma > 0$ such that $\sigma\tau\norms{\Id - \mbf{W}} < 2$.
		\STATE\hspace{0ex}Set $\mbf{w}_u^0 := \mbf{G}\mbf{u}^0 + \boldsymbol{\xi}^0$, $\hat{\mbf{u}}^{-1} = \mbf{u}^0$, $\hat{\mbf{u}}^0 = \mbf{u}^0 - (1-\beta_0)\tau \mbf{w}_u^0$, where $\beta_0 := \frac{\delta(t_0 - r)}{2(r-1)t_0}$. 
		\STATE\hspace{0ex}Set $\hat{\mbf{w}}_u^{-1} = \hat{\mbf{w}}_u^0 := \tau \mbf{w}_u^0$ and  $\mbf{d}^{-2} = \mbf{d}^{-1} := \frac{1}{2}(\Id - \mbf{W})(\mbf{u}^0 + \tau \mbf{w}_u^0)$. 
		\STATE\hspace{0ex}Set $\tilde{\mbf{v}}^{-1} = - \sigma \mbf{d}^{-1}$, $\tilde{\mbf{v}}^0 := -\sigma \beta_0 \mbf{d}^{-1}$.
		\STATE\hspace{0ex}\label{step:A1_loop}{\bfseries For $k = 0$ to $k_{\max}$ do}
		\vspace{0.25ex}   
		\STATE\hspace{3ex}Update $\gamma_k$, $\theta_k$, $\hat{\eta}_k$, $\hat{\gamma}_k$, and $\hat{\lambda}_k$ according to \eqref{eq:DFKM_params}.
		\STATE\hspace{3ex}Update the iterates as 
		\begin{equation}\label{eq:DFKM_scheme}\tag{ND-DFFP}
			\arraycolsep=0.2em
			\left\{
			\begin{array}{lcl}	
				\mbf{s}^k &:=& \hat{\mbf{u}}^k  + \gamma_k \hat{\mbf{w}}_u^k, \vspace{1ex}\\
				\mbf{d}^k &:=& \frac{1}{2}(\mbf{s}^k - \mbf{W}\mbf{s}^k), \vspace{1ex}\\
				\mbf{u}^{k+1} &:=& J_{\tau \mbf{T}} \left(\mbf{s}^k - \tau \tilde{\mbf{v}}^k - \tau \mbf{G}\hat{\mbf{u}}^k + \tau \sigma \gamma_k  \mbf{d}^{k-1} -  2 \tau \sigma \mbf{d}^k\right), \vspace{1ex}\\
				\hat{\mbf{w}}_u^{k+1} &:=& \hat{\mbf{u}}^k - \mbf{u}^{k+1} + \gamma_k \hat{\mbf{w}}_u^k, \vspace{1ex}\\
				\hat{\mbf{u}}^{k+1} &:=& \hat{\mbf{u}}^k + \theta_k \left(\hat{\mbf{u}}^k - \hat{\mbf{u}}^{k-1} \right) - \hat{\eta}_k\hat{\mbf{w}}_u^{k+1} + \hat{\gamma}_k \hat{\mbf{w}}_u^k  - \hat{\lambda}_k \hat{\mbf{w}}_u^{k-1}, \vspace{1ex}\\
				\tilde{\mbf{v}}^{k+1} &:=& \tilde{\mbf{v}}^k + \theta_k \left(\tilde{\mbf{v}}^k - \tilde{\mbf{v}}^{k-1} \right) + \sigma (\hat{\eta}_k\mbf{d}^k - \hat{\gamma}_k \mbf{d}^{k-1}  + \hat{\lambda}_k \mbf{d}^{k-2}).
			\end{array}
			\right.
		\end{equation}
		\STATE\hspace{0ex}{\bfseries End For}
	\end{algorithmic}
\end{algorithm}

Algorithm~\ref{alg:ND_DFFP} is a distributed version of \eqref{eq:DFKM_scheme0}, where the communication with neighbors only appears at the update of $\mbf{d}^k$ in the second line of \eqref{eq:DFKM_scheme}.
From conditions (ii) and (iv) of Definition~\ref{def:mixing_matrix}, we can show that the matrices $\mbf{K} = \mbf{K}^\top = \left(\frac{\Id - \mbf{W}}{2}\right)^{1/2}$ are well-defined. 
Moreover, we also have $\mathrm{Null}(\mbf{K}) = \mathrm{Null}(\Id - W)$, thus $\mathrm{Null}(\mbf{K}) = \mathrm{span}\sets{\mbf{1}}$.
Note that the stepsizes $\tau$ and $\sigma$ still depend on the mixing matrix $W$ of the network through the condition $\sigma\tau\norms{\Id - W} < 2$.
Therefore, we call this algorithm a \textit{network-dependent fast fixed-point-based algorithm}.

Finally, for implementation purposes, we can explicitly write \eqref{eq:DFKM_scheme} into the following form, where the updates can be executed at each agent $i \in [n]$:
\begin{equation*} 
	\arraycolsep=0.2em
	\left\{
	\begin{array}{lcl}	
		{s}_i^k &:=& \hat{{u}}_i^k  + \gamma_k \hat{{w}}_i^k, \vspace{1ex}\\
		{d}_i^k &:=& \frac{1}{2}({s}_i^k - \sum_{j=1}^n W_{ij} {s}_j^k), \vspace{1ex}\\
		{u}_i^{k+1} &:=& J_{\tau {T}_i} ({s}_i^k - \tau \tilde{{v}}_i^k - \tau {G}_i\hat{{u}}_i^k +  \tau \sigma \gamma_k  {d}_i^{k-1} -  2\tau \sigma {d}_i^k), \vspace{1ex}\\
		\hat{{w}}_i^{k+1} &:=& \hat{{u}}_i^k - {u}_i^{k+1}  + \gamma_k \hat{{w}}_i^k, \vspace{1ex}\\
		\hat{{u}}_i^{k+1} &:=& \hat{{u}}_i^k + \theta_k (\hat{{u}}_i^k - \hat{{u}}_i^{k-1} ) - \hat{\eta}_k\hat{{w}}_i^{k+1} + \hat{\gamma}_k \hat{{w}}_i^k  - \hat{\lambda}_k \hat{{w}}_i^{k-1}, \vspace{1ex}\\
		\tilde{{v}}_i^{k+1} &:=& \tilde{{v}}_i^k + \theta_k (\tilde{{v}}_i^k - \tilde{{v}}_i^{k-1} ) + \sigma (\hat{\eta}_k{d}_i^k - \hat{\gamma}_k {d}_i^{k-1}  + \hat{\lambda}_k {d}_i^{k-2}).
	\end{array}
	\right.
\end{equation*}
Here, $\hat{\mbf{w}}_u^k = [\hat{w}_1^k, \cdots, \hat{w}_n^k]$. Clearly, only the update of $d^k_i$ requires communication with agent $i$'s neighbors.

\beforesubsec
\subsection{Convergence analysis}\label{subsec:ND_case1_convergence_analysis}
\aftersubsec
We consider two cases: (i) $G_i$ is $L$-Lipschitz continuous and $G_i + T_i$ is maximally monotone, and (ii) $G_i$ is $\frac{1}{L}$-co-coercive and $T_i$ is maximally monotone for all $i \in [n]$.

\beforesubsubsec
\subsubsection{The monotone and Lipschitz continuous case}\label{subsubsec:ND_case1}
\aftersubsubsec
We will develop a distributed FFP method to solve \eqref{eq:DGE} under the following assumption.

\begin{assumption}\label{as:DGE_monotone_Lipschitz}
	The distributed composite inclusion \eqref{eq:DGE} satisfies the following conditions: 
	\begin{compactitem}
		\item[(i)] (\textbf{Local maximal monotonicity}) $\Phi_i := G_i + T_i$ is maximally monotone for all $i \in [n]$.
		\item[(ii)] (\textbf{Local Lipschitz continuity}) $G_i$ is $L$-Lipschitz continuous for all $i \in [n]$.
	\end{compactitem}
\end{assumption}
The following lemma transforms Assumption~\ref{as:DGE_monotone_Lipschitz} from $G$ and $T$ to $\mbb{P}^{-1} \mbb{G}$ and $\mbb{P}^{-1} \mbb{T}$, respectively, whose proof can be found in Appendix~\ref{apdx:le:DGE_assumption}. 
\begin{lemma}\label{le:DGE_assumption}
	Suppose that Assumption~\ref{as:DGE_monotone_Lipschitz} holds for \eqref{eq:DGE} and $\tau\sigma\norms{\Id - W} < 2$.  
	Then
	\begin{compactitem}
	\item[$\mathrm{(i)}$] 
	$\mbb{P}^{-1} \mbb{G}$ in \eqref{eq:precond_GE} is $\tilde{L}$-Lipschitz continuous w.r.t.  $\norms{\cdot}_{\mbb{P}}$, where $\tilde{L} := \frac{2\tau L}{2 - \tau \sigma \norms{\Id - W}}$.
	
	\item[$\mathrm{(ii)}$] 
	$\mbb{P}^{-1} \mbb{G} + \mbb{P}^{-1} \mbb{T}$ in \eqref{eq:precond_GE} is maximally monotone w.r.t. $\iprods{\cdot, \cdot}_{\mbb{P}}$.
	\end{compactitem}
\end{lemma}

Now, we can state the convergence of \eqref{eq:DFKM_scheme} in the following theorem.

\begin{theorem}[\textbf{Monotone and Lipschitz continuous setting}]\label{th:DFKM_convergence}
	For \eqref{eq:DGE}, suppose that Assumption~\ref{as:DGE_monotone_Lipschitz} holds. 
	Let $\sets{ (\mbf{u}^k, \tilde{\mbf{v}}^k) }$ be generated by \eqref{eq:DFKM_scheme} using $t_0$ given in \eqref{eq:FKM_key_est2_t0} and
	\begin{equation}\label{eq:DFKM_params}
		\arraycolsep=0.2em
		\hspace{-2ex}
		\begin{array}{c}
			r > 2, \qquad \delta \in (0, r-2), \qquad \omega := r^2 - 3r + 3, \qquad 0 < \nu < r-2-\delta,
			\vspace{1ex}\\
			2\tau L\sqrt{1+\omega} + \frac{1}{2}\tau \sigma (1 - \lambda_{\min}(W)) \leq 1, 
			\qquad 
			t_{k} := k + t_0, 
			\qquad 
			\gamma_k := 1, \vspace{1ex}\\
			\hat{\eta}_k := \frac{[2(r-1)-\delta](t_{k+1} - r)}{2(r-1)t_{k+1}} + \frac{\nu(t_k - 1)}{t_k t_{k+1}} + 1, \quad
			\hat{\gamma}_k := \frac{[4(r-1) - \delta] (t_k - r)}{2(r-1)t_{k+1}} + 1, \  \text{and} \
			\hat{\lambda}_k = \theta_k := \frac{t_k - r}{t_{k+1}}.
		\end{array}
		\hspace{-2ex}
	\end{equation}
	Then, the following statements hold.
	\begin{compactitem}
	\item[$~~~\mathrm{(i)}$~\textbf{$($Local iterates$)$}]
	The local iterates $\sets{u_{[i]}^k}$ and $\sets{\hat{u}_{[i]}^k}$ converge to the same $u^\star \in \zer{\Phi}$ of the original \eqref{eq:DGE} problem for all nodes $i \in [n]$.
	
	\item[$~~~\mathrm{(ii)}$~\textbf{$($Consensus errors$)$}]
	The consensus errors $\norms{\Pi \mbf{u}^k }$ and $\norms{\Pi \hat{\mbf{u}}^k }$ converge to zero at the rate of $\SmallOs{1/k}$.
	
	\item[$~~~\mathrm{(iii)}$~\textbf{$($Aggregated residual$)$}]
	The aggregated residual $\norms{\sum_{i=1}^n G_i u_{[i]}^k + \xi_i^k}$ converges to zero at the rate of $\SmallOs{1/k}$, where $\xi_i^k \in Tu_{[i]}^k$.
	
	\item[$~~~\mathrm{(iv)}$~\textbf{$($Restricted gap$)$}]
	If, in addition, Assumption~\ref{as:restricted_dual_gap} holds, then the restricted gap value $\mathrm{Gap}_B(\bar{u}^k)$ converges to zero at the rate of $\SmallOs{1/k}$, where $\bar{u}^k := \frac{1}{n}\sum_{i=1}^n u_{[i]}^k$.
	Otherwise, if, in addition, Assumption~\ref{as:boundary_regularity} holds, then $\mathrm{Gap}_B(p^k)$ converges to zero at an $\SmallOs{1/k}$ convergence rate, where $p^k := \proj_{\Bc}(\bar{u}^k)$.
	
	\item[$~~~\mathrm{(v)}$~\textbf{$($FBS residual$)$}]
	For a given $\lambda \in (0, \frac{1}{L})$, if, in addition, Assumption~\ref{as:restricted_dual_gap} holds, then $\norms{\Gc_{\lambda}\bar{u}^k}^2$ converges to zero at an $\SmallOs{1/k}$ convergence rate.
	Otherwise, if, in addition, Assumption~\ref{as:boundary_regularity} holds, then $\norms{\Gc_{\lambda}p^k}^2$ converges to zero at an $\SmallOs{1/k}$ convergence rate, where $p^k := \proj_{\Bc}(\bar{u}^k)$.

	\item[$~~~\mathrm{(vi)}$~\textbf{$($Special case$)$}]
	If $T_i = 0$ for all $i \in [n]$ $($i.e., \eqref{eq:DGE} reduces to \eqref{eq:DNE}$)$, then $\norms{\sum_{i=1}^n G_i \bar{u}^k}$ converges to zero at the rate of $\SmallOs{1/k}$.
	\end{compactitem}
\end{theorem}

For the norm-valued quantities in (ii), (iii), and (vi), the corresponding squared norms converge at the rates $\BigOs{1/k^2}$ and $\SmallOs{1/k^2}$, respectively. 
Item (v) already concerns the squared FBS residual.

\begin{proof}
	First, we can show that
	\begin{equation*} 
	\arraycolsep=0.2em
	\begin{array}{lcl}	
		\norms{\mbf{K}}^2 = \norms{\mbf{K}^2} = \frac{1}{2}\norms{\Id - W} = \frac{1}{2}\lambda_{\max}(\Id - W) = \frac{1 - \lambda_{\min}(W)}{2} > 0,
	\end{array}
	\end{equation*}
	where we have used $\Id \succeq W$ from Definition~\ref{def:mixing_matrix} in the last inequality.
	Using this relation, the condition $2\sqrt{1+\omega}\tau L + \frac{1}{2}\tau \sigma (1 - \lambda_{\min}(W)) \leq 1$ is equivalent to $2\sqrt{1+\omega}\tau L + \tau \sigma \norms{\mbf{K}}^2 \leq 1$, leading to $\tau \sigma \norms{\mbf{K}}^2 = \frac{1}{2}\sigma\tau\norms{\Id - W} < 1$, and thus Lemma~\ref{le:DGE_assumption} holds.
	Moreover, we can easily show that the conditions in Corollary~\ref{co:FKM_scheme3_convergence} hold for $\eta = 1$.
	Therefore, all the results from Corollary~\ref{co:FKM_scheme3_convergence} (or Theorem~\ref{th:FKM1_convergence}) can be applied to Algorithm~\ref{alg:ND_DFFP}.
	
	\vspace{0.75ex}
	\noindent$\mathrm{(i)}$~\textbf{The convergence of the local iterate sequences.}
	By Theorem~\ref{th:FKM1_convergence}(iii), both sequences $\sets{\mbf{x}^k}$ and $\sets{\mbf{y}^k}$ converge to $\mbf{x}^{\star} = [\mbf{u}^{\star}, \mbf{v}^{\star}]$, where $\mbf{u^{\star}}$ is a solution of \eqref{eq:GE2}.
	By Lemma~\ref{le:consensus}, $\mbf{u}^{\star} = [u^{\star}, \cdots, u^{\star}]^\top$ and $u^{\star}$ is a solution to \eqref{eq:DGE}.
	Since $\mbf{x}^k \to \mbf{x}^{\star}$ and $\mbf{y}^k \to \mbf{x}^{\star}$ as $k \to \infty$, we obtain $\mbf{u}^k \to \mbf{u}^{\star}$ and $\hat{\mbf{u}}^k \to \mbf{u}^{\star}$, thus $u^k_{[i]} \to u^{\star}$ and $\hat{u}^k_{[i]} \to u^{\star}$ for all $i \in [n]$.
	
	\vspace{0.75ex}
	\noindent$\mathrm{(ii)}$~\textbf{The convergence of consensus errors.}
	For $\mathbbm{w}^k \in \mbb{G}\mbf{x}^k + \mbb{T}\mbf{x}^k$, from Theorem~\ref{th:FKM1_convergence}, we have $\norms{\mbb{P}^{-1}\mathbbm{w}^k}_{\mbb{P}} = \BigOs{1/k}$ and $\norms{\mbb{P}^{-1}\mathbbm{w}^k}_{\mbb{P}} = \SmallOs{1/k}$.
	By definition, we get
	\begin{equation*}
	\arraycolsep=0.2em
	\begin{array}{lcl}	
		\norms{\mbb{P}^{-1}\mathbbm{w}^k}_{\mbb{P}}^2 &=& \iprods{\mathbbm{w}^k, \mbb{P}^{-1}\mathbbm{w}^k} \geq \lambda_{\min}(\mbb{P}^{-1})\norms{\mathbbm{w}^k}^2 = \frac{1}{\lambda_{\max}(\mbb{P})} \norms{\mathbbm{w}^k}^2,
	\end{array}
	\end{equation*}
	leading to $\norms{\mathbbm{w}^k}^2 \leq \lambda_{\max}(\mbb{P})\norms{ \mbb{P}^{-1}\mathbbm{w}^k }_{\mbb{P}}^2$.
	Thus, we conclude that $\norms{\mathbbm{w}^k} = \BigOs{1/k}$ and $\norms{\mathbbm{w}^k} = \SmallOs{1/k}$.
	However, since $\norms{\mathbbm{w}^k}^2 = \norms{\mbf{G}\mbf{u}^k + \boldsymbol{\xi}^k + \mbf{K}^\top \mbf{v}^k}^2 + \norms{\mbf{K}\mbf{u}^k}^2$ for $\boldsymbol{\xi}^k \in \mbf{T}\mbf{u}^k$, we also get
	\begin{equation}\label{eq:DFKM_th1_proof1}
	\arraycolsep=0.2em
	\begin{array}{rcl}	
		\norms{\mbf{G}\mbf{u}^k + \boldsymbol{\xi}^k + \mbf{K}^\top \mbf{v}^k} = \BigOs{1/k} 
		\qquad &\text{and}& \qquad 
		\norms{\mbf{G}\mbf{u}^k + \boldsymbol{\xi}^k + \mbf{K}^\top \mbf{v}^k} = \SmallOs{1/k}, \vspace{1ex}\\
		\norms{\mbf{K}\mbf{u}^k} = \BigOs{1/k} 
		\qquad & \text{and} & \qquad 
		\norms{\mbf{K}\mbf{u}^k} = \SmallOs{1/k}.
	\end{array}
	\end{equation}
	By a similar argument as above, we can show that $\norms{\hat{\mathbbm{w}}^k} = \BigOs{1/k}$ and $\norms{\hat{\mathbbm{w}}^k} = \SmallOs{1/k}$.
	However, since $\norms{\hat{\mathbbm{w}}_u^k} \leq \norms{\hat{\mathbbm{w}}^k}$, we also have $\norms{\hat{\mathbbm{w}}_u^k} = \BigOs{1/k}$ and $\norms{\hat{\mathbbm{w}}_u^k} = \SmallOs{1/k}$.
	Then, we can derive from the fourth line of \eqref{eq:DFKM_scheme} that
	\begin{equation*}
	\arraycolsep=0.2em
	\begin{array}{rcl}	
		\norms{\mbf{K}\hat{\mbf{u}}^k} = \norms{\mbf{K}\mbf{u}^{k+1} + \mbf{K}(\hat{\mbf{w}}_u^{k+1} - \hat{\mbf{w}}_u^k)} \leq \norms{\mbf{K}\mbf{u}^{k+1}} + \norms{\mbf{K}}\big(\norms{\hat{\mbf{w}}_u^{k+1}} + \norms{\hat{\mbf{w}}_u^k}\big)
	\end{array}
	\end{equation*}
	Combining this bound with the above results, we obtain 
	\begin{equation*}
		\arraycolsep=0.2em
		\begin{array}{rcl}				
			\norms{\mbf{K}\hat{\mbf{u}}^k} = \BigOs{1/k} 
			\qquad &\text{and}& \qquad 
			\norms{\mbf{K}\hat{\mbf{u}}^k} = \SmallOs{1/k}.
		\end{array}
	\end{equation*}
	Now, we can write $\mbf{u}^k = (\mbf{u}^k - \boldsymbol{1}\bar{u}^k) + \boldsymbol{1}\bar{u}^k$, where $\bar{u}^k := \frac{1}{n} \sum_{i=1}^n u_{[i]}^k$.
	Take an arbitrary $\mbf{s} = (s^\top, \cdots, s^\top)^\top \in \mathrm{Null}(\mbf{K})$, then we have
	\begin{equation*}
		\iprods{\mbf{u}^k - \boldsymbol{1}\bar{u}^k, \mbf{s}} 
		= \sum_{i=1}^n \iprods{u_{[i]}^k - \bar{u}^k, s} 
		= \iprod{\sum_{i=1}^n (u_{[i]}^k - \bar{u}^k), s} 
		= \iprods{0, s} = 0.
	\end{equation*}
	Thus, $\mbf{u}^k - \boldsymbol{1}\bar{u}^k \in \mathrm{Null}(\mbf{K})^{\perp}$.
	However,  one can easily prove that for any $\mbf{s} \in \mathrm{Null}(\mbf{K})^{\perp}$, we have $\norms{\mbf{K}\mbf{s}} \geq \sigma_{\min}^+(\mbf{K})\norms{\mbf{s}}$, where $\sigma_{\min}^+(\mbf{K}) > 0$ is the smallest positive singular value of $\mbf{K}$.
	Setting $\mbf{s} := \mbf{u}^k - \boldsymbol{1}\bar{u}^k$, we obtain
	\begin{equation*}
		\norms{\Pi\mbf{u}^k}^2 = \norms{\mbf{u}^k - \boldsymbol{1}\bar{u}^k}^2 \leq \frac{1}{\sigma_{\min}^+(\mbf{K})^2}\norms{\mbf{K}(\mbf{u}^k - \boldsymbol{1}\bar{u}^k)}^2 = \frac{1}{\sigma_{\min}^+(\mbf{K})^2}\norms{\mbf{K}\mbf{u}^k}^2,
	\end{equation*}
	where the last equality holds because $\mbf{K}\boldsymbol{1}\bar{u}^k = 0$ due to $\boldsymbol{1}\bar{u}^k \in \mathrm{Null}(\mbf{K})$.
	Thus, we conclude that the consensus error also converges to zero at the rates
	\begin{equation*}
	\arraycolsep=0.2em
	\begin{array}{rcl}	
		\norms{\Pi\mbf{u}^k} = \BigOs{1/k} \qquad &\text{and}& \qquad \norms{\Pi\mbf{u}^k}  = \SmallOs{1/k}.
	\end{array}
	\end{equation*}
	Moreover, since $\norms{\Pi\mbf{u}^k}^2 = \norms{\mbf{u}^k - \boldsymbol{1}\bar{u}^k}^2 = \sum_{i=1}^n \norms{u_{[i]}^k - \bar{u}^k}^2$, we also have
	\begin{equation*}
		\arraycolsep=0.2em
		\begin{array}{rcl}	
			\norms{u_{[i]}^k - \bar{u}^k} = \BigOs{1/k} \quad &\text{and}& \quad \norms{u_{[i]}^k - \bar{u}^k} = \SmallOs{1/k} \quad \text{for all $i \in [n]$}.
		\end{array}
	\end{equation*}
	The consensus of $\hat{\mbf{u}}^k$ can also be obtained similarly:
	\begin{equation*}
	\arraycolsep=0.2em
	\begin{array}{rcl}	
		\norms{\Pi\hat{\mbf{u}}^k} = \BigOs{1/k} \qquad &\text{and}& \qquad \norms{\Pi\hat{\mbf{u}}^k} = \SmallOs{1/k}.
	\end{array}
	\end{equation*}
	 \vspace{0.75ex}
	\noindent$\mathrm{(iii)}$~\textbf{The convergence of aggregated residuals.}
	Let us denote by $\mbf{w}_u := \mbf{G}\mbf{u}^k + \boldsymbol{\xi}^k + \mbf{K}^\top \mbf{v}^k$ for $\boldsymbol{\xi}^k \in \mbf{T}\mbf{u}^k$.
	Then, since $\boldsymbol{1}^\top \mbf{K}^\top = (\mbf{K} \boldsymbol{1}) ^\top = 0$, we have
	\begin{equation*}
	\arraycolsep=0.2em
	\begin{array}{rcl}	
		\boldsymbol{1}^\top \mbf{w}_u = \boldsymbol{1}^\top (\mbf{G}\mbf{u}^k + \boldsymbol{\xi}^k) + \boldsymbol{1}^\top \mbf{K}^\top \mbf{v}^k = \sum_{i=1}^n (G_i u_{[i]}^k + \xi_i^k),
	\end{array}
	\end{equation*}
	for $\xi_i^k \in T_i u_{[i]}^k$.
	Thus, by the Cauchy-Schwarz inequality, we can show that
	\begin{equation*}
	\arraycolsep=0.2em
	\begin{array}{rcl}	
		\norms{\sum_{i=1}^n (G_i u_{[i]}^k + \xi_i^k)}^2 = \norms{\boldsymbol{1}^\top \mbf{w}_u}^2 \leq \norms{\boldsymbol{1}}^2 \norms{\mbf{w}_u}^2 = n\norms{\mbf{w}_u}^2 = n \norms{\mbf{G}\mbf{u}^k + \boldsymbol{\xi}^k + \mbf{K}^\top \mbf{v}^k}^2.
	\end{array}
	\end{equation*}
	Combining this relation and the first line of \eqref{eq:DFKM_th1_proof1}, we obtain
	\begin{equation*}
	\arraycolsep=0.2em
	\begin{array}{rcl}	
		\norms{\sum_{i=1}^n (G_i u_{[i]}^k + \xi_i^k)} = \BigOs{1/k} 
		\qquad &\text{and}& \qquad 
		\norms{\sum_{i=1}^n (G_i u_{[i]}^k + \xi_i^k)} = \SmallOs{1/k}.
	\end{array}
	\end{equation*}
	\vspace{0.75ex}
	\noindent$\mathrm{(iv)}$~\textbf{The convergence of the restricted gap values.}
	We consider two cases as follows.
	
	\noindent\textbf{Case $\mathrm{(i)}$.}~
	Under Assumption~\ref{as:restricted_dual_gap},
	let us denote by $\mathrm{Gap}_{\Bc}(\bar{u}^k) := \sup_{u \in {\Bc}} \Gc(u, \bar{u}^k)$, 
	where $\Gc(u, \bar{u}^k) := \sum_{i=1}^n \iprods{G_i u + \xi_i, \bar{u}^k - u}$. 
	By inserting $u_{[i]}^k$, we obtain
	\begin{equation*} 
	\arraycolsep=0.2em
	\begin{array}{lcl}
	\Gc(u, \bar{u}^k) = \sum_{i=1}^n \iprods{G_i u + \xi_i, u_{[i]}^k - u} + \sum_{i=1}^n \iprods{G_i u + \xi_i, \bar{u}^k - u_{[i]}^k}.
	\end{array}
	\end{equation*}
	Since $G_i + T_i$ is monotone, for ${\xi}_i^k \in T_i {u}_i^k$, we have $\iprods{G_i u + \xi_i, {u}_i^k - u} \leq \iprods{G_i {u}_i^k + {\xi}_i^k, {u}_i^k - u}$.
	Substituting this relation into the last inequality, we get
	\begin{equation*}
	\arraycolsep=0.2em
	\begin{array}{lcl}
		\Gc(u, \bar{u}^k) &\leq& \sum_{i=1}^n \iprods{G_i {u}_i^k + {\xi}_i^k, {u}_i^k - u} + \sum_{i=1}^n \iprods{G_i u + \xi_i, \bar{u}^k - u_i^k} \vspace{1ex}\\
		&=& \sum_{i=1}^n \iprods{G_i {u}_i^k + {\xi}_i^k, {u}_i^k - \bar{u}^k} + \sum_{i=1}^n \iprods{G_i {u}_i^k + {\xi}_i^k, \bar{u}^k - u} + \sum_{i=1}^n \iprods{G_i u + \xi_i, \bar{u}^k - u_i^k} \vspace{1ex}\\
		&=& \iprod{\sum_{i=1}^n (G_i {u}_i^k + {\xi}_i^k), \bar{u}^k - u} + \sum_{i=1}^n \iprod{(G_i {u}_i^k + {\xi}_i^k) - (G_i u + \xi_i), {u}_i^k - \bar{u}^k} \vspace{1ex}\\
		&\leq& \norms{\sum_{i=1}^n (G_i {u}_i^k + {\xi}_i^k)} \norms{\bar{u}^k - u} + \sum_{i=1}^n \left( \norms{G_i {u}_i^k + {\xi}_i^k} + \norms{G_i u + \xi_i} \right) \norms{{u}_i^k - \bar{u}^k}.
	\end{array}
	\end{equation*}
	Taking supremum on both sides of this inequality over ${\Bc}$, we obtain
	\begin{equation*}
	\arraycolsep=0.2em
	\begin{array}{lcl}
		\mathrm{Gap}_B(\bar{u}^k) &\leq& \norms{\sum_{i=1}^n (G_i {u}_i^k + {\xi}_i^k)} \sup_{u \in B} \norms{\bar{u}^k - u} \vspace{1ex}\\
		&& + {~} \sum_{i=1}^n \sup_{u \in B} \left( \norms{G_i {u}_i^k + {\xi}_i^k} + \norms{G_i u + \xi_i} \right) \norms{{u}_i^k - \bar{u}^k}.
	\end{array}
	\end{equation*}
	Since $\sets{u_i^k}$ converges to $u^{\star}$ for all $i \in [n]$, $\bar{u}^k = \frac{1}{n}\sum_{i=1}^n u_i^k$ converges to $u^{\star}$, and thus there exists $M_1 > 0$ such that $\sup_{u \in {\Bc}} \norms{\bar{u}^k - u} \leq M_1$.
	Next, we have $\mbf{G}\mbf{u}^k + \boldsymbol{\xi}^k = (\mbf{G}\mbf{u}^k + \boldsymbol{\xi}^k + \mbf{K}^\top \mbf{v}^k) - \mbf{K}^\top \mbf{v}^k$, where the first term $\mbf{G}\mbf{u}^k + \boldsymbol{\xi}^k + \mbf{K}^\top \mbf{v}^k$ converges to zero as proved in part (b) and the second term $\mbf{K}^\top \mbf{v}^k$ converges to $\mbf{K}^\top \mbf{v}^{\star}$, leading to the boundedness of $\mbf{G}\mbf{u}^k + \boldsymbol{\xi}^k$.
	Thus, there exists $M_2 > 0$ such that $\norms{G_i {u}_i^k + {\xi}_i^k} \leq M_2$ for all $i \in [n]$.
	Finally, since ${\Bc} \subset \mathrm{int}(\dom{T_i})$, by Lemma~\ref{le:monotone_local_bounded}, there exists $M_3 > 0$ such that $\norms{G_i u + \xi_i} \leq M_3$ for all $i = 1, \cdots, n$.
	Substituting these bounds into the last inequality, we get
	\begin{equation*}
	\arraycolsep=0.2em
	\begin{array}{lcl}
		\mathrm{Gap}_{\Bc}(\bar{u}^k) &\leq& M_1 \norms{\sum_{i=1}^n (G_i {u}_i^k + {\xi}_i^k)} + \sum_{i=1}^n (M_2 + M_3) \norms{{u}_i^k - \bar{u}^k}.
	\end{array}
	\end{equation*}
	Combining this relation and the results in part (ii) and (iii), we conclude that 
	\begin{equation*}
	\arraycolsep=0.2em
	\begin{array}{rcl}	
		\mathrm{Gap}_{\Bc}(\bar{u}^k) = \BigOs{1/k} \qquad &\text{and}& \qquad \mathrm{Gap}_B(\bar{u}^k) = \SmallOs{1/k}.
	\end{array}
	\end{equation*}
	\noindent\textbf{Case~$\mathrm{(ii)}$.}
	Under Assumption~\ref{as:boundary_regularity}, we have $\dom{T_i} = \Cc_i$, and thus $\Bc = \bigcap_{i=1}^n \Cc_i$ is nonempty, compact, and convex, and $\bigcap_{i=1}^n \ri{\Cc_i} \ne \emptyset$.
	By \cite[Corollary 3]{bauschke1999strong}, the collection $\sets{\Cc_1, \cdots, \Cc_n}$ is boundedly linear regular, i.e., for every bounded set $\Qc \subseteq \R^p$, there exists $\kappa_{\Qc} > 0$ such that 
	\begin{equation*}
		\dist{u, \Bc} \leq \kappa_{\Qc} \max_{i=1, \cdots, n} \dist{u, \Cc_i}, \qquad u \in \Qc.
	\end{equation*}
	Since $\bar{u}^k \to u^{\star}$, $\sets{\bar{u}^k}_{k \geq 0}$ is bounded.
	Taking $\Qc := \sets{\bar{u}^k}_{k \geq 0}$.
	Then, there exists $\kappa > 0$ such that 
	\begin{equation*}
	\arraycolsep=0.2em
	\begin{array}{rcl}	
		\dist{\bar{u}^k, \Bc} \leq \kappa \max\limits_{i=1, \cdots, n} \dist{\bar{u}^k, \Cc_i}, \qquad \forall k \geq 0.
	\end{array}
	\end{equation*}
	Next, let $p^k := \proj_{\Bc}(\bar{u}^k)$, then we have $\norms{p^k - \bar{u}^k} = \dist{\bar{u}^k, \Bc}$.
	Since $u_{[i]}^k \in \dom{T_i} = \Cc_i$, we obtain from the last inequality that
	\begin{equation*}
	\arraycolsep=0.2em
	\begin{array}{rcl}	
		\norms{p^k - \bar{u}^k} \leq \kappa \max\limits_{i=1, \cdots, n} \dist{\bar{u}^k, \Cc_i} 
		\leq \kappa \max\limits_{i=1, \cdots, n} \norms{\bar{u}^k - u_{[i]}^k} 
		= \kappa \norms{\Pi\mbf{u}^k}_{\infty} \leq \kappa \norms{\Pi\mbf{u}^k}.
	\end{array}
	\end{equation*}
	Using this relation, triangle inequality, and Cauchy-Schwarz inequality, we can derive that
	\begin{equation*}
	\arraycolsep=0.2em
	\begin{array}{rcl}	
		\sum_{i=1}^n \norms{p^k - u_{[i]}^k} &\leq& \sum_{i=1}^n \big(\norms{p^k - \bar{u}^k} + \norms{\bar{u}^k - u_{[i]}^k}\big) \vspace{1ex}\\
		&=& n\norms{p^k - \bar{u}^k} + \sum_{i=1}^n \norms{\bar{u}^k - u_{[i]}^k} \vspace{1ex}\\
		&\leq& n\norms{p^k - \bar{u}^k} + \sqrt{n}\big(\sum_{i=1}^n \norms{\bar{u}^k - u_{[i]}^k}^2\big)^{1/2} \vspace{1ex}\\
		&\leq& (n\kappa + \sqrt{n})\norms{\Pi \mbf{u}^k}.
	\end{array}
	\end{equation*}
	Now, fix arbitrary $u \in \Bc$ and arbitrary $\xi_i := s_i + n_i \in T_iu$, where $s_i \in S_iu$ and $n_i \in \Nc_{\Cc_i}u$.
	Then, since $p^k \in \Bc$, we have $\iprods{n_i, p^k - u} \leq 0$, leading to
	\begin{equation}\label{eq:DFKM_th1_proof2}
	\arraycolsep=0.2em
	\begin{array}{lcl}	
		\iprods{G_iu + \xi_i, p^k - u} &=& \iprods{G_iu + s_i, p^k - u} + \iprods{n_i, p^k - u} \vspace{1ex}\\
		&\leq& \iprods{G_iu + s_i, p^k - u_{[i]}^k} + \iprods{G_iu + s_i, u_{[i]}^k - u}.
	\end{array}
	\end{equation}
	Since $0 \in \Nc_{\Cc_i}(u)$ for all $u \in \Cc_i$, $i = 1, \cdots, n$, we have $\gra{G_i + S_i} \subseteq \gra{\Phi_i}$.
	Thus, the monotonicity of $\Phi_i$ immediately implies that $G_i + S_i$ is monotone.
	Moreover, since $u \in \Bc \subseteq \Cc_i$ and $n_i^k \in \Nc_{\Cc_i}u_{[i]}^k$, we have $\iprods{n_i^k, u_{[i]}^k - u} \geq 0$.
	Using this relation and the monotonicity of $G_i + S_i$, we have
	\begin{equation}\label{eq:DFKM_th1_proof3}
	\arraycolsep=0.2em
	\begin{array}{lcl}	
		\iprods{G_i u_{[i]}^k + \xi_i^k, u_{[i]}^k - u} &=& \iprods{G_i u_{[i]}^k + s_i^k, u_{[i]}^k - u} + \iprods{n_i^k, u_{[i]}^k - u} \vspace{1ex}\\
		&\geq& \iprods{G_i u_{[i]}^k + s_i^k, u_{[i]}^k - u} \vspace{1ex}\\
		&\geq& \iprods{G_iu + s_i, u_{[i]}^k - u}.
	\end{array}
	\end{equation}
	From \eqref{eq:DFKM_th1_proof1}, since $\norms{\mbf{G}\mbf{u}^k + \boldsymbol{\xi}^k + \mbf{K}^\top \mbf{v}^k} \to 0$ and $\mbf{v}^k \to \mbf{v}^{\star}$, we obtain
	\begin{equation*}
	\arraycolsep=0.1em
	\begin{array}{lcl}	
		\norms{\mbf{G}\mbf{u}^k + \boldsymbol{\xi}^k} &=& \norms{\mbf{G}\mbf{u}^k + \boldsymbol{\xi}^k + \mbf{K}^\top \mbf{v}^k - \mbf{K}^\top \mbf{v}^k}  
		\leq  \norms{\mbf{G}\mbf{u}^k + \boldsymbol{\xi}^k + \mbf{K}^\top \mbf{v}^k} + \norms{\mbf{K}} \norms{\mbf{v}^k} < M_1 < +\infty.
	\end{array}
	\end{equation*}
	Since $G_i u_{[i]}^k + \xi_i^k$ is the $i$-th block of $\mbf{G}\mbf{u}^k + \boldsymbol{\xi}^k$, we also have 
	\begin{equation*}
	\arraycolsep=0.2em
	\begin{array}{lcl}	
		\norms{G_i u_{[i]}^k + \xi_i^k} \leq \norms{\mbf{G}\mbf{u}^k + \boldsymbol{\xi}^k} < M_1 < +\infty \qquad \forall i = 1, \cdots, n.
	\end{array}
	\end{equation*}
	Since $\Bc$ is compact, we have $D_{\Bc} := \sup_{u, u' \in \Bc} \norms{u - u'} < +\infty$.
	Then, since $p^k, u \in \Bc$, we have $\norms{p^k - u} \leq D_{\Bc}$.
	
	Using the continuity of $G_i$ (due to its Lipschitz continuity), the boundedness of $u \in \Bc$, and \eqref{eq:S_bounded}, there exists $M_2 > 0$ such that $\norms{G_iu + s_i} < M_2 < +\infty$.
	Substituting \eqref{eq:DFKM_th1_proof3} into \eqref{eq:DFKM_th1_proof2} and then using the Cauchy-Schwarz inequality and the last relation, we can show that
	\begin{equation*}
	\arraycolsep=0.2em
	\begin{array}{lcl}	
		\iprods{G_iu + \xi_i, p^k - u} &\leq& \norms{G_iu + s_i}\norms{p^k - u_{[i]}^k} + \iprods{G_i u_{[i]}^k + \xi_i^k, u_{[i]}^k - u} \vspace{1ex}\\
		&\leq& M_2\norms{p^k - u_{[i]}^k}  + \iprods{G_i u_{[i]}^k + \xi_i^k, u_{[i]}^k - u}.
	\end{array}
	\end{equation*}
	Summing the last inequality from $i = 1$ to $n$, we get
	\begin{equation*}
	\arraycolsep=0.2em
	\begin{array}{lcl}	
		\sum_{i=1}^n \iprods{G_iu + \xi_i, p^k - u} &\leq& \sum_{i=1}^n \iprods{G_i u_{[i]}^k + \xi_i^k, u_{[i]}^k - u} + M_2 \sum_{i=1}^n \norms{p^k - u_{[i]}^k} \vspace{1ex}\\
		&=& \sum_{i=1}^n \iprods{G_i u_{[i]}^k + \xi_i^k, u_{[i]}^k - p^k} +  \iprods{\sum_{i=1}^n (G_i u_{[i]}^k + \xi_i^k), p^k - u} \vspace{1ex}\\
		&& + {~} M_2 \sum_{i=1}^n \norms{p^k - u_{[i]}^k} \vspace{1ex}\\
		&\leq& \sum_{i=1}^n \norms{G_i u_{[i]}^k + \xi_i^k} \norms{u_{[i]}^k - p^k} + \norms{\sum_{i=1}^n (G_i u_{[i]}^k + \xi_i^k)} \norms{p^k - u} \vspace{1ex}\\
		&& + {~} M_2 \sum_{i=1}^n \norms{p^k - u_{[i]}^k} \vspace{1ex}\\
		&\leq& (M_1 + M_2) \sum_{i=1}^n \norms{u_{[i]}^k - p^k} + D_{\Bc}\norms{\sum_{i=1}^n (G_i u_{[i]}^k + \xi_i^k)} \vspace{1ex}\\
		&\leq& (M_1 + M_2) (n\kappa + \sqrt{n})\norms{\Pi \mbf{u}^k} + D_{\Bc}\norms{\sum_{i=1}^n (G_i u_{[i]}^k + \xi_i^k)}.
	\end{array}
	\end{equation*}
	Taking the supremum on both sides of the last inequality, we obtain
	\begin{equation*}
	\arraycolsep=0.2em
	\begin{array}{lcl}	
		\mathrm{Gap}_{\Bc}(p^k) &\leq& (M_1 + M_2) (n\kappa + \sqrt{n})\norms{\Pi \mbf{u}^k} + D_{\Bc}\norms{\sum_{i=1}^n (G_i u_{[i]}^k + \xi_i^k)}.
	\end{array}
	\end{equation*}
	Combining this relation and the results in part (ii) and (iii), we conclude that
	\begin{equation*}
	\arraycolsep=0.2em
	\begin{array}{rcl}	
		\mathrm{Gap}_{\Bc}(p^k) = \BigOs{1/k} \qquad &\text{and}& \qquad \mathrm{Gap}_B(p^k) = \SmallOs{1/k}.
	\end{array}
	\end{equation*}
	\vspace{0.75ex}
	\noindent$\mathrm{(v)}$~\textbf{The convergence of the FBS residual.}
	Let $\bar{G}u := \frac{1}{n}\sum_{i=1}^n G_iu$ and $\bar{T}u := \frac{1}{n}\sum_{i=1}^n T_iu$.
	Consider the FBS residual $\Gc_{\lambda}u := \frac{1}{\lambda}(u - J_{\lambda \bar{T}}(u - \lambda \bar{G}u))$.
	Under Assumption~\ref{as:restricted_dual_gap}, it is obvious that $\bar{G}$ is also $L$-Lipschitz continuous and $\bar{\Phi} := \bar{G} + \bar{T}$ is maximally monotone.
	We consider two cases corresponding to  Assumption~\ref{as:restricted_dual_gap} and Assumption~\ref{as:boundary_regularity}.
	
	\noindent\textbf{Case $\mathrm{(i)}$.}
	Under Assumption~\ref{as:restricted_dual_gap}, let $q^k := J_{\lambda \bar{T}}(\bar{u}^k - \lambda \bar{G}\bar{u}^k)$.
	Since $\sets{\bar{u}^k}$ converges to $u^{\star}$ and $J_{\lambda \bar{T}}$ is continuous, we have $\lim_{k \to \infty} q^k = J_{\lambda \bar{T}}(u^{\star} - \lambda \bar{G}u^{\star})$.
	However, since $u^{\star}$ is a solution of \eqref{eq:DGE}, we also have $0 \in \bar{G}u^{\star} + \bar{T}u^{\star}$, meaning that $u^{\star}$ is also a solution of $0 \in \bar{G}u^{\star} + \bar{T}u^{\star}$.
	This implies $u^{\star} = J_{\lambda \bar{T}}(u^{\star} - \lambda \bar{G}u^{\star})$, and therefore, $\sets{q^k}$ converges to $u^{\star}$.
	In particular, both $\sets{q^k}$ and $\sets{\bar{u}^k - q^k}$ are bounded, and thus there exists $M_1 > 0$ such that $\norms{\bar{u}^k - q^k} \leq M_1$ and $\norms{q^k - u^{\star}} \leq M_1$ for all $k \geq 0$.
	
	Next, by the definition of $q^k$, we have $\frac{1}{\lambda}(\bar{u}^k - q^k) - \bar{G}\bar{u}^k \in \bar{T}q^k = \frac{1}{n}\sum_{i=1}^n T_i q^k$.
	Thus, there exist $\zeta_i^k \in T_i q^k$ for all $i = 1, \cdots, n$, such that $\frac{1}{n} \sum_{i=1}^n \zeta_i^k = \frac{1}{\lambda}(\bar{u}^k - q^k) - \bar{G}\bar{u}^k$.
	Let $\breve{w}_i^k := G_i q^k + \zeta_i^k \in \Phi_i q^k$.
	Then, we have $\sum_{i=1}^n \breve{w}_i^k = \frac{n}{\lambda}(\bar{u}^k - q^k) + n(\bar{G}q^k - \bar{G}\bar{u}^k)$.
	Since $\sets{\bar{u}^k - q^k}$ is bounded and $\bar{G}$ is Lipschitz continuous, we also have $\sets{\sum_{i=1}^n \breve{w}_i^k}$ is bounded, i.e., there exists $M_2 > 0$ such that $\norms{\sum_{i=1}^n \breve{w}_i^k} \leq M_2$ for all $k \geq 0$.
	
	We next show that each $\sets{\breve{w}_i^k}$ is bounded for all $i = 1, \cdots, n$.
	Recall the interior solution $u^{\dagger}$ from Assumption~\ref{as:restricted_dual_gap}, there exists $\delta > 0$ such that $B_{\delta}(u^{\dagger}) \subset \bigcap_{i=1}^n \mathrm{int}(\dom{\Phi_i})$, where $B_{\delta}(u^{\dagger}) = \sets{u \in \R^p: \norms{u - u^{\dagger}} \leq \delta}$.
	By Lemma~\ref{le:monotone_local_bounded}, there exists $M_3 > 0$ such that 
	\begin{equation*}
	\arraycolsep=0.2em
	\begin{array}{rcl}	
		\norms{\hat{w}_i} \leq M_3, \quad \hat{w}_i \in \Phi_iy, \quad y \in B_{\delta}(u^{\dagger}), \quad i = 1, \cdots, n.
	\end{array}
	\end{equation*}
	For $\breve{w}_i^k \neq 0$, let $e_i^k := \frac{\breve{w}_i^k}{\norms{\breve{w}_i^k}}$, define $y_i^k := u^{\dagger} + \delta e_i^k \in B_{\delta}(u^{\dagger})$, and choose $\hat{w}_i^k \in \Phi_iy^k$.
	By the monotonicity of $\Phi_i$, we have $0 \leq \iprods{\breve{w}_i^k - \hat{w}_i^k, q^k - y_i^k} = \iprods{\breve{w}_i^k - \hat{w}_i^k, q^k - u^{\dagger} - \delta e_i^k}$.
	Expanding this expression with a notice that $\iprods{\breve{w}_i^k, e_i^k} = \norms{\breve{w}_i^k}$ and then rearranging the result, we obtain
	\begin{equation*}
	\arraycolsep=0.2em
	\begin{array}{rcl}	
		\delta \norms{\breve{w}_i^k} &\leq& \iprods{\breve{w}_i^k, q^k - u^{\dagger}} - \iprods{\hat{w}_i^k, q^k - u^{\dagger}} + \delta\iprods{\hat{w}_i^k, e_i^k} \vspace{1ex}\\
		&\leq& \iprods{\breve{w}_i^k, q^k - u^{\dagger}} + \norms{\hat{w}_i^k} \norms{q^k - u^{\dagger}} + \delta\norms{\hat{w}_i^k}\norms{e_i^k} \vspace{1ex}\\
		&\leq& \iprods{\breve{w}_i^k, q^k - u^{\dagger}} + M_3\norms{q^k - u^{\dagger}} + \delta M_3.
	\end{array}
	\end{equation*}
	Summing the last inequality from $i=1$ to $n$, we get
	\begin{equation*}
	\arraycolsep=0.2em
	\begin{array}{rcl}	
		\delta \sum_{i=1}^n \norms{\breve{w}_i^k} &\leq& \iprods{\sum_{i=1}^n \breve{w}_i^k, q^k - u^{\dagger}} + nM_3\norms{q^k - u^{\dagger}} + n\delta M_3 \vspace{1ex}\\
		&\leq& \norms{\sum_{i=1}^n \breve{w}_i^k}(\norms{q^k - u^{\star}} + \norms{u^{\star} - u^{\dagger}}) + nM_3(\norms{q^k - u^{\star}} + \norms{u^{\star} - u^{\dagger}}) + n\delta M_3 \vspace{1ex}\\
		&\leq& (M_2 + nM_3)(M_1 + \norms{u^{\star} - u^{\dagger}}) + n\delta M_3.
	\end{array}
	\end{equation*}
	Combining the boundedness of $\sets{G_i u_{[i]}^k+\xi_i^k}$ established in part (iv) with the preceding bound on $\sum_{i=1}^n \norms{\breve{w}_i^k} = \sum_{i=1}^n \norms{G_i q^k + \zeta_i^k}$, there exists $M_4 > 0$ such that $\sum_{i=1}^n \big(\norms{G_i u_{[i]}^k + \xi_i^k} + \norms{G_i q^k + \zeta_i^k}\big) \leq M_4$ for all $k \geq 0$.
	
	Now, for $\xi_i^k \in T_iu_{[i]}^k$, by the monotonicity of $G_i + T_i$, we have
	\begin{equation*}
	\arraycolsep=0.2em
	\begin{array}{rcl}	
		0 &\leq& \sum_{i=1}^n \iprods{(G_i u_{[i]}^k + \xi_i^k) - (G_i q^k + \zeta_i^k), u_{[i]}^k - q^k} \vspace{1ex}\\
		&=& \sum_{i=1}^n \iprods{(G_i u_i^k + \xi_i^k) - (G_i q^k + \zeta_i^k), (u_i^k - \bar{u}^k) + (\bar{u}^k - q^k)}.
		\end{array}
	\end{equation*}
	Partially expanding this inner product and then using $\sum_{i=1}^n \zeta_i^k = \frac{n}{\lambda}(\bar{u}^k - q^k) - n\bar{G}\bar{u}^k$, we get
	\begin{equation*}
		\arraycolsep=0.2em
		\begin{array}{rcl}	
			\frac{n}{\lambda} \norms{\bar{u}^k - q^k}^2 &\leq& \iprod{\sum_{i=1}^n (G_i u_{[i]}^k + \xi_i^k), \bar{u}^k - q^k} + n \iprods{\bar{G}\bar{u}^k - \bar{G}q^k, \bar{u}^k - q^k} \vspace{1ex}\\
			&& + {~} \sum_{i=1}^n \iprod{(G_i u_{[i]}^k + \xi_i^k) - (G_i q^k + \zeta_i^k), u_i^k - \bar{u}^k}.
		\end{array}
	\end{equation*}
	From the $L$-Lipschitz continuity of $G_i$, we have 
	$n \iprods{\bar{G}\bar{u}^k - \bar{G}q^k, \bar{u}^k - q^k} = \sum_{i=1}^n \iprods{G_i \bar{u}^k - G_i q^k, \bar{u}^k - q^k} \leq n L\norms{\bar{u}^k - q^k}^2$.
	Substituting this bound into the last inequality, we obtain
	\begin{equation*}
	\arraycolsep=0.2em
	\begin{array}{rcl}	
		n(\frac{1}{\lambda} - L) \norms{\bar{u}^k - q^k}^2 &\leq& \iprods{\sum_{i=1}^n (G_i u_{[i]}^k + \xi_i^k), \bar{u}^k - q^k} +  \sum_{i=1}^n \iprods{(G_i u_{[i]}^k + \xi_i^k) - (G_i q^k + \zeta_i^k), u_i^k - \bar{u}^k} \vspace{1ex}\\
		
		&\leq& \norms{\sum_{i=1}^n (G_i u_{[i]}^k + \xi_i^k)} \norms{\bar{u}^k - q^k} +  \sum_{i=1}^n \big(\norms{G_i u_{[i]}^k + \xi_i^k} + \norms{G_i q^k + \zeta_i^k}\big) \norms{u_i^k - \bar{u}^k} \vspace{1ex}\\
		
		&\leq& M_1 \norms{\sum_{i=1}^n (G_i u_{[i]}^k + \xi_i^k)} + M_4 \sum_{i=1}^n \norms{u_i^k - \bar{u}^k}.
		\end{array}
	\end{equation*}
	Combining this inequality with the results of part $\mathrm{(ii)}$ and $\mathrm{(iii)}$, we conclude that
	\begin{equation*}
		\arraycolsep=0.2em
		\begin{array}{rcl}	
			\norms{\bar{u}^k - q^k}^2 = \BigOs{1/k} \qquad \text{and} \qquad \norms{\bar{u}^k - q^k}^2 = \SmallOs{1/k}.
		\end{array}
	\end{equation*}
	Finally, since $\norms{\Gc_{\lambda}\bar{u}^k}^2 = \frac{1}{\lambda^2}\norms{\bar{u}^k - J_{\lambda \bar{T}}(\bar{u}^k - \lambda \bar{G}\bar{u}^k)}^2 = \frac{1}{\lambda^2}\norms{\bar{u}^k - q^k}^2$, we also get
	\begin{equation*}
	\arraycolsep=0.2em
	\begin{array}{rcl}	
		\norms{\Gc_{\lambda}\bar{u}^k}^2 = \BigOs{1/k} \qquad \text{and} \qquad \norms{\Gc_{\lambda}\bar{u}^k}^2 = \SmallOs{1/k}.
	\end{array}
	\end{equation*}
	\vspace{0.75ex}
	\noindent
	\textbf{Case $\mathrm{(ii)}$.}
	Under Assumption~\ref{as:boundary_regularity}, let $q^k := J_{\lambda \bar{T}}(p^k - \lambda \bar{G}p^k)$, where $p^k := \proj_{\Bc}(\bar{u}^k)$ as defined in part $\mathrm{(iv)}$.
	Then, there exist $\zeta_i^k \in T_i q^k$ such that $\frac{1}{n}\sum_{i=1}^n \zeta_i^k = \frac{1}{\lambda}(p^k - q^k) - \bar{G}p^k$.
	Using the definition of the restricted gap function, the last expression, Cauchy-Schwarz inequality, and the $L$-Lipschitz continuity of $\bar{G}$, we have
	\begin{equation*}
	\arraycolsep=0.2em
	\begin{array}{rcl}	
		\mathrm{Gap}_{\Bc}(p^k) &\geq& \sum_{i=1}^n \iprods{G_iq^k + \zeta_i^k, p^k - q^k} \vspace{1ex}\\
		&=& n\iprods{\bar{G}q^k + \frac{1}{\lambda}(p^k - q^k) - \bar{G}p^k, p^k - q^k} \vspace{1ex}\\
		&=& \frac{n}{\lambda}\norms{p^k - q^k}^2 - n\iprods{\bar{G}p^k - \bar{G}q^k, p^k - q^k} \vspace{1ex}\\
		&\geq&  \frac{n}{\lambda}\norms{p^k - q^k}^2 - n\norms{\bar{G}p^k - \bar{G}q^k}\norms{p^k - q^k} \vspace{1ex}\\
		&\geq&  n\big(\frac{1}{\lambda} - L\big)\norms{p^k - q^k}^2.
	\end{array}
	\end{equation*}
	Combining this with the results in part (iv), we conclude that
	\begin{equation*}
	\arraycolsep=0.2em
	\begin{array}{rcl}	
		\norms{p^k - q^k}^2 = \BigOs{1/k} \qquad \text{and} \qquad \norms{p^k - q^k}^2 = \SmallOs{1/k}.
	\end{array}
	\end{equation*}
	Finally, since $\norms{\Gc_{\lambda}p^k}^2 = \frac{1}{\lambda^2}\norms{p^k - J_{\lambda \bar{T}}(p^k - \lambda \bar{G}p^k)}^2 = \frac{1}{\lambda^2}\norms{p^k - q^k}^2$, we also get
	\begin{equation*}
	\arraycolsep=0.2em
	\begin{array}{rcl}	
		\norms{\Gc_{\lambda}p^k}^2 = \BigOs{1/k} \qquad \text{and} \qquad \norms{\Gc_{\lambda}p^k}^2 = \SmallOs{1/k}.
	\end{array}
	\end{equation*}
\vspace{0.75ex}
\noindent$\mathrm{(vi)}$~\textbf{The convergence of operator residual $\norms{\sum_{i=1}^n G_i \bar{u}^k}$ when $T_i = 0$ for all $i \in [n]$.}
Using the triangle inequality and the $L$-Lipschitz continuity of $G_i$, we have
\begin{equation*}
\arraycolsep=0.2em
\begin{array}{rcl}	
	\norms{\sum_{i=1}^n G_i \bar{u}^k} &=& \norm{\sum_{i=1}^n G_i \bar{u}^k - \sum_{i=1}^n G_i u_{[i]}^k + \sum_{i=1}^n G_i u_{[i]}^k } \vspace{1ex}\\
	&\leq& \sum_{i=1}^n \norms{G_i \bar{u}^k - G_i u_{[i]}^k} + \norm{\sum_{i=1}^n G_i u_{[i]}^k} \vspace{1ex}\\
	&\leq& L\sum_{i=1}^n \norms{\bar{u}^k - u_{[i]}^k} + \norm{\sum_{i=1}^n G_i u_{[i]}^k}.
\end{array}
\end{equation*}
From the results of parts (b) and (c), the last inequality implies that
\begin{equation*}
\arraycolsep=0.2em
\begin{array}{rcl}	
	\norms{\sum_{i=1}^n G_i \bar{u}^k} = \BigOs{1/k} \qquad \text{and} \qquad \norms{\sum_{i=1}^n G_i \bar{u}^k} = \SmallOs{1/k},
\end{array}
\end{equation*}
which completes the proof.
\end{proof}

\beforesubsubsec
\subsubsection{The co-coercive case}\label{subsubsec:ND_case2}
\aftersubsubsec
Now, we design a distributed FFP method to solve  \eqref{eq:DGE} under the following assumption.

\begin{assumption}\label{as:DGE_cocoercive}
	The distributed composite inclusion \eqref{eq:DGE} satisfies the following conditions:
	\begin{compactitem}
		\item[(i)] (\textbf{Local co-coercivity}) $G_i$ is $\frac{1}{L}$-co-coercive, for all $i \in [n]$.
		\item[(ii)] (\textbf{Local maximal monotonicity}) $T_i$ is maximally monotone for all $i \in [n]$.
	\end{compactitem}
\end{assumption}
First, we establish the following lemma to transform \eqref{eq:DGE} to \eqref{eq:precond_GE}.
The proof of this lemma is given in Appendix~\ref{apdx:le:DGE_cocoercive}.
\begin{lemma}\label{le:DGE_cocoercive}
	Let $G$ and $T$ in \eqref{eq:DGE} satisfy  Assumption~\ref{as:DGE_cocoercive} and $\tau \sigma \norms{\Id - W} < 2$. 
	Then
	\begin{compactitem}
		\item[$\mathrm{(i)}$] 
		The operator $\mbb{P}^{-1} \mbb{G}$ in \eqref{eq:precond_GE} is $\frac{1}{\tilde{L}}$-co-coercive w.r.t. $\norms{\cdot}_{\mbb{P}}$, where $\tilde{L} := \frac{2\tau L}{2 - \tau \sigma \norms{\Id - W}}$.
		
		\item[$\mathrm{(ii)}$] 
		The operator $\mbb{P}^{-1} \mbb{T}$ in \eqref{eq:precond_GE} is maximally monotone w.r.t. $\iprods{\cdot, \cdot}_{\mbb{P}}$.
	\end{compactitem}
\end{lemma}
Next, utilizing the parameters from Corollary~\ref{co:FKM_scheme2+_convergence}, the scheme \eqref{eq:DFKM_scheme} reduces to
\begin{equation}\label{eq:DFKM_scheme2}
\tag{ND-DFFP2}
\arraycolsep=0.2em
\left\{
\begin{array}{lcl}	
	\mbf{d}^k &:=& \frac{1}{2}(\hat{\mbf{u}}^k - W\hat{\mbf{u}}^k), \vspace{1ex}\\
	\mbf{u}^{k+1} &:=& J_{\tau \mbf{T}} \left(\hat{\mbf{u}}^k - \tau \tilde{\mbf{v}}^k - \tau \mbf{G}\hat{\mbf{u}}^k  -  2 \tau \sigma \mbf{d}^k\right), \vspace{1ex}\\
	\hat{\mbf{w}}^{k+1} &:=& \hat{\mbf{u}}^k - \mbf{u}^{k+1}, \vspace{1ex}\\
	\hat{\mbf{u}}^{k+1} &:=& \hat{\mbf{u}}^k + \theta_k \left(\hat{\mbf{u}}^k - \hat{\mbf{u}}^{k-1} \right) - \hat{\eta}_k\hat{\mbf{w}}^{k+1} + \hat{\gamma}_k \hat{\mbf{w}}^k , \vspace{1ex}\\
	\tilde{\mbf{v}}^{k+1} &:=& \tilde{\mbf{v}}^k + \theta_k \left(\tilde{\mbf{v}}^k - \tilde{\mbf{v}}^{k-1} \right) + \sigma (\hat{\eta}_k\mbf{d}^k - \hat{\gamma}_k \mbf{d}^{k-1}),
\end{array}
\right.
\end{equation}
where the iterates are initialized as $\hat{\mbf{u}}^{-1} = \hat{\mbf{u}}^0 = \mbf{u}^0$, $\tilde{\mbf{v}}^{-1} = \tilde{\mbf{v}}^0 = 0$, $\hat{\mbf{w}}^0 = 0$, and $\mbf{d}^{-1} = 0$, and the parameters are updated according to \eqref{eq:DFKM2_params}.
Then, the convergence of \eqref{eq:DFKM_scheme2} can be stated in the following theorem.
\begin{theorem}[\textbf{Co-coercive setting}]\label{th:DFKM_convergence_cocoercive}
	For \eqref{eq:DGE}, suppose that Assumption~\ref{as:DGE_cocoercive} holds. 
	Let $\sets{(\mbf{u}^k, \tilde{\mbf{v}}^k)}$ be generated by \eqref{eq:DFKM_scheme2} using $t_0$ given in \eqref{eq:t0_cond2} and
	\begin{equation}\label{eq:DFKM2_params}
		\arraycolsep=0.2em
		\hspace{-2ex}
		\begin{array}{c}
			r > 1, \qquad \nu \in (0, r-1), \qquad 			
			\tau L + \tau \sigma (1 - \lambda_{\min}(W)) < 2, \vspace{1ex}\\
			t_{k} := k + t_0, \quad \theta_k := \frac{t_k - r}{t_{k+1}}, \qquad  \hat{\eta}_k = \frac{t_{k+1} - r}{t_{k+1}} + \frac{\nu (t_k - 1)}{t_k t_{k+1}}, \qquad \text{and} \qquad \hat{\gamma}_k = \theta_k.
		\end{array}
		\hspace{-2ex}
	\end{equation}
	Then, the following statements hold.
	\begin{compactitem}
	\item[$~~~\mathrm{(i)}$~\textbf{$($Local iterates$)$}]
	The local iterates $\sets{u_{[i]}^k}$ and $\sets{\hat{u}_{[i]}^k}$ converge to the same $u^\star \in \zer{\Phi}$ of the original \eqref{eq:DGE} problem for all nodes $i \in [n]$.
	
	\item[$~~~\mathrm{(ii)}$~\textbf{$($Consensus errors$)$}]
	The consensus errors $\norms{\Pi \mbf{u}^k }$ and $\norms{\Pi \hat{\mbf{u}}^k }$ converge to zero at the rate of $\SmallOs{1/k}$.
	
	\item[$~~~\mathrm{(iii)}$~\textbf{$($Aggregated residual$)$}]
	The aggregated residual $\norms{\sum_{i=1}^n G_i u_{[i]}^k + \xi_i^k}$ converges to zero at the rate of $\SmallOs{1/k}$, where $\xi_i^k \in Tu_{[i]}^k$.
	
	\item[$~~~\mathrm{(iv)}$~\textbf{$($Restricted gap$)$}]
	If, in addition, Assumption~\ref{as:restricted_dual_gap} holds, then the restricted gap value $\mathrm{Gap}_B(\bar{u}^k)$ converges to zero at the rate of $\SmallOs{1/k}$, where $\bar{u}^k := \frac{1}{n}\sum_{i=1}^n u_{[i]}^k$.
	Otherwise, if, in addition, Assumption~\ref{as:boundary_regularity} holds, then $\mathrm{Gap}_B(p^k)$ converges to zero at the rate of $\SmallOs{1/k}$, where $p^k := \proj_{\Bc}(\bar{u}^k)$.
	
	\item[$~~~\mathrm{(v)}$~\textbf{$($FBS residual$)$}]
	For a given $\lambda \in (0, \frac{1}{L})$, if, in addition, Assumption~\ref{as:restricted_dual_gap} holds, then $\norms{\Gc_{\lambda}\bar{u}^k}^2$ converges to zero at the rate of $\SmallOs{1/k}$.
	Otherwise, if, in addition, Assumption~ \ref{as:boundary_regularity} holds, then $\norms{\Gc_{\lambda}p^k}^2$ converges to zero at the rate of $\SmallOs{1/k}$, where $p^k := \proj_{\Bc}(\bar{u}^k)$.
	
	\item[$~~~\mathrm{(vi)}$~\textbf{$($Special case$)$}]
	If $T_i = 0$ for all $i \in [n]$ $($i.e., \eqref{eq:DGE} reduces to \eqref{eq:DNE}$)$, then $\norms{\sum_{i=1}^n G_i \bar{u}^k}$ converges to zero at the rate of $\SmallOs{1/k}$.
	\end{compactitem}
\end{theorem}

\begin{proof}
	The proof of this theorem is similar to Theorem~\ref{th:DFKM_convergence}, but using Corollary~\ref{co:FKM_scheme2+_convergence} (or Theorem~\ref{th:FKM2_convergence}).
	Thus, we omit the details here.
\end{proof}

\beforesec
\section{Distributed FFP Algorithm with NI Heterogeneous Stepsizes}\label{sec:NI-DFKM}
\aftersec
Finally, we develop a new distributed fast fixed-point-based (FFP) algorithm where each agent uses its own stepsize depending solely on the local information and independent of the network.
We name this algorithm a \textit{distributed FFP algorithm with network-independent stepsizes}.

\beforesubsec
\subsection{Reformulation of \eqref{eq:DGE}}\label{subsec:NI_reform}
\aftersubsec
Our first step is to reformulate the decentralized composite inclusion \eqref{eq:DGE} into a centralized three-operator inclusion.
The following lemma from \cite{tam2025decentralised} provides such a reformulation.

\begin{lemma}[\cite{tam2025decentralised}]\label{le:DFKM_3op_reform}
	Let $\gamma_1, \cdots, \gamma_n > 0$, $\mbf{\Gamma} := \diag{\gamma_1, \cdots, \gamma_n}$, and $\mbf{u} = (u_{[1]}, \cdots, u_{[n]}) \in \R^{n\times p}$, where $u_{[i]} \in \R^p$ is the local variable of agent $i \in [n]$.
	Define $T: \R^{n\times p} \rightrightarrows \R^{n\times p}$ and $G: \R^{n\times p} \to \R^{n\times p}$ as in \eqref{eq:GE_notation}.
	Let $\mbf{K} := \left(\frac{\Id - \mbf{W}}{2}\right)^{1/2} \succeq 0$ and $\mbf{M} := (\kappa \Id - \mbf{K}\mbf{\Gamma}^2 \mbf{K})^{1/2} \succ 0$, where $\kappa > \norms{\mbf{\Gamma} \mbf{K}^2 \mbf{\Gamma}}$.
	Let ${A}, {C}: \R^{n\times p} \times \R^{n\times p} \rightrightarrows \R^{n\times p} \times \R^{n\times p}$ and ${B}: \R^{n\times p} \times \R^{n\times p} \to \R^{n\times p} \times \R^{n\times p}$:
	\begin{equation}\label{eq:DFKM_3op_operators}
	\arraycolsep=0.2em
	\begin{array}{lcl}
		{A}(\mbf{a}, \breve{\mbf{a}}) & := & \left(\mbf{\Gamma} T(\mbf{\Gamma} \mbf{a}), \ \Nc_{\sets{0}}(\breve{\mbf{a}}) \right), \vspace{1ex}\\
		{B}(\mbf{a}, \breve{\mbf{a}}) & := & \left(\mbf{\Gamma} G (\mbf{\Gamma} \mbf{a}), \ 0 \right), \vspace{1ex}\\
		{C}(\mbf{a}, \breve{\mbf{a}}) & := & \left(\mbf{\Gamma} \mbf{K} \Nc_{\sets{0}}(\mbf{K} \mbf{\Gamma} \mbf{a} + \mbf{M}\breve{\mbf{a}}), \ 
		\mbf{M}\Nc_{\sets{0}}(\mbf{K}\mbf{\Gamma} \mbf{a} + \mbf{M}\breve{\mbf{a}}) \right).
	\end{array}
	\end{equation}
	Then, the following statements hold.
	\begin{compactitem}
		\item[$\mathrm{(i)}$] 
		The point $(\mbf{a}, \breve{\mbf{a}}) \in \R^{n\times p} \times \R^{n\times p}$ satisfies
		\begin{equation}\label{eq:le_DFKM_3op_reform1}
		\arraycolsep=0.2em
		\begin{array}{lcl}
			0 \in ({A} + {B} + {C})(\mbf{a}, \breve{\mbf{a}}) \subset \R^{n\times p} \times \R^{n\times p}
		\end{array}
		\end{equation}
		if and only if $\mbf{a} = \mbf{\Gamma}^{-1}(u, \cdots, u)$ for some $u \in \R^p$ and $u$ solves $0 \in \sum_{i=1}^n\big(G_iu + T_iu\big)$.
		
		\item[$\mathrm{(ii)}$] 
		Suppose $T_i$ for all $i \in [n]$ are maximally monotone, and $G_i$ are monotone and $L_i$-Lipschitz continuous for all $i \in [n]$.
		Then, ${A}$ and ${C}$ are maximally monotone, and ${B}$ is monotone and $L_B$-Lipschitz continuous with $L_B = \max_{ i \in [n]} \sets{\gamma_i^2 L_i}$.
	\end{compactitem}
\end{lemma}
This lemma serves as a key template, which allows us to apply any suitable method to three-operator inclusions to developing new distributed fixed-point-based methods for solving the distributed composite monotone inclusion \eqref{eq:DGE}.

\beforesubsec
\subsection{Derivation of the algorithm}\label{subsec:NI_DFFP_derivation}
\aftersubsec
In what follows, we apply \eqref{eq:FKM_3OP_scheme} to the three-operator inclusion \eqref{eq:le_DFKM_3op_reform1} to derive a new distributed algorithm for the distributed composite monotone inclusion \eqref{eq:DGE}.

First, let us initialize the variable $\breve{\mbf{a}}^0 = 0$.
Next, we perform the following steps.
\begin{compactitem}
	\item 
	\textbf{Step 1:} 
	Let $u = (\mbf{a}, \breve{\mbf{a}})$, $\hat{u} = (\hat{\mbf{a}}, \hat{\breve{\mbf{a}}})$, $v = (\mbf{b}, \breve{\mbf{b}})$, 
	$\hat{v} = (\hat{\mbf{b}}, \hat{\breve{\mbf{b}}})$, $z_u = (z_{\mbf{a}}, z_{\breve{\mbf{a}}})$, 
	and $z_v = (z_{\mbf{b}}, z_{\breve{\mbf{b}}})$ in $\R^{n\times p} \times \R^{n\times p}$. 
	We update
	\begin{equation*}
	\arraycolsep=0.2em
	\begin{array}{lcllcl}
		\hat{\mbf{a}}^k &=& \frac{t_k - r}{t_k} \mbf{a}^k + \frac{r}{t_k} z_{\mbf{a}}^k, 
		\qquad \text{and} \qquad
		\hat{\breve{\mbf{a}}}^k &=& \frac{t_k - r}{t_k} \breve{\mbf{a}}^k + \frac{r}{t_k} z_{\breve{\mbf{a}}}^k, \vspace{1ex}\\
		\hat{\mbf{b}}^k &=& \frac{t_k - r}{t_k} \mbf{b}^k + \frac{r}{t_k} z_{\mbf{b}}^k, 
		\qquad \text{and} \qquad
		\hat{\breve{\mbf{b}}}^k &=& \frac{t_k - r}{t_k} \breve{\mbf{b}}^k + \frac{r}{t_k} z_{\breve{\mbf{b}}}^k.
	\end{array}
	\end{equation*}
	
	\item \textbf{Step 2:} 
	Let us denote by $y_u = (y_{\mbf{a}}, y_{\breve{\mbf{a}}})$, $\tilde{w}_u = (\tilde{w}_{\mbf{a}}, \tilde{w}_{\breve{\mbf{a}}})$, and $\tilde{w}_v = (\tilde{w}_{\mbf{b}}, \tilde{w}_{\breve{\mbf{b}}})$ in $\R^{n\times p} \times \R^{n\times p}$.
	We update
	\begin{equation*}
	\arraycolsep=0.2em
	\begin{array}{lcl}
		y_{\mbf{a}}^k = \hat{\mbf{a}}^k - \frac{\gamma_k - \beta_k}{1 - \tau \sigma} (\tau \tilde{w}_{\mbf{a}}^k + \tau \sigma \tilde{w}_{\mbf{b}}^k)
		\qquad \text{and} \qquad 
		y_{\breve{\mbf{a}}}^k = \hat{\breve{\mbf{a}}}^k - \frac{\gamma_k - \beta_k}{1 - \tau \sigma} (\tau \tilde{w}_{\breve{\mbf{a}}}^k + \tau \sigma \tilde{w}_{\breve{\mbf{b}}}^k).
	\end{array}
	\end{equation*}
	
	\item \textbf{Step 3:} 
	By the definition of ${A}(\mbf{a}, \breve{\mbf{a}}) = (\mbf{\Gamma} T(\mbf{\Gamma} \mbf{a}), \Nc_{\sets{0}}(\breve{\mbf{a}}))$, we have $J_{\tau A} = (J_{\tau \mbf{\Gamma} T \mbf{\Gamma}}, 0)$. 
	Thus, using ${B}(\mbf{a}, \breve{\mbf{a}}) = (\mbf{\Gamma} G (\mbf{\Gamma} \mbf{a}), 0)$, we update
	\begin{equation*}
	\arraycolsep=0.2em
	\begin{array}{lcl}
		\mbf{a}^{k+1} = J_{\tau \mbf{\Gamma} T \mbf{\Gamma}} (\hat{\mbf{a}}^k - \tau \hat{\mbf{b}}^k - \tau \mbf{\Gamma} G (\mbf{\Gamma} y_{\mbf{a}}^k) + \tau \beta_k \tilde{w}_{\mbf{a}}^k)
		\qquad \text{and} \qquad
		\breve{\mbf{a}}^{k+1} = 0.
	\end{array}
	\end{equation*}
	By induction, we can prove that $\breve{\mbf{a}}^k = 0$, $\hat{\breve{\mbf{a}}}^k = 0$ (and $z_{\breve{\mbf{a}}}^k = 0$ in the later step).
	This proof is nontrivial and can be found in Appendix~\ref{le:NI_DFFP_derivation3}.
	
	\item 
	\textbf{Step 4:} Let $\tilde{v} = (\tilde{\mbf{b}}, \tilde{\breve{\mbf{b}}})$ and $v = (\mbf{b}, \breve{\mbf{b}})$  in $\R^{n\times p} \times \R^{n\times p}$.
	First, we update 
\begin{equation*}
\arraycolsep=0.2em
\begin{array}{lcl}
	\tilde{\mathbf{b}}^{k+1} = \hat{\mathbf{b}}^k + \sigma(2\mathbf{a}^{k+1} - \hat{\mathbf{a}}^k) + \sigma\beta_k \tilde{w}_{\mathbf{b}}^k
	\ \  \text{and} \ \
	\tilde{\breve{\mbf{b}}}^{k+1} = \hat{\breve{\mbf{b}}}^k + \sigma(2\breve{\mbf{a}}^{k+1} - \hat{\breve{\mbf{a}}}^k) + \sigma\beta_k \tilde{w}_{\breve{\mbf{b}}}^k. 
\end{array}
\end{equation*}
	Next, to obtain the update of $v^{k+1}$, we need an explicit formula for $J_{\sigma^{-1} C}$.
	Let $\mbf{E} = \begin{bmatrix}
		\mbf{K}\mbf{\Gamma} & \mbf{M}
	\end{bmatrix}$, then for $\mbf{v} = \begin{bmatrix}
		\mbf{b} \\ \breve{\mbf{b}}
	\end{bmatrix}$, we can write $C(\mbf{v}) = \mbf{E}^\top \Nc_{\sets{0}}(\mbf{E}\mbf{v})$.
	By definition, $\mbf{r} = J_{\sigma^{-1}C}(\mbf{t})$ if and only if $\mbf{t} \in \mbf{r} + \sigma^{-1}C(\mbf{r})$, or equivalently, $\mbf{t} \in \mbf{r} + \mbf{E}^\top \Nc_{\sets{0}} (\mbf{E}\mbf{r})$.
	Thus, we must have $\mbf{E}\mbf{r} = 0$, and there exists $\mbf{s} \in \Nc_{\sets{0}} (\mbf{E}\mbf{r})$ such that $\mbf{t} = \mbf{r} + \mbf{E}^\top\mbf{s}$.
	Multiplying this by $\mbf{E}$ and using $\mbf{E}\mbf{r} = 0$ and $\mbf{E}\mbf{E}^\top = \mbf{K}\mbf{\Gamma}^2\mbf{K} + \mbf{M}^2 = \kappa \Id$, we obtain $\mbf{E}\mbf{t} = \kappa \mbf{s}$, or $\mbf{s} = \kappa^{-1}\mbf{E}\mbf{t}$.
	Therefore, we can show that
	\begin{equation*}
		J_{\sigma^{-1}C}(\mbf{t}) = \mbf{r} = \mbf{t} - \mbf{E}^\top \mbf{s} = \mbf{t} - \kappa^{-1}\mbf{E}^\top \mbf{E}\mbf{t} = (\Id - \kappa^{-1}\mbf{E}^\top \mbf{E})\mbf{t}.
	\end{equation*}
	Using this expression, we obtain
	\begin{equation}\label{eq:DFKM_ni_proof1}
	\arraycolsep=0.2em
	\begin{array}{lcl}
		v^{k+1} &=& \tilde{v}^{k+1} - \sigma \left( \frac{1}{\sigma}\tilde{v}^{k+1} - \frac{\kappa^{-1}}{\sigma} \mbf{E}^\top \mbf{E}\tilde{v}^{k+1} \right) = \kappa^{-1}\mbf{E}^\top \mbf{E} \tilde{v}^{k+1},
	\end{array}
	\end{equation}
	where we have the following expression from the fifth line of \eqref{eq:FKM_3OP_scheme}:
	\begin{equation}\label{eq:DFKM_ni_proof2}
	\arraycolsep=0.2em
	\begin{array}{lcl}
		\mbf{E}\tilde{v}^{k+1} = \mbf{E}\hat{v}^k + \sigma \mbf{E}(2u^{k+1} - \hat{u}^k) + \sigma \beta_k \mbf{E}\tilde{w}_v^k.
	\end{array}
	\end{equation}
	Let us break this expression down.
	First, since $u^{k+1} = \begin{bmatrix}
		\mbf{a}^{k+1} \\ 0
	\end{bmatrix}$ and $\hat{u}^k = \begin{bmatrix}
		\hat{\mbf{a}}^k \\ 0
	\end{bmatrix}$, we have 
	\begin{equation}\label{eq:DFKM_ni_proof3}
	\arraycolsep=0.2em
	\begin{array}{lcl}
		\mbf{E}(2u^{k+1} - \hat{u}^k) = \mbf{K} \mbf{\Gamma} (2\mbf{a}^{k+1} - \hat{\mbf{a}}^k).
	\end{array}
	\end{equation}
	Next, from the fifth line of \eqref{eq:FKM_3OP_scheme}, we get
	\begin{equation*}
		\arraycolsep=0.2em
		\begin{array}{lcl}
			\frac{1}{\sigma}(\tilde{v}^{k} - v^k) = \frac{1}{\sigma}(\hat{v}^{k-1} - v^k) + 2u^{k} - \hat{u}^{k-1} + \beta_{k-1} \tilde{w}_v^{k-1},
		\end{array}
	\end{equation*}
	From the eighth line of \eqref{eq:FKM_3OP_scheme}, we have $\tilde{w}_v^k = -(\hat{u}^{k-1} - u^k) + \frac{1}{\sigma}(\hat{v}^{k-1} - v^k) + \beta_{k-1} \tilde{w}_v^{k-1}$.
	Taking the difference of the last two relations and rearranging the result, we obtain
	\begin{equation*}
	\arraycolsep=0.2em
	\begin{array}{lcl}
		\frac{1}{\sigma}(\tilde{v}^{k} - v^k) = \tilde{w}_v^k + u^k.
	\end{array}
	\end{equation*}
	Since $v^k = J_{\sigma C^{-1}}(\tilde{v}^k)$, we have $\frac{1}{\sigma}(\tilde{v}^k - v^k) \in C^{-1}(v^k)$, which implies that
	\begin{equation*}
	\arraycolsep=0.2em
	\begin{array}{lcl}
		v^k \in C\left(\frac{1}{\sigma}(\tilde{v}^k - v^k)\right) = \mbf{E}^\top \Nc_{\sets{0}} (\mbf{E}(\tilde{w}_v^k + u^k)).
	\end{array}
	\end{equation*}
	Thus, we must have $\mbf{E}(\tilde{w}_v^k + u^k) = 0$, which implies that
	\begin{equation}\label{eq:DFKM_ni_proof4}
	\arraycolsep=0.2em
	\begin{array}{lcl}
		\mbf{E}\tilde{w}_v^k = - \mbf{E}u^k = -\begin{bmatrix}
			\mbf{K}\mbf{\Gamma} & \mbf{M}
		\end{bmatrix} \begin{bmatrix}
		\mbf{a}^k \\ 0
		\end{bmatrix} = - \mbf{K}\mbf{\Gamma} \mbf{a}^k.
	\end{array}
	\end{equation}
	Substituting \eqref{eq:DFKM_ni_proof3} and \eqref{eq:DFKM_ni_proof4} into \eqref{eq:DFKM_ni_proof2}, we get
	\begin{equation}\label{eq:DFKM_ni_proof5}
	\arraycolsep=0.2em
	\begin{array}{lcl}
		\mbf{E}\tilde{v}^{k+1}  &=& \mbf{E}\hat{v}^k + \sigma \mbf{K} \mbf{\Gamma} (2\mbf{a}^{k+1} - \hat{\mbf{a}}^k - \beta_k \mbf{a}^k).
	\end{array}
	\end{equation}
	Substituting this into \eqref{eq:DFKM_ni_proof1}, we get
	\begin{equation}\label{eq:DFKM_ni_proof6}
	\arraycolsep=0.2em
	\begin{array}{lcl}
		v^{k+1} &=& \kappa^{-1}\mbf{E}^\top\mbf{E} \hat{v}^k + \kappa^{-1}\sigma \mbf{E}^\top \mbf{K} \mbf{\Gamma} (2\mbf{a}^{k+1} - \hat{\mbf{a}}^k - \beta_k \mbf{a}^k).
	\end{array}
	\end{equation}
	Now, substituting $\kappa^{-1}\mbf{E}^\top\mbf{E}\hat{v}^k = \hat{v}^k$ from Lemma~\ref{le:NI_DFFP_derivation2} into \eqref{eq:DFKM_ni_proof6}, we get
	\begin{equation*} 
	\arraycolsep=0.2em
	\begin{array}{lcl}
		v^{k+1} &=& \hat{v}^k + \kappa^{-1}\sigma \mbf{E}^\top \mbf{K} \mbf{\Gamma} (2\mbf{a}^{k+1} - \hat{\mbf{a}}^k - \beta_k \mbf{a}^k).
	\end{array}
	\end{equation*}
	This expression can be explicitly rewritten as
	\begin{equation*} 
	\arraycolsep=0.2em
	\begin{array}{lcl}
		\mbf{b}^{k+1} &=& \hat{\mbf{b}}^k + \kappa^{-1}\sigma \mbf{\Gamma} \mbf{K}^2 \mbf{\Gamma} (2\mbf{a}^{k+1} - \hat{\mbf{a}}^k - \beta_k \mbf{a}^k), \vspace{1ex}\\
		\breve{\mbf{b}}^{k+1} &=& \utilde{\hat{\mbf{b}}}^k + \kappa^{-1}\sigma \mbf{M} \mbf{K} \mbf{\Gamma} (2\mbf{a}^{k+1} - \hat{\mbf{a}}^k - \beta_k \mbf{a}^k)
	\end{array}
	\end{equation*}
	
	\item \textbf{Step 5:} 
	We update
	\begin{equation*}
	\arraycolsep=0.2em
	\left\{
	\begin{array}{lcl}
		\tilde{w}_{\mbf{a}}^{k+1} = \frac{1}{\tau}(\hat{\mbf{a}}^k - \mbf{a}^{k+1}) - (\hat{\mbf{b}}^k - \mbf{b}^{k+1}) + \beta_k \tilde{w}_{\mbf{a}}^k, \vspace{1ex}\\
		\tilde{w}_{\breve{\mbf{a}}}^{k+1} = \frac{1}{\tau}(\hat{\breve{\mbf{a}}}^k - \breve{\mbf{a}}^{k+1}) - (\hat{\breve{\mbf{b}}}^k - \breve{\mbf{b}}^{k+1}) + \beta_k \tilde{w}_{\breve{\mbf{a}}}^k, \vspace{1ex}\\
		\tilde{w}_{\mbf{b}}^{k+1} = -(\hat{\mbf{a}}^k - \mbf{a}^{k+1}) + \frac{1}{\sigma}(\hat{\mbf{b}}^k - \mbf{b}^{k+1}) + \beta_k \tilde{w}_{\mbf{b}}^k, \vspace{1ex}\\
		\tilde{w}_{\breve{\mbf{b}}}^{k+1} = -(\hat{\breve{\mbf{a}}}^k - \breve{\mbf{a}}^{k+1}) + \frac{1}{\sigma}(\hat{\utilde{\bf{b}}}^k - \breve{\mbf{b}}^{k+1}) + \beta_k \tilde{w}_{\breve{\mbf{b}}}^k.
	\end{array}
	\right.
	\end{equation*}
	
	\item \textbf{Step 6:} 
	We update
	\begin{equation*}
	\arraycolsep=0.2em
	\begin{array}{lcl}
		z_{\mbf{a}}^{k+1} = z_{\mbf{a}}^k - \frac{\nu_k}{r(1-\tau\sigma)}(\tau\tilde{w}_{\mbf{a}}^{k+1} + \tau\sigma\tilde{w}_{\mbf{b}}^{k+1})
		\ \ \text{and} \ \
		z_{\breve{\mbf{a}}}^{k+1} = z_{\breve{\mbf{a}}}^k - \frac{\nu_k}{r(1-\tau\sigma)}(\tau\tilde{w}_{\breve{\mbf{a}}}^{k+1} + \tau\sigma\tilde{w}_{\breve{\mbf{b}}}^{k+1}), \vspace{1ex}\\ 
		z_{\mbf{b}}^{k+1} = z_{\mbf{b}}^k - \frac{\nu_k}{r(1-\tau\sigma)}(\tau\sigma\tilde{w}_{\mbf{a}}^{k+1} + \sigma\tilde{w}_{\mbf{b}}^{k+1})
		\ \ \text{and} \ \
		z_{\breve{\mbf{b}}}^{k+1} = z_{\breve{\mbf{b}}}^k - \frac{\nu_k}{r(1-\tau\sigma)}(\tau\sigma\tilde{w}_{\breve{\mbf{a}}}^{k+1} + \sigma\tilde{w}_{\breve{\mbf{b}}}^{k+1}).
	\end{array}
	\end{equation*}	
\end{compactitem}
Putting everything together, we obtain the detailed scheme \eqref{eq:DFKM_NI_scheme_init} in Appendix~\ref{apdx:DFKM_NI_scheme_init}.
Next, we also use appropriate linear variable transformations and the facts that $\mbf{\Gamma}^2 = \mbf{\Lambda} = \diag{\alpha_1, \cdots, \alpha_n}$ and $\mbf{K} = \left(\frac{\Id -  W}{2}\right)^{1/2}$, to obtain a simplified scheme \eqref{eq:DFKM_NI_scheme} in Appendix~\ref{apdx:DFKM_NI_scheme_init}.

\begin{algorithm}[ht]\caption{(Network-Independent Fast Fixed-Point-Based Algorithm for \eqref{eq:DGE})}\label{alg:NI_DFFP}
	\normalsize
	\begin{algorithmic}[1]
		\STATE\label{step:A2_i0}{\bfseries Initialization:} 
		Choose $\mbf{x}^0 \in \dom{\mbf{T}}$ and $\boldsymbol{\xi}^0 \in \mbf{T}\mbf{x}^0$ arbitrarily. Set $\mbf{w}_{\mbf{x}}^0 := \mbf{G}\mbf{x}^0 + \mbf{\xi}^0$.
		\STATE\hspace{0ex}Choose parameters $r > 0$, $\tau > 0$, and $\sigma > 0$ with $\tau \sigma < 1$.
		\STATE\hspace{0ex}Choose stepsizes $\alpha_1, \cdots, \alpha_n \in \R_{++}$ and set $\mbf{\Lambda} = \diag{\alpha_1, \cdots, \alpha_n}$.
		\STATE\hspace{0ex}Choose a mixing matrix $W$ and parameter $2\kappa > \norms{\mbf{\Lambda}^{1/2}(\Id - \mbf{W})\mbf{\Lambda}^{1/2}}$.
		\STATE Set $\mbf{s}_{\mbf{x}}^0 = \mbf{x}^0$, $\mbf{v}^0 = \mbf{s}_{\mbf{v}}^0 = \frac{\sigma}{2\kappa}\mbf{\Lambda}(\Id - \mbf{W})(\mbf{x}^0 - \tau \mbf{\Lambda}\mbf{w}_{\mbf{x}}^0)$, 
		$\mbf{r}_{\mbf{x}}^0 = \mbf{\Lambda}\mbf{w}_{\mbf{x}}^0 + \mbf{v}^0$, and $\mbf{r}_{\mbf{v}}^0 = -\tau \mbf{\Lambda}\mbf{w}_{\mbf{x}}^0 - \frac{1}{\sigma}\mbf{v}^0$. \hspace{-1cm}
		\STATE\hspace{0ex}\label{step:A2_loop}{\bfseries For $k = 0$ to $k_{\max}$ do}
		\vspace{0.25ex}   
		\STATE\hspace{3ex}Update  $t_k$, $\gamma_k$, $\beta_k$, and $\nu_k$ according to \eqref{eq:NI_DFKM_params}.
		\STATE\hspace{3ex}Update the iterates as
		\begin{equation*}
			\tag{\ref{eq:DFKM_NI_scheme}}
			\arraycolsep=0.2em
			\left\{
			\begin{array}{lcl}
				\hat{\mbf{x}}^k & := & \frac{t_k - r}{t_k} \mbf{x}^k + \frac{r}{t_k} \mbf{s}_{\mbf{x}}^k, \vspace{1ex}\\
				\hat{\mbf{v}}^k & := & \frac{t_k - r}{t_k} \mbf{v}^k + \frac{r}{t_k} \mbf{s}_{\mbf{v}}^k, \vspace{1ex}\\
				\mbf{t}^k & := & \hat{\mbf{x}}^k - \frac{\gamma_k - \beta_k}{1 - \tau \sigma} (\tau \mbf{r}_{\mbf{x}}^k + \tau \sigma \mbf{r}_{\mbf{v}}^k), \vspace{1ex}\\			
				\mbf{x}^{k+1} & := & J_{\tau \mbf{\Lambda} T} (\hat{\mbf{x}}^k - \tau \hat{\mbf{v}}^k - \tau \mbf{\Lambda} G{\mbf{t}^k} + \tau \beta_k \mbf{r}_{\mbf{x}}^k), \vspace{1ex}\\
				\mbf{v}^{k+1} & := & \hat{\mbf{v}}^k + \frac{\sigma}{2\kappa} \mbf{\Lambda} (\Id - W) (2\mbf{x}^{k+1} - \hat{\mbf{x}}^k - \beta_k \mbf{x}^k), \vspace{1ex}\\			
				\mbf{r}_{\mbf{x}}^{k+1} & := & \frac{1}{\tau}(\hat{\mbf{x}}^k - \mbf{x}^{k+1}) - (\hat{\mbf{v}}^k - \mbf{v}^{k+1}) + \beta_k \mbf{r}_{\mbf{x}}^k, \vspace{1ex}\\
				\mbf{r}_{\mbf{v}}^{k+1} & := & -(\hat{\mbf{x}}^k - \mbf{x}^{k+1}) + \frac{1}{\sigma}(\hat{\mbf{v}}^k - \mbf{v}^{k+1}) + \beta_k \mbf{r}_{\mbf{v}}^k, \vspace{1ex}\\			
				\mbf{s}_{\mbf{x}}^{k+1} & := & \mbf{s}_{\mbf{x}}^k - \frac{\nu_k}{r(1-\tau\sigma)}(\tau\mbf{r}_{\mbf{x}}^{k+1} + \tau\sigma\mbf{r}_{\mbf{v}}^{k+1}), \vspace{1ex}\\ 
				\mbf{s}_{\mbf{v}}^{k+1} & := & \mbf{s}_{\mbf{v}}^k - \frac{\nu_k}{r(1-\tau\sigma)}(\tau\sigma \mbf{r}_{\mbf{x}}^{k+1} + \sigma\mbf{r}_{\mbf{v}}^{k+1}).
			\end{array}
			\right.
		\end{equation*}
		\STATE\hspace{0ex}{\bfseries End For}
	\end{algorithmic}
\end{algorithm}

In particular, if $\gamma_k = \beta_k = 0$, then the scheme \eqref{eq:DFKM_NI_scheme} reduces to
\begin{equation}\label{eq:DFKM_NI_scheme2}\tag{NI-DFFP2}
	\arraycolsep=0.2em
	\left\{
	\begin{array}{lcl}
		\hat{\mbf{x}}^k & := & \frac{t_k - r}{t_k} \mbf{x}^k + \frac{r}{t_k} \mbf{s}_{\mbf{x}}^k, \vspace{1ex}\\
		\hat{\mbf{v}}^k & := & \frac{t_k - r}{t_k} \mbf{v}^k + \frac{r}{t_k} \mbf{s}_{\mbf{v}}^k, \vspace{1ex}\\
		\mbf{x}^{k+1} & := & J_{\tau \mbf{\Lambda} T} (\hat{\mbf{x}}^k - \tau \hat{\mbf{v}}^k - \tau \mbf{\Lambda} G \hat{\mbf{x}}^k), \vspace{1ex}\\
		\mbf{v}^{k+1} & := & \hat{\mbf{v}}^k + \frac{\sigma}{2\kappa} \mbf{\Lambda} (\Id - \mbf{W}) (2\mbf{x}^{k+1} - \hat{\mbf{x}}^k), \vspace{1ex}\\			
		\mbf{r}_{\mbf{x}}^{k+1} & := & \frac{1}{\tau}(\hat{\mbf{x}}^k - \mbf{x}^{k+1}) - (\hat{\mbf{v}}^k - \mbf{v}^{k+1}), \vspace{1ex}\\
		\mbf{r}_{\mbf{v}}^{k+1} & := & -(\hat{\mbf{x}}^k - \mbf{x}^{k+1}) + \frac{1}{\sigma}(\hat{\mbf{v}}^k - \mbf{v}^{k+1}), \vspace{1ex}\\			
		\mbf{s}_{\mbf{x}}^{k+1} & := & \mbf{s}_{\mbf{x}}^k - \frac{\nu_k}{r(1-\tau\sigma)}(\tau\mbf{r}_{\mbf{x}}^{k+1} + \tau\sigma\mbf{r}_{\mbf{v}}^{k+1}), \vspace{1ex}\\ 
		\mbf{s}_{\mbf{v}}^{k+1} & := & \mbf{s}_{\mbf{v}}^k - \frac{\nu_k}{r(1-\tau\sigma)}(\tau\sigma \mbf{r}_{\mbf{x}}^{k+1} + \sigma\mbf{r}_{\mbf{v}}^{k+1}),
	\end{array}
	\right.
\end{equation}
where the iterates are initialized as $\mbf{s}_{\mbf{x}}^0 = \mbf{x}^0$ and $\mbf{v}^0 = \mbf{s}_{\mbf{v}}^0 = 0$.
This scheme is simpler and possesses fewer parameters than \eqref{eq:DFKM_NI_scheme}.
However, as we will see in Subsection~\ref{subsec:NI_convergence}, it requires stronger assumptions to converge. 

\beforesubsec
\subsection{Convergence analysis}\label{subsec:NI_convergence}
\aftersubsec
The convergence of \eqref{eq:DFKM_NI_scheme} and \eqref{eq:DFKM_NI_scheme2} can be stated as follows.

\begin{theorem}[\textbf{Monotone and Lipschitz continuous case}]\label{th:NI_DFKM_convergence}
	For \eqref{eq:DGE}, suppose that $G_i$ is monotone and $L_i$-Lipschitz continuous, and $T_i$ is maximally monotone for $i \in [n]$. 
	Let $\sets{\mbf{x}^k}$ be generated by \eqref{eq:DFKM_NI_scheme} using $t_0$ given in \eqref{eq:FKM_key_est2_t0} and
	\begin{equation}\label{eq:NI_DFKM_params}
		\arraycolsep=0.2em
		\hspace{-2ex}
		\begin{array}{c}
			r > 2, \quad \delta \in (0, r - 2), \quad \omega := r^2 - 3r + 3, \quad 0 < \nu < r - 2 - \delta, \vspace{1ex}\\
			0 < \alpha_i \leq \frac{1 - \tau \sigma}{2 \sqrt{1+\omega} \tau L_i}, \quad 2\kappa > \norms{\mbf{\Lambda}^{1/2}(\Id -  \mbf{W})\mbf{\Lambda}^{1/2}},  \vspace{1ex}\\
			t_{k} := k +  t_0, \quad \beta_k := \frac{\delta(t_k - r)}{2(r-1)t_k}, \quad \nu_k := \frac{\nu(t_k - 1)}{t_k}, \quad \text{and} \quad \gamma_k := 1,
		\end{array}
		\hspace{-2ex}
	\end{equation}
	where $\mbf{\Lambda} = \diag{\alpha_1, \cdots, \alpha_n}$.
	Then, the following statements hold.
	\begin{compactitem}
	\item[$~~~\mathrm{(i)}$~\textbf{$($Local iterates$)$}]
	The local iterates $\sets{x_{[i]}^k}$ converge to the same $u^\star \in \zer{\Phi}$ of the original \eqref{eq:DGE} problem for all nodes $i \in [n]$.
	
 	\item[$~~~\mathrm{(ii)}$~\textbf{$($Consensus error$)$}]
	The consensus error $\norms{\Pi \mbf{x}^k }$ converges to zero at the rate of $\SmallOs{1/k}$.
	
	\item[$~~~\mathrm{(iii)}$~\textbf{$($Aggregated residual$)$}]
	The aggregated residual $\norms{\sum_{i=1}^n G_i x_{[i]}^k + \xi_i^k}$ converges to zero at the rate of $\SmallOs{1/k}$, where $\xi_i^k \in Tx_{[i]}^k$.
		
	\item[$~~~\mathrm{(iv)}$~\textbf{$($Restricted gap$)$}]
	If, in addition, Assumption~\ref{as:restricted_dual_gap} holds, then the restricted gap value $\mathrm{Gap}_B(\bar{x}^k)$ converges to zero at the rate of $\SmallOs{1/k}$, where $\bar{x}^k := \frac{1}{n}\sum_{i=1}^n x_{[i]}^k$.
	Otherwise, if, in addition, Assumption~ \ref{as:boundary_regularity} holds, then $\mathrm{Gap}_B(p^k)$ converges to zero at the rate of $\SmallOs{1/k}$, where $p^k := \proj_{\Bc}(\bar{x}^k)$.

	\item[$~~~\mathrm{(v)}$~\textbf{$($FBS residual$)$}]
	For a given $\lambda \in (0, \frac{1}{\bar{L}})$ with $\bar{L} = \frac{1}{n}\sum_{i=1}^n L_i$, if, in addition, Assumption~\ref{as:restricted_dual_gap} holds, then $\norms{\Gc_{\lambda}\bar{x}^k}^2$ converges to zero at the rate of $\SmallOs{1/k}$.
	Otherwise, if, in addition, Assumption~\ref{as:boundary_regularity} holds, then $\norms{\Gc_{\lambda}p^k}^2$ converges to zero at the rate of $\SmallOs{1/k}$, where $p^k := \proj_{\Bc}(\bar{x}^k)$.
	
	\item[$~~~\mathrm{(vi)}$~\textbf{$($Special case$)$}]
	If $T_i = 0$ for all $i \in [n]$ $($i.e., \eqref{eq:DGE} reduces to \eqref{eq:DNE}$)$, then $\norms{\sum_{i=1}^n G_i \bar{x}^k}$ converges to zero at the rate of $\SmallOs{1/k}$.
	\end{compactitem}
\end{theorem}

\begin{proof}
	For the sake of clarity, we divide this proof into the following steps.
	
	\vspace{0.75ex}
	\noindent\textbf{(i)~The convergence of iterate sequences.}
	Denote $\mbf{\Gamma} = \mbf{\Lambda}^{1/2} \succ 0$, $\mbf{K} = (\frac{\Id - W}{2})^{1/2} \succeq 0$, 
	and $\mbf{M} = (\kappa\Id - \mbf{K}\mbf{\Gamma}^2 \mbf{K})^{1/2} \succ 0$.
	From Lemma~\ref{le:DFKM_nihs1}, $\mbf{M}$ is well-defined.
	Besides, for the operators $A$, $B$, and $C$ defined in \eqref{eq:DFKM_3op_operators}, by Lemma~\ref{le:DFKM_3op_reform}, we know that $A$ and $C$ are maximally monotone, and $B$ is monotone and $L$-Lipschitz continuous with $L_B = \max_i \sets{\alpha_i L_i}$.
	Moreover, $(\mbf{a}^{\star}, \breve{\mbf{a}}^{\star}) \in \Hc^n \times \Hc^n$ satisfies
	\begin{equation*} 
		\arraycolsep=0.2em
		\begin{array}{lcl}
			0 \in ({A} + {B} + {C})(\mbf{a}^{\star}, \breve{\mbf{a}}^{\star}) \subset \Hc^n \times \Hc^n
		\end{array}
	\end{equation*}
	if and only if $\mbf{a}^{\star} = \Gamma^{-1}(u^{\star}, \cdots, u^{\star})$ for some $u^{\star} \in \R^p$ and $u^{\star}$ solves $0 \in \sum_{i=1}^n\big(G_iu^{\star} + T_iu^{\star}\big)$.
	
	Note that since the stepsizes $\alpha_i \leq \frac{1 - \tau \sigma}{2 \sqrt{1+\omega} \tau L_i}$ for all $i \in [n]$, we have $2\sqrt{1+\omega}\tau L_B + \tau \sigma \leq 1$.
	Thus, applying Theorem~\ref{th:FKM_3OP_convergence}, we conclude that $\sets{\mbf{a}^k}$, $\sets{y_{\mbf{a}}^k}$, $\sets{z_{\mbf{a}}^k}$, and $\sets{\hat{\mbf{a}}^k}$ converge to $\mbf{a}^{\star} = \Gamma^{-1}(u^{\star}, \cdots, u^{\star})$.
	Since $\mbf{x}^k = \Gamma \mbf{a}^k$, $\hat{\mbf{x}}^k = \Gamma \hat{\mbf{a}}^k$, $\mbf{t}^k = \Gamma y_{\mbf{a}}^k$, and $\mbf{s}_{\mbf{z}}^k = \Gamma z_{\mbf{a}}^k$, we also get that $\sets{\mbf{x}^k}$, $\sets{\mbf{t}^k}$, $\sets{\mbf{s}_{\mbf{z}}^k}$, and $\sets{\hat{\mbf{x}}^k}$ converge to $\mbf{u}^{\star} = (u^{\star}, \cdots, u^{\star})$, where $u^{\star} \in \zer{\sum_{i=1}^n (G_i + T_i)}$.
		
	\vspace{0.75ex}
	\noindent\textbf{(ii)~The convergence of consensus error.}
	From \eqref{eq:FKM_3OP_convergence_proof2} of Theorem~\ref{th:FKM_3OP_convergence}, we know that 
	\begin{equation}\label{eq:DFKM_NI_convergence_proof1}
	\arraycolsep=0.2em
	\begin{array}{ccc}
		\norms{\vartheta^k - u^k} = \BigOs{1/k} \quad &\text{and}& \quad \norms{\vartheta^k - u^k} = \SmallOs{1/k},
	\end{array}
	\end{equation}
	where $\vartheta^k \in C^{-1}v^k$, which is equivalent to $v^k \in C\vartheta^k$.
	By the definition of $C$, we have $C\vartheta^k = \mbf{E}^\top \Nc_{\sets{0}} (\mbf{E}\vartheta^k)$, where $\mbf{E} = \begin{bmatrix}
		\mbf{K}\mbf{\Gamma}, & \mbf{M}
	\end{bmatrix}$.
	Thus, for the existence of $v^k$, we must have $\mbf{E}\vartheta^k = 0$.
	Since $\mbf{E}\mbf{E}^\top = \kappa \Id$, we have $\norms{\mbf{E}} = \sqrt{\lambda_{\max}(\mbf{E}\mbf{E}^\top)} = \sqrt{\kappa}$.
	Then, we can derive that
	\begin{equation}\label{eq:DFKM_NI_convergence_proof2}
	\arraycolsep=0.2em
	\begin{array}{ccc}
		\norms{\mbf{E}u^k} = \norms{\mbf{E}(\vartheta^k - u^k)} \leq \norms{\mbf{E}}\norms{\vartheta^k - u^k} = \sqrt{\kappa}\norms{\vartheta^k - u^k}.
	\end{array}
	\end{equation}
	Recall that $\mbf{E} = \begin{bmatrix} \mbf{K}\mbf{\Gamma}, & \mbf{M} \end{bmatrix}$, 
	$u^k = \begin{bmatrix} \mbf{a}^k \\ 0 \end{bmatrix}$, 
	and $\mbf{x}^k = \Gamma \mbf{a}^k$, we can compute that $\mbf{E}u^k = \mbf{K}\mbf{\Gamma} \mbf{a}^k = \mbf{K}\mbf{x}^k$.
	Combining this fact, \eqref{eq:DFKM_NI_convergence_proof1}, and \eqref{eq:DFKM_NI_convergence_proof2}, we get
	\begin{equation}\label{eq:DFKM_NI_convergence_proof3}
	\arraycolsep=0.2em
	\begin{array}{ccc}
		\norms{\mbf{K}\mbf{x}^k} = \BigOs{1/k} \quad &\text{and}& \quad \norms{\mbf{K}\mbf{x}^k} = \SmallOs{1/k},
	\end{array}
	\end{equation}
	Using these results, similarly to part (b) of Theorem~\ref{th:DFKM_convergence}, 
	we can also prove that the consensus error also converges to zero as follows:
	\begin{equation*}
	\arraycolsep=0.2em
	\begin{array}{rcl}	
		\norms{\Pi\mbf{x}^k} = \BigOs{1/k} \qquad &\text{and}& \qquad \norms{\Pi\mbf{x}^k} = \SmallOs{1/k}.
	\end{array}
	\end{equation*}	
	Moreover, since $\norms{\Pi\mbf{x}^k}^2 = \norms{\mbf{x}^k - \boldsymbol{1}\bar{x}^k}^2 = \sum_{i=1}^n \norms{x_{[i]}^k - \bar{x}^k}^2$, we also have
	\begin{equation*}
	\arraycolsep=0.2em
	\begin{array}{rcl}	
		\norms{x_i^k - \bar{x}^k} = \BigOs{1/k} \quad & \text{and} & \quad \norms{x_i^k - \bar{x}^k} = \SmallOs{1/k} \quad \text{for all $i \in [n]$}.
	\end{array}
	\end{equation*}
	
	\vspace{0.75ex}
	\noindent\textbf{(iii)~The convergence of aggregated  residuals.}
	From \eqref{eq:FKM_3OP_convergence_proof2} of Theorem~\ref{th:FKM_3OP_convergence}, we have
	\begin{equation*} 
	\arraycolsep=0.2em
	\begin{array}{ccc}
		\norms{\zeta^k + Bu^k + v^k} = \BigOs{1/k} \quad &\text{and}& \quad \norms{\zeta^k + Bu^k + v^k} = \SmallOs{1/k}.
	\end{array}
	\end{equation*}
	Using $Bu^k = \begin{bmatrix}
		\mbf{\Gamma} G(\mbf{\Gamma} \mbf{a}^k) \\ 0
	\end{bmatrix} = \begin{bmatrix}
		\mbf{\Gamma} G\mbf{x}^k \\ 0
	\end{bmatrix}$, $\zeta^k = \begin{bmatrix}
		\mbf{\Gamma} \boldsymbol{\xi}^k \\ 0
	\end{bmatrix}$ for $\boldsymbol{\xi}^k \in T(\mbf{\Gamma} \mbf{a}^k) = T\mbf{x}^k$, and $v^k = \begin{bmatrix}
		\mbf{b}^k \\ \breve{\mbf{b}}^k
	\end{bmatrix}$, we obtain from the last relation that
	\begin{equation*} 
	\arraycolsep=0.2em
	\begin{array}{ccc}
		\norms{\mbf{\Gamma} \boldsymbol{\xi}^k + \mbf{\Gamma} G\mbf{x}^k + \mbf{b}^k} = \BigOs{1/k} \quad &\text{and}& \quad \norms{\mbf{\Gamma} \boldsymbol{\xi}^k + \mbf{\Gamma} G\mbf{x}^k + \mbf{b}^k} = \SmallOs{1/k}.
	\end{array}
	\end{equation*}
	Moreover, since $\norms{\boldsymbol{\xi}^k + G\mbf{x}^k + \mbf{\Gamma}^{-1}\mbf{b}^k} = \norms{\mbf{\Gamma}^{-1}(\mbf{\Gamma} \boldsymbol{\xi}^k + \mbf{\Gamma} G\mbf{x}^k + \mbf{b}^k)} \leq \norms{\mbf{\Gamma}^{-1}} \norms{\mbf{\Gamma} \boldsymbol{\xi}^k + \mbf{\Gamma} G\mbf{x}^k + \mbf{b}^k} = (\min_i \sqrt{\alpha_i})^{-1} \norms{\mbf{\Gamma} \boldsymbol{\xi}^k + \mbf{\Gamma} G\mbf{x}^k + \mbf{b}^k}$, we deduce from the last relation that
	\begin{equation*} 
	\arraycolsep=0.2em
	\begin{array}{ccc}
		\norms{\boldsymbol{\xi}^k + G\mbf{x}^k + \mbf{\Gamma}^{-1}\mbf{b}^k} = \BigOs{1/k} \quad &\text{and}& \quad \norms{\boldsymbol{\xi}^k + G\mbf{x}^k + \mbf{\Gamma}^{-1}\mbf{b}^k} = \SmallOs{1/k}.
	\end{array}
	\end{equation*}
	Now, since $v^k \in C\vartheta^k = \mbf{E}^\top \Nc_{\sets{0}}(\mbf{E}\vartheta^k)$, there exists some $\mbf{s}^k \in \Nc_{\sets{0}}(0)$ such that $v^k = \mbf{E}^\top\mbf{s}^k$.
	Using $\mbf{E}^\top = \begin{bmatrix}
		\Gamma K \\ M
	\end{bmatrix}$, we get $\mbf{b}^k = \mbf{\Gamma} \mbf{K} s^k$, thus $\mbf{\Gamma}^{-1}\mbf{b^k} = \mbf{K}\mbf{s}^k$.
	Substituting this relation into the last expression, we get
	\begin{equation*} 
	\arraycolsep=0.2em
	\begin{array}{ccc}
		\norms{\boldsymbol{\xi}^k + G\mbf{x}^k + \mbf{K}\mbf{s}^k} = \BigOs{1/k} \quad &\text{and}& \quad \norms{\boldsymbol{\xi}^k + G\mbf{x}^k + \mbf{K}\mbf{s}^k} = \SmallOs{1/k}.
	\end{array}
	\end{equation*}
	Utilizing these results, similar to part (c) of Theorem~\ref{th:DFKM_convergence}, we can also prove that the aggregated residual converges to zero as
	\begin{equation*}
		\arraycolsep=0.2em
		\begin{array}{rcl}	
			\norms{\sum_{i=1}^n (G_i x_{[i]}^k + \xi_i^k)} = \BigOs{1/k} 
			\qquad &\text{and}& \qquad 
			\norms{\sum_{i=1}^n (G_i x_{[i]}^k + \xi_i^k)} = \SmallOs{1/k}.
		\end{array}
	\end{equation*}
	Finally, it is easy to prove that the average operator $\bar{G}$ is $\bar{L}$-Lipschitz continuous with $\bar{L} = \frac{1}{n}\sum_{i=1}^n L_i$.
	Then, using the above convergence results together with Theorem \ref{th:FKM_3OP_convergence}(ii), all the ingredients required in the proofs of Theorem ~\ref{th:DFKM_convergence}(iv)–(vi) are satisfied. 
	Hence, parts (iv)–(vi) follow by similar arguments, and we omit the details.
\end{proof}

\begin{theorem}[\textbf{Co-coercive case}]\label{th:NI_DFKM2_convergence}
	For \eqref{eq:DGE}, suppose that $G_i$ is $\frac{1}{L}$-co-coercive and $T_i$ is maximally monotone for all $i \in [n]$. 
	Let $\sets{\mbf{x}^k}$ be generated by \eqref{eq:DFKM_NI_scheme2} using  $t_0$ given in \eqref{eq:t0_cond2} and the following parameters:
	\begin{equation}\label{eq:NI-DFKM2_params}
		\arraycolsep=0.2em
		\begin{array}{c}
			r > 1, \qquad 0 < \nu < r - 1, \qquad 0 < \alpha_i < \frac{2(1 - \tau \sigma)}{\tau L_i}, \vspace{1ex}\\
			2\kappa > \norms{\mbf{\Lambda}^{1/2}(\Id - W)\mbf{\Lambda}^{1/2}}, \qquad 
			t_{k} := k +  t_0, \qquad \text{and} \qquad \nu_k := \frac{\nu(t_k - 1)}{t_k},
		\end{array}
	\end{equation}
	where $\mbf{\Lambda} := \diag{\alpha_1, \cdots, \alpha_n}$.
	Then, the following statements hold.
	\begin{compactenum}
	\item[$~~~\mathrm{(i)}$~\textbf{$($Local iterates$)$}]
	The local iterates $\sets{x_{[i]}^k}$ converge to the same $u^\star \in \zer{\Phi}$ of the original \eqref{eq:DGE} problem for all nodes $i \in [n]$.
	
	\item[$~~~\mathrm{(ii)}$~\textbf{$($Consensus error$)$}]
	The consensus error $\norms{\Pi \mbf{x}^k }$ converges to zero at the rate of $\SmallOs{1/k}$.
	
	\item[$~~~\mathrm{(iii)}$~\textbf{$($Aggregated residual$)$}]
	The aggregated residual $\norms{\sum_{i=1}^n G_i x_{[i]}^k + \xi_i^k}$ converges to zero at the rate of $\SmallOs{1/k}$, where $\xi_i^k \in Tx_{[i]}^k$.
	
	\item[$~~~\mathrm{(iv)}$~\textbf{$($Restricted gap$)$}]
	If, in addition, Assumption~\ref{as:restricted_dual_gap} holds, then $\mathrm{Gap}_B(\bar{x}^k)$ converges to zero at the rate of $\SmallOs{1/k}$, where $\bar{x}^k := \frac{1}{n}\sum_{i=1}^n x_{[i]}^k$.
	Otherwise, if, in addition, Assumption~\ref{as:boundary_regularity} holds, then $\mathrm{Gap}_B(p^k)$ converges to zero at the rate of $\SmallOs{1/k}$, where $p^k := \proj_{\Bc}(\bar{x}^k)$.
	
	\item[$~~~\mathrm{(v)}$~\textbf{$($FBS residual$)$}]
	For a given $\lambda \in (0, \frac{1}{\bar{L}})$ with $\bar{L} = \frac{1}{n}\sum_{i=1}^n L_i$, if, in addition, Assumption~\ref{as:restricted_dual_gap} holds, then $\norms{\Gc_{\lambda}\bar{x}^k}^2$ converges to zero at the rate of $\SmallOs{1/k}$.
	Otherwise, if, in addition, Assumption~\ref{as:boundary_regularity} holds, then $\norms{\Gc_{\lambda}p^k}^2$ converges to zero at the rate of $\SmallOs{1/k}$, where $p^k := \proj_{\Bc}(\bar{x}^k)$.
	
	\item[$~~~\mathrm{(vi)}$~\textbf{$($Special case$)$}]
	If $T_i = 0$ for all $i \in [n]$ $($i.e., \eqref{eq:DGE} reduces to \eqref{eq:DNE}$)$, then $\norms{\sum_{i=1}^n G_i \bar{x}^k}$ converges to zero at the rate of $\SmallOs{1/k}$.
	\end{compactenum}
\end{theorem}

\begin{proof}
	Since $0 < \alpha_i < \frac{2(1 - \tau \sigma)}{\tau L_i}$, we can show that $\frac{1}{2}\tau L_B + \tau \sigma < 1$.
	Therefore, all the conditions of \eqref{eq:FKM_3OP_th2_params} in Theorem~\ref{th:FKM_3OP_convergence2} are satisfied.
	The remainder of this proof is very similar to that of Theorem~\ref{th:NI_DFKM_convergence}, and hence is omitted. 
\end{proof}

\begin{remark}
	Although Algorithm~\ref{alg:NI_DFFP} requires $2\kappa > \norms{\mbf{\Lambda}^{1/2}(\Id - \mbf{W})\mbf{\Lambda}^{1/2}}$, in practice, we can choose a valid value $\kappa = \max_{ i \in [n]} \alpha_i$ which does not requires spectral information of the network.
	More importantly, the local stepsize bound $0 < \alpha_i \leq \frac{1 - \tau \sigma}{2 \sqrt{1+\omega} \tau L_i}$ or $0 < \alpha_i < \frac{2(1 - \tau \sigma)}{\tau L_i}$ depends only on $L_i$, allowing each agent to exploit its own local information instead of using a common stepsize determined by the worst local operator and network topology.
\end{remark}
Theorems \ref{th:NI_DFKM_convergence} and \ref{th:NI_DFKM2_convergence} show that our methods are optimal w.r.t. the restricted gap metric both for Lipschitz continuous and co-coercive operators, see \cite{Nesterov2007a}, answering positively to the two questions posed in the introduction: design fully distributed algorithms having optimal rates in the original primal space.

\beforesec
\section{Numerical Experiments} \label{sec:numerical_experiments}
\aftersec
We conduct extensive numerical experiments to evaluate the algorithms developed in the previous sections for the distributed composite monotone inclusion \eqref{eq:DGE}.
The two examples are adapted from \cite{tam2025decentralised} and extended to cover additional settings.
For completeness and readability, we restate the necessary details to make the presentation self-contained.
All algorithms are implemented in Python and run on the Longleaf computing cluster using NVIDIA A100 (40GB) and L40 (80GB) GPUs.

\beforesubsec
\subsection{Bilinear matrix games}\label{subsec:exam1_bilinear_matrix_game}
\aftersubsec
We consider a zero-sum matrix game between two teams of $n$ players, where team one has the payoff matrix $M := \sum_{i=1}^n M_i \in \R^{p_2 \times p_1}$ and team two has the payoff matrix $-M$.
Each player in team one is paired with a player from team two, where pair $i$ is associated with the payoff matrices $(M_i, -M_i)$.
A Nash equilibrium of this game can be characterized as a saddle point of the following bilinear minimax problem:
\begin{equation}\label{eq:decentralized_matrix_game}
\min\limits_{u \in \Delta_{p_1}} \max\limits_{v \in \Delta_{p_2}}\Big\{ \Lc(u, v) :=  \sum_{i=1}^n \iprods{M_i u, v} \Big\},
\end{equation}
where $\Delta_p := \sets{ u = (u_1, \cdots, u_p) \in \R^p_+ : \sum_{j=1}^p u_j = 1}$ denotes the unit simplex in $\R^p$.

For all $i \in [n] := \sets{1, \cdots, n}$, if we define
\begin{equation*}
\arraycolsep=0.2em
\begin{array}{c}
x := \begin{bmatrix} u \\ v \end{bmatrix} \in \mathbb{R}^{p_1 + p_2}, \quad
G_i(x) := \begin{bmatrix} 0 & M_i^{\top} \\ - M_i & 0 \end{bmatrix} x, \quad \text{and} \quad 
T_i(x) := \begin{bmatrix} \mathcal{N}{\Delta{p_1}}(u) \\ \mathcal{N}{\Delta{p_2}}(v) \end{bmatrix},
\end{array}
\end{equation*}
then \eqref{eq:decentralized_matrix_game} can be formulated as a special case of \eqref{eq:DGE}.

\vspace{0.75ex}
\noindent\textbf{$\mathrm{(a)}$~Data generation.}
We choose $p_1 = p_2 = p = 100$ and construct the matrices $M_i \in \R^{p \times p}$ as $M_i = s_i \Id - K_i$, where $s_i > \norms{K_i}$ and $K_i$ has positive entries.
Under this construction, the bilinear matrix game problem \eqref{eq:decentralized_matrix_game} is \textit{completely mixed} \cite{cohen1986perturbation} and therefore admits a unique Nash equilibrium \cite{raghavan1978completely}.
We compute this equilibrium using the \texttt{linprog} solver from \texttt{SciPy} and use it to evaluate the relative optimality gap.
Specifically, we generate $K_i = i U_i$, where the entries of $U_i \in \R^{p \times p}$ are independently sampled from $\mathcal{U}(0,1)$, and set $s_i := 1.1\norms{K_i}$.

\vspace{0.75ex}
\noindent\textbf{$\mathrm{(b)}$~Network topologies.}
We consider three different communication networks: cycle, barbell, and 2D grid networks, with a total of $n$ agents, where $n \in \sets{10, 20, 50, 100, 500}$.
For the mixing matrix $W$, we use a Laplacian-based constant edge-weight matrix with $\gamma := 0.505\lambda_{\max}(\Lc)$.

\vspace{0.75ex}
\noindent\textbf{$\mathrm{(c)}$~Algorithms and parameters.}
We implement two proposed algorithms: Algorithm~\ref{alg:ND_DFFP} (called \texttt{ND-DFFP}) and Algorithm~\ref{alg:NI_DFFP} (called \texttt{NI-DFFP}), with the parameters chosen as follows:
\begin{compactitem}
	\item 
	For \texttt{ND-DFFP}: $r := 3.3$, $\delta := 10^{-6}$, $\nu := r - 2 - 1.001\delta$, $\sigma := 10^2$, $\tau := \frac{2}{4\sqrt{1 + \omega}L_{\max} + \sigma (1-\lambda_{\min}(W))}$.
	
	\item 
	For \texttt{NI-DFFP}: $r := 6.2$, $\delta := 0.115$, $\nu := r - 2 - 1.001\delta$, $\tau := 0.1$, $\sigma := 2.8$, the heterogeneous stepsizes $\alpha_i := \frac{0.99(1 - \tau \sigma)}{2\sqrt{1+\omega}\tau L_i}$, and $\kappa := 0.5(1+10^{-10})\norms{\Lambda^{1/2}(\Id - W)\Lambda^{1/2}}$.
\end{compactitem}
Here, recall that $\Lambda := \diag{\alpha_1, \cdots, \alpha_n}$.

For benchmarking, we also implement four algorithms in \cite{tam2025decentralised}: one is \eqref{eq:malitsky2026_alg} and three different variants of \eqref{eq:tam2026_alg}.
We choose  the parameters as proposed in \cite{tam2025decentralised}:
\begin{compactitem}
	\item \texttt{NI-DFB} with stepsizes $\alpha_i := \frac{0.9}{8L_i}$ and parameter $\beta = \beta_{\max} := 0.9(\max_i\sets{\alpha_i})^{-1}$.
	\item \texttt{NI-DFB} with stepsizes $\alpha_i := \frac{0.9}{8L_i}$ and parameter $\beta = \beta_{\text{norm}} := 0.45\norms{\Lambda^{1/2}(\Id - W)\Lambda^{1/2}}^{-1}$.
	\item \texttt{NI-DFB} with stepsize $\alpha_i := \frac{0.9}{8\max_i \sets{L_i}}$ and parameter $\beta = \beta_{\max} := 0.9(\max_i\sets{\alpha_i})^{-1}$.
	\item \texttt{ND-DFRB} with stepsize $\alpha := \frac{0.9(1 + \lambda_{\min}(W))}{4\max_i\sets{L_i}}$.
\end{compactitem}
\vspace{0.75ex}
\noindent\textbf{$\mathrm{(d)}$~Reported metrics.}
The performance of each algorithm is evaluated across four metrics:
\begin{compactenum}
	\item 
	Relative consensus error (\texttt{Cons.~Err.}): 
	$\frac{\norms{\mbf{x}^k - \mbf{1}\bar{x}^k}}{\max\sets{\norms{\mbf{x}^0 - \mbf{1}\bar{x}^0}, 1.0}}$, 
	where $\mbf{x}^k = (x_{[1]}^k, \cdots, x_{[n]}^k)$ and $\bar{x}^k := \frac{1}{n}\sum_{i=1}^n x_{[i]}^k$.
	
	\item 
	Relative aggregated residual (\texttt{Agg.~Res.}): 
	$\frac{\norms{\sum_{i=1}^n G_i x_i^k + \xi_i^k}}{\max\sets{\norms{\sum_{i=1}^n G_i x_i^0 + \xi_i^0}, 1.0}}$.
	
	\item 
	Relative FBS residual evaluated at the average iterate (\texttt{FBS~Res.}): 
	$\frac{\norms{\Gc_{\tau}(\bar{x}^k)}}{\max\sets{\norms{\Gc_{\tau}(\bar{x}^0)}, 1.0}}$.
	
	\item 
	Relative optimality gap (\texttt{Opt.~Gap}): 
	$\frac{\norms{\mbf{x}^k - \mbf{1}x^{\star}}}{\norms{\mbf{1}x^{\star}}}$.
\end{compactenum}
For each experiment, we run all algorithms on $50$ problem instances and report the mean performance metrics against the number of iterations.

\vspace{0.75ex}
\noindent\textbf{$\mathrm{(e)}$~Numerical results.}
The numerical results of all the algorithms after $10{,}000$ iterations are reported in Tables~\ref{tab:matrix_games_results_cycle}, \ref{tab:matrix_games_results_grid2d}, and \ref{tab:matrix_games_results_barbell}.
In addition, Figure~\ref{fig:cycle20_matrix_games} illustrates the performance of the six methods on the cycle network with $20$ agents.

\begin{table}[ht]
	\centering
	\caption{Numerical results for the cycle network after $10^4$ iterations, averaged over $50$ instances.}
	\label{tab:matrix_games_results_cycle}
	\resizebox{\textwidth}{!}{
		\begin{tabular}{llcccccc}
			\toprule
			$n$ & {Metrics} & \texttt{NI-DFFP} & \texttt{ND-DFFP} & \makecell{\texttt{NI-DFB}\\($\beta = \beta_{\max}$)} & \makecell{\texttt{NI-DFB}\\($\beta = \beta_{\text{norm}}$)} & \makecell{\texttt{NI-DFB}\\($\alpha_1 = \cdots = \alpha_n$)} & \texttt{ND-DFRB} \\
			\midrule
			\multirow{4}{*}{10} & Cons. Err. & $1.88 \times 10^{-14}$ & $1.56 \times 10^{-7}$ & $3.02 \times 10^{-16}$ & $\mathbf{1.70 \times 10^{-16}}$ & $1.80 \times 10^{-10}$ & $1.63 \times 10^{-6}$ \\
			& Agg. Res. & $3.37 \times 10^{-12}$ & $9.91 \times 10^{-7}$ & $\mathbf{2.78 \times 10^{-12}}$ & $3.56 \times 10^{-12}$ & $2.98 \times 10^{-8}$ & $9.03 \times 10^{-1}$ \\
			& FBS Res. & $2.26 \times 10^{-11}$ & $4.38 \times 10^{-5}$ & $\mathbf{1.47 \times 10^{-12}}$ & $1.91 \times 10^{-12}$ & $3.88 \times 10^{-7}$ & $7.39 \times 10^{-1}$ \\
			& Opt. Gap & $1.95 \times 10^{-12}$ & $1.53 \times 10^{-6}$ & $\mathbf{3.53 \times 10^{-14}}$ & $4.53 \times 10^{-14}$ & $9.21 \times 10^{-9}$ & $5.21 \times 10^{-1}$ \\
			\midrule
			\multirow{4}{*}{20} & Cons. Err. & $\mathbf{1.10 \times 10^{-6}}$ & $1.17 \times 10^{-5}$ & $4.21 \times 10^{-4}$ & $5.99 \times 10^{-5}$ & $4.76 \times 10^{-4}$ & $5.88 \times 10^{-6}$ \\
			& Agg. Res. & $\mathbf{3.91 \times 10^{-8}}$ & $1.41 \times 10^{-6}$ & $2.44 \times 10^{-5}$ & $6.00 \times 10^{-6}$ & $7.21 \times 10^{-4}$ & $9.16 \times 10^{-1}$ \\
			& FBS Res. & $\mathbf{3.52 \times 10^{-5}}$ & $1.94 \times 10^{-3}$ & $5.78 \times 10^{-2}$ & $8.10 \times 10^{-3}$ & $5.83 \times 10^{-2}$ & $8.71 \times 10^{-1}$ \\
			& Opt. Gap & $\mathbf{8.56 \times 10^{-6}}$ & $8.64 \times 10^{-5}$ & $3.22 \times 10^{-3}$ & $4.80 \times 10^{-4}$ & $3.54 \times 10^{-3}$ & $5.26 \times 10^{-1}$ \\
			\midrule
			\multirow{4}{*}{50} & Cons. Err. & $4.10 \times 10^{-4}$ & $1.29 \times 10^{-3}$ & $2.18 \times 10^{-3}$ & $1.84 \times 10^{-3}$ & $2.03 \times 10^{-3}$ & $\mathbf{3.61 \times 10^{-5}}$ \\
			& Agg. Res. & $\mathbf{6.85 \times 10^{-7}}$ & $1.38 \times 10^{-5}$ & $2.18 \times 10^{-5}$ & $1.95 \times 10^{-5}$ & $9.04 \times 10^{-4}$ & $9.22 \times 10^{-1}$ \\
			& FBS Res. & $\mathbf{2.79 \times 10^{-2}}$ & $2.34 \times 10^{-1}$ & $3.05 \times 10^{-1}$ & $2.94 \times 10^{-1}$ & $2.92 \times 10^{-1}$ & $9.98 \times 10^{-1}$ \\
			& Opt. Gap & $\mathbf{3.10 \times 10^{-3}}$ & $9.51 \times 10^{-3}$ & $1.60 \times 10^{-2}$ & $1.36 \times 10^{-2}$ & $1.49 \times 10^{-2}$ & $5.29 \times 10^{-1}$ \\
			\midrule
			\multirow{4}{*}{100} & Cons. Err. & $1.41 \times 10^{-3}$ & $1.84 \times 10^{-3}$ & $2.94 \times 10^{-3}$ & $2.52 \times 10^{-3}$ & $2.85 \times 10^{-3}$ & $\mathbf{1.46 \times 10^{-4}}$ \\
			& Agg. Res. & $\mathbf{5.69 \times 10^{-7}}$ & $8.84 \times 10^{-6}$ & $1.92 \times 10^{-5}$ & $1.62 \times 10^{-5}$ & $8.35 \times 10^{-4}$ & $9.22 \times 10^{-1}$ \\
			& FBS Res. & $\mathbf{1.46 \times 10^{-1}}$ & $2.97 \times 10^{-1}$ & $4.25 \times 10^{-1}$ & $4.17 \times 10^{-1}$ & $4.20 \times 10^{-1}$ & $1.07$ \\
			& Opt. Gap & $\mathbf{1.04 \times 10^{-2}}$ & $1.34 \times 10^{-2}$ & $2.13 \times 10^{-2}$ & $1.84 \times 10^{-2}$ & $2.07 \times 10^{-2}$ & $5.30 \times 10^{-1}$ \\
			\midrule
			\multirow{4}{*}{500} & Cons. Err. & $\mathbf{2.20 \times 10^{-3}}$ & $2.75 \times 10^{-3}$ & $4.50 \times 10^{-3}$ & $3.98 \times 10^{-3}$ & $4.32 \times 10^{-3}$ & $5.64 \times 10^{-3}$ \\
			& Agg. Res. & $\mathbf{4.09 \times 10^{-7}}$ & $6.00 \times 10^{-6}$ & $1.59 \times 10^{-5}$ & $1.22 \times 10^{-5}$ & $2.53 \times 10^{-4}$ & $8.63 \times 10^{-1}$ \\
			& FBS Res. & $\mathbf{1.84 \times 10^{-1}}$ & $4.06 \times 10^{-1}$ & $6.26 \times 10^{-1}$ & $6.42 \times 10^{-1}$ & $5.93 \times 10^{-1}$ & $9.91 \times 10^{-1}$ \\
			& Opt. Gap & $\mathbf{1.58 \times 10^{-2}}$ & $1.98 \times 10^{-2}$ & $3.23 \times 10^{-2}$ & $2.86 \times 10^{-2}$ & $3.09 \times 10^{-2}$ & $5.07 \times 10^{-1}$ \\
			\bottomrule
		\end{tabular}
	}
\end{table}

\begin{table}[ht]
	\centering
	\caption{Numerical results for the 2D Grid network after $10^4$ iterations, averaged over $50$ instances.}
	\label{tab:matrix_games_results_grid2d}
	\resizebox{\textwidth}{!}{
		\begin{tabular}{llcccccc}
			\toprule
			$n$ & {Metrics} & \texttt{NI-DFFP} & \texttt{ND-DFFP} & \makecell{\texttt{NI-DFB}\\($\beta = \beta_{\max}$)} & \makecell{\texttt{NI-DFB}\\($\beta = \beta_{\text{norm}}$)} & \makecell{\texttt{NI-DFB}\\($\alpha_1 = \cdots = \alpha_n$)} & \texttt{ND-DFRB} \\
			\midrule
			\multirow{4}{*}{10} & Cons. Err. & $1.02 \times 10^{-14}$ & $1.81 \times 10^{-7}$ & $5.14 \times 10^{-13}$ & $\mathbf{1.31 \times 10^{-16}}$ & $4.09 \times 10^{-10}$ & $1.43 \times 10^{-6}$ \\
			& Agg. Res. & $3.38 \times 10^{-12}$ & $9.84 \times 10^{-7}$ & $\mathbf{3.18 \times 10^{-12}}$ & $3.47 \times 10^{-12}$ & $1.75 \times 10^{-8}$ & $9.03 \times 10^{-1}$ \\
			& FBS Res. & $2.26 \times 10^{-11}$ & $4.07 \times 10^{-5}$ & $4.40 \times 10^{-11}$ & $\mathbf{1.95 \times 10^{-12}}$ & $2.24 \times 10^{-7}$ & $7.39 \times 10^{-1}$ \\
			& Opt. Gap & $1.94 \times 10^{-12}$ & $1.71 \times 10^{-6}$ & $4.23 \times 10^{-12}$ & $\mathbf{4.61 \times 10^{-14}}$ & $6.20 \times 10^{-9}$ & $5.20 \times 10^{-1}$ \\
			\midrule
			\multirow{4}{*}{20} & Cons. Err. & $5.78 \times 10^{-10}$ & $1.79 \times 10^{-6}$ & $1.03 \times 10^{-5}$ & $\mathbf{3.99 \times 10^{-13}}$ & $1.92 \times 10^{-5}$ & $2.13 \times 10^{-6}$ \\
			& Agg. Res. & $9.71 \times 10^{-11}$ & $1.05 \times 10^{-6}$ & $1.57 \times 10^{-6}$ & $\mathbf{2.68 \times 10^{-12}}$ & $2.39 \times 10^{-5}$ & $9.16 \times 10^{-1}$ \\
			& FBS Res. & $1.04 \times 10^{-8}$ & $2.87 \times 10^{-4}$ & $8.06 \times 10^{-4}$ & $\mathbf{3.19 \times 10^{-11}}$ & $1.11 \times 10^{-3}$ & $8.71 \times 10^{-1}$ \\
			& Opt. Gap & $4.86 \times 10^{-9}$ & $1.44 \times 10^{-5}$ & $8.66 \times 10^{-5}$ & $\mathbf{3.35 \times 10^{-12}}$ & $1.58 \times 10^{-4}$ & $5.26 \times 10^{-1}$ \\
			\midrule
			\multirow{4}{*}{50} & Cons. Err. & $7.49 \times 10^{-6}$ & $7.81 \times 10^{-5}$ & $1.03 \times 10^{-3}$ & $4.72 \times 10^{-4}$ & $1.00 \times 10^{-3}$ & $\mathbf{4.13 \times 10^{-6}}$ \\
			& Agg. Res. & $\mathbf{7.72 \times 10^{-8}}$ & $2.75 \times 10^{-6}$ & $1.81 \times 10^{-5}$ & $1.32 \times 10^{-5}$ & $6.77 \times 10^{-4}$ & $9.23 \times 10^{-1}$ \\
			& FBS Res. & $\mathbf{2.16 \times 10^{-4}}$ & $1.31 \times 10^{-2}$ & $1.50 \times 10^{-1}$ & $7.46 \times 10^{-2}$ & $1.37 \times 10^{-1}$ & $9.97 \times 10^{-1}$ \\
			& Opt. Gap & $\mathbf{5.94 \times 10^{-5}}$ & $6.17 \times 10^{-4}$ & $7.83 \times 10^{-3}$ & $3.75 \times 10^{-3}$ & $7.60 \times 10^{-3}$ & $5.30 \times 10^{-1}$ \\
			\midrule
			\multirow{4}{*}{100} & Cons. Err. & $3.24 \times 10^{-5}$ & $6.12 \times 10^{-4}$ & $1.51 \times 10^{-3}$ & $8.21 \times 10^{-4}$ & $1.57 \times 10^{-3}$ & $\mathbf{1.37 \times 10^{-5}}$ \\
			& Agg. Res. & $\mathbf{1.31 \times 10^{-7}}$ & $1.31 \times 10^{-5}$ & $1.82 \times 10^{-5}$ & $1.05 \times 10^{-5}$ & $8.02 \times 10^{-4}$ & $9.25 \times 10^{-1}$ \\
			& FBS Res. & $\mathbf{3.28 \times 10^{-3}}$ & $2.15 \times 10^{-1}$ & $3.69 \times 10^{-1}$ & $3.15 \times 10^{-1}$ & $3.77 \times 10^{-1}$ & $1.06$ \\
			& Opt. Gap & $\mathbf{2.66 \times 10^{-4}}$ & $4.97 \times 10^{-3}$ & $1.16 \times 10^{-2}$ & $6.55 \times 10^{-3}$ & $1.18 \times 10^{-2}$ & $5.32 \times 10^{-1}$ \\
			\midrule
			\multirow{4}{*}{500} & Cons. Err. & $7.93 \times 10^{-4}$ & $1.56 \times 10^{-3}$ & $3.94 \times 10^{-3}$ & $2.01 \times 10^{-3}$ & $3.83 \times 10^{-3}$ & $\mathbf{5.46 \times 10^{-5}}$ \\
			& Agg. Res. & $\mathbf{2.14 \times 10^{-7}}$ & $5.70 \times 10^{-6}$ & $1.84 \times 10^{-5}$ & $1.01 \times 10^{-5}$ & $3.35 \times 10^{-4}$ & $9.29 \times 10^{-1}$ \\
			& FBS Res. & $\mathbf{1.83 \times 10^{-1}}$ & $4.09 \times 10^{-1}$ & $6.58 \times 10^{-1}$ & $6.11 \times 10^{-1}$ & $6.52 \times 10^{-1}$ & $1.02$ \\
			& Opt. Gap & $\mathbf{5.97 \times 10^{-3}}$ & $1.16 \times 10^{-2}$ & $2.82 \times 10^{-2}$ & $1.46 \times 10^{-2}$ & $2.74 \times 10^{-2}$ & $5.35 \times 10^{-1}$ \\
			\bottomrule
		\end{tabular}
	}
\end{table}

\begin{table}[ht]
	\centering
	\caption{Numerical results for the Barbell network after $10^4$ iterations, averaged over $50$ instances.}
	\label{tab:matrix_games_results_barbell}
	\resizebox{\textwidth}{!}{
		\begin{tabular}{llcccccc}
			\toprule
			$n$ & {Metrics} & \texttt{NI-DFFP} & \texttt{ND-DFFP} & \makecell{\texttt{NI-DFB}\\($\beta = \beta_{\max}$)} & \makecell{\texttt{NI-DFB}\\($\beta = \beta_{\text{norm}}$)} & \makecell{\texttt{NI-DFB}\\($\alpha_1 = \cdots = \alpha_n$)} & \texttt{ND-DFRB} \\
			\midrule
			\multirow{4}{*}{10} & Cons. Err. & $4.81 \times 10^{-9}$ & $2.45 \times 10^{-6}$ & $1.43 \times 10^{-6}$ & $\mathbf{8.64 \times 10^{-11}}$ & $1.59 \times 10^{-5}$ & $7.43 \times 10^{-6}$ \\
			& Agg. Res. & $5.68 \times 10^{-10}$ & $1.12 \times 10^{-6}$ & $2.54 \times 10^{-7}$ & $\mathbf{3.22 \times 10^{-11}}$ & $7.39 \times 10^{-5}$ & $9.03 \times 10^{-1}$ \\
			& FBS Res. & $2.21 \times 10^{-7}$ & $4.02 \times 10^{-4}$ & $2.41 \times 10^{-4}$ & $\mathbf{1.46 \times 10^{-8}}$ & $1.79 \times 10^{-3}$ & $7.39 \times 10^{-1}$ \\
			& Opt. Gap & $4.17 \times 10^{-8}$ & $2.11 \times 10^{-5}$ & $1.24 \times 10^{-5}$ & $\mathbf{7.49 \times 10^{-10}}$ & $1.29 \times 10^{-4}$ & $5.20 \times 10^{-1}$ \\
			\midrule
			\multirow{4}{*}{20} & Cons. Err. & $5.32 \times 10^{-5}$ & $6.97 \times 10^{-5}$ & $1.47 \times 10^{-3}$ & $1.30 \times 10^{-3}$ & $1.34 \times 10^{-3}$ & $\mathbf{4.26 \times 10^{-5}}$ \\
			& Agg. Res. & $\mathbf{3.01 \times 10^{-7}}$ & $4.83 \times 10^{-6}$ & $1.90 \times 10^{-5}$ & $2.05 \times 10^{-5}$ & $1.33 \times 10^{-3}$ & $9.15 \times 10^{-1}$ \\
			& FBS Res. & $\mathbf{3.98 \times 10^{-3}}$ & $2.25 \times 10^{-2}$ & $2.40 \times 10^{-1}$ & $2.26 \times 10^{-1}$ & $2.17 \times 10^{-1}$ & $8.70 \times 10^{-1}$ \\
			& Opt. Gap & $\mathbf{4.57 \times 10^{-4}}$ & $5.95 \times 10^{-4}$ & $1.24 \times 10^{-2}$ & $1.12 \times 10^{-2}$ & $1.11 \times 10^{-2}$ & $5.25 \times 10^{-1}$ \\
			\midrule
			\multirow{4}{*}{50} & Cons. Err. & $1.51 \times 10^{-3}$ & $1.60 \times 10^{-3}$ & $2.16 \times 10^{-3}$ & $2.15 \times 10^{-3}$ & $2.10 \times 10^{-3}$ & $\mathbf{5.44 \times 10^{-4}}$ \\
			& Agg. Res. & $\mathbf{5.86 \times 10^{-7}}$ & $8.20 \times 10^{-6}$ & $1.78 \times 10^{-5}$ & $1.70 \times 10^{-5}$ & $9.80 \times 10^{-4}$ & $9.14 \times 10^{-1}$ \\
			& FBS Res. & $\mathbf{1.63 \times 10^{-1}}$ & $3.10 \times 10^{-1}$ & $3.33 \times 10^{-1}$ & $3.47 \times 10^{-1}$ & $3.28 \times 10^{-1}$ & $1.02$ \\
			& Opt. Gap & $\mathbf{1.19 \times 10^{-2}}$ & $1.23 \times 10^{-2}$ & $1.67 \times 10^{-2}$ & $1.67 \times 10^{-2}$ & $1.63 \times 10^{-2}$ & $5.25 \times 10^{-1}$ \\
			\midrule
			\multirow{4}{*}{100} & Cons. Err. & $\mathbf{1.63 \times 10^{-3}}$ & $1.76 \times 10^{-3}$ & $2.46 \times 10^{-3}$ & $2.43 \times 10^{-3}$ & $2.41 \times 10^{-3}$ & $8.18 \times 10^{-3}$ \\
			& Agg. Res. & $\mathbf{5.73 \times 10^{-7}}$ & $3.10 \times 10^{-6}$ & $2.02 \times 10^{-5}$ & $1.93 \times 10^{-5}$ & $8.11 \times 10^{-4}$ & $8.59 \times 10^{-1}$ \\
			& FBS Res. & $\mathbf{1.64 \times 10^{-1}}$ & $3.11 \times 10^{-1}$ & $4.00 \times 10^{-1}$ & $3.97 \times 10^{-1}$ & $4.09 \times 10^{-1}$ & $1.09$ \\
			& Opt. Gap & $\mathbf{1.31 \times 10^{-2}}$ & $1.45 \times 10^{-2}$ & $2.15 \times 10^{-2}$ & $2.11 \times 10^{-2}$ & $2.09 \times 10^{-2}$ & $5.04 \times 10^{-1}$ \\
			\midrule
			\multirow{4}{*}{500} & Cons. Err. & $\mathbf{2.61 \times 10^{-3}}$ & $2.73 \times 10^{-3}$ & $3.06 \times 10^{-3}$ & $3.03 \times 10^{-3}$ & $3.06 \times 10^{-3}$ & $3.66 \times 10^{-2}$ \\
			& Agg. Res. & $\mathbf{3.75 \times 10^{-7}}$ & $2.29 \times 10^{-6}$ & $1.26 \times 10^{-5}$ & $1.20 \times 10^{-5}$ & $4.06 \times 10^{-4}$ & $5.24 \times 10^{-1}$ \\
			& FBS Res. & $\mathbf{1.72 \times 10^{-1}}$ & $3.85 \times 10^{-1}$ & $5.22 \times 10^{-1}$ & $5.22 \times 10^{-1}$ & $5.37 \times 10^{-1}$ & $9.35 \times 10^{-1}$ \\
			& Opt. Gap & $\mathbf{2.80 \times 10^{-2}}$ & $2.94 \times 10^{-2}$ & $3.05 \times 10^{-2}$ & $3.04 \times 10^{-2}$ & $3.03 \times 10^{-2}$ & $4.16 \times 10^{-1}$ \\
			\bottomrule
		\end{tabular}
	}
\end{table}

\begin{figure}[h]
	\centering
	\includegraphics[width=\textwidth]{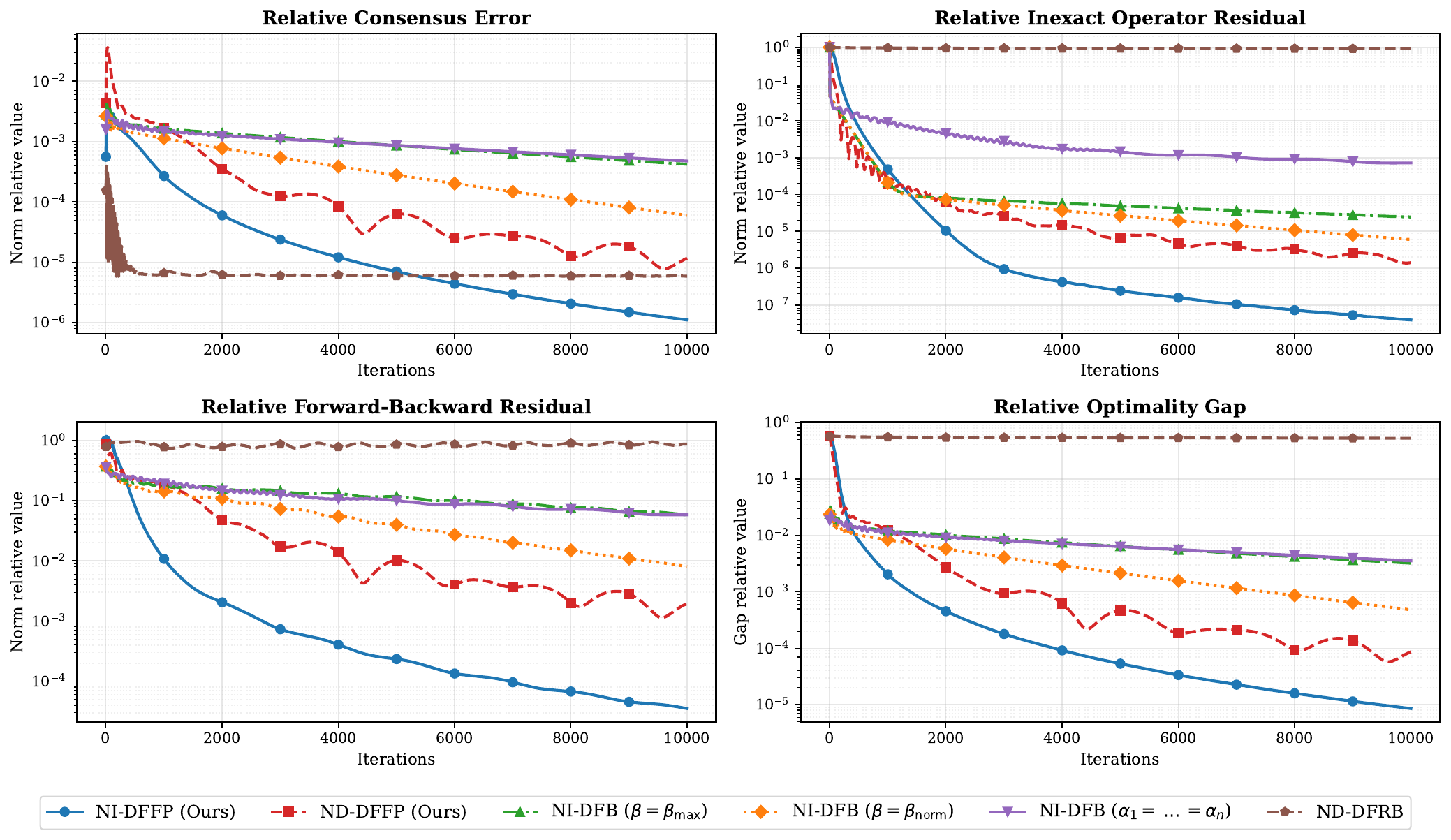}
	\caption{Comparison of the $6$ algorithms for solving \eqref{eq:decentralized_matrix_game} over a cycle $C(20)$ network}
	\label{fig:cycle20_matrix_games}
\end{figure}

As we can observe from our results, for a small-scale network with $n = 10$, the competitor \texttt{NI-DFB}, particularly with heterogeneous stepsizes, initially demonstrates highly competitive performance, achieving fast convergence toward an optimal solution and low residual values.
Nevertheless, our proposed algorithms also perform well. 
In particular, \texttt{NI-DFFP} exhibits competitive performance compared with \texttt{NI-DFB} using $\beta = \beta_{\max}$.

As the network size increases to an intermediate regime with $n \geq 20$ (or $n \geq 50$ for 2D grid networks), a distinct shift in performance becomes evident.
Our proposed algorithms consistently outperform the benchmark methods across the main evaluation metrics.
Notably, this advantage is maintained even when compared with baseline methods that use network-independent stepsizes.
These results suggest that our parameter choices remain effective as the network grows and that the proposed algorithms are less sensitive to the structural bottlenecks and information-mixing delays arising in larger and more complex networks.

In large-scale network scenarios, particularly at $n = 500$, the deteriorating algebraic connectivity of the cycle graph significantly affects the performance of all methods.
The convergence rates degrade across all evaluated algorithms, reflecting the communication limitations and unfavorable spectral properties of large, sparsely connected cycle networks.
Nevertheless, our algorithms exhibit greater robustness under these challenging conditions.
Despite the overall slowdown in convergence, the proposed methods maintain a clear and consistent performance advantage over the competing methods throughout the iteration history.

Furthermore, the empirical observations on cycle graphs remain largely consistent across two structurally different network topologies: 2D Grid and Barbell networks.
The 2D Grid network exhibits a different algebraic connectivity decay rate from the cycle graph, yet our proposed \texttt{NI-DFFP} algorithm consistently achieves faster convergence and lower optimality gaps for intermediate and large network sizes.
More importantly, the Barbell graph presents a severe structural bottleneck due to the sparse bridge connecting two densely connected subgraphs, which substantially slows network-wide information diffusion.
Despite this challenging topology, our algorithms remain competitive and robust.
While the convergence of all evaluated methods deteriorates at the large scale, $n = 500$, due to this communication bottleneck, the proposed methods continue to maintain a clear performance advantage over the baselines throughout the iteration history.
The consistency of these results across structurally diverse networks indicates that the accelerated technique is not specific to a particular topology.
Instead, it demonstrates robust performance under a broad range of network-induced communication delays and spectral limitations.

\beforesubsec
\subsection{Virtual power plants}\label{subsec:EX2_virtual_PP}
\aftersubsec
We consider a virtual power plant (VPP) consisting of $n$ noncooperative players, each operating a power bank, which are connected via the local power grid (further details on VPPs can be found in \cite{wang2019review}).
In each time period $t \in [p]$, players must decide how much electricity to buy (charging their power banks) and how much to sell (discharging their power banks) to satisfy their personal electricity usage while minimizing their total cost over the time horizon.
We seek a vector that determines when and how much players charge and discharge their power banks such that players have no incentive to individually deviate from this schedule.

\vspace{0.75ex}
\noindent\textbf{$\mathrm{(a)}$~Mathematical model.}
We represent a charging schedule for the player population using $u = (u_1, \cdots, u_n)^\top \in \R^{2p} \times \cdots \times \R^{2p}$, where the charging ($u_i^+$) and discharging ($u_i^-$) decisions for player $i$ are denoted as
\begin{equation*}
\arraycolsep=0.2em
\begin{array}{lcl}
	u_i = \begin{bmatrix} u_i^+ \\ u_i^- \end{bmatrix} \in \R^{2p}, \quad 
	u_i^+ = \begin{bmatrix} u_i^+(1) \\ \vdots \\ u_i^+(p) \end{bmatrix} \in \R^p, \quad 
	u_i^- = \begin{bmatrix} u_i^-(1) \\ \vdots \\ u_i^-(p) \end{bmatrix} \in \R^p,
\end{array}
\end{equation*}
where $u_i^+(t) \in \R$ (resp. $u_i^-(t)$) denotes the amount of electricity drawn from the grid (resp. discharged to the grid) by player $i \in [n]$ in time period $t \in [p]$.
Under this notation, the VPP coordination problem can be formulated as an aggregative game, where the problem for player $i$ is given by

\begin{equation}\label{eq:VPP_playeri}
	\arraycolsep=0.2em
	\left\{
	\begin{array}{ccl}
		\min\limits_{u_i \in \R^{2p}} && g_i(u_i) + f_i(u_i, \sum_{j=1}^n u_j),  \vspace{1ex}\\
		\text{s.t.} && u_i \in \Omega_i, \vspace{1ex}\\
		&& M\sum_{j=1}^n u_j \leq b,
	\end{array}
	\right.
\end{equation}
in which the elements of \eqref{eq:VPP_playeri} are defined as
\begin{equation}\label{eq:VPP_playeri_elements}
\arraycolsep=0.2em
\begin{array}{c}
	g_i(u_i) = u_i^\top Q_i u_i + p_i^\top u_i, \quad f_i\left(u_i, \sum_{j=1}^n u_j\right) = \iprod{\begin{bmatrix} \Id & -\Id \end{bmatrix} u_i, \begin{bmatrix} \Id & -\Id \end{bmatrix} \sum_{j=1}^n u_j + m}, \vspace{1ex}\\
	\Omega_i = \set{u_i \in \R^{2p} \mid 0 \leq u_i \leq \begin{bmatrix} \bar{u}_i^+ \\ \bar{u}_i^- \end{bmatrix}, \ l_i^{\text{low}} \leq \begin{bmatrix} e_i^+ R - \frac{1}{e_i^-}R\end{bmatrix} u_i \leq l_i^\text{up}}, \vspace{1ex}\\
	M = \begin{bmatrix} \Id & -\Id \\ -\Id & \Id \end{bmatrix} \in \R^{2p \times 2p}, \quad b = \begin{bmatrix} K - m \\ m \end{bmatrix} \in \R^{2p},
\end{array}
\end{equation}
where $R = (r_{ij}) \in \R^{p \times p}$ is a lower triangular matrix with $r_{ij} = 1$ for $i \geq j$ and $0$ otherwise.
The details of this formulation (particularly $Q_i \in \R^{2p \times 2p}$, $p_i \in \R^{2p}_+$, $m = (m(t)) \in \R^p$, $\bar{u}_i^+ = (\bar{u}_i^+(t)) \in \R^p$, $\bar{u}_i^- = (\bar{u}_i^-(t)) \in \R^p$, $l_i^{\text{low}} = (l_i^{\text{low}}(t)) \in \R^p_+$, $l_i^{\text{up}} = (l_i^{\text{up}}(t)) \in \R^p_+$, $e_i^+, e_i^- \in (0,1)$, and $K = (K(t)) \in \R^p$) can be found in \cite{tam2025decentralised}.  

According to \cite[Proposition 5.1]{tam2025decentralised}, $u = (u_1, \cdots, u_n) \in \R^{2p} \times \cdots \times \R^{2p}$ is a generalized Nash equilibrium of \eqref{eq:VPP_playeri} if and only if there exists a dual variable $v = (v_1, \cdots, v_n) \in \R^{2p} \times \cdots \times \R^{2p}$ such that $(u, v)$ satisfies
\begin{equation}\label{eq:VPP_inclusion}
\begin{bmatrix} 0 \\ 0 \end{bmatrix} \in \begin{bmatrix}
	\partial (g_i + \delta_{\Omega_i})(u_i) \\ \Nc_{\R^{2p}_+}(v_i)
\end{bmatrix} + \begin{bmatrix}
\nabla_{u_i} f_i(u_i, \sum_{j=1}^n u_j) + M^\top v_i \\
b - M\sum_{j=1}^n u_j
\end{bmatrix}, \quad \forall i \in [n].
\end{equation}
Thus, let $\Hc = \R^{2p} \times \R^{2p}$ and define the operators $G_i: \Hc^n \to \Hc^n$ and $T_i: \Hc^n \rightrightarrows \Hc^n$ as
\begin{equation*} 
\arraycolsep=0.2em
\begin{array}{c}
	z_i = \begin{bmatrix} u_i \\ v_i \end{bmatrix} \in \Hc, \quad 
	z = \begin{bmatrix} z_1 \\ \vdots \\ z_n \end{bmatrix} \in \Hc^n, \quad
	G_i(z) = \begin{bmatrix} G_{i1} (z_1, \sum_{j=1}^n z_j) \\ \vdots \\ G_{in} (z_n, \sum_{j=1}^n z_j) \end{bmatrix}, \quad 
	T_i(z) = \begin{bmatrix} T_{i1} (z_1) \\ \vdots \\ T_{in} (z_n) \end{bmatrix},
\end{array}
\end{equation*}
where the operators $G_{ij}: \Hc^2 \to \Hc^2$ and $T_{ij}: \Hc \rightrightarrows \Hc$ are defined by
\begin{equation*} 
\arraycolsep=0.2em
\begin{array}{rclrcl}
	G_{ii}(z_i, \sum_{j=1}^n z_j) &=& \begin{bmatrix}
		\nabla_{u_i} f_i(u_i, s) + \nabla_s f_i(u_i, s) + M^\top v_i \\ b - Ms
	\end{bmatrix},  &\quad
	G_{ij}(z_i, \sum_{j=1}^n z_j) &=& 0 \quad \forall j \ne i, \vspace{1ex}\\
	T_{ii}(z_i) &=& \begin{bmatrix} \partial (g_i + \delta_{\Omega_i}) (u_i) \\ \Nc_{\R^{2p}_+}(v_i) \end{bmatrix}, \quad & T_{ij}(z_i) &=& 0 \quad \forall j \ne i,
\end{array}
\end{equation*}
where $s = \sum_{j=1}^n u_j$ in $G_{ii}(z_i, \sum_{j=1}^n z_j)$.
Then, \eqref{eq:VPP_inclusion} is exactly a special case of \eqref{eq:DGE}.

\vspace{0.75ex}
\noindent\textbf{$\mathrm{(b)}$~Data generation.}
The matrix $Q_i = \diag{q_1, \cdots, q_{2p}}$ has its diagonal entries drawn from $U(0.1, 4)$, while the elements of $p_i$ are sampled from $U(0.2, 2)$.
The entries of $c_i^+$ and $c_i^-$ are generated from $U(0,2)$ and $U(-2,0)$, respectively.
For $\bar{u}_i^+$ and $\bar{u}_i^-$, each entry is independently set to $0$ with probability $0.2$, and otherwise sampled from $U(1,5)$ with probability $0.8$.
The efficiencies $e_i^{\pm}$ are drawn from $U(0.5, 1)$, whereas the elements of $l_i^{\text{low}}$ and $l_i^{\text{up}}$ are sampled from $U(0,1)$ and $U(1,3)$, respectively.
The non-VPP demand $m(t)$ follows \cite{ma2010decentralized}, and the grid capacity is set to $K(t) = 0.55 + \max_t m(t)$.

\vspace{0.75ex}
\noindent\textbf{$\mathrm{(c)}$~Network topologies.}
We continue utilizing three types of communication networks: cycle, barbell, and 2D grid network, with a total of $n$ agents, $n \in \sets{20, 40, 60, 80, 100}$, and the Laplacian-based mixing matrix with $\gamma = 0.505\lambda_{\max}(\Lc)$.

\vspace{0.75ex}
\noindent\textbf{$\mathrm{(d)}$~Algorithms and parameters.}
We implement two proposed algorithms: \texttt{ND-DFFP} and \texttt{NI-DFFP}, with the parameters chosen as follows:
\begin{compactitem}
	\item For \texttt{ND-DFFP}: $r = 3.0$, $\delta = 10^{-5}$, $\nu = r - 2 - 1.001\delta$, $\sigma = 10$, $\tau = \frac{2}{4\sqrt{1 + \omega}L_{\max} + \sigma (1-\lambda_{\min}(W))}$.
	\item For \texttt{NI-DFFP}: $r = 2.8$, $\delta = 0.115$, $\nu = r - 2 - 1.001\delta$, $\tau = 0.3$, $\sigma = \frac{1}{30}$, heterogeneous stepsizes $\alpha_i = \frac{0.99(1 - \tau \sigma)}{2\sqrt{1+\omega}\tau L_i}$, and $\kappa = (1+10^{-10})\norms{\Lambda^{1/2}((I-W)/2)\Lambda^{1/2}}$.
\end{compactitem}
Here, recall that $\Lambda = \diag{\alpha_1, \cdots, \alpha_n}$.
For benchmarking, we also implement \eqref{eq:malitsky2026_alg} and one variant of \eqref{eq:tam2026_alg}, with the parameters chosen as proposed in \cite{tam2025decentralised}:
\begin{compactitem}
	\item \texttt{NI-DFB} with heterogeneous stepsizes $\alpha_i = \frac{0.9}{8L_i}$ and parameter $\beta = 0.9\norms{\Lambda^{1/2}((I-W)/2)\Lambda^{1/2}}^{-1}$.
	\item \texttt{ND-DFRB} with stepsize $\alpha = \frac{0.9(1 + \lambda_{\min}(W))}{4\max_i\sets{L_i}}$.
\end{compactitem}

\vspace{0.75ex}
\noindent\textbf{$\mathrm{(e)}$~Numerical results.}
The performance of each algorithm is evaluated using the relative optimality gap: $\frac{\norms{\mbf{z}^k - \mbf{1}z^{\star}}}{\norms{\mbf{1}z^{\star}}}$, where $z^{\star}$ is computed using a centralized QP solver from \texttt{CvxOpt}.
Each test was repeated five times, and we report the mean performance across the five runs.
The numerical results after $5{,}000$ iterations are reported in 
Tables~\ref{tab:VPP_cycle_results}, \ref{tab:VPP_grid_results}, and \ref{tab:VPP_barbell_results}.
Figure~\ref{fig:grid100_VPP} illustrates the performance of the four methods operating on a grid network of $100$ agents.

\begin{table*}[ht]
	\centering
	\caption{Numerical results of relative optimality gap and execution time for the Cycle graph.}
	\label{tab:VPP_cycle_results}
	\resizebox{\textwidth}{!}{
		\begin{tabular}{l c c c c c c c c}
			\toprule
			& \multicolumn{2}{c}{NI-DFFP (Ours)} & \multicolumn{2}{c}{ND-DFFP (Ours)} & \multicolumn{2}{c}{NI-DFB} & \multicolumn{2}{c}{ND-DFRB} \\
			\cmidrule(lr){2-3} \cmidrule(lr){4-5} \cmidrule(lr){6-7} \cmidrule(lr){8-9}
			$n$ & Gap & Time (s) & Gap & Time (s) & Gap & Time (s) & Gap & Time (s) \\
			\midrule
			20 & $3.77_{{\pm 0.36}} \times 10^{-1}$ & $253.54_{{\pm 1.28}}$ & $2.47_{{\pm 0.25}} \times 10^{-1}$ & $236.58_{{\pm 1.06}}$ & $\mathbf{2.37}_{{\pm 0.29}} \times 10^{-1}$ & $255.07_{{\pm 1.65}}$ & $3.52_{{\pm 0.18}} \times 10^{0}$ & $\mathbf{199.13}_{{\pm 1.14}}$ \\
			40 & $8.71_{{\pm 0.53}} \times 10^{-1}$ & $480.36_{{\pm 1.49}}$ & $7.69_{{\pm 0.60}} \times 10^{-1}$ & $\mathbf{443.87}_{{\pm 1.51}}$ & $\mathbf{7.53}_{{\pm 0.63}} \times 10^{-1}$ & $496.42_{{\pm 1.52}}$ & $2.76_{{\pm 0.14}} \times 10^{0}$ & $450.52_{{\pm 1.05}}$ \\
			60 & $\mathbf{8.76}_{{\pm 0.43}} \times 10^{-1}$ & $591.07_{{\pm 3.10}}$ & $9.46_{{\pm 0.50}} \times 10^{-1}$ & $\mathbf{548.45}_{{\pm 2.52}}$ & $9.52_{{\pm 0.51}} \times 10^{-1}$ & $637.19_{{\pm 1.28}}$ & $2.31_{{\pm 0.10}} \times 10^{0}$ & $617.79_{{\pm 1.81}}$ \\
			80 & $\mathbf{8.69}_{{\pm 0.45}} \times 10^{-1}$ & $923.29_{{\pm 6.61}}$ & $1.00_{{\pm 0.05}} \times 10^{0}$ & $\mathbf{861.26}_{{\pm 3.27}}$ & $1.03_{{\pm 0.05}} \times 10^{0}$ & $1021.05_{{\pm 3.59}}$ & $2.06_{{\pm 0.09}} \times 10^{0}$ & $1016.47_{{\pm 3.29}}$ \\
			100 & $\mathbf{8.90}_{{\pm 0.44}} \times 10^{-1}$ & $945.32_{{\pm 2.37}}$ & $1.10_{{\pm 0.05}} \times 10^{0}$ & $\mathbf{897.07}_{{\pm 2.12}}$ & $1.06_{{\pm 0.05}} \times 10^{0}$ & $1066.86_{{\pm 2.69}}$ & $1.89_{{\pm 0.08}} \times 10^{0}$ & $1070.50_{{\pm 4.12}}$ \\
			\bottomrule
		\end{tabular}
	}
\end{table*}

\begin{table*}[ht]
	\centering
	\caption{Numerical results of relative optimality gap and execution time for the Grid graph.}
	\label{tab:VPP_grid_results}
	\resizebox{\textwidth}{!}{
		\begin{tabular}{l c c c c c c c c}
			\toprule
			& \multicolumn{2}{c}{NI-DFFP (Ours)} & \multicolumn{2}{c}{ND-DFFP (Ours)} & \multicolumn{2}{c}{NI-DFB} & \multicolumn{2}{c}{ND-DFRB} \\
			\cmidrule(lr){2-3} \cmidrule(lr){4-5} \cmidrule(lr){6-7} \cmidrule(lr){8-9}
			$n$ & Gap & Time (s) & Gap & Time (s) & Gap & Time (s) & Gap & Time (s) \\
			\midrule
			20 & $3.70_{{\pm 0.36}} \times 10^{-1}$ & $211.55_{{\pm 1.36}}$ & $2.46_{{\pm 0.25}} \times 10^{-1}$ & $197.56_{{\pm 0.31}}$ & $\mathbf{2.37}_{{\pm 0.29}} \times 10^{-1}$ & $213.04_{{\pm 0.57}}$ & $3.52_{{\pm 0.18}} \times 10^{0}$ & $\mathbf{161.17}_{{\pm 0.86}}$ \\
			40 & $8.57_{{\pm 0.53}} \times 10^{-1}$ & $398.43_{{\pm 1.21}}$ & $7.67_{{\pm 0.60}} \times 10^{-1}$ & $370.36_{{\pm 0.93}}$ & $\mathbf{7.54}_{{\pm 0.64}} \times 10^{-1}$ & $406.25_{{\pm 1.87}}$ & $2.76_{{\pm 0.14}} \times 10^{0}$ & $\mathbf{330.98}_{{\pm 1.37}}$ \\
			60 & $\mathbf{8.59}_{{\pm 0.42}} \times 10^{-1}$ & $581.67_{{\pm 2.66}}$ & $9.41_{{\pm 0.50}} \times 10^{-1}$ & $540.19_{{\pm 2.31}}$ & $9.54_{{\pm 0.51}} \times 10^{-1}$ & $609.34_{{\pm 2.44}}$ & $2.30_{{\pm 0.10}} \times 10^{0}$ & $\mathbf{536.68}_{{\pm 0.65}}$ \\
			80 & $\mathbf{8.60}_{{\pm 0.45}} \times 10^{-1}$ & $765.07_{{\pm 3.65}}$ & $1.00_{{\pm 0.05}} \times 10^{0}$ & $\mathbf{717.98}_{{\pm 3.12}}$ & $1.02_{{\pm 0.05}} \times 10^{0}$ & $820.48_{{\pm 2.67}}$ & $2.05_{{\pm 0.09}} \times 10^{0}$ & $766.01_{{\pm 3.36}}$ \\
			100 & $\mathbf{8.72}_{{\pm 0.43}} \times 10^{-1}$ & $938.39_{{\pm 2.52}}$ & $1.01_{{\pm 0.05}} \times 10^{0}$ & $\mathbf{892.67}_{{\pm 5.69}}$ & $1.04_{{\pm 0.05}} \times 10^{0}$ & $1037.44_{{\pm 3.30}}$ & $1.86_{{\pm 0.07}} \times 10^{0}$ & $997.49_{{\pm 3.82}}$ \\
			\bottomrule
		\end{tabular}
	}
\end{table*}

\begin{table*}[ht]
	\centering
	\caption{Numerical results of relative optimality gap and execution time for the Barbell graph.}
	\label{tab:VPP_barbell_results}
	\resizebox{\textwidth}{!}{
		\begin{tabular}{l c c c c c c c c}
			\toprule
			 & \multicolumn{2}{c}{NI-DFFP (Ours)} & \multicolumn{2}{c}{ND-DFFP (Ours)} & \multicolumn{2}{c}{NI-DFB} & \multicolumn{2}{c}{ND-DFRB} \\
			\cmidrule(lr){2-3} \cmidrule(lr){4-5} \cmidrule(lr){6-7} \cmidrule(lr){8-9}
			$n$ & Gap & Time (s) & Gap & Time (s) & Gap & Time (s) & Gap & Time (s) \\
			\midrule
			20 & $3.93_{{\pm 0.35}} \times 10^{-1}$ & $212.14_{{\pm 1.09}}$ & $2.49_{{\pm 0.25}} \times 10^{-1}$ & $197.77_{{\pm 0.53}}$ & $\mathbf{2.36}_{{\pm 0.29}} \times 10^{-1}$ & $214.29_{{\pm 0.81}}$ & $3.52_{{\pm 0.18}} \times 10^{0}$ & $\mathbf{178.48}_{{\pm 0.40}}$ \\
			40 & $9.15_{{\pm 0.54}} \times 10^{-1}$ & $485.86_{{\pm 2.54}}$ & $7.78_{{\pm 0.61}} \times 10^{-1}$ & $\mathbf{446.34}_{{\pm 1.29}}$ & $\mathbf{7.68}_{{\pm 0.63}} \times 10^{-1}$ & $495.41_{{\pm 0.81}}$ & $2.76_{{\pm 0.15}} \times 10^{0}$ & $461.46_{{\pm 1.00}}$ \\
			60 & $\mathbf{9.84}_{{\pm 0.48}} \times 10^{-1}$ & $595.80_{{\pm 5.03}}$ & $1.10_{{\pm 0.05}} \times 10^{0}$ & $\mathbf{548.54}_{{\pm 3.20}}$ & $1.03_{{\pm 0.05}} \times 10^{0}$ & $611.72_{{\pm 0.69}}$ & $2.29_{{\pm 0.10}} \times 10^{0}$ & $588.22_{{\pm 1.09}}$ \\
			80 & $\mathbf{1.05}_{{\pm 0.05}} \times 10^{0}$ & $774.57_{{\pm 2.50}}$ & $1.20_{{\pm 0.06}} \times 10^{0}$ & $\mathbf{719.63}_{{\pm 3.39}}$ & $1.13_{{\pm 0.06}} \times 10^{0}$ & $801.80_{{\pm 0.56}}$ & $2.05_{{\pm 0.09}} \times 10^{0}$ & $791.51_{{\pm 1.92}}$ \\
			100 & $1.21_{{\pm 0.06}} \times 10^{0}$ & $951.15_{{\pm 4.08}}$ & $1.37_{{\pm 0.06}} \times 10^{0}$ & $\mathbf{893.60}_{{\pm 3.20}}$ & $\mathbf{1.06}_{{\pm 0.05}} \times 10^{0}$ & $994.48_{{\pm 2.32}}$ & $1.89_{{\pm 0.08}} \times 10^{0}$ & $1001.43_{{\pm 4.69}}$ \\
			\bottomrule
		\end{tabular}
	}
\end{table*}

\begin{figure}[h]
	\centering
	\includegraphics[width=\textwidth]{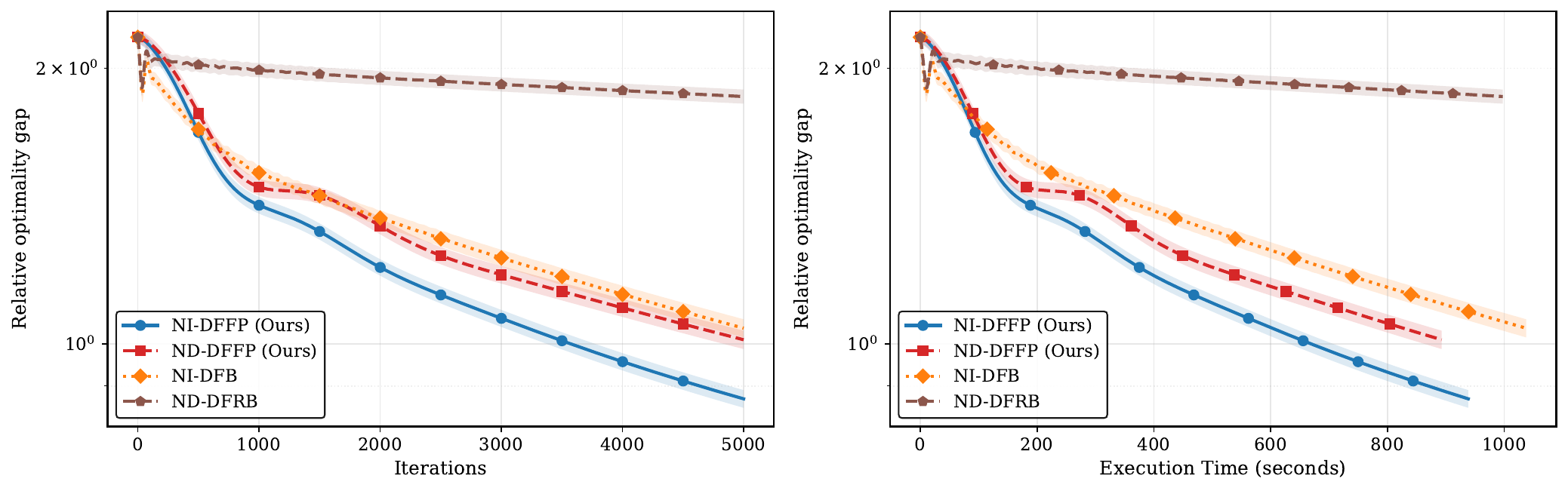}
	\caption{Comparison of the $4$ algorithms for solving \eqref{eq:VPP_playeri} over a grid network of $100$ agents}
	\label{fig:grid100_VPP}
\end{figure}

The numerical experiments for the VPP coordination problem reinforce the scalability and stability trends previously observed in the distributed matrix game example.
On the one hand, \texttt{NI-DFFP} exhibits increasingly competitive performance as the network size grows, particularly for sparse network topologies. 
While all algorithms perform comparably in small networks ($n=20$), the performance gap widens significantly as the dimension scales up. 
For $n \geq 40$, the baseline methods (\texttt{ND-DFRB} and \texttt{NI-DFB}) experience severe deceleration, whereas \texttt{NI-DFFP} maintains a rapid, consistent descent. 
\texttt{ND-DFRB} struggles to converge competitively across all tested scales in these topologies. 

On the other hand, the barbell graph introduces a severe communication bottleneck between two dense cliques. 
The benchmark \texttt{NI-DFB} outperforms \texttt{NI-DFFP} for $n=100$ in this particular configuration.
However, this descent comes at the cost of severe numerical instability; as the network scales to $n \geq 60$, \texttt{NI-DFB}'s convergence trajectory becomes highly erratic and oscillatory. 
In contrast, \texttt{NI-DFFP} continues to outperform \texttt{NI-DFB} at intermediate network sizes and delivers a more stable convergence path despite the extreme network bottleneck. 

\vspace{1ex}
\noindent\textbf{Acknowledgements.}
This work is  partially supported by the National Science Foundation (NSF), grant no. RTG-DMS-2134107 and the Office of Naval Research (ONR), grant No. N00014-23-1-2588 (2023-2026).

\vspace{1ex}
\noindent{\textbf{Conflicts of interest/competing interests.}}
The authors declare that they have no conflicts of interest or competing interests related to this work.

\vspace{1ex}
\noindent{\textbf{Data availability.}}
This paper uses only synthetic data.
The procedure for generating synthetic data is clearly described in the paper.

\appendix
\beforesec
\section{Appendix: Technical Lemmas and Preliminary Results}\label{apdx:sec:tech_lemmas}
\aftersec
This appendix recalls necessary technical lemmas and provides a proof of Lemma~\ref{le:gap_function}.

\beforesubsec
\subsection{Technical lemmas}\label{apdx:subsec:tech_lemmas}
\aftersubsec
First, we recall necessary technical lemmas used in the convergence analysis of this paper. 

\begin{lemma}[\cite{Bauschke2011}, Lemma 5.31]\label{le:A1}
	Let $\sets{u_k}$, $\sets{v_k}$, $\sets{\gamma_k}$, and $\sets{\varepsilon_k}$ be nonnegative sequences such that $\sum_{k=0}^{\infty}\gamma_k < +\infty$ and $\sum_{k=0}^{\infty}\varepsilon_k < +\infty$.
	In addition, for all $k \geq 0$, we assume that
	\begin{equation*}
		u_{k+1} \leq (1+\gamma_k) u_k - v_k + \varepsilon_k.
	\end{equation*}
	Then, we conclude that $\lim_{k \to \infty} u_k$ exists and $\sum_{k=0}^\infty v_k < +\infty$.
\end{lemma}

\begin{lemma}[\cite{TranDinh2025a}, Lemma 29]\label{le:A2}
Let $\sets{\xi_k}$ be a nonnegative sequence and $s \in \R$ such that $\lim_{k \to \infty} k^{s+1} \xi_k$ exists and $\sum_{k=0}^{\infty} k^s \xi_k < +\infty$.
Then, we have $\lim_{k \to \infty} k^{s+1}\xi_k = 0$.
\end{lemma}

\begin{lemma}[see \cite{Golub1996}, Sect. 8.1]\label{le:norm_equivalence}
	Let $P \in \R^{p \times p}$ be a symmetric positive definite matrix.
	Then 
	\begin{equation}\label{eq:norm_equivalence}
		\lambda_{\min}(P) \norms{P^{-1}z}_P^2 \leq \norms{z}_2^2 \leq \lambda_{\max}(P) \norms{P^{-1}z}_P^2, \quad \textrm{for any $z \in \R^p$},
	\end{equation}
	where $\lambda_{\min}(P)$ and $\lambda_{\max}(P)$ are the smallest and largest eigenvalues of $P$, respectively.
\end{lemma}

\begin{lemma}[\cite{Bauschke2011}, Corollary 21.19]\label{le:monotone_local_bounded}
	Let $T: \R^p \rightrightarrows \R^p$ be monotone and $\Xc$ be a compact subset in $\mathrm{int}(\dom{T})$. 
	Then, $T(\Xc)$ is bounded.
\end{lemma}

\begin{lemma}[\cite{Rockafellar1997}, Corollary 12.44]\label{le:sum_maximal_monotone}
	Let $T = T_1 + T_2$ for $T_i: \R^p \rightrightarrows \R^p$ maximally monotone for $i=1,2$.
	If $\ri{\dom{T_1}} \cap \ri{\dom{T_2}} \ne \emptyset$, then $T$ is maximally monotone.
\end{lemma}

\begin{lemma}[see \cite{tam2025decentralised}]\label{le:DFKM_nihs1}
	Let $S, T: \R^p \to \R^p$ be self-adjoint operators and $\kappa \in \R_{++}$ be given.
	If $\norms{TS^2T} < \kappa$, then $\kappa\Id - ST^2S \succ 0$.
\end{lemma}

\begin{lemma}[see \cite{Golub1996}]
	Let $\mbf{L} \in \R^{m \times n}$ be a matrix with $0 < \mathrm{rank}(\mbf{L}) \leq \min\sets{m,n}$.
	Then, for any $\mbf{s} \in \mathrm{Null}(\mbf{L})^{\perp}$, we have 
	\begin{equation*}
		\norms{\mbf{L}\mbf{s}} \geq \sigma_{\min}^+(\mbf{L})\norms{\mbf{s}}, 
	\end{equation*}		
	where $\sigma_{\min}^+(\mbf{L})$ is the smallest positive singular value of $\mbf{L}$.
\end{lemma}

\subsection{Proof of Lemma~\ref{le:gap_function} --- Properties of the gap function}\label{apdx:le:gap_function}
	\noindent$\mathrm{(i)}$~
	First, due to the construction of $\Bc$, the supremum in the definition of $\mathrm{Gap}_{\Bc}$ is taken over a nonempty set.
	Second, let $v \in \dom{\Phi}$ and $\xi_i^v \in T_iv$ arbitrarily, $i = 1, \cdots, n$.
	Then, for any $u \in \Bc$ and $\xi_i \in T_iu$, using the monotonicity of $\Phi_i$ and Cauchy-Schwarz inequality, we have
	\begin{equation*}
	\begin{array}{lcl}
		\sum_{i=1}^n \iprods{G_iu + \xi_i, v - u} \leq \iprods{\sum_{i=1}^n (G_iv + \xi_i^v), v - u} \leq \norms{\sum_{i=1}^n (G_iv + \xi_i^v)}\norms{v-u}.
	\end{array}
	\end{equation*}
	Taking supremum over $u \in \Bc$ and $\xi_i \in T_iu$ both sides of the last inequality and then 
	\begin{equation*}
	\begin{array}{lcl}
		\mathrm{Gap}_{\Bc}(v) \leq \norms{\sum_{i=1}^n (G_iv + \xi_i^v)} \sup_{u \in \Bc}\norms{v-u} < +\infty,
	\end{array}
	\end{equation*}
	where the last inequality holds due to the boundedness of $\Bc$.
	Therefore, $\mathrm{Gap}_{\Bc}$ is well-defined.
	
	Next, under Assumption~\ref{as:restricted_dual_gap} and Assumption~\ref{as:boundary_regularity}$\mathrm{(i)}$, we observe that $\zer{\Phi} \cap \Bc \ne \emptyset$.
	Take $u^{\dagger} \in \zer{\Phi} \cap \Bc$ arbitrarily.
	Then, there exist $\xi_i^{\dagger} \in T_i u^{\dagger}$, $i = 1, \cdots, n$, such that $\sum_{i=1}^n (G_iu^{\dagger} + \xi_i^{\dagger}) = 0$.
	Hence, for all $v \in \dom{\Phi}$, we can easily show from \eqref{eq:res_gap_func} that
	\begin{equation*}
		\arraycolsep=0.2em
		\begin{array}{lcl}
			\mathrm{Gap}_{\Bc}(v)  \geq&\sum_{i=1}^n \iprods{G_i u^{\dagger} + \xi_i^{\dagger}, v - u^{\dagger}} 
			= \iprods{\sum_{i=1}^n (G_i u^{\dagger} + \xi_i^{\dagger}), v - u^{\dagger}} 
			= 0. 
		\end{array}
	\end{equation*}
	\vspace{1ex}
	\noindent$\mathrm{(ii)}$~
	Suppose that $u^{\star} \in \zer{\Phi}$, then there exist $\xi_i^{\star} \in T_iu^{\star} $ such that $\sum_{i=1}^n (G_iu^{\star} + \xi_i^{\star}) = 0$.
	Let $u \in \Bc$ and $\xi_i \in T_i u$ arbitrarily, $i = 1, \cdots, n$.
	By the monotonicity of $G_i + T_i$, we can show that
	\begin{equation*}
		\arraycolsep=0.2em
		\begin{array}{lcl}
			\sum_{i=1}^n \iprods{G_i u + \xi_i, u^{\star} - u} \leq \sum_{i=1}^n \iprods{ G_iu^{\star} + \xi_i^{\star}, u^{\star} - u} =  \iprods{\sum_{i=1}^n (G_iu^{\star} + \xi_i^{\star}), u^{\star} - u} = 0.
		\end{array}
	\end{equation*}
	Taking supremum of both sides over all $u \in {\Bc}$ and $\xi_i \in T_iu$, we obtain $\mathrm{Gap}_{\Bc}(u^{\star}) \leq 0$.
	However, since $\mathrm{Gap}_{\Bc}(u^{\star}) \geq 0$ from part (i), we must have $\mathrm{Gap}_{\Bc}(u^{\star}) = 0$.
	
	\vspace{1ex}
	\noindent$(\mathrm{iii})$~Suppose that $u^{\star} \in {\Bc}$ and $\mathrm{Gap}_{\Bc}(u^{\star}) = 0$.
	By the definition of the restricted gap function, for every $u \in {\Bc}$ and $\xi_i \in T_i u$, we have $\sum_{i=1}^n \iprods{G_i u + \xi_i, u^{\star} - u} \leq  0$, which is equivalent to the following Minty's variational inequality:
	\begin{equation}\label{eq:le_gap_function_proof1}
		\arraycolsep=0.2em
		\begin{array}{lcl}
			\sum_{i=1}^n \iprods{G_i u + \xi_i, u - u^{\star}} \geq 0 \qquad \forall u \in \Bc, \quad \xi_i \in T_i u.
		\end{array}
	\end{equation}
	Next, since $\dom{\Phi_i} = \dom{T_i}$, under either Assumption~\ref{as:restricted_dual_gap} or Assumption~\ref{as:boundary_regularity}$\mathrm{(i)}$, we always have $\bigcap_{i=1}^n \ri{\dom{\Phi_i}} \ne \emptyset$.
	Since $\Phi_i$ is maximally monotone for all $i = 1, \cdots, n$, applying Lemma~\ref{le:sum_maximal_monotone}, we can show that $\Phi := \sum_{i=1}^n \Phi_i$ is also maximally monotone.
	Moreover, $\dom{\Phi} = \bigcap_{i=1}^n \dom{T_i}$.
	Now, we consider two cases as follows.
	
	\vspace{1ex}
	\noindent\textbf{\textit{Case 1. Interior-solution case.}}
	Let $\Nc_{\Bc}$ be the normal cone of $\Bc$, then since $\Bc$ is closed and convex set, $\Nc_{\Bc}$ is maximally monotone.
	From the construction of $\Bc$ under Assumption~\ref{as:restricted_dual_gap}, we have 
	\begin{equation*}
	\begin{array}{lcl}
		\emptyset \ne \ri{\Bc} \subseteq \Bc \subset \bigcap_{i=1}^n \mathrm{int}(\dom{T_i}) = \mathrm{int}\big(\bigcap_{i=1}^n \dom{T_i}\big) = \mathrm{int}\big(\dom{\Phi}\big) \subseteq \ri{\dom{\Phi}},
	\end{array}
	\end{equation*}
	Thus, applying Lemma~\ref{le:sum_maximal_monotone}, the operator $\Psi := \Phi + \Nc_B$ is maximally monotone.
	
	Now, let $u \in \Bc$, $\xi_i \in T_i u$, and $w_u \in \Nc_{\Bc}(u)$ arbitrarily.
	Using the definition of normal cone and $u^{\star} \in \Bc$, we have $\iprods{w_u, u - u^{\star}} \geq 0$.
	Adding this inequality to \eqref{eq:le_gap_function_proof1} yields
	\begin{equation*}
	\arraycolsep=0.2em
	\begin{array}{lcl}
		\iprod{\sum_{i=1}^n (G_i u + \xi_i) + w_u - 0, u - u^{\star}} \geq 0.
	\end{array}
	\end{equation*}
	Since this holds for every $u \in \Bc$, $\xi_i \in T_iu$, and $w_u \in \Nc_{\Bc}(u)$, by the maximality of $\Phi + \Nc_{\Bc}$, we can conclude that the point $(u^{\star}, 0) \in \gra{\Phi + \Nc_{\Bc}}$, or equivalently, $0 \in \Phi u^{\star} + \Nc_{\Bc}(u^{\star})$.	
	Thus, there exists $\xi_i^{\star} \in T_i u^{\star}$ such that $-\sum_{i=1}^n (G_i u^{\star} + \xi_i^{\star}) \in \Nc_B(u^{\star})$, leading to
	\begin{equation}\label{eq:le_gap_function_proof2}
		\arraycolsep=0.2em
		\begin{array}{lcl}
			\sum_{i=1}^n \iprods{G_i u^{\star} + \xi_i^{\star}, u - u^{\star}} \geq 0 \quad \text{for all}\  u \in \Bc.
		\end{array}
	\end{equation}    
	Particularly, considering $u^{\dagger} \in \Bc$, the last inequality yields
	\begin{equation*}
	\arraycolsep=0.2em
	\begin{array}{lcl}
		\sum_{i=1}^n \iprods{G_i u^{\star} + \xi_i^{\star}, u^{\star} - u^{\dagger}} \leq 0.
	\end{array}
	\end{equation*}   
	However, since $\sum_{i=1}^n (G_iu^{\dagger} + \xi_i^{\dagger}) = 0$ for $\xi_i^{\dagger} \in T_iu^{\dagger}$, the monotonicity of $\Phi$ gives
	\begin{equation*}
	\arraycolsep=0.2em
	\begin{array}{lcl}
		\sum_{i=1}^n \iprods{G_i u^{\star} + \xi_i^{\star},  u^{\star} - u^{\dagger}} \geq 0.
	\end{array}
	\end{equation*}    
	Therefore, we must have 
	\begin{equation}\label{eq:le_gap_function_proof3}
	\arraycolsep=0.2em
	\begin{array}{lcl}
		\sum_{i=1}^n \iprods{G_i u^{\star} + \xi_i^{\star},  u^{\star} - u^{\dagger}} = 0.
	\end{array}
	\end{equation}    
	Finally, we will show that $\sum_{i=1}^n (G_i u^{\star} + \xi_i^{\star}) = 0$.
	Since $u^{\dagger} \in \mathrm{int}(\Bc)$, there exists $\varepsilon > 0$ sufficiently small such that $u_{\varepsilon} := u^{\dagger} - \varepsilon\sum_{i=1}^n (G_i u^{\star} + \xi_i^{\star}) \in \Bc$.
	Using $u_{\varepsilon} \in \Bc$ in \eqref{eq:le_gap_function_proof2}, we obtain
	\begin{equation*}
	\arraycolsep=0.2em
	\begin{array}{lcl}
		0 &\leq& \sum_{i=1}^n \iprods{G_i u^{\star} + \xi_i^{\star}, u^{\dagger} - \varepsilon\sum_{i=1}^n (G_i u^{\star} + \xi_i^{\star}) - u^{\star}} \vspace{1ex}\\
		&=& \sum_{i=1}^n \iprods{G_i u^{\star} + \xi_i^{\star}, u^{\dagger} - u^{\star}} - \varepsilon\norms{\sum_{i=1}^n (G_i u^{\star} + \xi_i^{\star})}^2 \vspace{1ex}\\
		&\stackrel{\eqref{eq:le_gap_function_proof3}}{=}& - \varepsilon\norms{\sum_{i=1}^n (G_i u^{\star} + \xi_i^{\star})}^2 \leq 0.
	\end{array}
	\end{equation*} 
	Therefore, we have $\norms{\sum_{i=1}^n (G_i u^{\star} + \xi_i^{\star})}^2 = 0$, or equivalently, $\sum_{i=1}^n (G_i u^{\star} + \xi_i^{\star}) = 0$.
	Since $\xi_i^{\star} \in T_iu^{\star}$, the last equality implies that $0 \in \sum_{i=1}^n (G_i u^{\star} + T_i u^{\star})$, and thus $u^{\star} \in \zer{\Phi}$.
	
	\vspace{1ex}
	\noindent\textbf{\textit{Case 2. Bounded common-domain case.}}
	From the construction of $\Bc$ under Assumption~\ref{as:boundary_regularity}$\mathrm{(i)}$, we have
	\begin{equation*}
	\begin{array}{lcl}
		\Bc = \bigcap_{i=1}^n \dom{T_i} = \dom{\Phi}.
	\end{array}
	\end{equation*}
	Therefore, the inequality \eqref{eq:le_gap_function_proof1} holds on the entire domain $\dom{\Phi}$, leading to
	\begin{equation*} 
	\arraycolsep=0.2em
	\begin{array}{lcl}
		\iprods{\sum_{i=1}^n (G_i u + \xi_i) - 0, u - u^{\star}} \geq 0 \qquad \forall u \in \dom{\Phi}, \quad \xi_i \in T_i u.
	\end{array}
	\end{equation*}
	Since $\Phi$ is maximally monotone, the last inequality implies that $(u^{\star}, 0) \in \gra{\Phi}$.
	Thus, we have $0 \in \Phi(u^{\star})$, or equivalently, $u^{\star} \in \zer{\Phi}$.
\Eproof		

\beforesec
\section{Appendix: Proof of Technical Lemmas in Section~\ref{sec:FKM}}\label{apdx:sec:FKM_alg}
\aftersec
This appendix presents the full proof of technical lemmas in Section~\ref{sec:FKM}.

\beforesubsec
\subsection{Proof of Lemma~\ref{le:FKM_key_est1} --- Descent lemma for the Lyapunov function}\label{apdx:le:FKM_key_est1}
\aftersubsec
From the first and the third lines of \eqref{eq:FKM_scheme2}, we have $t_k x^{k+1} = (t_k - r)x^k + rz^k - \eta t_k \hat{w}^{k+1} + \beta_k t_k \hat{w}^k$.
	Using this expression and the fourth line of \eqref{eq:FKM_scheme2}, we can show that
	\begin{equation}\label{eq:FKM_key_est1_proof1}
		\arraycolsep=0.2em
		\left\{
		\begin{array}{lcl}
			t_k (t_k - r) (x^{k+1} - x^k) &=& r(t_k - r)(z^k - x^k) - \eta t_k (t_k - r)\hat{w}^{k+1} + \beta_k t_k (t_k - r) \hat{w}^k, \vspace{1ex}\\
			t_k (t_k - r) (x^{k+1} - x^k) &=& r t_k (z^k - x^{k+1}) - \eta t_k^2 \hat{w}^{k+1} + \beta_k t_k^2 \hat{w}^k \vspace{1ex}\\
			&=& r t_k (z^{k+1} - x^{k+1}) - (\eta t_k - \nu_k) t_k \hat{w}^{k+1} + \beta_k t_k^2 \hat{w}^k.
		\end{array}
		\right.
	\end{equation}
	From Assumption~\ref{as:FKM_assumption}, we have
	\begin{equation*}
		\arraycolsep=0.2em
		\begin{array}{lcl}
			t_k (t_k - r) \iprods{w^{k+1}, x^{k+1} - x^k} - t_k (t_k - r)\iprods{w^k, x^{k+1} - x^k} &\geq& 0.
		\end{array}
	\end{equation*}
	Substituting \eqref{eq:FKM_key_est1_proof1} into the last inequality, and then using $\hat{w}^k = w^k + e^k$, we can show that
	\begin{equation*}
		\arraycolsep=0.2em
		\begin{array}{lcl}
			\hat{\Tc}_{[1]} &:=& r(t_k - r) \iprods{w^k, x^k - z^k} - rt_k \iprods{w^{k+1}, x^{k+1} - z^{k+1}} \vspace{1ex}\\
			&\geq& (\eta t_k - \nu_k) t_k \iprods{w^{k+1}, \hat{w}^{k+1}} - \beta_k t_k^2 \iprods{w^{k+1}, \hat{w}^k} - \eta t_k (t_k - r)\iprods{w^k, \hat{w}^{k+1}} \vspace{1ex}\\
			&& + {~} \beta_k t_k (t_k - r)\iprods{w^k, \hat{w}^k} \vspace{1ex}\\
			&=& (\eta t_k - \nu_k) t_k \iprods{w^{k+1}, \hat{w}^{k+1}} - \beta_k t_k^2 \iprods{w^{k+1}, w^k} - \eta t_k (t_k - r)\iprods{w^k, \hat{w}^{k+1}} \vspace{1ex}\\
			&& + {~}  \beta_k t_k (t_k - r)\norms{w^k}^2  -  \beta_k t_k \iprods{t_k w^{k+1} - (t_k - r)w^k, e^k}.
		\end{array}
	\end{equation*}
	Using Young's inequality, for any $c_1 > 0$, we have
	\begin{equation*}
		\arraycolsep=0.2em
		\begin{array}{lcl}
			- \iprods{t_k w^{k+1} - (t_k - r)w^k, e^k} &\geq&  - \frac{c_1 \eta}{2\beta_k t_k} \norms{t_k w^{k+1} - (t_k - r)w^k}^2 - \frac{\beta_k t_k}{2c_1 \eta} \norms{e^k}^2 \vspace{1ex}\\
			&=& -\frac{c_1 \eta t_k}{2\beta_k}\norms{w^{k+1}}^2 + \frac{c_1 \eta (t_k - r)}{\beta_k}\iprods{w^{k+1}, w^k} \vspace{1ex}\\
			&& - {~} \frac{c_1 \eta (t_k - r)^2}{2\beta_k t_k}\norms{w^k}^2 - \frac{\beta_k t_k}{2c_1 \eta} \norms{e^k}^2.
		\end{array}
	\end{equation*}
	Substituting this relation into $\hat{\Tc}_{[1]}$, we obtain
	\begin{equation*}
		\arraycolsep=0.2em
		\begin{array}{lcl}
			\hat{\Tc}_{[1]} &:=& r(t_k - r) \iprods{w^k, x^k - z^k} - rt_k \iprods{w^{k+1}, x^{k+1} - z^{k+1}} \vspace{1ex}\\
			&\geq& (\eta t_k - \nu_k) t_k \iprods{w^{k+1}, \hat{w}^{k+1}} - \eta t_k (t_k - r)\iprods{w^k, \hat{w}^{k+1}}  - \frac{t_k (t_k - r)}{2} \left[\frac{c_1 \eta (t_k - r)}{t_k} - 2\beta_k \right] \norms{w^k}^2\vspace{1ex}\\
			&&  - {~} \frac{\beta_k^2 t_k^2}{2c_1 \eta}\norms{e^k}^2 - \frac{c_1 \eta t_k^2}{2}\norms{w^{k+1}}^2 
			+  t_k [c_1 \eta (t_k - r) - \beta_k t_k] \iprods{w^{k+1}, w^k}.
		\end{array}
	\end{equation*}
	Let us choose $\beta_k := \frac{c_1 \eta(t_k - r)}{t_k}$.
	Then, by $t_0 > r$ from \eqref{eq:FKM_key_est1_params}, we have $\beta_k > 0$.
	Moreover, we have $c_1 \eta (t_k - r) - \beta_k t_k = 0$.
	Partially substituting this $\beta_k$ and the identity $\iprods{w^k, \hat{w}^{k+1}} = \frac{1}{2}\left(\norms{w^k}^2 + \norms{\hat{w}^{k+1}}^2 - \norms{\hat{w}^{k+1} - w^k}^2\right)$ into $\hat{\Tc}_{[1]}$, we can further evaluate it as
	\begin{equation*}
		\arraycolsep=0.2em
		\begin{array}{lcl}
			\hat{\Tc}_{[1]} &:=& r(t_k - r) \iprods{w^k, x^k - z^k} - rt_k \iprods{w^{k+1}, x^{k+1} - z^{k+1}} \vspace{1ex}\\
			&\geq& (\eta t_k - \nu_k) t_k \iprods{w^{k+1}, \hat{w}^{k+1}} - \frac{\eta t_k (t_k - r)}{2}\norms{\hat{w}^{k+1}}^2 + \frac{\eta t_k (t_k - r)}{2}\norms{\hat{w}^{k+1} - w^k}^2 - \frac{\beta_k^2 t_k^2}{2c_1 \eta}\norms{e^k}^2  \vspace{1ex}\\
			&& - {~} \frac{\eta t_k (t_k - r)}{2} \left[ 1 - \frac{c_1 (t_k - r)}{t_k} \right] \norms{w^k}^2 - \frac{c_1 \eta t_k^2}{2}\norms{w^{k+1}}^2.
		\end{array}
	\end{equation*}
	From the second and the third lines of \eqref{eq:FKM_scheme2}, we have
	\begin{equation*}
		\arraycolsep=0.2em
		\begin{array}{lcl}
			x^{k+1} - y^k &:=& \eta (\hat{w}^k - \hat{w}^{k+1}) = \eta [(w^k - \hat{w}^{k+1}) + e^k].
		\end{array}
	\end{equation*}
	Using the $L$-Lipschitz continuity of $G$, the last expression, Young's inequality, and the identity $-2\iprods{w^k - \hat{w}^{k+1}, w^k} = \norms{\hat{w}^{k+1}}^2 - \norms{w^k}^2 - \norms{w^k - \hat{w}^{k+1}}^2$, for any $\omega > 0$, we can show that
	\begin{equation*}
		\arraycolsep=0.2em
		\begin{array}{lcl}
			\norms{w^{k+1} - \hat{w}^{k+1}}^2 &=& (1+\omega)\norms{Gx^{k+1} - Gy^k}^2 - \omega \norms{Gx^{k+1} - Gy^k}^2 \vspace{1ex}\\
			&\leq& (1+\omega)L^2 \norms{x^{k+1} - y^k}^2 - \omega \norms{e^{k+1}}^2 \vspace{1ex}\\
			&=& (1+\omega)L^2 \eta^2 \norms{(w^k - \hat{w}^{k+1}) + e^k}^2  - \omega \norms{e^{k+1}}^2\vspace{1ex}\\
			&\leq& 2(1+\omega)L^2 \eta^2 \norms{w^k - \hat{w}^{k+1}}^2 + 2(1+\omega)L^2\eta^2 \norms{e^k}^2  - \omega \norms{e^{k+1}}^2.
		\end{array}
	\end{equation*}
	Expanding the left-hand side of this inequality and then rearranging the result, we get
	\begin{equation*}
		\arraycolsep=0.2em
		\begin{array}{lcl}
			0 &\geq& \norms{w^{k+1}}^2 + \norms{\hat{w}^{k+1}}^2 - 2\iprods{w^{k+1}, \hat{w}^{k+1}} - M^2 \eta^2 \norms{w^k - \hat{w}^{k+1}}^2 - M^2\eta^2 \norms{e^k}^2 + \omega\norms{e^{k+1}}^2,
		\end{array}
	\end{equation*}
	where $M^2 := 2(1+\omega)L^2$.	

	Suppose that $0 < \nu_k \leq \eta t_k$.
	Multiplying the last inequality by $\frac{(\eta t_k - \nu_k)t_k}{2} \geq 0$ then adding the result to $\hat{\Tc}_{[1]}$, we obtain
	\begin{equation*}
		\arraycolsep=0.2em
		\begin{array}{lcl}
			\hat{\Tc}_{[1]} &:=& r(t_k - r) \iprods{w^k, x^k - z^k} - rt_k \iprods{w^{k+1}, x^{k+1} - z^{k+1}} \vspace{1ex}\\
			&\geq&  \frac{t_k}{2} (r \eta - \nu_k) \norms{\hat{w}^{k+1}}^2 + \frac{t_k}{2}[\eta (t_k - r) - M^2 \eta^2 (\eta t_k - \nu_k)] \norms{\hat{w}^{k+1} - w^k}^2 \vspace{1ex}\\
			&& - {~} \frac{\eta t_k (t_k - r)}{2} \left[ 1 - \frac{c_1 (t_k - r)}{t_k} \right] \norms{w^k}^2 + \frac{t_k}{2}[(1-c_1)\eta t_k - \nu_k]\norms{w^{k+1}}^2 \vspace{1ex}\\
			&& - {~}  \frac{t_k}{2} \left[ \frac{\beta_k^2 t_k}{c_1 \eta} + M^2 \eta^2 (\eta t_k - \nu_k) \right]\norms{e^k}^2 + \frac{\omega (\eta t_k - \nu_k)t_k}{2}\norms{e^{k+1}}^2.
		\end{array}
	\end{equation*}
	Using $t_{k+1} = t_k + 1$, we obtain from the last inequality that 
	\begin{equation*}
		\arraycolsep=0.2em
		\begin{array}{lcl}
			\Tc_{[1]} &:=& r(t_k - r) \iprods{w^k, x^k - z^k} - r(t_{k+1}-r) \iprods{w^{k+1}, x^{k+1} - z^{k+1}} \vspace{1ex}\\
			&=&\hat{\Tc}_{[1]} + r(r-1)\iprods{w^{k+1}, x^{k+1} - z^{k+1}} \vspace{1ex}\\
			&\geq& r(r-1)\iprods{w^{k+1}, x^{k+1} - z^{k+1}} + \frac{t_k (r \eta - \nu_k)}{2} \norms{\hat{w}^{k+1}}^2 \vspace{1ex}\\
			&& + {~} \frac{t_k}{2}[\eta (t_k - r) - M^2 \eta^2 (\eta t_k - \nu_k)] \norms{\hat{w}^{k+1} - w^k}^2 \vspace{1ex}\\
			&& - {~} \frac{\eta t_k (t_k - r)}{2} \left[ 1 - \frac{c_1 (t_k - r)}{t_k} \right] \norms{w^k}^2 + \frac{t_k}{2}[(1-c_1)\eta t_k - \nu_k]\norms{w^{k+1}}^2 \vspace{1ex}\\
			&& - {~} \frac{t_k}{2} \left[ \frac{\beta_k^2 t_k}{c_1 \eta} + M^2 \eta^2 (\eta t_k - \nu_k) \right]\norms{e^k}^2  +  \frac{\omega (\eta t_k - \nu_k)t_k}{2}\norms{e^{k+1}}^2.
		\end{array}
	\end{equation*}
	Next, utilizing $r(z^{k+1} - z^k) = -\nu_k \hat{w}^{k+1}$ from the last line of \eqref{eq:FKM_scheme2} and Young's inequality, for any $c_k > 0$, we can derive that
	\begin{equation*}
		\arraycolsep=0.2em
		\begin{array}{lcl}
			\hat{\Tc}_{[2]} &:=& \frac{r^2(r-1)}{2\nu_k}\norms{z^k - \xopt}^2 - \frac{r^2(r-1)}{2\nu_k}\norms{z^{k+1} - \xopt}^2 \vspace{1ex}\\
			&=& -\frac{r^2(r-1)}{\nu_k}\iprods{z^{k+1} - z^k, z^{k+1} - \xopt} + \frac{r^2(r-1)}{2\nu_k}\norms{z^{k+1} - z^k}^2 \vspace{1ex}\\
			&=& r(r-1)\iprods{\hat{w}^{k+1}, z^{k+1} - \xopt} + \frac{(r-1)\nu_k}{2}\norms{\hat{w}^{k+1}}^2 \vspace{1ex}\\
			&=& r(r-1)\iprods{w^{k+1} + e^{k+1}, z^{k+1} - \xopt} +  \frac{\nu_k(r-1)}{2}\norms{\hat{w}^{k+1}}^2 \vspace{1ex}\\
			&\geq&  r(r-1)\iprods{w^{k+1}, z^{k+1} - \xopt} - \frac{r^2(r-1)c_k}{2}\norms{z^{k+1} - \xopt}^2 + \frac{\nu_k(r-1)}{2}\norms{\hat{w}^{k+1}}^2 - \frac{r-1}{2c_k}\norms{e^{k+1}}^2.
		\end{array}
	\end{equation*}
	Let us choose $\nu_k := \frac{\nu(t_k - 1)}{t_k}$ for some $\nu > 0$.
	Then, by $t_0 > 1$ from \eqref{eq:FKM_key_est1_params}, we get $\nu_k > 0$.
	In addition, we have $\nu_k \leq \nu \leq \eta t_0 \leq \eta t_k$ due to the choice of $t_0$ as in \eqref{eq:FKM_key_est1_params}.
	Since $t_{k+1} = t_k + 1$, we can easily prove that
	\begin{equation*}
		\arraycolsep=0.2em
		\begin{array}{lcl}
			\frac{1}{2\nu_k} - \frac{1}{2\nu_{k+1}} = \frac{1}{2\nu}\left(\frac{t_k}{t_k - 1} - \frac{t_{k+1}}{t_{k+1}-1} \right) = \frac{1}{2\nu t_k (t_k-1)}.
		\end{array}
	\end{equation*}
	Let us choose $c_k := \frac{1}{\nu_k} - \frac{1}{\nu_{k+1}} = \frac{1}{\nu t_k (t_k-1)} > 0$.
	Then, we have $\frac{c_k}{2} - \frac{1}{2\nu_k} = -\frac{1}{2\nu_{k+1}}$.
	Using this relation, we obtain from $\hat{\Tc}_{[2]}$ that
	\begin{equation*}
		\arraycolsep=0.2em
		\begin{array}{lcl}
			\Tc_{[2]} &:=& \frac{r^2(r-1)}{2\nu_k}\norms{z^k - \xopt}^2 - \frac{r^2(r-1)}{2\nu_{k+1}}\norms{z^{k+1} - \xopt}^2 \vspace{1ex}\\
			&\geq&  r(r-1)\iprods{w^{k+1}, z^{k+1} - \xopt} + \frac{(r-1)\nu (t_k-1)}{2t_k}\norms{\hat{w}^{k+1}}^2 - \frac{(r-1)\nu t_k (t_k-1)}{2}\norms{e^{k+1}}^2.
		\end{array}
	\end{equation*}
	Adding $\Tc_{[2]}$ to $\Tc_{[1]}$, we have
	\begin{equation*}
		\arraycolsep=0.2em
		\begin{array}{lcl}
			\Tc_{[3]} &:=& r(t_k - r) \iprods{w^k, x^k - z^k} + \frac{r^2(r-1)}{2\nu_k}\norms{z^k - \xopt}^2 - r(t_{k+1}-r) \iprods{w^{k+1}, x^{k+1} - z^{k+1}} \vspace{1ex}\\
			&& - {~} \frac{r^2(r-1)}{2\nu_{k+1}}\norms{z^{k+1} - \xopt}^2 \vspace{1ex}\\
			&\geq& r(r-1)\iprods{w^{k+1}, x^{k+1} - \xopt} + \left[\frac{t_k (r \eta - \nu_k)}{2} + \frac{(r-1)\nu (t_k-1)}{2t_k} \right]\norms{\hat{w}^{k+1}}^2 \vspace{1ex}\\
			&& + {~} \frac{t_k}{2}[\eta (t_k - r) - M^2 \eta^2 (\eta t_k - \nu_k)] \norms{\hat{w}^{k+1} - w^k}^2 \vspace{1ex}\\
			&& - {~} \frac{\eta t_k (t_k - r)}{2} \left[ 1 - \frac{c_1 (t_k - r)}{t_k} \right] \norms{w^k}^2 + \frac{t_k}{2}[(1-c_1)\eta t_k - \nu_k]\norms{w^{k+1}}^2 \vspace{1ex}\\
			&& - {~} \frac{t_k}{2} \left[ \frac{\beta_k^2 t_k}{c_1 \eta} + M^2 \eta^2 (\eta t_k - \nu_k) \right]\norms{e^k}^2 + \frac{t_k}{2}\big[\omega (\eta t_k - \nu_k)
			- (r-1)\nu (t_k-1)\big]\norms{e^{k+1}}^2.
		\end{array}
	\end{equation*}
	Let us introduce the following coefficients:
	\begin{equation*}
	\arraycolsep=0.2em
	\left\{\begin{array}{lcllcl}
		a_k &:=& \eta t_k (t_k - r) \left[ 1 - \frac{c_1 (t_k - r)}{t_k} \right], 
		\qquad
		&\hat{a}_{k+1} &:=& t_k [(1-c_1)\eta t_k - \nu_k], \vspace{1ex}\\
		b_k &:=& t_k (r \eta - \nu_k) + \frac{(r-1)\nu (t_k-1)}{t_k}, \qquad
		&\phi_k &:=& t_k[\eta (t_k - r) - M^2 \eta^2 (\eta t_k - \nu_k)], \vspace{1ex}\\
		d_k &:=& t_k \left[ \frac{\beta_k^2 t_k}{c_1 \eta} + M^2 \eta^2 (\eta t_k - \nu_k) \right], 
		\qquad \qquad
		&\hat{d}_{k+1} &:=& t_k \big[\omega (\eta t_k - \nu_k) - (r-1)\nu (t_k-1)\big],
	\end{array}\right.
	\end{equation*}
	and potential function
	\begin{equation*}
		\arraycolsep=0.2em
		\begin{array}{lcl}
			\Pc_k &:=& \frac{a_k}{2}\norms{w^k}^2 + r(t_k - r) \iprods{w^k, x^k - z^k} + \frac{r^2(r-1)}{2\nu_k}\norms{z^k - \xopt}^2 + \frac{d_k}{2}\norms{e^k}^2.
		\end{array}
	\end{equation*}
	Then, the last inequality implies  exactly \eqref{eq:FKM_key_est1}.
\Eproof	

\beforesubsec
\subsection{Proof of Lemma~\ref{le:FKM_key_est2} --- Simplified descent lemma}\label{apdx:le:FKM_key_est2}
\aftersubsec
	First, using the coefficients from Lemma~\ref{le:FKM_key_est1}, we can compute that
	\begin{equation*}
		\arraycolsep=0.2em
		\left\{\begin{array}{lcl}
			\hat{a}_{k+1} - a_{k+1}  &=& [(r - 2(r-1)c_1 - 2)\eta - \nu]t_{k+1} + [1 + (r^2 - 1)c_1]\eta + 2\nu, \vspace{1ex}\\
			b_k  &=& (r\eta - \nu)t_{k+1} - (r\eta - 2\nu) + \frac{(r-1)\nu (t_k - 1)}{t_k}, \vspace{1ex}\\
			\phi_k &=& (1 - M^2 \eta^2)\eta t_k^2 + (M^2 \eta \nu - r) \eta t_k - M^2 \eta^2 \nu, \vspace{1ex}\\
			\hat{d}_{k+1} - d_{k+1} &=& [\omega - c_1 - \frac{(r-1)\nu}{\eta} - M^2 \eta^2] \eta t_{k+1}^2 + [2(\omega - r + 1) - M^2 \eta^2]\nu \vspace{1ex}\\
			&& - {~} \big[2(\omega - rc_1) \eta + \big(\omega - 3(r-1) - M^2 \eta^2\big)\nu \big] t_{k+1}  + (\omega - r^2 c_1) \eta.
		\end{array}\right.
	\end{equation*}
	To achieve our goal, we need to guarantee the following conditions:
	\begin{equation}\label{eq:FKM_key_est2_proof1}
	\arraycolsep=0.2em
	\begin{array}{ll}
		&\left(r - 2(r-1)c_1 - 2\right)\eta > \nu, \quad 
		r\eta > \nu,	\quad
		1 - 2(1+\omega)L^2 \eta^2 > 0, \vspace{1ex}\\
		& \omega - c_1 - \frac{(r-1)\nu}{\eta} - 2(1+\omega)L^2\eta^2 > 0.
	\end{array}
	\end{equation}
	Let us choose $c_1 := \frac{\delta}{2(r-1)}$ for some small $\delta > 0$.
	Then, the first condition of \eqref{eq:FKM_key_est2_proof1} becomes $\nu < (r - 2 - \delta)\eta$, provided that $r > 2$ and $\delta < r - 2$.
	The second condition of \eqref{eq:FKM_key_est2_proof1} can be guaranteed by the first condition.
	If we choose $\eta>0$ such that $2(1 + \omega)L^2 \eta^2 \leq \frac{1}{2}$, then the third condition of \eqref{eq:FKM_key_est2_proof1} holds.
	Finally, utilizing $c_1 := \frac{\delta}{2(r-1)}$, the upper bound of $\nu$, and $2(1 + \omega)L^2 \eta^2 \leq \frac{1}{2}$, we can see that the fourth condition of \eqref{eq:FKM_key_est2_proof1} holds if
	\begin{equation*}
		\arraycolsep=0.2em
		\begin{array}{lcl}
			\omega \geq \frac{\delta}{2(r-1)} + (r-1)(r-2-\delta) + \frac{1}{2}.
		\end{array}
	\end{equation*}
	Since $\delta < r-2 < r-1$, the last condition holds if $\omega \geq (r-1)(r-2) + 1 - \delta(r-1)$.
	Moreover, it is easy to show that the choice $\omega := (r-1)(r-2) + 1$ satisfies this condition.
	
	Under these choice of parameters, we can bound the coefficients in \eqref{eq:FKM_key_est1} as follows:
	\begin{equation*} 
		\arraycolsep=0.2em
		\left\{
		\begin{array}{lcl}
			\hat{a}_{k+1} - a_{k+1} &\geq& [(r-2-\delta)\eta - \nu]t_{k+1} + \frac{[2 + (r+1)\delta]\eta}{2} + 2\nu \geq [(r-2-\delta)\eta - \nu]t_{k+1}, \vspace{1ex}\\
			b_k &\geq& (2+\delta)\eta t_{k+1} - (r\eta - 2\nu) + \frac{(r-1)\nu (t_k - 1)}{t_k} \stackrel{\tiny\textcircled{1}}{\geq} \eta t_{k+1},	 \vspace{1ex}\\	
			\phi_k &\geq& \frac{\eta t_k^2}{2} - (r - M^2 \eta \nu)\eta t_k - M^2 \eta^2 \nu \stackrel{\tiny\textcircled{2}}{\geq} \frac{\eta t_k^2}{4}, \vspace{1ex}\\
			\hat{d}_{k+1} - d_{k+1} &\geq& \frac{\eta t_{k+1}^2}{2} - r\omega\eta t_{k+1} \stackrel{\tiny\textcircled{3}}{\geq}  \frac{\eta t_{k+1}^2}{4},
		\end{array}
		\right.
	\end{equation*}
	where {\tiny\textcircled{1}} holds due to $t_{k+1} \geq t_0 \geq  r\eta - 2\nu$, {\tiny\textcircled{2}} holds due to $t_k \geq t_0 \geq 4r$, and {\tiny\textcircled{3}} holds due to $t_{k+1} \geq t_0 \geq  4r\omega$ as in \eqref{eq:FKM_key_est2_t0}.
	Therefore, the expression \eqref{eq:FKM_key_est1} can be simplified to
	\begin{equation*}
		\arraycolsep=0.2em
		\begin{array}{lcl}
			\Pc_k - \Pc_{k+1} &\geq& \frac{[(r-2-\delta)\eta - \nu]t_{k+1}}{2} \norms{w^{k+1}}^2 + r(r-1)\iprods{w^{k+1}, x^{k+1} - \xopt} + \frac{\eta t_{k+1}}{2} \norms{\hat{w}^{k+1}}^2  \vspace{1ex}\\ 
			&& + {~} \frac{\eta t_k^2}{8} \norms{\hat{w}^{k+1} - w^k}^2 + \frac{\eta t_{k+1}^2}{8}\norms{e^{k+1}}^2.
		\end{array}
	\end{equation*}
	Substituting $\iprods{w^{k+1}, x^{k+1} - \xopt} \geq 0$ from Assumption \ref{as:FKM_assumption} into the last relation, we get \eqref{eq:FKM_key_est2}.
	Finally, we note that the choice of $t_0$ from \eqref{eq:FKM_key_est2_t0} satisfies the condition in Lemma~\ref{le:FKM_key_est1}.
\Eproof	

\beforesubsec
\subsection{Proof of Lemma~\ref{le:FKM_key_est3} --- Lower bound of Lyapunov function}\label{apdx:le:FKM_key_est3}
\aftersubsec
	Exploiting $\iprods{w^k, x^k - \xopt} \geq 0$ from Assumption~\ref{as:FKM_assumption}, we can show from \eqref{eq:FKM_Lyapunov_func} that
	\begin{equation}\label{eq:FKM_key_est3_proof1}
		\arraycolsep=0.2em
		\begin{array}{lcl}
			\Pc_k &=& \frac{a_k}{2}\norms{w^k}^2 - r(t_k - r)\iprods{w^k, z^k - \xopt} + \frac{r^2(r-1)t_k}{2\nu(t_k - 1)}\norms{z^k - \xopt}^2 \vspace{1ex}\\
			&& + {~} \frac{d_k}{2}\norms{e^k}^2 + r(t_k - r)\iprods{w^k, x^k - \xopt} \vspace{1ex}\\
			&=& \frac{a_k - c_2 (t_k - r)^2}{2} \norms{w^k}^2 + \frac{r^2}{2}\left[\frac{(r-1)t_k}{\nu(t_k - 1)} - \frac{1}{c_2}\right]\norms{z^k - \xopt}^2 + \frac{d_k}{2}\norms{e^k}^2  \vspace{1ex}\\
			&& + {~}  r(t_k - r)\iprods{w^k, x^k - \xopt} + \frac{1}{2c_2}\norms{c_2(t_k - r)w^k - r(z^k - \xopt)}^2. \vspace{1ex}\\
			&\geq& \frac{a_k - c_2 (t_k - r)^2}{2} \norms{w^k}^2 + \frac{r^2}{2}\left[\frac{(r-1)t_k}{\nu(t_k - 1)} - \frac{1}{c_2}\right]\norms{z^k - \xopt}^2 + \frac{d_k}{2}\norms{e^k}^2.
		\end{array}
	\end{equation}
	Let us choose $c_2 := \frac{\nu}{r-2}$. 
	Then, from the condition $\nu < [r-2-2(r-1)c_1]\eta$, we can show that $c_2 = \frac{\nu}{r-2} < [1 - \frac{2(r-1)c_1}{r-2}]\eta$.
	Using this relation and $c_1 = \frac{\delta}{2(r-1)}$, we can bound
	\begin{equation*}
		\arraycolsep=0.2em
		\begin{array}{lcl}
			a_k - c_2 (t_k - r)^2  
			&=&  [(1-c_1)\eta - c_2] (t_k - r)^2 + r\eta (t_k - r) 
			\geq \frac{\delta \eta}{2(r-1)}(t_k - r)^2,
		\end{array}
	\end{equation*}
	and
	\begin{equation*}
		\arraycolsep=0.2em
		\begin{array}{lcl}
			\frac{(r-1)t_k}{\nu(t_k - 1)} - \frac{1}{c_2} = \frac{(r-1)t_k}{\nu(t_k - 1)} - \frac{r-2}{\nu} \geq \frac{r-1}{\nu} - \frac{r-2}{\nu} = \frac{1}{\nu}.
		\end{array}
	\end{equation*}
	Substituting these bounds into \eqref{eq:FKM_key_est3_proof1} and noting that $e^k = Gy^{k-1} - Gx^k$, we get \eqref{eq:FKM_key_est3}.
\Eproof	

\beforesubsec
\subsection{Proof of Lemma~\ref{le:FKM2_key_est1} --- Descent lemma for the co-coercive case}\label{apdx:le:FKM2_key_est1}
	From the first two lines of \eqref{eq:FKM2_scheme_reformulation}, we have $t_{k+1}y^{k+1} = (t_{k+1} - r)y^k + rz^{k+1} - \eta (t_{k+1} - r)\Gc_{\eta}y^k$.
	Using this expression and the third line of \eqref{eq:FKM2_scheme_reformulation}, we can show that
	\begin{equation}\label{eq:FKM2_key_est1_proof1}
		\arraycolsep=0.2em
		\hspace{-1ex}\left\{ 
		\begin{array}{lcl}
			t_{k+1}(t_{k+1} - r) (y^{k+1} - y^k)&=& r t_{k+1}(z^{k+1} - y^{k+1}) - \eta t_{k+1} (t_{k+1} - r) \Gc_{\eta}y^k, \vspace{1ex}\\
			t_{k+1}(t_{k+1} - r) (y^{k+1} - y^k)&=& r(t_{k+1} - r)(z^{k+1} - y^k) - \eta (t_{k+1} - r)^2 \Gc_{\eta}y^k \vspace{1ex}\\
			&=& r(t_{k+1} - r)(z^k - y^k) - [\eta (t_{k+1} - r) + \nu_k](t_{k+1} - r) \Gc_{\eta}y^k.
		\end{array}
		\right.\hspace{-1ex}
	\end{equation}
	From the $\bar{\beta}$-co-coercivity of $\Gc_{\eta}$, for any $\beta \in (0, \bar{\beta})$, we have
	\begin{equation*}
		\arraycolsep=0.2em
		\begin{array}{lcl}
			\tilde{\Tc}_{[1]} &:=& t_{k+1}(t_{k+1} - r)\iprods{\Gc_{\eta}y^{k+1}, y^{k+1} - y^k} - t_{k+1}(t_{k+1} - r)\iprods{\Gc_{\eta}y^k, y^{k+1} - y^k} \vspace{1ex}\\
			&\geq& (\bar{\beta} - \beta) t_{k+1}(t_{k+1} - r)\norms{\Gc_{\eta}y^{k+1} - \Gc_{\eta}y^k}^2 + \beta t_{k+1}(t_{k+1} - r)\norms{\Gc_{\eta}y^{k+1} - \Gc_{\eta}y^k}^2.
		\end{array}
	\end{equation*}
	Substituting \eqref{eq:FKM2_key_est1_proof1} into the last inequality, then rearranging the result, we obtain
	\begin{equation*}
		\arraycolsep=0.2em
		\hspace{-5ex}
		\begin{array}{lcl}
			\tilde{\Tc}_{[1]} &:=& r(t_{k+1} - r)\iprods{\Gc_{\eta}y^k, y^k - z^k} - rt_{k+1}\iprods{\Gc_{\eta}y^{k+1}, y^{k+1} - z^{k+1}} \vspace{1ex}\\
			&\geq& \eta t_{k+1}(t_{k+1} - r)\iprods{\Gc_{\eta}y^{k+1}, \Gc_{\eta}y^k} - [\eta (t_{k+1} - r) + \nu_k] (t_{k+1} - r) \norms{\Gc_{\eta}y^k}^2 \vspace{1ex}\\
			&& + {~} (\bar{\beta} - \beta) t_{k+1}(t_{k+1} - r)\norms{\Gc_{\eta}y^{k+1} - \Gc_{\eta}y^k}^2 + \beta t_{k+1}(t_{k+1} - r)\norms{\Gc_{\eta}y^{k+1} - \Gc_{\eta}y^k}^2.
		\end{array}
		\hspace{-5ex}
	\end{equation*}
	Using this inequality and $t_{k+1} = t_k + 1$, we can show that
	\begin{equation*}
		\arraycolsep=0.2em
		\hspace{-2ex}\begin{array}{lcl}
			\breve{\Tc}_{[1]} &:=& rt_k\iprods{\Gc_{\eta}y^k, y^k - z^k} - rt_{k+1}\iprods{\Gc_{\eta}y^{k+1}, y^{k+1} - z^{k+1}} \vspace{1ex}\\
			&\geq& \eta t_{k+1}(t_{k+1} - r)\iprods{\Gc_{\eta}y^{k+1}, \Gc_{\eta}y^k} - [\eta (t_{k+1} - r) + \nu_k] (t_{k+1} - r) \norms{\Gc_{\eta}y^k}^2 \vspace{1ex}\\
			&& + {~} (\bar{\beta} - \beta) t_{k+1}(t_{k+1} - r)\norms{\Gc_{\eta}y^{k+1} - \Gc_{\eta}y^k}^2 + \beta t_{k+1}(t_{k+1} - r)\norms{\Gc_{\eta}y^{k+1} - \Gc_{\eta}y^k}^2 \vspace{1ex}\\
			&& + {~} r(t_k - t_{k+1} + r)\iprods{\Gc_{\eta}y^k, y^k - z^k} \vspace{1ex}\\
			&=& \beta t_{k+1}(t_{k+1} - r)\norms{\Gc_{\eta}y^{k+1}}^2 - [(\eta - \beta) t_{k+1} + \nu_k - r\eta] (t_{k+1} - r) \norms{\Gc_{\eta}y^k}^2 \vspace{1ex}\\
			&& + {~} (\eta - 2\beta) t_{k+1}(t_{k+1} - r) \iprods{\Gc_{\eta}y^{k+1}, \Gc_{\eta}y^k} + (\bar{\beta} - \beta) t_{k+1}(t_{k+1} - r)\norms{\Gc_{\eta}y^{k+1} - \Gc_{\eta}y^k}^2\vspace{1ex}\\
			&& + {~} r(r-1)\iprods{\Gc_{\eta}y^k, y^k - z^k}.
		\end{array}\hspace{-5ex}
	\end{equation*}
	Next, utilizing $z^{k+1} - z^k = -\frac{\nu_k}{r}\Gc_{\eta}y^k$ from the third line of \eqref{eq:FKM2_scheme_reformulation}, we have
	\begin{equation*}
		\arraycolsep=0.1em
		\begin{array}{lcl}
			\breve{\Tc}_{[2]} &:=& \frac{r^2(r-1)}{2\nu_k} \norms{z^k - \xopt}^2 - \frac{r^2(r-1)}{2\nu_{k+1}} \norms{z^{k+1} - \xopt}^2 \vspace{1ex}\\
			&=& \frac{r^2(r-1)}{2\nu_k} \left(\norms{z^k - \xopt}^2 - \norms{z^{k+1} - \xopt}^2\right) + \frac{r^2(r-1)}{2}\left(\frac{1}{\nu_k} - \frac{1}{\nu_{k+1}}\right) \norms{z^{k+1} - \xopt}^2 \vspace{1ex}\\
			&=& -\frac{r^2(r-1)}{\nu_k} \iprods{z^{k+1} - z^k, z^k - \xopt} - \frac{r^2(r-1)}{2\nu_k} \norms{z^{k+1} - z^k}^2 + \frac{r^2(r-1)}{2}\left(\frac{1}{\nu_k} - \frac{1}{\nu_{k+1}}\right) \norms{z^{k+1} - \xopt}^2 \vspace{1ex}\\
			&=& r(r-1)\iprods{\Gc_{\eta}y^k, z^k - \xopt} - \frac{(r-1)\nu_k}{2} \norms{\Gc_{\eta}y^k}^2 + \frac{r^2(r-1)}{2}\left(\frac{1}{\nu_k} - \frac{1}{\nu_{k+1}}\right) \norms{z^{k+1} - \xopt}^2.
		\end{array}
	\end{equation*}
	Adding $\breve{\Tc}_{[1]}$ and $\breve{\Tc}_{[2]}$, we obtain
	\begin{equation*}
		\arraycolsep=0.2em
		\hspace{-1ex}\begin{array}{lcl}
			\breve{\Tc}_{[3]} &:=& rt_k\iprods{\Gc_{\eta}y^k, y^k - z^k} + \frac{r^2(r-1)}{2\nu_k} \norms{z^k - \xopt}^2 - rt_{k+1}\iprods{\Gc_{\eta}y^{k+1}, y^{k+1} - z^{k+1}} \vspace{1ex}\\
			&& - {~} \frac{r^2(r-1)}{2\nu_{k+1}} \norms{z^{k+1} - \xopt}^2 \vspace{1ex}\\
			&\geq& \beta t_{k+1}(t_{k+1} - r)\norms{\Gc_{\eta}y^{k+1}}^2 - \set{[(\eta - \beta) t_{k+1} + \nu_k - r\eta] (t_{k+1} - r) + \frac{(r-1)\nu_k}{2}} \norms{\Gc_{\eta}y^k}^2 \vspace{1ex}\\
			&& + {~} (\eta - 2\beta) t_{k+1}(t_{k+1} - r) \iprods{\Gc_{\eta}y^{k+1}, \Gc_{\eta}y^k} + (\bar{\beta} - \beta) t_{k+1}(t_{k+1} - r)\norms{\Gc_{\eta}y^{k+1} - \Gc_{\eta}y^k}^2\vspace{1ex}\\
			&& + {~} \frac{r^2(r-1)}{2}\left(\frac{1}{\nu_k} - \frac{1}{\nu_{k+1}}\right) \norms{z^{k+1} - \xopt}^2 + r(r-1)\iprods{\Gc_{\eta}y^k, y^k - \xopt}.
		\end{array}\hspace{-1ex}
	\end{equation*}
	Since $\beta := \frac{\eta}{2}$, we have $\eta - 2\beta = 0$, and the last inequality becomes
	\begin{equation*}
		\arraycolsep=0.2em
		\begin{array}{lcl}
			\breve{\Tc}_{[3]} &:=& rt_k\iprods{\Gc_{\eta}y^k, y^k - z^k} + \frac{r^2(r-1)}{2\nu_k} \norms{z^k - \xopt}^2 - rt_{k+1}\iprods{\Gc_{\eta}y^{k+1}, y^{k+1} - z^{k+1}} \vspace{1ex}\\
			&& - {~} \frac{r^2(r-1)}{2\nu_{k+1}} \norms{z^{k+1} - \xopt}^2 \vspace{1ex}\\
			&\geq& \frac{\eta t_{k+1}(t_{k+1} - r)}{2} \norms{\Gc_{\eta}y^{k+1}}^2 - \set{\left[\frac{\eta (t_{k+1} - 2r)}{2} + \nu_k\right] (t_{k+1} - r) + \frac{(r-1)\nu_k}{2}} \norms{\Gc_{\eta}y^k}^2 \vspace{1ex}\\
			&& + {~} \frac{(2\bar{\beta} - \eta) t_{k+1}(t_{k+1} - r)}{2} \norms{\Gc_{\eta}y^{k+1} - \Gc_{\eta}y^k}^2 + \frac{r^2(r-1)}{2}\left(\frac{1}{\nu_k} - \frac{1}{\nu_{k+1}}\right) \norms{z^{k+1} - \xopt}^2 \vspace{1ex}\\
			&& + {~} r(r-1)\iprods{\Gc_{\eta}y^k, y^k - \xopt}.
		\end{array}
	\end{equation*}
	Let us choose $\nu_k := \frac{\nu(t_k - 1)}{t_k}$.
	Then, we get $\frac{1}{\nu_k} - \frac{1}{\nu_{k+1}} = \frac{1}{\nu t_k (t_k - 1)}$.
	Moreover, we can further lower bound $\breve{\Tc}_{[3]}$ as follows:
	\begin{equation*}
		\arraycolsep=0.1em
		\begin{array}{lcl}
			\breve{\Tc}_{[3]} &:=& rt_k\iprods{\Gc_{\eta}y^k, y^k - z^k} + \frac{r^2(r-1)}{2\nu_k} \norms{z^k - \xopt}^2 - rt_{k+1}\iprods{\Gc_{\eta}y^{k+1}, y^{k+1} - z^{k+1}} \vspace{1ex}\\
			&& - {~} \frac{r^2(r-1)}{2\nu_{k+1}} \norms{z^{k+1} - \xopt}^2 \vspace{1ex}\\
			&\geq& \frac{\eta t_{k+1}(t_{k+1} - r)}{2} \norms{\Gc_{\eta}y^{k+1}}^2 - \left[\frac{\eta (t_{k+1} - 2r)(t_{k+1} - r)}{2} + \frac{\nu(t_{k} - \frac{r-1}{2})(t_k - 1)}{t_k}\right] \norms{\Gc_{\eta}y^k}^2 \vspace{1ex}\\
			&& + {~} \frac{(2\bar{\beta} - \eta) t_{k+1}(t_{k+1} - r)}{2} \norms{\Gc_{\eta}y^{k+1} - \Gc_{\eta}y^k}^2 + \frac{r^2(r-1)}{2\nu t_k (t_k - 1)} \norms{z^{k+1} - \xopt}^2 + r(r-1)\iprods{\Gc_{\eta}y^k, y^k - \xopt}.
		\end{array}
	\end{equation*}
	Using the definition of $\Lc_k$ from \eqref{eq:FKM2_Lyapunov_function}, the last inequality leads to
	\begin{equation}\label{eq:FKM2_key_est1_proof2}
		\arraycolsep=0.2em
		\begin{array}{lcl}
			\Lc_k - \Lc_{k+1} &\geq& \Big[\frac{\eta [t_k (t_k - r) - (t_{k+1} - 2r)(t_{k+1} - r)]}{2} - \frac{\nu(t_{k} - \frac{r-1}{2})(t_k - 1)}{t_k} \Big] \norms{\Gc_{\eta}y^k}^2  \vspace{1ex}\\
			&& + {~} r(r-1)\iprods{\Gc_{\eta}y^k, y^k - \xopt}  +  \frac{(2\bar{\beta} - \eta) t_{k+1}(t_{k+1} - r)}{2} \norms{\Gc_{\eta}y^{k+1} - \Gc_{\eta}y^k}^2 \vspace{1ex}\\
			&& + {~} \frac{r^2(r-1)}{2\nu t_k (t_k - 1)} \norms{z^{k+1} - \xopt}^2.
		\end{array}
	\end{equation}
	To guarantee $2\bar{\beta} - \eta > 0$, we require $\eta < 2\bar{\beta} = \frac{\eta (4 - L\eta) - 4\rho}{2(1 - L\rho)}$, which is equivalent to $2\rho < \eta < \frac{2}{L}$ as stated in \eqref{eq:FKM2_key_est1_params}.
	Since $(2\rho, \frac{2}{L}) \subset \left(\frac{2(1 - \sqrt{1 - L\rho})}{L}, \frac{2(1 + \sqrt{1 - L\rho})}{L}\right)$, $\eta$ also satisfies the condition of Lemma~\ref{le:FB_cocoercive}, which guarantee the $\bar{\beta}$-co-coercivity of $\Gc_{\eta}$.
	Moreover, we have
	\begin{equation*}
		\arraycolsep=0.2em
		\begin{array}{lcl}
			\hat{a}_k &:=& \frac{\eta [t_k (t_k - r) - (t_{k+1} - 2r)(t_{k+1} - r)]}{2} - \frac{\nu(t_{k} - \frac{r-1}{2})(t_k - 1)}{t_k} \vspace{1ex}\\
			&=& [(r-1)\eta - \nu]t_k - \frac{(r-1)(2r-1)\eta}{2} + \frac{\nu}{2}\left(r+1 - \frac{r-1}{t_k} \right) \vspace{1ex}\\
			&\geq& \frac{(r-1)\eta - \nu}{2}t_k,
		\end{array}
	\end{equation*}
	due to $0 \leq \nu < (r-1)\eta$ and $t_k \geq t_0 \geq \max\set{\frac{(r-1)(2r-1)\eta}{(r-1)\eta - \nu}, \frac{r-1}{r+1}}$.
	Using this bound and the relation $\iprods{\Gc_{\eta}y^k, y^k - \xopt} \geq \bar{\beta}\norms{\Gc_{\eta}y^k}^2$ from the $\bar{\beta}$-co-coercivity of $\Gc_{\eta}$, we obtain  from \eqref{eq:FKM2_key_est1_proof2} the desired estimate \eqref{eq:FKM2_key_est1}.
\Eproof	

\beforesubsec
\subsection{Proof of Lemma~\ref{le:FKM2_key_est2} --- Lower bound of Lyapunov function}\label{apdx:le:FKM2_key_est2}
\aftersubsec
	By $\nu_k := \frac{\nu(t_k - 1)}{t_k}$ and the $\bar{\beta}$-co-coercivity of $\Gc_{\eta}$, we can show from \eqref{eq:FKM2_Lyapunov_function} that
	\begin{equation*} 
		\arraycolsep=0.2em
		\begin{array}{lcl}
			\Lc_k &:=& \frac{\eta a_k}{2} \norms{\Gc_{\eta}y^k}^2 + rt_k \iprods{\Gc_{\eta}y^k, y^k - z^k} + \frac{r^2(r-1)t_k}{2\nu (t_k - 1)} \norms{z^k - \xopt}^2 \vspace{1ex}\\
			&=& \frac{\eta(2a_k - t_k^2)}{4}\norms{\Gc_{\eta}y^k}^2 + rt_k\iprods{\Gc_{\eta}y^k, y^k - \xopt} + \frac{r^2[(r-1)\eta t_k - 2\nu(t_k - 1)]}{2\nu\eta(t_k - 1)}  \norms{z^k - \xopt}^2 \vspace{1ex}\\
			&& + {~} \frac{1}{4\eta}\norms{2r(z^k - \xopt) - \eta t_k \Gc_{\eta}y^k}^2 \vspace{1ex}\\
			&\geq& \frac{\eta(2a_k - t_k^2) + 4r\bar{\beta}t_k}{4}\norms{\Gc_{\eta}y^k}^2 + \frac{r^2[((r-1)\eta  - 2\nu) t_k + 2\nu]}{2\nu\eta(t_k - 1)} \norms{z^k - \xopt}^2.
		\end{array}
	\end{equation*}
	Let $A_k := \eta(2a_k - t_k^2) + 4r\bar{\beta}t_k$.
	Then, $A_k = \eta t_k (t_k - 2r) + 4r\bar{\beta}t_k = \eta t_k^2 + 2r (2\bar{\beta} - \eta)t_k \geq \eta t_k^2$ due to $2\bar{\beta} - \eta > 0$.
	Using this bound, we obtain \eqref{eq:FKM2_key_est2} from the last inequality.
\Eproof	

\beforesec
\section{Proof of Technical Lemmas in Section~\ref{sec:ND_DFFP_alg}}\label{apdx:sec:ND_DFFP_alg}
\aftersec
This appendix provides the full proofs of technical results in Section~\ref{sec:ND_DFFP_alg}.

\beforesubsec
\subsection{Proof of Lemma~\ref{le:DGE_assumption} --- The monotone and Lipschitz continuous case}\label{apdx:le:DGE_assumption}
\aftersubsec
The proof of this lemma follows from standard arguments in primal-dual methods. However, due to the different structures of the operators involved, we provide a complete proof here.

\noindent$\mathrm{(i)}$~
	By the $L$-Lipschitz continuity of all $G_i$, for any $\mbf{u} = (u_{[1]}, \cdots, u_{[n]})$ and $\hat{\mbf{u}} = (\hat{u}_{[1]}, \cdots, \hat{u}_{[n]})$ in $\R^{np}$, we have 
	\begin{equation*}
		\arraycolsep=0.2em
		\begin{array}{lcl}	
			\norms{\mbf{G}\mbf{u} - \mbf{G}\hat{\mbf{u}}}^2 = \sum_{i=1}^n \norms{G_i u_{[i]} - G_i \hat{u}_{[i]}}^2 \leq L^2 \sum_{i=1}^n \norms{u_{[i]} - \hat{u}_{[i]}}^2 = L^2 \norms{\mbf{u} - \hat{\mbf{u}}}^2.
		\end{array}
	\end{equation*}
	Thus, $\mbf{G}$ is also $L$-Lipschitz continuous.
	
	Next, for any $\mbf{x} := [\mbf{u}^\top, \mbf{v}^\top]^\top$ and $\hat{\mbf{x}} := [\hat{\mbf{u}}^\top, \hat{\mbf{v}}^\top]^\top$, from the definition of $\mbb{G}$ in \eqref{eq:DFKM_components}, we have
	\begin{equation*}
		\arraycolsep=0.2em
		\begin{array}{lcl}	
			\norms{\mbb{G}\mbf{x} - \mbb{G}\hat{\mbf{x}}}^2 = \norms{\mbf{G}\mbf{u} - \mbf{G}\hat{\mbf{u}}}^2 \leq L^2\norms{\mbf{u} - \hat{\mbf{u}}}^2.
		\end{array}
	\end{equation*}
	Since $\norms{\mbf{u} - \hat{\mbf{u}}}^2 \leq \norms{\mbf{u} - \hat{\mbf{u}}}^2 + \norms{\mbf{v} - \hat{\mbf{v}}}^2 = \norms{\mbf{x} - \hat{\mbf{x}}}^2$, we obtain from the last inequality that
	\begin{equation*}
		\arraycolsep=0.2em
		\begin{array}{lcl}	
			\norms{\mbb{G}\mbf{x}_1 - \mbb{G}\mbf{x}_2}^2 \leq L^2 \norms{\mbf{x}_1 - \mbf{x}_2}^2.
		\end{array}
	\end{equation*}
	This shows that $\mbb{G}$ is also $L$-Lipschitz continuous.
	
	Now, on the one hand, exploiting the structure of $\mbb{G}$ and applying the standard formula for the inverse of a $2\times2$ block matrix to $\mbb{P}^{-1}$, we can show that 
	\begin{equation*}
		\arraycolsep=0.2em
		\begin{array}{lcl}	
			\norms{\mbb{P}^{-1}\mbb{G}\mbf{x} - \mbb{P}^{-1}\mbb{G}\hat{\mbf{x}}}_{\mbb{P}}^2 &=& \iprods{\mbb{P}^{-1}\mbb{G}\mbf{x} - \mbb{P}^{-1}\mbb{G}\hat{\mbf{x}}, \mbb{P}^{-1}\mbb{G}\mbf{x} - \mbb{P}^{-1}\mbb{G}\hat{\mbf{x}}}_{\mbb{P}} \vspace{1ex}\\
			&=& \iprods{\mbb{G}\mbf{x} - \mbb{G}\hat{\mbf{x}}, \mbb{P}^{-1}(\mbb{G}\mbf{x} - \mbb{G}\hat{\mbf{x}})} \vspace{1ex}\\
			&=& \tau \iprods{\mbf{G}\mbf{u} - \mbf{G}\hat{\mbf{u}}, (\Id - \tau \sigma \mbf{K}^\top \mbf{K})^{-1}(\mbf{G}\mbf{u} - \mbf{G}\hat{\mbf{u}})} \vspace{1ex}\\
			&\leq& \frac{\tau}{1 - \tau \sigma \norms{\mbf{K}}^2} \norms{\mbf{G}\mbf{u} - \mbf{G}\hat{\mbf{u}}}^2 \vspace{1ex}\\
			&\leq& \frac{\tau L^2}{1 - \tau \sigma \norms{\mbf{K}}^2} \norms{\mbf{u} - \hat{\mbf{u}}}^2.
		\end{array}
	\end{equation*}
	On the other hand, using $\mbb{P}$, Young's inequality, and the Cauchy-Schwarz inequality, we have
	\begin{equation*}
		\arraycolsep=0.2em
		\begin{array}{lcl}	
			\norms{\mbf{x} - \hat{\mbf{x}}}_{\mbb{P}}^2 &=& \iprods{\mbb{P}(\mbf{x} - \hat{\mbf{x}}), \mbf{x} - \hat{\mbf{x}}} \vspace{1ex}\\
			&=& \frac{1}{\tau}\norms{\mbf{u} - \hat{\mbf{u}}}^2 + 2 \iprods{\mbf{K}(\mbf{u} - \hat{\mbf{u}}), \mbf{v} - \hat{\mbf{v}}} + \frac{1}{\sigma}\norms{\mbf{v} - \hat{\mbf{v}}}^2 \vspace{1ex}\\
			&\geq& \frac{1}{\tau}\norms{\mbf{u} - \hat{\mbf{u}}}^2 - \sigma \norms{\mbf{K}(\mbf{u} - \hat{\mbf{u}})}^2 - \frac{1}{\sigma}\norms{\mbf{v} - \hat{\mbf{v}}}^2 + \frac{1}{\sigma}\norms{\mbf{v} - \hat{\mbf{v}}}^2 \vspace{1ex}\\
			&\geq& \frac{1}{\tau}\norms{\mbf{u} - \hat{\mbf{u}}}^2 - \sigma \norms{\mbf{K}}^2\norms{\mbf{u} - \hat{\mbf{u}}}^2 \vspace{1ex}\\
			&=& \frac{1 - \tau \sigma \norms{\mbf{K}}^2}{\tau} \norms{\mbf{u} - \hat{\mbf{u}}}^2.
		\end{array}
	\end{equation*}
	Substituting this relation into the last inequality, we obtain
	\begin{equation*}
		\arraycolsep=0.2em
		\begin{array}{lcl}	
			\norms{\mbb{P}^{-1}\mbb{G}\mbf{x} - \mbb{P}^{-1}\mbb{G}\hat{\mbf{x}}}_{\mbb{P}}^2 &\leq& \frac{\tau^2 L^2}{(1 - \tau \sigma \norms{\mbf{K}}^2)^2} \norms{\mbf{x} - \hat{\mbf{x}}}_{\mbb{P}}^2.
		\end{array}
	\end{equation*}
	This expression shows that $\mbb{P}^{-1}\mbb{G}$ is $\tilde{L}$-Lipschitz continuous w.r.t. the weighted norm $\norms{\cdot}_{\mbb{P}}$, where $\tilde{L} := \frac{\tau L}{1 - \tau \sigma \norms{\mbf{K}}^2} > 0$, provided that $\tau\sigma\norms{K}^2 < 1$.
	
	\vspace{1ex}
	\noindent$\mathrm{(ii)}$~
	Since each $\Phi_i := G_i + T_i$ is maximally monotone, for $w_i \in \Phi_i(u_{[i]})$, $\hat{w}_i \in \Phi_i(\hat{u}_{[i]})$, we have
	\begin{equation*}
		\iprods{w_i - \hat{w}_i, u_{[i]} - \hat{u}_{[i]}} \geq 0, \qquad \forall i \in [n].
	\end{equation*} 
	Now, let $\mbf{w} = (w_1, \cdots, w_n) \in \mbf{G}\mbf{u} + \mbf{T}\mbf{u}$, $\hat{\mbf{w}} = (\hat{w}_1, \cdots, \hat{w}_n) \in \mbf{G}\hat{\mbf{u}} + \mbf{T}\hat{\mbf{u}}$.
	Then, we have
	\begin{equation*}
		\arraycolsep=0.2em
		\begin{array}{lcl}	
			\iprods{\mbf{w} - \hat{\mbf{w}}, \mbf{u} - \hat{\mbf{u}}} &=& \sum_{i=1}^n \iprods{w_i - \hat{w}_i, u_{[i]} - \hat{u}_{[i]}} \geq 0.
		\end{array}
	\end{equation*}
	Thus, $\mbf{G} + \mbf{T}$ is monotone in $\R^{n \times p}$.
	Moreover, since each $\Phi_i$ is maximally monotone, we have $\range{\Id + \Phi_i} = \R^p$ for all $i$.
	Consequently, we can easily show that $\range{\Id + \mbf{G} + \mbf{T}} = \R^{n\times p}$.
	Therefore, Minty's theorem gives that $\mbf{G} + \mbf{T}$ is maximally monotone.
	
	Next, in the primal-dual space $\R^{2n \times p}$, for $\mbf{x} := [\mbf{u}^\top, \mbf{v}^\top]^\top \in \R^{2n \times p}$, we have 
	$$\mbb{G}(\mbf{x}) = \begin{bmatrix}
		\mbf{G}\mbf{u} \\ 0
	\end{bmatrix}, \qquad \mbb{T}(\mbf{x}) = \begin{bmatrix}
		\mbf{T}\mbf{u} + \mbf{K}^\top \mbf{v} \\ -\mbf{K}\mbf{u}
	\end{bmatrix}.$$
	Let $\mbf{x} := [\mbf{u}^\top, \mbf{v}^\top]^\top$, $\hat{\mbf{x}} := [\hat{\mbf{u}}^\top, \hat{\mbf{v}}^\top]^\top$, and take $\mathbbm{w} \in \mbb{G}(\mbf{x}) + \mbb{T}(\mbf{x})$, $\hat{\mathbbm{w}} \in \mbb{G}(\hat{\mbf{x}}) + \mbb{T}(\hat{\mbf{x}})$.
	If we denote $\Delta \mathbbm{w} := \mathbbm{w} - \hat{\mathbbm{w}}$, $\Delta \mbf{u} := \mbf{u} - \hat{\mbf{u}}$, $\Delta \mbf{v} := \mbf{v} - \hat{\mbf{v}}$, and $\Delta \mbf{w} := \mbf{w} - \hat{\mbf{w}}$, then  we can compute 
	\begin{equation*}
		\arraycolsep=0.2em
		\begin{array}{lcl}	
			\Delta \mathbbm{w} = \mathbbm{w} - \hat{\mathbbm{w}} = \begin{bmatrix}
				\Delta \mbf{w} + \mbf{K}^\top \Delta \mbf{v}\\
				-\mbf{K}\Delta \mbf{u}
			\end{bmatrix}.
		\end{array}
	\end{equation*}
	Using the maximal monotonicity of $\mbf{G} + \mbf{T}$ in $\R^{n \times p}$, we have
	\begin{equation*}
		\arraycolsep=0.2em
		\begin{array}{lcl}	
			\iprods{\mathbbm{w} - \hat{\mathbbm{w}}, \mbf{x} - \hat{\mbf{x}}} &=& \iprods{\Delta \mbf{w} + \mbf{K}^\top \Delta \mbf{v}, \Delta \mbf{u}} + \iprods{-\mbf{K}\Delta \mbf{u}, \Delta \mbf{v}} = \iprods{\Delta \mbf{w}_u, \Delta \mbf{u}} \geq 0.
		\end{array}
	\end{equation*}
	Hence, $\mbb{G} + \mbb{T}$ is monotone.
	The maximality of $\mbb{G} + \mbb{T}$ can be obtained by using the maximality of $\mbf{G} + \mbf{T}$.
	Thus, we conclude that $\mbb{G} + \mbb{T}$ is maximally monotone in $\R^{2n \times p}$.
	
	Finally, the last inequality also implies that
	\begin{equation*}
		\arraycolsep=0.2em
		\begin{array}{lcl}	
			\iprods{\mbb{P}^{-1}\mathbbm{w} - \mbb{P}^{-1}\hat{\mathbbm{w}}, \mbf{x} - \hat{\mbf{x}}}_{\mbb{P}} &=& \iprods{\mbb{P}(\mbb{P}^{-1}\mathbbm{w} - \mbb{P}^{-1}\hat{\mathbbm{w}}), \mbf{x} - \hat{\mbf{x}}} = \iprods{\mathbbm{w} - \hat{\mathbbm{w}}, \mbf{x} - \hat{\mbf{x}}} \geq 0.
		\end{array}
	\end{equation*}
	Moreover, since $P^{-1}$ is symmetric and positive definite, using \cite[Lemma 3.7$\mathrm{(i)}$]{combettes2014variable} and the last inequality, we conclude that $\mbb{P}^{-1}(\mbb{G} + \mbb{T})$ is maximally monotone with respect to $\iprods{\cdot, \cdot}_{\mbb{P}}$.
\Eproof	

\beforesubsec
\subsection{Proof of Lemma~\ref{le:DGE_cocoercive} --- The co-coercive case}\label{apdx:le:DGE_cocoercive}
\aftersubsec
	\vspace{0.75ex}
	\noindent
	$\mathrm{(i)}$~
	Since each $G_i$ is $\frac{1}{L}$-co-coercive, $\mbf{G}$ is also $\frac{1}{L}$-co-coercive.
	Thus, we can evaluate that
	\begin{equation*}
		\arraycolsep=0.2em
		\begin{array}{lcl}	
			\iprods{\mbb{P}^{-1}\mbb{G}\mbf{x}_1 - \mbb{P}^{-1}\mbb{G}\mbf{x}_2, \mbf{x}_1 - \mbf{x}_2}_{\mbb{P}} &=& \iprods{\mbb{P}(\mbb{P}^{-1}\mbb{G}\mbf{x}_1 - \mbb{P}^{-1}\mbb{G}\mbf{x}_2), \mbf{x}_1 - \mbf{x}_2} \vspace{1ex}\\
			&=& \iprods{\mbb{G}\mbf{x}_1 - \mbb{G}\mbf{x}_2, \mbf{x}_1 - \mbf{x}_2} \vspace{1ex}\\
			&=& \iprods{\mbf{G}\mbf{u}_1 - \mbf{G}\mbf{u}_2, \mbf{u}_1 - \mbf{u}_2} \vspace{1ex}\\
			&\geq& \frac{1}{L} \norms{\mbf{G}\mbf{u}_1 - \mbf{G}\mbf{u}_2}^2.
		\end{array}
	\end{equation*}
	Next, from Lemma~\ref{le:DGE_assumption}, we have proved that
	\begin{equation*}
		\arraycolsep=0.2em
		\begin{array}{lcl}	
			\norms{\mbb{P}^{-1}\mbb{G}\mbf{x}_1 - \mbb{P}^{-1}\mbb{G}\mbf{x}_2}_{\mbb{P}}^2 &\leq& \frac{\tau}{1 - \tau \sigma \norms{\mbf{K}}^2} \norms{\mbf{G}\mbf{u}_1 - \mbf{G}\mbf{u}_2}^2.
		\end{array}
	\end{equation*}
	Combining these relations, we obtain
	\begin{equation*}
		\arraycolsep=0.2em
		\begin{array}{lcl}	
			\iprods{\mbb{P}^{-1}\mbb{G}\mbf{x}_1 - \mbb{P}^{-1}\mbb{G}\mbf{x}_2, \mbf{x}_1 - \mbf{x}_2}_{\mbb{P}} &\geq& \frac{1 - \tau \sigma \norms{\mbf{K}}^2}{L\tau} \norms{\mbb{P}^{-1}\mbb{G}\mbf{x}_1 - \mbb{P}^{-1}\mbb{G}\mbf{x}_2}_{\mbb{P}}^2,
		\end{array}
	\end{equation*}
	proving the $\frac{1}{\tilde{L}}$-co-coercivity of $\mbb{P}^{-1}\mbb{G}$ with  $\tilde{L} := \frac{\tau L}{1 - \tau \sigma \norms{\mbf{K}}^2}$.
	
	\vspace{0.75ex}
	\noindent
	\noindent$\mathrm{(ii)}$
	This can be proved similarly as in Lemma~\ref{le:DGE_assumption}(ii), and thus we omit.
\Eproof	

\beforesec
\section{Supporting Results for Section~\ref{sec:NI-DFKM}}\label{apdx:sec:NI_DFKM}
\aftersec
This section presents the missing proofs of  technical results used in Section~\ref{sec:NI-DFKM}.

\beforesubsec
\subsection{Supporting results for the derivation of \eqref{eq:DFKM_NI_scheme}}\label{apdx:subsec:DFKM_NI_scheme}
\aftersubsec
\begin{lemma}\label{le:NI_DFFP_derivation2}
	Let $\sets{(\hat{v}^k, z^k)}$ be generated by \eqref{eq:DFKM_NI_scheme}.
	Then, we have $\hat{v}^k, z_v^k \in \mathrm{Im}(\mbf{E}^{\top})$ and $\kappa^{-1}\mbf{E}^\top\mbf{E}\hat{v}^k = \hat{v}^k$ for all $k \geq 0$.
\end{lemma}

\begin{proof}
	In \eqref{eq:FKM_3OP_scheme}, let us denote $z_v^{k+1} = z_v^k - \frac{\nu_k}{r(1 - \tau\sigma)}Q_{k+1}$, where $Q_{k+1} := \tau \sigma \tilde{w}_u^{k+1} + \sigma \tilde{w}_v^{k+1}$.
	From the seventh and eighth lines of \eqref{eq:FKM_3OP_scheme}, we can derive that
	\begin{equation*}
		\arraycolsep=0.2em
		\begin{array}{lcl}
			Q_{k+1} &=& (1 - \tau \sigma)(\hat{v}^k - v^{k+1}) + \beta_k (\tau \sigma \tilde{w}_u^k + \sigma \tilde{w}_v^k) = (1 - \tau \sigma)(\hat{v}^k - v^{k+1}) + \beta_k Q_k.
		\end{array}
	\end{equation*}
	We will prove that $\hat{v}^k, z_v^k, Q_k \in \mathrm{Im}(\mbf{E}^{\top})$ for all $k \geq 0$ by induction.
	\begin{compactitem}
		\item \textit{The base step}: 
		From the initialization of \eqref{eq:FKM_scheme}, we have $z_v^0 = v^0$, and thus, $\hat{v}^0 = \frac{t_0 - r}{t_0}v^0 + \frac{r}{t_0}z_v^0 = v^0$.
		Let us choose $v^0 = \frac{\sigma}{\kappa}\mbf{E}^\top \mbf{E}\Big(u^0 - \tau \begin{bmatrix}
			\boldsymbol{\Gamma}\mbf{h}^0 \\ 0
		\end{bmatrix}\Big)$, then $v^0 \in \mathrm{Im}(\mbf{E}^\top)$, and thus $\hat{v}^0 = z_v^0 \in \mathrm{Im}(\mbf{E}^\top)$.
		Next, let $c^0 := u^0 - \tau \begin{bmatrix}
			\boldsymbol{\Gamma}\mbf{h}^0 \\ 0
		\end{bmatrix} - \frac{1}{\sigma}v^0$.
		Then, we can compute that
		\begin{equation*}
		\arraycolsep=0.2em
		\begin{array}{lcl}
			\mbf{E}c^0 = \mbf{E}\left(u^0 - \tau \begin{bmatrix}
				\boldsymbol{\Gamma}\mbf{h}^0 \\ 0
			\end{bmatrix}\right) - \frac{1}{\sigma}\mbf{E}v^0
			= \mbf{E}\left(u^0 - \tau \begin{bmatrix}
				\boldsymbol{\Gamma}\mbf{h}^0 \\ 0
			\end{bmatrix}\right) - \frac{1}{\kappa}\mbf{E}\mbf{E}^\top \mbf{E}\left(u^0 - \tau \begin{bmatrix}
			\boldsymbol{\Gamma}\mbf{h}^0 \\ 0
			\end{bmatrix}\right) = 0.
		\end{array}
		\end{equation*}
		Therefore, $v^0 \in Cc^0$, hence we can choose $\tilde{w}_u^0 = \begin{bmatrix}
			\boldsymbol{\Gamma}\mbf{h}^0 \\ 0
		\end{bmatrix} + v^0$ and $\tilde{w}_v^0 = c^0 - u^0 = -\tau \begin{bmatrix}
		\boldsymbol{\Gamma}\mbf{h}^0 \\ 0
		\end{bmatrix} - \frac{1}{\sigma}v^0$.
		Thus, we can compute that 
		\begin{equation*}
		\arraycolsep=0.2em
		\begin{array}{lcl}
			Q_0 &=& \tau\sigma\tilde{w}_u^0 + \sigma\tilde{w}_v^0 = \tau\sigma\left(\begin{bmatrix}
				\boldsymbol{\Gamma}\mbf{h}^0 \\ 0
			\end{bmatrix} + v^0\right) + \sigma\left(-\tau \begin{bmatrix}
			\boldsymbol{\Gamma}\mbf{h}^0 \\ 0
			\end{bmatrix} - \frac{1}{\sigma}v^0\right) = -(1 - \tau\sigma)v^0.
		\end{array}
		\end{equation*}
		Since $v^0 \in \mathrm{Im}(\mbf{E}^\top)$, the last expression implies that $Q_0 \in \mathrm{Im}(\mbf{E}^\top)$.
		
		\item \textit{The induction step}: 
		Since $v^{k+1} = \kappa^{-1}\mbf{E}^{\top}\mbf{E}\tilde{v}^{k+1} = \mbf{E}^{\top}\left(\kappa^{-1}\mbf{E}\tilde{v}^{k+1}\right)$, we have $v^{k+1} \in \mathrm{Im}(\mbf{E}^{\top})$.
		From the inductive hypothesis, we already have $\hat{v}^k, z_v^k, Q_k \in \mathrm{Im}(\mbf{E}^{\top})$.
		Thus, we conclude that $Q_{k+1} = (1 - \tau \sigma)(\hat{v}^k - v^{k+1}) + \beta_k Q_k \in \mathrm{Im}(\mbf{E}^{\top})$.
		Using this fact, we also get $z_v^{k+1} = z_v^k - \frac{\nu_k}{r(1 - \tau \sigma)}Q_{k+1} \in \mathrm{Im}(\mbf{E}^{\top})$, and $\hat{v}^{k+1} = \frac{t_{k+1} - r}{t_{k+1}} v^{k+1} + \frac{r}{t_{k+1}}z_v^{k+1} \in \mathrm{Im}(\mbf{E}^{\top})$.
	\end{compactitem}
	Finally, since $\hat{v}^k \in \mathrm{Im}(\mbf{E}^{\top})$, there exists $\xi^k$ such that $\hat{v}^k = \mbf{E}^{\top}\xi^k$.
	Then, using the fact that $\mbf{E}\mbf{E}^{\top} = \kappa\Id$, we have $\kappa^{-1}\mbf{E}^{\top}\mbf{E}\hat{v}^k = \kappa^{-1}\mbf{E}^{\top}\mbf{E}\mbf{E}^{\top}\xi^k = \kappa^{-1}\mbf{E}^{\top}(\kappa\Id)\xi^k = \mbf{E}^{\top}\xi^k = \hat{v}^k$.
\end{proof}

\begin{lemma}\label{le:NI_DFFP_derivation3}
	In the scheme \eqref{eq:DFKM_NI_scheme_init}, we always have $\breve{\mbf{a}}^k = 0$, $\hat{\breve{\mbf{a}}}^k = 0$, and $z_{\breve{\mbf{a}}}^k = 0$, for all $k \geq 0$.
\end{lemma}

\begin{proof}
	Denote by $z_u^{k+1} = z_u^k - \frac{\nu_k}{r(1 - \tau\sigma)}R_{k+1}$ from the scheme \eqref{eq:FKM_3OP_scheme}, where $R_{k+1} := \tau \tilde{w}_u^{k+1} + \tau \sigma \tilde{w}_v^{k+1}$.
	From  the seventh and eighth lines of \eqref{eq:FKM_3OP_scheme}, we can show that
	\begin{equation}\label{eq:NI_DFFP_derivation3_proof1}
	\arraycolsep=0.2em
	\begin{array}{lcl}
		R_{k+1}  &=& (1 - \tau \sigma)(\hat{u}^k - u^{k+1}) + \beta_k (\tau \tilde{w}_u^k + \tau \sigma \tilde{w}_v^k) 
		= (1 - \tau \sigma)(\hat{u}^k - u^{k+1}) + \beta_k R_k.
	\end{array}
	\end{equation}
	Let $u^k = (\mbf{a}^k, \breve{\mbf{a}}^k)$  and other notations be  as in Subsection~\ref{subsec:NI_DFFP_derivation}.
	Let $\utilde{R}_{k+1}$ be the second (fictitious) block of $R_k$ corresponding to $\breve{\mbf{a}}^k$.
	We will prove by induction that $\breve{\mbf{a}}^k = 0$, $\hat{\breve{\mbf{a}}}^k = 0$, $z_{\breve{\mbf{a}}}^k = 0$, and $\utilde{R}_k = 0$ for all $k \geq 0$.
	Indeed, we proceed as follows.
	\begin{compactitem}
		\item \textit{The base step:} 
		Since we initialize $\breve{\mbf{a}}^0 = 0$ and $z_{\breve{\mbf{a}}}^0 = 0$, we have $\hat{\breve{\mbf{a}}}^0 = \frac{t_0 - r}{t_0}\breve{\mbf{a}}^0 + \frac{r}{t_0}z_{\breve{\mbf{a}}}^0 = 0$.
		Moreover, from the base step of Lemma~\ref{le:NI_DFFP_derivation2}, we can compute that
		\begin{equation*}
		\arraycolsep=0.2em
		\begin{array}{lcl}
			R_0 &=& \tau\tilde{w}_u^0 + \tau\sigma\tilde{w}_v^0 = \tau\left(\begin{bmatrix}
				\boldsymbol{\Gamma}\mbf{h}^0 \\ 0
			\end{bmatrix} + v^0\right) + \tau\sigma\left(-\tau \begin{bmatrix}
				\boldsymbol{\Gamma}\mbf{h}^0 \\ 0
			\end{bmatrix} - \frac{1}{\sigma}v^0\right) = \tau(1 - \tau\sigma)\begin{bmatrix}
			\boldsymbol{\Gamma}\mbf{h}^0 \\ 0
			\end{bmatrix}.
		\end{array}
		\end{equation*}
		Thus, the second block $\utilde{R}_0 = 0$.
		\item \textit{The induction step:} 
		Assume that for $k \geq 0$, we have $\breve{\mbf{a}}^k = 0$, $\hat{\breve{\mbf{a}}}^k = 0$, $z_{\breve{\mbf{a}}}^k = 0$, and $\utilde{R}_k = 0$.
		From \textbf{Step 3} of Subsection~\ref{subsec:NI_DFFP_derivation}, we already know that $\breve{\mbf{a}}^{k+1} = 0$.
		Next, we can compute the fictitious block of \eqref{eq:NI_DFFP_derivation3_proof1} as $\utilde{R}_{k+1} = (1 - \tau \sigma)(\hat{\breve{\mbf{a}}}^k - \breve{\mbf{a}}^{k+1}) + \beta_k \utilde{R}_k = 0$.
		Moreover, we also get $z_{\breve{\mbf{a}}}^{k+1} = z_{\breve{\mbf{a}}}^a - \frac{\nu_k}{r(1 - \tau\sigma)}\utilde{R}_{k+1} = 0$ and $\hat{\breve{\mbf{a}}}^{k+1} = \frac{t_{k+1} - r}{t_{k+1}}\breve{\mbf{a}}^{k+1} + \frac{r}{t_{k+1}}z_{\breve{\mbf{a}}}^{k+1} = 0$.
	\end{compactitem}
	By induction, we complete the proof.
\end{proof}

\beforesubsec
\subsection{The initial form of \eqref{eq:DFKM_NI_scheme} and its simplification}\label{apdx:DFKM_NI_scheme_init}
\aftersubsec
After gathering all the derivation steps in Subsection~\ref{subsec:NI_DFFP_derivation}, we first come up with the following detailed scheme, which requires a further simplification:
\begin{equation}\label{eq:DFKM_NI_scheme_init}
	\arraycolsep=0.2em
	\left\{
	\begin{array}{lcl}
		\hat{\mbf{a}}^k &=& \frac{t_k - r}{t_k} \mbf{a}^k + \frac{r}{t_k} z_{\mbf{a}}^k, \vspace{1ex}\\
		\hat{\mbf{b}}^k &=& \frac{t_k - r}{t_k} \mbf{b}^k + \frac{r}{t_k} z_{\mbf{b}}^k, \vspace{1ex}\\
		\hat{\breve{\mbf{b}}}^k &=& \frac{t_k - r}{t_k} \breve{\mbf{b}}^k + \frac{r}{t_k} z_{\breve{\mbf{b}}}^k, \vspace{1ex}\\
		y_{\mbf{a}}^k &=& \hat{\mbf{a}}^k - \frac{\gamma_k - \beta_k}{1 - \tau \sigma} (\tau \tilde{w}_{\mbf{a}}^k + \tau \sigma \tilde{w}_{\mbf{b}}^k), \vspace{1ex}\\
		y_{\breve{\mbf{a}}}^k &=& \hat{\breve{\mbf{a}}}^k - \frac{\gamma_k - \beta_k}{1 - \tau \sigma} (\tau \tilde{w}_{\breve{\mbf{a}}}^k + \tau \sigma \tilde{w}_{\breve{\mbf{b}}}^k), \vspace{1ex}\\
		\mbf{a}^{k+1} &=& J_{\tau \mbf{\Gamma} T \mbf{\Gamma} } (\hat{\mbf{a}}^k - \tau \hat{\mbf{b}}^k - \tau \mbf{\Gamma} G (\mbf{\Gamma} y_{\mbf{a}}^k) + \tau \beta_k \tilde{w}_{\mbf{a}}^k), \vspace{1ex}\\
		\tilde{\mathbf{b}}^{k+1} &=& \hat{\mathbf{b}}^k + \sigma(2\mathbf{a}^{k+1} - \hat{\mathbf{a}}^k) + \sigma\beta_k \tilde{w}_{\mathbf{b}}^k, \vspace{1ex} \\
		\tilde{\breve{\mbf{b}}}^{k+1} &=& \hat{\breve{\mbf{b}}}^k + \sigma(2\breve{\mbf{a}}^{k+1} - \hat{\breve{\mbf{a}}}^k) + \sigma\beta_k \tilde{w}_{\breve{\mbf{b}}}^k, \vspace{1ex}\\
		\mbf{b}^{k+1} &=& \hat{\mbf{b}}^k + \kappa^{-1}\sigma \mbf{\Gamma} \mbf{K}^2 \mbf{\Gamma} (2\mbf{a}^{k+1} - \hat{\mbf{a}}^k - \beta_k \mbf{a}^k), \vspace{1ex}\\
		\breve{\mbf{b}}^{k+1} &=& \utilde{\hat{\mbf{b}}}^k + \kappa^{-1}\sigma \mbf{M} \mbf{K} \mbf{\Gamma} (2\mbf{a}^{k+1} - \hat{\mbf{a}}^k - \beta_k \mbf{a}^k), \vspace{1ex}\\
		\tilde{w}_{\mbf{a}}^{k+1} &=& \frac{1}{\tau}(\hat{\mbf{a}}^k - \mbf{a}^{k+1}) - (\hat{\mbf{b}}^k - \mbf{b}^{k+1}) + \beta_k \tilde{w}_{\mbf{a}}^k, \vspace{1ex}\\
		\tilde{w}_{\breve{\mbf{a}}}^{k+1} &=& \frac{1}{\tau}(\hat{\breve{\mbf{a}}}^k - \breve{\mbf{a}}^{k+1}) - (\hat{\breve{\mbf{b}}}^k - \breve{\mbf{b}}^{k+1}) + \beta_k \tilde{w}_{\breve{\mbf{a}}}^k, \vspace{1ex}\\
		\tilde{w}_{\mbf{b}}^{k+1} &=& -(\hat{\mbf{a}}^k - \mbf{a}^{k+1}) + \frac{1}{\sigma}(\hat{\mbf{b}}^k - \mbf{b}^{k+1}) + \beta_k \tilde{w}_{\mbf{b}}^k, \vspace{1ex}\\
		\tilde{w}_{\breve{\mbf{b}}}^{k+1} &=& -(\hat{\breve{\mbf{a}}}^k - \breve{\mbf{a}}^{k+1}) + \frac{1}{\sigma}(\hat{\utilde{\bf{b}}}^k - \breve{\mbf{b}}^{k+1}) + \beta_k \tilde{w}_{\breve{\mbf{b}}}^k, \vspace{1ex}\\
		z_{\mbf{a}}^{k+1} &=& z_{\mbf{a}}^k - \frac{\nu_k}{r(1-\tau\sigma)}(\tau\tilde{w}_{\mbf{a}}^{k+1} + \tau\sigma\tilde{w}_{\mbf{b}}^{k+1}), \vspace{1ex}\\ 
		z_{\mbf{b}}^{k+1} &=& z_{\mbf{b}}^k - \frac{\nu_k}{r(1-\tau\sigma)}(\tau\sigma\tilde{w}_{\mbf{a}}^{k+1} + \sigma\tilde{w}_{\mbf{b}}^{k+1}), \vspace{1ex}\\ 
		z_{\breve{\mbf{b}}}^{k+1} &=& z_{\breve{\mbf{b}}}^k - \frac{\nu_k}{r(1-\tau\sigma)}(\tau\sigma\tilde{w}_{\breve{\mbf{a}}}^{k+1} + \sigma\tilde{w}_{\breve{\mbf{b}}}^{k+1}).
	\end{array}
	\right.
\end{equation}
where $\breve{\mbf{a}}^k = 0$, $\hat{\breve{\mbf{a}}}^k = 0$, and $z_{\breve{\mbf{a}}}^k = 0$ for all $k \geq 0$ by Lemma~\ref{le:NI_DFFP_derivation3}.

Now, we simplify this scheme further as follows.
Since our main primal variable is $\mbf{a}$, we can remove all the unnecessary variables and obtain the following scheme:
\begin{equation*}
	\arraycolsep=0.2em
	\left\{
	\begin{array}{lcl}
		\hat{\mbf{a}}^k &=& \frac{t_k - r}{t_k} \mbf{a}^k + \frac{r}{t_k} z_{\mbf{a}}^k, \vspace{1ex}\\
		\hat{\mbf{b}}^k &=& \frac{t_k - r}{t_k} \mbf{b}^k + \frac{r}{t_k} z_{\mbf{b}}^k, \vspace{1ex}\\
		y_{\mbf{a}}^k &=& \hat{\mbf{a}}^k - \frac{\gamma_k - \beta_k}{1 - \tau \sigma} (\tau \tilde{w}_{\mbf{a}}^k + \tau \sigma \tilde{w}_{\mbf{b}}^k), \vspace{1ex}\\			
		\mbf{a}^{k+1} &=& J_{\tau \mbf{\Gamma} T \mbf{\Gamma}} (\hat{\mbf{a}}^k - \tau \hat{\mbf{b}}^k - \tau \mbf{\Gamma} G (\mbf{\Gamma} y_{\mbf{a}}^k) + \tau \beta_k \tilde{w}_{\mbf{a}}^k), \vspace{1ex}\\
		\mbf{b}^{k+1} &=& \hat{\mbf{b}}^k + \kappa^{-1}\sigma \mbf{\Gamma} \mbf{K}^2 \mbf{\Gamma} (2\mbf{a}^{k+1} - \hat{\mbf{a}}^k - \beta_k \mbf{a}^k), \vspace{1ex}\\			
		\tilde{w}_{\mbf{a}}^{k+1} &=& \frac{1}{\tau}(\hat{\mbf{a}}^k - \mbf{a}^{k+1}) - (\hat{\mbf{b}}^k - \mbf{b}^{k+1}) + \beta_k \tilde{w}_{\mbf{a}}^k, \vspace{1ex}\\
		\tilde{w}_{\mbf{b}}^{k+1} &=& -(\hat{\mbf{a}}^k - \mbf{a}^{k+1}) + \frac{1}{\sigma}(\hat{\mbf{b}}^k - \mbf{b}^{k+1}) + \beta_k \tilde{w}_{\mbf{b}}^k, \vspace{1ex}\\			
		z_{\mbf{a}}^{k+1} &=& z_{\mbf{a}}^k - \frac{\nu_k}{r(1-\tau\sigma)}(\tau\tilde{w}_{\mbf{a}}^{k+1} + \tau\sigma\tilde{w}_{\mbf{b}}^{k+1}), \vspace{1ex}\\ 
		z_{\mbf{b}}^{k+1} &=& z_{\mbf{b}}^k - \frac{\nu_k}{r(1-\tau\sigma)}(\tau\sigma\tilde{w}_{\mbf{a}}^{k+1} + \sigma\tilde{w}_{\mbf{b}}^{k+1}).
	\end{array}
	\right.
\end{equation*}
Next, let us define the following transformations:
\begin{equation*}
	\arraycolsep=0.2em
	\begin{array}{c}
		\mbf{x}^k := \mbf{\Gamma} \mbf{a}^k, \quad
		\hat{\mbf{x}}^k := \mbf{\Gamma} \hat{\mbf{a}}^k, \quad
		\mbf{v}^k := \mbf{\Gamma} \mbf{b}^k, \quad 
		\hat{\mbf{v}}^k := \mbf{\Gamma} \hat{\mbf{b}}^k, \quad
		\mbf{t}^k := \mbf{\Gamma} y_{\mbf{a}}^k, \vspace{1ex}\\
		\mbf{s}_{\mbf{x}}^k := \mbf{\Gamma} z_{\mbf{a}}^k, \quad 
		\mbf{s}_{\mbf{v}}^k := \mbf{\Gamma} z_{\mbf{b}}^k, \quad
		\mbf{r}_{\mbf{x}}^k := \mbf{\Gamma} \tilde{w}_{\mbf{a}}^{k}, \quad \text{and} \quad
		\mbf{r}_{\mbf{v}}^k := \mbf{\Gamma} \tilde{w}_{\mbf{b}}^k.
	\end{array}
\end{equation*}
Utilizing these transformation and $\mbf{\Gamma} J_{\tau \mbf{\Gamma} T \mbf{\Gamma}}(\cdot) = J_{\tau \mbf{\Gamma}^2 T}(\mbf{\Gamma}(\cdot))$, the last scheme reduces to
\begin{equation*}
	\arraycolsep=0.2em
	\left\{
	\begin{array}{lcl}
		\hat{\mbf{x}}^k &=& \frac{t_k - r}{t_k} \mbf{x}^k + \frac{r}{t_k} \mbf{s}_{\mbf{x}}^k, \vspace{1ex}\\
		\hat{\mbf{v}}^k &=& \frac{t_k - r}{t_k} \mbf{v}^k + \frac{r}{t_k} \mbf{s}_{\mbf{v}}^k, \vspace{1ex}\\
		\mbf{t}^k &=& \hat{\mbf{x}}^k - \frac{\gamma_k - \beta_k}{1 - \tau \sigma} (\tau \mbf{r}_{\mbf{x}}^k + \tau \sigma \mbf{r}_{\mbf{v}}^k), \vspace{1ex}\\			
		\mbf{x}^{k+1} &=& J_{\tau \mbf{\Gamma}^2 T} (\hat{\mbf{x}}^k - \tau \hat{\mbf{v}}^k - \tau \mbf{\Gamma}^2 G\mbf{t}^k + \tau \beta_k \mbf{r}_{\mbf{x}}^k), \vspace{1ex}\\
		\mbf{v}^{k+1} &=& \hat{\mbf{v}}^k + \kappa^{-1}\sigma \mbf{\Gamma}^2 \mbf{K}^2 (2\mbf{x}^{k+1} - \hat{\mbf{x}}^k - \beta_k \mbf{x}^k), \vspace{1ex}\\			
		\mbf{r}_{\mbf{x}}^{k+1} &=& \frac{1}{\tau}(\hat{\mbf{x}}^k - \mbf{x}^{k+1}) - (\hat{\mbf{v}}^k - \mbf{v}^{k+1}) + \beta_k \mbf{r}_{\mbf{x}}^k, \vspace{1ex}\\
		\mbf{r}_{\mbf{v}}^{k+1} &=& -(\hat{\mbf{x}}^k - \mbf{x}^{k+1}) + \frac{1}{\sigma}(\hat{\mbf{v}}^k - \mbf{v}^{k+1}) + \beta_k \mbf{r}_{\mbf{v}}^k, \vspace{1ex}\\			
		\mbf{s}_{\mbf{x}}^{k+1} &=& \mbf{s}_{\mbf{x}}^k - \frac{\nu_k}{r(1-\tau\sigma)}(\tau\mbf{r}_{\mbf{x}}^{k+1} + \tau\sigma\mbf{r}_{\mbf{v}}^{k+1}), \vspace{1ex}\\ 
		\mbf{s}_{\mbf{v}}^{k+1} &=& \mbf{s}_{\mbf{v}}^k - \frac{\nu_k}{r(1-\tau\sigma)}(\tau\sigma \mbf{r}_{\mbf{x}}^{k+1} + \sigma\mbf{r}_{\mbf{v}}^{k+1}).
	\end{array}
	\right.
\end{equation*}
Finally, by the facts that $\mbf{\Gamma}^2 = \mbf{\Lambda} = \diag{\alpha_1, \cdots, \alpha_n}$ and $\mbf{K} = \left(\frac{\Id - W}{2}\right)^{1/2}$, we get
\begin{equation}\label{eq:DFKM_NI_scheme}  
\tag{NI-DFFP}
	\arraycolsep=0.2em
	\left\{
	\begin{array}{lcl}
		\hat{\mbf{x}}^k &=& \frac{t_k - r}{t_k} \mbf{x}^k + \frac{r}{t_k} \mbf{s}_{\mbf{x}}^k, \vspace{1ex}\\
		\hat{\mbf{v}}^k &=& \frac{t_k - r}{t_k} \mbf{v}^k + \frac{r}{t_k} \mbf{s}_{\mbf{v}}^k, \vspace{1ex}\\
		\mbf{t}^k &=& \hat{\mbf{x}}^k - \frac{\gamma_k - \beta_k}{1 - \tau \sigma} (\tau \mbf{r}_{\mbf{x}}^k + \tau \sigma \mbf{r}_{\mbf{v}}^k), \vspace{1ex}\\			
		\mbf{x}^{k+1} &=& J_{\tau \mbf{\Lambda} T} (\hat{\mbf{x}}^k - \tau \hat{\mbf{v}}^k - \tau \mbf{\Lambda} G\mbf{t}^k + \tau \beta_k \mbf{r}_{\mbf{x}}^k), \vspace{1ex}\\
		\mbf{v}^{k+1} &=& \hat{\mbf{v}}^k + \frac{\sigma}{2\kappa} \mbf{\Lambda} (\Id - W) (2\mbf{x}^{k+1} - \hat{\mbf{x}}^k - \beta_k \mbf{x}^k), \vspace{1ex}\\			
		\mbf{r}_{\mbf{x}}^{k+1} &=& \frac{1}{\tau}(\hat{\mbf{x}}^k - \mbf{x}^{k+1}) - (\hat{\mbf{v}}^k - \mbf{v}^{k+1}) + \beta_k \mbf{r}_{\mbf{x}}^k, \vspace{1ex}\\
		\mbf{r}_{\mbf{v}}^{k+1} &=& -(\hat{\mbf{x}}^k - \mbf{x}^{k+1}) + \frac{1}{\sigma}(\hat{\mbf{v}}^k - \mbf{v}^{k+1}) + \beta_k \mbf{r}_{\mbf{v}}^k, \vspace{1ex}\\			
		\mbf{s}_{\mbf{x}}^{k+1} &=& \mbf{s}_{\mbf{x}}^k - \frac{\nu_k}{r(1-\tau\sigma)}(\tau\mbf{r}_{\mbf{x}}^{k+1} + \tau\sigma\mbf{r}_{\mbf{v}}^{k+1}), \vspace{1ex}\\ 
		\mbf{s}_{\mbf{v}}^{k+1} &=& \mbf{s}_{\mbf{v}}^k - \frac{\nu_k}{r(1-\tau\sigma)}(\tau\sigma \mbf{r}_{\mbf{x}}^{k+1} + \sigma\mbf{r}_{\mbf{v}}^{k+1}).
	\end{array}
	\right.
\end{equation}
This scheme is exactly the one implemented in Algorithm~\ref{alg:NI_DFFP}.

\vspace{0.75ex}
\noindent\textbf{Initialization.}
The scheme \eqref{eq:DFKM_NI_scheme} is initialized as follows.
First, we choose $\mbf{x}^0 \in \dom{\mbf{T}}$ and $\boldsymbol{\xi}^0 \in \mbf{T}\mbf{x}^0$, and set $\mbf{w}_{\mbf{x}}^0 = \mbf{G}\mbf{x}^0 + \mbf{\xi}^0$.
Let $u^0 = \begin{bmatrix}
	\mbf{\Gamma}^{-1}\mbf{x}^0 \\ 0
\end{bmatrix}$ and $g^0 = \begin{bmatrix}
	\mbf{\Gamma} \mbf{w}_{\mbf{x}}^0 \\ 0
\end{bmatrix} \in (A+B)u^0$.
From Step 2 of Subsection 3.4.2, we need to guarantee that $\tilde{w}^0 = P\hat{w}^0 \in Gx^0 + Tx^0$, where $\tilde{w}^0 = \begin{bmatrix}
	g^0 + v^0 \\ c^0 - u^0 
\end{bmatrix}$ with $c^0 \in C^{-1}v^0$.
If we choose $\hat{w}^0 = \begin{bmatrix}
	\tau g^0 \\ -v^0
\end{bmatrix}$, then the last condition implies $c^0 - u^0 = -\tau g^0 - \frac{1}{\sigma}v^0$, or equivalently, $u^0 - \tau g^0 = c^0 + \frac{1}{\sigma}v^0$.
Since $C = \mbf{E}^\top \Nc_{\sets{0}}\mbf{E}$, $\mbf{E}\mbf{E}^\top = \kappa \Id$, and $(\ker \mbf{E})^{\perp} = \mathrm{Im}(\mbf{E}^\top)$, the last condition is satisfied by the orthogonal decomposition:
\begin{equation*}
	\arraycolsep=0.2em
	\begin{array}{lcl}
		\frac{1}{\sigma}v^0 = \frac{1}{\kappa}\mbf{E}^\top \mbf{E}(u^0 - \tau g^0), \qquad c^0 = \big(\Id - \frac{1}{\kappa}\mbf{E}^\top \mbf{E}\big)(u^0 - \tau g^0).
	\end{array}
\end{equation*}
Hence, $v^0 = \frac{\sigma}{\kappa}\mbf{E}^\top \mbf{E}(u^0 - \tau g^0)$, where $u^0 - \tau g^0 = \begin{bmatrix}
	\mbf{\Gamma}^{-1}(\mbf{x}^0 - \tau \mbf{\Lambda}\mbf{w}_{\mbf{x}}^0) \\ 0
\end{bmatrix}$.
Since $\mbf{E} = \begin{bmatrix}
	\mbf{K}\mbf{\Gamma} & \mbf{M}
\end{bmatrix}$, we get $\mbf{E}(u^0 - \tau g^0) = \mbf{K}(\mbf{x}^0 - \tau \mbf{\Lambda}\mbf{w}_{\mbf{x}}^0)$, and thus the first block of $v^0 = (\mbf{b}^0, \breve{\mbf{b}}^0)$ is $\mbf{b}^0 = \frac{\sigma}{\kappa}\mbf{\Gamma}\mbf{K}^2(\mbf{x^0} - \tau \mbf{\Lambda}\mbf{w}_{\mbf{x}}^0)$.
Using the transformation in Appendix~\ref{apdx:DFKM_NI_scheme_init}, we get $\mbf{v}^0 = \mbf{\Gamma}\mbf{b}^0 = \frac{\sigma}{\kappa}\mbf{\Lambda}\mbf{K}^2(\mbf{x}^0 - \tau \mbf{\Lambda}\mbf{w}_{\mbf{x}}^0)$.
Since $\mbf{K}^2 = \frac{\Id - \mbf{W}}{2}$, we get  $\mbf{v}^0 = \frac{\sigma}{2\kappa}\mbf{\Lambda}(\Id - \mbf{W})(\mbf{x}^0 - \tau \mbf{\Lambda}\mbf{w}_{\mbf{x}}^0)$.
Since we choose $z^0 = x^0$ in \eqref{eq:FKM_scheme}, we eventually get $\mbf{s}_{\mbf{x}}^0 = \mbf{x}^0$ and $\mbf{s}_{\mbf{v}}^0 = \mbf{v}^0$.

Next, since the $\mbf{a}$-block of $\tilde{w}^0$ is $\tilde{w}_{\mbf{a}}^0 = \mbf{\Gamma}\mbf{w}_{\mbf{x}}^0 + \mbf{b}^0$, we have $\mbf{r}_{\mbf{x}}^0 = \mbf{\Gamma}\tilde{w}_{\mbf{a}}^0 = \mbf{\Lambda}\mbf{w}_{\mbf{x}}^0 + \mbf{v}^0$.

Finally, since the $\mbf{b}$-block of $\tilde{w}^0$ is $\tilde{w}_{\mbf{b}}^0 = -\tau \mbf{\Gamma} \mbf{w}_{\mbf{x}}^0 - \frac{1}{\sigma}\mbf{b}^0$, we have $\mbf{r}_{\mbf{v}}^0 = \mbf{\Gamma}\tilde{w}_{\mbf{b}}^0 = -\tau \mbf{\Lambda}\mbf{w}_{\mbf{x}}^0 - \frac{1}{\sigma}\mbf{v}^0$.

\bibliographystyle{plain}

\end{document}